\documentclass[11pt]{amsart}
\usepackage{etex}
\usepackage[all]{xy}
\usepackage{amssymb,amscd,amsmath,amsfonts,amsthm}
\usepackage{stmaryrd,mathabx,tikz,tikz-cd,mathrsfs,float,enumitem}
\usepackage{ulem}
\usepackage{graphpap, color}

\usepackage{pstricks}
\usepackage{cancel}

\usepackage[verbose]{hyperref}

\numberwithin{equation}{section}

\newcommand{\rd}{\operatorname{red}}

\newcommand{\A }{\mathbb{A}}

\newcommand{\PP}{\mathbb{P}}

\newcommand{\ZZ}{\mathbb{Z}}

\newcommand{\bR}{\mathbf{R}}
\newcommand{\bT}{\mathbf{T}}

\newcommand{\bk}{\mathbf{k}}

\newcommand{\bw}{\mathbf{w}}

\newcommand{\kk}{\bk}

\def\bH{\mathbf{H}}

\newcommand{\cal}{\mathcal}

\def\cA{{\cal A}}

\def\cB{{\cal B}}
\def\cC{{\cal C}}
\def\cD{{\cal D}}

\def\cE{{\cal E}}

\def\cH{{\cal H}}
\def\cK{{\cal K}}

\def\cM{{\cal M}}
\def\cN{{\cal N}}

\def\cQ{{\cal Q}}
\def\cR{{\cal R}}

\def\cU{{\cal U}}
\def\cV{{\cal V}}
\def\cW{{\cal W}}

\def\cX{{\cal X}}
\def\cY{{\cal Y}}
\def\cZ{{\cal Z}}
\def\cG{{\cal G}}

\def\sE{{\mathscr E}}

\def\sH{{\mathscr H}}

\def\sL{{\mathscr L}}
\def\sM{{\mathscr M}}

\def\sO{{\mathscr O}}

\def\fE{\mathfrak{E}}
\def\fF{\mathfrak{F}}

\def\fM{\mathfrak{M}}
\def\2M{M}

\def\fN{\mathfrak{N}}

\def\fV{\mathfrak{V}}

\def\ff{\mathfrak{f}}

\def\fn{\mathfrak n}
\def\fp{\mathfrak{p}}

\def\mapright#1{\,\smash{\mathop{\lra}\limits^{#1}}\,}

\def\fdd{\fM_2^{\rm div}}

\def\dual{^{\vee}}

\def\sta{^\ast}
\def\st{^{\mathrm{st}}}

\def\upmo{^{-1}}
\def\sta{^{\ast}}

\def\sta{^*}

\def\lra{\longrightarrow}

\def\lsta{_{\ast}}

\newcommand{\Si}{\Sigma}
\newcommand{\Ga}{\Gamma}

\newcommand{\si}{\sigma}

\def\begeq{\begin{equation}}
\def\endeq{\end{equation}}
\def\and{\quad{\rm and}\quad}
\def\bl{\bigl(}
\def\br{\bigr)}
\def\defeq{:=}

\def\sub{\subset}
\def\Ao{{\mathbb A}^{\!1}}

\def\and{\quad\text{and}\quad}
\def\mapright#1{\,\smash{\mathop{\lra}\limits^{#1}}\,}

\def\lalp{_\alpha}

\DeclareMathOperator{\pr}{pr}

\DeclareMathOperator{\rank}{rank}

\newtheorem{prop}{Proposition}[section]
\newtheorem{theo}[prop]{Theorem}
\newtheorem{lemm}[prop]{Lemma}
\newtheorem{coro}[prop]{Corollary}
\newtheorem{defi}[prop]{Definition}

\newtheorem{assu}[prop]{Assumption}

\theoremstyle{definition}

\newtheorem{rema}[prop]{Remark}
\newtheorem{exam}[prop]{Example}

\def\Po{{\mathbb P^1}}

\def\Pn{{\mathbb P}^n}
\def\PP{{\mathbb P}}

\def\AA{{\mathbb A}}

\def\MPdd{\overline{M}_2(\Pn,d)}

\def\fM{\mathfrak{M}}

\def\sta{^\ast}

\def\beq{\begin{equation}}
\def\eeq{\end{equation}}

\def\vsp{\vskip5pt}

\let\ga=\gamma

\def\fN{{\mathfrak N}}

\def\sss{\mathrm{ss}}
\def\bcD{\overline{\cD}}
\def\bcC{\overline{\cC}}
\def\bcA{\overline{\cA}}
\def\bcB{\overline{\cB}}
\def\barM{\overline{M}}

\def\fq{\mathfrak q}
\def\wv{\widetilde{\mathcal V}}

\def\wh#1{\widehat{#1}}
\def\wc#1{\widecheck{#1}}
\def\ti#1{\widetilde{#1}}
\def\ov#1{\overline{#1}}
\def\ud#1{\underline{#1}}

\def\eref#1{(\ref{#1})}

\def\tn{\textnormal}
\def\ts{\textsf}

\def\inn{\!\in\!}
\def\eq{\!=\!}
\def\bsl{\backslash}

\def\lr#1{\langle{#1}\rangle}

\def\lrbr#1{[{#1}]}

\def\mwt{\fM_2^{\tn{wt}}}
\def\mn{\tn{mn}}

\def\nM{\textnormal M}
\def\P{\mathbb P}
\def\fE{\mathfrak E}
\def\bE{\mathbf E}

\def\se{\mathsf e}

\def\fK{\mathfrak{K}}

\def\De{\Delta}
\def\Up{\Upsilon}
\def\Th{\Theta}

\def\al{\alpha}
\def\be{\beta}
\def\ga{\gamma}
\def\de{\delta}

\def\ka{\kappa}
\def\la{\lambda}

\def\si{\sigma}
\def\th{\theta}
\def\ve{\varepsilon}

\def\vt{\vartheta}

\def\u{\upsilon}
\def\ze{\zeta}

\def\ord{\rho}

\def\dt#1{\theta_{#1}}
\def\wdt#1{\widehat{\theta}_{#1}}

\def\hmr{\varphi_{\le r}}

\def\ohmr{{\overline\varphi}_{\le r}}

\def\rd{\mathsf r}
\def\ph{\mathsf p}
\def\st{\mathsf s}

\def\chk{\raisebox{0.3ex}{\scalebox{0.66}{$\vee$}}}

\def\Mw{{\fM_2^{\rm wt}}}
\def\Md{\fM_2^{\rm div}}

\def\wt{^{\rm wt}}

\def\ve{\varepsilon}

\def\fD{\mathfrak{I}}
\def\ex{\tn{gr}}
\def\cht{\mathbf{m}}
\def\im{\tn{i-m}}
\def\PT{\tn{PT}}

\def\tMPdd{\widetilde{M}_2(\Pn,d)}

\newcommand\undermat[2]{%
	\makebox[0pt][l]{$\smash{\underbrace{\phantom{%
					\begin{matrix}#2\end{matrix}}}_{\text{$#1$}}}$}#2}

\usetikzlibrary{arrows}
\usetikzlibrary{decorations.pathreplacing}
\usetikzlibrary{decorations.pathmorphing}
\usetikzlibrary{patterns}
\tikzcdset{arrow style=tikz, diagrams={>= stealth }}

\makeatletter
\newcommand{\doublewidetilde}[1]{{%
  \mathpalette\double@widetilde{#1}%
}}
\newcommand{\double@widetilde}[2]{%
  \sbox\z@{$\m@th#1\widetilde{#2}$}%
  \ht\z@=.9\ht\z@
  \widetilde{\box\z@}%
}
\makeatother

\title[Modular Resolution of Moduli Space of Stable Maps]
{Genus Two Stable Maps, Local Equations and Modular Resolutions}
\date{}
\author{Yi Hu}
\address{Department of Mathematics, University of Arizona, USA.}
\email{yhu@math.arizona.edu}
\author{Jun Li}
\address{Stanford University, USA.
\newline\indent
Current address:
Shanghai Center for Mathematical Sciences, China}
\email{lijun2210@fudan.edu.cn}

\author{Jingchen Niu}
\address{
Department of Mathematics, University of Arizona, USA.
\newline\indent
Current address:
Santakatu 12 G 109, Helsinki 00180, Finland}
\email{jingchen.niu@mail.huji.ac.il}

\begin{document}

\begin{abstract} 
We provide a geometric construction of a sequence of modular blowups
of the Artin stack parameterizing pre-stable pairs consisting of a  genus-two nodal curve and a smooth divisor. The resulting stack locally diagonalizes the tautological derived objects associated
with the moduli of stable maps from genus-two curves to projective space. As a consequence,
the singularities of the main component of the moduli space of stable maps are resolved, and the
entire space admits only normal crossing singularities. Our approach is expected to generalize to
higher genera.	
\end{abstract}

\maketitle
\setcounter{tocdepth}{2}
\tableofcontents

\section{Introduction}
\label{SecIntro}

Moduli problems are of central importance in algebraic geometry.
Among them, the moduli spaces $\barM_{g,k}(\PP^n,d)$ of 
genus $g$ degree $d$ stable maps into 
projective space with $k$ marked points
are particularly important. For example,
the Gromov-Witten theory is an intersection theory on these fundamental moduli spaces
(introduced and constructed in \cite{K, A, FP}).

This article focuses on the deep geometry of the moduli space $\barM_{g,k}(\PP^n,d)$.
By Murphy's law of Vakil (\cite{Vakil06}), 
the moduli space
$\barM_{g,k}(\PP^n,d)$ is arbitrarily singular: every singularity of finite type
over $\ZZ$ appears for some $n$, $g$, and $d$.
The resolution of singularity 
is arguably one of the hardest problems in algebraic geometry
(\cite{Hironaka64a, Hironaka64b, deJong96, ATW, McG}, etc.).
Several moduli problems 
offer platforms for geometric solutions,
and $\barM_{g,k}(\PP^n,d)$ is one of them  (and this is {\it the pivotal motivation} for 
one of us to study
the moduli  $\barM_{g,k}(\PP^n,d)$). 
A canonical resolution of  $\barM_{g,k}(\PP^n,d)$ is also desirable for applications to 
computing GW invariants,
even though these invariants can be introduced without using any resolution 
(\cite{BF} and \cite{LT}).
When the genus $g$ is $0$,
the moduli space $\barM_{0,k}(\PP^n,d)$ is smooth. {For $g\eq 1$,
the moduli space $\barM_{1,k}(\PP^n,d)$ is
singular and the resolution of the main component of $\barM_{1,k}(\PP^n,d)$ 
was constructed by Vakil and Zinger in \cite{VZ08}, followed by
an algebraic approach to 
 the entire moduli $\barM_{1}(\PP^n,d)$ (without marked points)} by Hu and Li \cite{HL10}.

Our current article begins to treat the 
 more difficult cases of general genera.
We start with
the moduli space $\barM_2(\PP^n,d)$;
for conciseness and without loss of generality in terms of singularity, 
we still work with  curves without marked points.
The goal of this paper is
to {\it geometrically} describe  a canonical sequence of blowups of $\barM_2(\PP^n,d)$
such that the resulted stack admits only normal crossing singularities 
and each of its irreducible components 
is smooth.
This is the best that one can hope for as far as smoothness is concerned.

In the subsequent works \cite{HN1,HN2}, the first and third named authors proved that the blowup stacks constructed in both \cite{HL10} and the present article admit modular interpretations: they parameterize certain geometric objects. This perspective has the advantage of being more amenable
to generalization to higher genus.

Although not a primary motivation of this series of works, 
the smooth blowup of $\barM_2(\PP^n,d)$ obtained in this paper can be used, for instance, 
to prove the ``hyperplane'' property of the Gromov-Witten invariants of quintic Calabi-Yau threefolds
conjectured in \cite{LZ} (along with
other possible calculations). 
Further, the method provided in this paper
may also be helpful for the resolution of other singular moduli used in the recent progresses in mirror symmetry.

\smallskip
To state our main results,
consider the smooth Artin stack $\fM_2^{\rm div}$ (written as $\mathfrak D_2$ in~\cite{HN2}) consisting  of pre-stable pairs
$\big(C, D\big)$ where~$C$  are connected genus 2 nodal curves 
and $D\!\subset\!C$ are simple divisors away from the nodes of $C$. 
A pair $(C,D)$ is \ts{pre-stable} if any smooth rational
component of $C$ missing the divisor $D$ has at least three nodes.

To resolve $\MPdd$, we construct a sequence of smooth blowups of $\fM_2^{\rm div}$, then pull back this sequence to $\MPdd$. The process would be more straightforward if there existed a natural global morphism from $\MPdd$ to $\fM_2^{\rm div}$. Fortunately, although such a global morphism is not
available, a rich supply of local morphisms suffices for our purpose.

For every $[C,u]\inn \ov M_2(\P^n,d)$, we have
$\big(C,u^{-1}(H)\big)\inn\fM_2^{\rm div}$ for a generic hyperplane $H\!\subset\!\P^n$.
Let $\breve{\P}^n\eq \tn{Gr}(n,n\!+\!1)$ be the space of the hyperplanes of $\P^n$.
Then, for every neighborhood $U\!\subset\!\ov M_2(\P^n,d)$ of $[C,u]$, we set
\begin{align}\label{Eqn:H}
	\bH_U:=
	\big\{\,H\inn\breve{\P}^n\,:\,
	\big(C', (u')^{-1}(H)\big)\inn\fM_2^{\rm div}~
	\forall~[C',u']\inn U\,
	\big\}\,,
\end{align}
which is nonempty as long as $U$ is small.
Each $H\inn \bH_U$ gives rise to a morphism 
\begin{align}\label{toP0}
	f_{U,H}:\,
	U\lra \fM_2^{\rm div}, \qquad
	[C,u]\mapsto \big(C, u^{-1}(H)\big).
\end{align}

We denote by 
\begin{align}
	\label{Eqn:univ_fami}
	\pi: \cX \lra \MPdd,\qquad 
	\ff: \cX \lra \Pn
\end{align}
the universal family of the moduli $\MPdd$.
The main technical goal of this article is to locally diagonalize the 
derived objects $\bR \pi_* \ff^* \sO_{\Pn}(k)$, $k\!\ge\!1$,
in the sense of Definition~\ref{dObject}.

\vsp\noindent
{\bf Theorem 1.}
{\it
There exist a  sequence of blowups $\widetilde \fM_2^{\rm div}\!\lra\!\fM_2^{\rm div}$,
a Deligne-Mumford stack $\ti M_2(\P^n,d)$,
and a morphism $\ti M_2(\P^n,d)\!\lra\!\ov M_2(\P^n,d)$
such that for every small open $U\!\subset\!\ov M_2(\P^n,d)$ and every $H\inn \bH_U$,
there exists a commutative diagram
\[
	\begin{tikzcd}[column sep=.5em,row sep=1.5em]
		U\times_{f_{U,H};\,\fM_2^{\rm div}}\ti\fM_2^{\rm div}  \arrow[dr,""] \arrow[rr,"\sim"] && U\times_{\ov M_2(\P^n,d)}\ti M_2(\P^n,d)
		\arrow[dl,""]
		\\
		& U 
	\end{tikzcd}
\]
where the horizontal arrow is an isomorphism (of stacks over $U$),
and the other two arrows are natural morphisms.

Moreover, the pullback of the derived object $\bR \pi_* \ff^* \sO_{\Pn}(1)$ to 
$\ti M_2(\P^n,d)$
becomes locally diagonalizable. 
For every integer $k\!>\!1$, 
the pullback of the derived object $\bR \pi_* \ff^* \sO_{\Pn}(k)$ to the main component $\ti M_2(\P^n,d)^{\rm mc}$ of 
$\ti M_2(\P^n,d)$
is also locally diagonalizable.}
\vsp

Here, the \ts{main components} $$\MPdd^{\rm mc}\subset \MPdd
\qquad\tn{and}\qquad\ti M_2(\P^n,d)^{\rm mc}\subset\ti M_2(\P^n,d)$$ are respectively the irreducible components whose general points
are stable maps with smooth domain curves.

The stack $\ti M_2(\P^n,d)$ is obtained by gluing $U\!\times_{f_{U,H};\,\fM_2^{\rm div}}\!\ti\fM_2^{\rm div}$ for all small open $U$ and $H\inn\bH_U$.
For the precise construction of $\ti M_2(\P^n,d)$,
see Corollary~\ref{Crl:Gluing}.

The case of Theorem 1 when $k\eq 1$ leads to the following conclusion.

\vsp\noindent
{\bf Corollary 2.} 
{\it 
$\widetilde{M}_2(\Pn,d)$
has smooth irreducible components and at worst normal crossing singularities. 
For $d\!>\!2$, the main component is of the expected dimension~$d(n\!+\!1)\!-\!n\!+\!3$.}
\vsp

Theorem 1 and Corollary 2 are proved in \S\ref{SecLocalEqns};
see Corollary~\ref{Crl:Gluing}, Proposition~\ref{Prp:Diag}, and  Theorem~\ref{inducedImmersion}.
Theorem~1 also gives rise to the following statement on the desingularization of the sheaves $\pi_\ast\ff^*\sO_{\Pn}(k)$.
It is proved in Corollary~\ref{Crl:Diag}.

\vsp\noindent
{\bf Corollary 3.}   
{\it
Let $(\ti\pi^{\rm mc}, \ti\ff^{\rm mc})$ be the pullback to $\ti M_2(\P^n,d)^{\rm mc}$ of the universal family $(\pi, \ff)$. Then the direct image sheaf $(\ti\pi^{\rm mc})_{\ast} (\ti\ff^{\rm mc})^*
\sO_{\Pn}(k)$  is locally free for all $k\!\ge\! 1$.
It is of rank $kd\!-\!1$ if $d\!>\!2$.}

\vsp 
As a by-product of the construction of $\ti\fM_2^{\rm div}$,
in Proposition~\ref{Prp:Diag_k},
we construct a Deligne-Mumford stack $\ti M_2(\P^n,d;k)$ for every integer $k\!>\!1$ in a manner similar to $\ti M_2(\P^n,d)$.
The stacks $\ti M_2(\P^n,d;k)$, albeit less smooth than $\ti M_2(\P^n,d)$,
behaves nicer as far as the local diagonalizability of the pullback of the derived object $\bR \pi_* \ff^* \sO_{\Pn}(k)$ is concerned.
The following statement is a recapitulation of Propositions~\ref{Prp:Diag_k}.

\vsp\noindent
{\bf Theorem 4.}
{\it
For every integer $k\!>\!1$,
the pullback of  $\bR \pi_* \ff^* \sO_{\Pn}(k)$ to 
$\ti M_2(\P^n,d;k)$
becomes locally diagonalizable. 
Moreover,
for every irreducible component $N$ of $\ti M_2(\P^n,d;k)$,
with $(\ti\pi_N, \ti\ff_N)$ denoting the pullback to $N$ of the universal family (\ref{Eqn:univ_fami}), the direct image sheaf $(\ti\pi_N)_{\ast} (\ti\ff_N)^*
\sO_{\Pn}(k)$  is locally free.
It is of rank $kd\!-\!1$ if $d\!>\!2$ and the general points of $N$ have smooth domain curves.
}

\vsp 
Our approach to $\ti\fM_2^{\rm div}/\fdd$ improves the technique of \cite{HL10} and
combines with the idea of the derived modular. To explain this,
notice the stack $\fdd$ is equipped with a universal curve $\rho: \cC \to \fdd$ as well as a universal divisor $\cD\!\subset\!\cC$.
We study the parallel diagonalization problem of
the derived object $\rho_* \sO_\cC(\cD)$.
Locally over a smooth chart $\cV$ of $\fdd$, the object $\rho_* \sO_\cC(\cD)|_\cV$ can be represented by a
two-term structural homomorphism
\beq\label{PHI}
\varphi: \sO_\cV^{\oplus m+1} \lra \sO_\cV^{\oplus 2},
\eeq
where $m$ is the degree of $\cD$. 
{\it Throughout this article, $m$ can be any positive integer;
it corresponds to the case of $\pi_*\ff^* \sO_{\Pn}(k)$ when $m=kd$.}

We provide a detailed
analysis of this homomorphism in terms of local modular parameters 
in \S\ref{SecPhi} (and later in \S\ref{SecChangeOfPhi}). 
The analysis of $\varphi$,
summarized in Proposition~\ref{Prp:phi_key}, guides our construction of the modular blowups 
in \S \ref{global}.
The goal is to find global {\it smooth} blowups of $\fdd$ such that all the pullback of (\ref{PHI})
can be diagonalized everywhere, locally.
As proved in \S \ref{SecLocalEqns}, this implies Theorem 1 as well as Corollaries 2 and 3.
The procedure involves three \ts{rounds},
the first and the last of which further contain several \ts{phases}.
When the degree $d$ is fixed,
each round or phase consists of finitely many \ts{steps}.
For conciseness,
we hereafter write
\begin{equation}\label{e_rps}
 \rd_i,\quad
 \ph_j,\quad
 \st_\ell
\end{equation}
respectively for the $i$-th round, the $j$-th phase, and the $\ell$-th step whenever applicable.

We point out here that in order to locally diagonalize $\bR \pi_* \ff^* \sO_{\Pn}(k)$  with $k\!\ge\!2$,
 $(\rd_1\ph_5)$ and $(\rd_3\ph_4)$ can be omitted;
to locally diagonalize $\bR \pi_* \ff^* \sO_{\Pn}(k)$  with $k\!\ge\!3$,
 $(\rd_3\ph_3)$  can be omitted as well.

Many sequences of our blowups are performed in the Artin stack 
$\Mw$ of  stable weighted curves of genus 2 (i.e.~genus 2 nodal curves whose
irreducible components are decorated by weights in~$\ZZ_{\ge 0}$, satisfying certain stability condition; c.f.~\S\ref{sheafStructures}), 
which comes equipped with the morphisms
\begin{align}\label{MPddToWeight0}
&\fdd \lra  \fM_2^{\rm wt}, \qquad(C,D) \mapsto \big(C, c_1(D)\big),\nonumber
\\
&\varrho: \MPdd \lra \fM_2^{\rm wt}, \qquad
[C,u]\mapsto \big(C, c_1(u^*\sO_{\Pn}(1))\big),
\end{align}
where for every small open $U\!\subset\!\MPdd$ and $H\inn\bH_U$, the composition of the natural inclusion $U\!\to\!\MPdd$ and the latter morphism factors through~\eref{toP0} and the former.
Each blowup $\ti\fM'\!\to\!\mwt$ determines the blowup $\fdd\!\times_{\mwt}\ti\fM'$ of~$\fdd$.

\vsp
\noindent
{\it The first round of blowup $(\rd_1)$. }  
\vsp
We first blow up $\mwt$ successively along the proper transforms of the substacks that are illustrated in Figure~\ref{figBlowup1} and precisely described in \S\ref{rd1}.
This round consists of five phases of sequential blowups: the phases are classified by 
the types of the \ts{cores} of curves (i.e.~the smallest subcurve of genus 2;
c.f.~Definition \ref{Dfn:nodal_curve})
in the corresponding centers
as well as the distribution of weights (see~Figure~\ref{figBlowup1}). 
Within each given phase, the blowup centers
are classified by the number of rational tails attached to the cores (also see Figure~\ref{figBlowup1}).

\begin{figure}[htp]
\begin{center}
\begin{tikzpicture}[scale=0.28]
 \def\1a{
       (0,0.8)..controls (-0.9,0.8) and (-1.5,0.36)..
       (-1.5,0)..controls (-1.5,-0.36) and (-0.9,-0.8)..
       (0,-0.8)..controls (0.96,-0.8) and (1.36,-0.42)..
       (1.36,-0.36)..controls (1.36,-0.3) and (1.3,-0.2)..
       (1.2,-0.2)..controls (1.1,-0.2) and (0.96,-0.36)..
       (0.76,-0.36)..controls (0.56,-0.36) and (0.3,-0.2)..
       (0.3,0)..controls (0.3,0.2) and (0.56,0.36)..
       (0.76,0.36)..controls (0.96,0.36) and (1.1,0.2)..
       (1.2,0.2)..controls (1.3,0.2) and (1.36,0.3)..
       (1.36,0.36)..controls (1.36,0.42) and (0.96,0.8)..
       (0,0.8)
       (-0.88,0)..controls (-0.64,0.12)..(-0.4,0)
       (-0.4,0)..controls (-0.64,-0.12)..(-0.88,0)
       (-1.04,0.08)--(-0.88,0)
       (-0.4,0)--(-0.24,0.08)
       }
 \draw[xshift=0,yshift=-12cm] \1a;
 \draw[xshift=1.2cm,yshift=-12cm,fill=black!42] (0.31,0) circle (0.3cm);
       (0,1.2cm) circle (0.4cm);
 
 \draw[xshift=8cm,yshift=-12cm] \1a;
 \draw[xshift=9.2cm,yshift=-12cm,fill=black!42] (0.31,0) circle (0.3cm);
 \draw[xshift=8cm,yshift=-12cm,fill=black!42]
       (0,1.2cm) circle (0.4cm);

 \draw[xshift=16cm,yshift=-12cm] \1a;
 \draw[xshift=17.2cm,yshift=-12cm,fill=black!42] (0.31,0) circle (0.3cm);
 \draw[xshift=16cm,yshift=-12cm,fill=black!42]
       (-.6,1.16cm) circle (0.4cm)
       (0.6,1.16cm) circle (0.4cm);

 \def\g1{
       (-1,0) ellipse (1 and 0.5)
       (-1.4,0)..controls(-1,0.1)..(-0.6,0)
       (-0.6,0)..controls(-1,-0.1)..(-1.4,0)
       (-0.5,0.05)--(-0.6,0)
       (-1.4,0)--(-1.5,0.05)
       }
 \draw[xshift=.5cm, yshift=-18cm] \g1;
 \draw[xshift=2.5cm, yshift=-18cm,fill=black!42] \g1;
 
 \draw[xshift=8.5cm, yshift=-18cm] \g1;
 \draw[xshift=8.5cm, yshift=-18cm, fill=black!42] (-1cm,0.9cm) circle (0.4cm);
 \draw[xshift=10.5cm, yshift=-18cm,fill=black!42] \g1;
 
 \draw[xshift=16.5cm, yshift=-18cm] \g1;
 \draw[xshift=16.5cm, yshift=-18cm, fill=black!42]
 (-1.5cm,0.87cm) circle (0.4cm)
 (-0.5cm,0.87cm) circle (0.4cm);
 \draw[xshift=18.5cm, yshift=-18cm,fill=black!42] \g1;
 
 \draw[yshift=-6cm] \g1;
 \draw[xshift=-.1cm,yshift=-6cm,fill=black!42] (0.5cm,0) circle (0.4cm);
 \draw[xshift=2.8cm,yshift=-6cm] \g1;
 
 \draw[xshift=8cm, yshift=-6cm] \g1;
 \draw[xshift=7.9cm,yshift=-6cm,fill=black!42] (0.5cm,0) circle (0.4cm);
 \draw[xshift=8cm, yshift=-6cm,fill=black!42] (-1cm,0.9cm) circle (0.4cm);
 \draw[xshift=10.8cm,yshift=-6cm] \g1;

 \draw[xshift=16cm, yshift=-6cm] \g1;
 \draw[xshift=15.9cm,yshift=-6cm,fill=black!42] (0.5cm,0) circle (0.4cm);
 \draw[xshift=16cm, yshift=-6cm,fill=black!42] (-1.5cm,0.87cm) circle (0.4cm)
 (-.5cm,0.87cm) circle (0.4cm);
 \draw[xshift=18.8cm,yshift=-6cm] \g1; 
 
 \draw[xshift=22.7cm, yshift=-6cm] \g1;
 \draw[xshift=22.6cm,yshift=-6cm,fill=black!42] (0.5cm,0) circle (0.4cm);
 \draw[xshift=22.7cm, yshift=-6cm,fill=black!42] (-1cm,0.9cm) circle (0.4cm)
 (1.8cm,0.9cm) circle (0.4cm);
 \draw[xshift=25.5cm,yshift=-6cm] \g1;
 
 \def\g2left{
       (0,0.8) arc (90:270:1.6 and 0.8)
       (-1.04,0.08)--(-0.88,0)
       ..controls (-0.64,-0.12)..(-0.4,0)
       --(-0.24,0.08)
       (-0.88,0)..controls (-0.64,0.12)..(-0.4,0)
       }
 \draw \g2left
       [xscale=-1] \g2left;
 \draw[fill=black!42] (0,1.2cm) circle (0.4cm);
 
 \draw [xshift=8cm]\g2left
       [xscale=-1] \g2left;
 \draw[fill=black!42, xshift=8cm]
  (-.8cm,1.13cm) circle (0.4cm)
  (.8cm,1.13cm) circle (0.4cm);
 
 \draw [xshift=16cm]\g2left
       [xscale=-1] \g2left;
 \draw[fill=black!42, xshift=16cm]
  (0,1.2cm) circle (0.4cm)
  (-1cm,1.08cm) circle (0.4cm)
  (1cm,1.08cm) circle (0.4cm);
 
 \filldraw [xshift=28.5cm]
  (0,0) circle (2pt) 
  (.5,0) circle (2pt) 
  (1,0) circle (2pt);
  
 \filldraw [xshift=28.5cm, yshift=-6cm]
  (0,0) circle (2pt) 
  (.5,0) circle (2pt) 
  (1,0) circle (2pt);
  
 \filldraw [xshift=28.5cm, yshift=-12cm]
  (0,0) circle (2pt) 
  (.5,0) circle (2pt) 
  (1,0) circle (2pt);
 
 \filldraw [xshift=28.5cm, yshift=-18cm]
  (0,0) circle (2pt) 
  (.5,0) circle (2pt) 
  (1,0) circle (2pt);
  
 \draw[->,>=stealth, very thick]
  (3,0)->(5.1,0);
 \draw[->,>=stealth, very thick]
  (11,0)->(13.1,0);
 \draw[->,>=stealth, very thick]
  (18.9,0)->(27.4,0);
 
 \draw[->,>=stealth, very thick, yshift=-6cm]
  (4,0)->(5,0);
 \draw[->,>=stealth, very thick, yshift=-6cm]
  (12,0)->(13,0);
 \draw[->,>=stealth, very thick, yshift=-6cm]
  (26.6,0)->(27.4,0);
  
 \draw[->,>=stealth, very thick, yshift=-12cm]
  (3,0)->(5.3,0);
 \draw[->,>=stealth, very thick, yshift=-12cm]
  (11,0)->(13.3,0);
 \draw[->,>=stealth, very thick, yshift=-12cm]
  (19,0)->(27.4,0);
 
 \draw[->,>=stealth, very thick, yshift=-18cm]
  (3.8,0)->(5.2,0);
 \draw[->,>=stealth, very thick, yshift=-18cm]
  (11.8,0)->(13.2,0);
 \draw[->,>=stealth, very thick, yshift=-18cm]
  (19.8,0)->(27.4,0);
 
 \draw[->,>=stealth,very thick, rounded corners]
  (30.7,0)--(31.4,0)--(31.4,-3.1)--(-3.9,-3.1)--(-3.9,-6)--(-3,-6);
 \draw[->,>=stealth,very thick, yshift=-6cm, rounded corners]
  (30.7,0)--(31.4,0)--(31.4,-3.1)--(-3.9,-3.1)--(-3.9,-6)--(-2.7,-6);
 \draw[->,>=stealth,very thick, yshift=-12cm, rounded corners]
  (30.7,0)--(31.4,0)--(31.4,-3.1)--(-3.9,-3.1)--(-3.9,-6)--(-2.7,-6);
 
 \draw[yshift=-0.8cm]
  node[below] {\scriptsize{$\ov\fM_{(1,1)}$}};
 \draw[xshift=8cm,yshift=-0.8cm]
  node[below] {\scriptsize{$\ov\fM_{(1,2)}$}};
 \draw[xshift=16cm,yshift=-0.8cm]
  node[below] {\scriptsize{$\ov\fM_{(1,3)}$}};
  
  \draw[decorate,decoration=brace]
   (-1.9,-0.7)--(-1.9,1.4);
  \draw[decorate,decoration=brace]
   (1.9,1.4)--(1.9,-.7);
  \draw (-1.9,1.9)--(1.9,1.9);
  
  \draw[xshift=8cm,decorate,decoration=brace]
   (-1.9,-0.7)--(-1.9,1.4);
  \draw[xshift=8cm,decorate,decoration=brace]
   (1.9,1.4)--(1.9,-.7);
  \draw[xshift=8cm] (-1.9,1.9)--(1.9,1.9);
  
  \draw[xshift=16cm,decorate,decoration=brace]
   (-1.9,-0.7)--(-1.9,1.4);
  \draw[xshift=16cm,decorate,decoration=brace]
   (1.9,1.4)--(1.9,-.7);
  \draw[xshift=16cm] (-1.9,1.9)--(1.9,1.9);
  
 \draw[xshift=0.5cm, yshift=-18.6cm]
  node[below] {\scriptsize{$\ov\fM_{(4,1)}$}};
 \draw[xshift=8.5cm, yshift=-18.6cm]
  node[below] {\scriptsize{$\ov\fM_{(4,2)}$}};
 \draw[xshift=16.5cm, yshift=-18.6cm]
  node[below] {\scriptsize{$\ov\fM_{(4,3)}$}};  
  
  \draw[decorate,decoration=brace,xshift=0.5cm,yshift=-18.6cm]
   (-2.3,0)--(-2.3,1.2);
  \draw[decorate,decoration=brace,xshift=0.5cm,yshift=-18.6cm]
   (2.3,1.2)--(2.3,0);
  \draw[xshift=0.5cm,yshift=-18.6cm] (-2.3,1.6)--(2.3,1.6);
  
  \draw[decorate,decoration=brace,xshift=8.5cm,yshift=-18.6cm]
   (-2.3,0)--(-2.3,1.9);
  \draw[decorate,decoration=brace,xshift=8.5cm,yshift=-18.6cm]
   (2.3,1.9)--(2.3,0);
  \draw[xshift=8.5cm,yshift=-18.6cm] (-2.3,2.2)--(2.3,2.2);
  
  \draw[decorate,decoration=brace,xshift=16.5cm,yshift=-18.6cm]
   (-2.3,0)--(-2.3,1.9);
  \draw[decorate,decoration=brace,xshift=16.5cm,yshift=-18.6cm]
   (2.3,1.9)--(2.3,0);
  \draw[xshift=16.5cm,yshift=-18.6cm] (-2.3,2.2)--(2.3,2.2);
  
 \draw[xshift=0, yshift=-12.8cm]
  node[below] {\scriptsize{$\ov\fM_{(3,1)}$}};
 \draw[xshift=8cm, yshift=-12.8cm]
  node[below] {\scriptsize{$\ov\fM_{(3,2)}$}};
 \draw[xshift=16cm, yshift=-12.8cm]
  node[below] {\scriptsize{$\ov\fM_{(3,3)}$}};
  
  \draw[decorate,decoration=brace,xshift=0cm,yshift=-12.8cm]
   (-1.8,.1)--(-1.8,1.6);
  \draw[decorate,decoration=brace,xshift=0cm,yshift=-12.8cm]
   (2,1.6)--(2,.1);
  \draw[xshift=0cm,yshift=-12.8cm] (-1.8,2)--(2,2);
  
  \draw[decorate,decoration=brace,xshift=8cm,yshift=-12.8cm]
   (-1.8,.1)--(-1.8,2.3);
  \draw[decorate,decoration=brace,xshift=8cm,yshift=-12.8cm]
   (2,2.3)--(2,.1);
  \draw[xshift=8cm,yshift=-12.8cm] (-1.8,2.7)--(2,2.7);
  
  \draw[decorate,decoration=brace,xshift=16cm,yshift=-12.8cm]
   (-1.8,.1)--(-1.8,2.3);
  \draw[decorate,decoration=brace,xshift=16cm,yshift=-12.8cm]
   (2,2.3)--(2,.1);
  \draw[xshift=16cm,yshift=-12.8cm] (-1.8,2.7)--(2,2.7);
  
 \draw[xshift=.4cm, yshift=-6.6cm]
  node[below] {\scriptsize{$\ov\fM_{(2,1)}$}};
 \draw[xshift=8.4cm, yshift=-6.6cm]
  node[below] {\scriptsize{$\ov\fM_{(2,2)}$}};
 \draw[xshift=19cm, yshift=-6.6cm]
  node[below] {\scriptsize{$\ov\fM_{(2,3)}$}};
 \draw[xshift=19.77cm, yshift=-5.8cm]
  node {\scriptsize$\cup$};
  
  \draw[decorate,decoration=brace,xshift=.4cm,yshift=-6.6cm]
   (-2.6,0)--(-2.6,1.2);
  \draw[decorate,decoration=brace,xshift=.4cm,yshift=-6.6cm]
   (2.6,1.2)--(2.6,0);
  \draw[xshift=.4cm,yshift=-6.6cm] (-2.6,1.6)--(2.6,1.6);
  
  \draw[decorate,decoration=brace,xshift=8.4cm,yshift=-6.6cm]
   (-2.6,0)--(-2.6,1.9);
  \draw[decorate,decoration=brace,xshift=8.4cm,yshift=-6.6cm]
   (2.6,1.9)--(2.6,0);
  \draw[xshift=8.4cm,yshift=-6.6cm] (-2.6,2.2)--(2.6,2.2);
  
  \draw[decorate,decoration=brace,xshift=19cm,yshift=-6.6cm]
   (-5.2,0)--(-5.2,1.9);
  \draw[decorate,decoration=brace,xshift=19cm,yshift=-6.6cm]
   (6.7,1.9)--(6.7,0);
  \draw[decorate,decoration=brace,xshift=19cm,yshift=-6.6cm]
   (1.6,0)--(1.6,1.9);
  \draw[decorate,decoration=brace,xshift=19cm,yshift=-6.6cm]
   (-.05,1.9)--(-.05,0);
  \draw[xshift=19cm,yshift=-6.6cm] (-5.2,2.2)--(-.05,2.2);
  \draw[xshift=19cm,yshift=-6.6cm] (1.6,2.2)--(6.7,2.2);
  
 \draw (-7.5,0) node{\small\ts{Phase 1:}};
 \draw (-7.5,-1.2) node{\small{($\rd_1\ph_1$)}};
 \draw (-7.5,-6) node{\small\ts{Phase 2:}};
 \draw (-7.5,-7.2) node{\small{($\rd_1\ph_2$)}};
 \draw (-7.5,-12) node{\small\ts{Phase 3:}};
 \draw (-7.5,-13.2) node{\small{($\rd_1\ph_3$)}};
 \draw (-7.5,-18) node{\small\ts{Phase 4:}};
 \draw (-7.5,-19.2) node{\small{($\rd_1\ph_4$)}};
 
 \draw (-7.5,-24) node{\small\ts{Phase 5:}};
 \draw (-7.5,-25.2) node{\small{($\rd_1\ph_5$)}};
  \def\g2{
       (0,0.8) arc (90:450:1.6 and 0.8)
       (0.4,0)..controls (0.64,0.12)..(0.88,0)
       (0.88,0)..controls (0.64,-0.12)..(0.4,0)       
       (0.24,0.08)--(0.4,0)
       (0.88,0)--(1.04,0.08)
       (-0.4,0)..controls (-0.64,-0.12)..(-0.88,0)
       (-0.88,0)..controls (-0.64,0.12)..(-0.4,0)       
       (-0.24,0.08)--(-0.4,0)
       (-0.88,0)--(-1.04,0.08)
       }
  \shadedraw [shading=radial,yshift=-24cm] \g2;
  \draw[yshift=-24cm,fill=black!42] (0,1.2cm) circle (0.4cm);
  
  \shadedraw [shading=radial,xshift=8cm,yshift=-24cm] \g2;
  \draw[yshift=-24cm,xshift=8cm,fill=black!42]
  (-.8cm,1.13cm) circle (0.4cm)
  (.8cm,1.13cm) circle (0.4cm);
 
 \shadedraw [shading=radial,xshift=16cm,yshift=-24cm] \g2;  
 \draw[fill=black!42,yshift=-24cm,xshift=16cm]
  (0,1.2cm) circle (0.4cm)
  (-1cm,1.08cm) circle (0.4cm)
  (1cm,1.08cm) circle (0.4cm);
  
 \draw[yshift=-24.8cm]
  node[below] {\scriptsize{$\ov\fM_{(5,1)}$}};
 \draw[xshift=8cm,yshift=-24.8cm]
  node[below] {\scriptsize{$\ov\fM_{(5,2)}$}};
 \draw[xshift=16cm,yshift=-24.8cm]
  node[below] {\scriptsize{$\ov\fM_{(5,3)}$}};
    
  \draw[decorate,decoration=brace,yshift=-24cm]
   (-1.9,-0.7)--(-1.9,1.4);
  \draw[decorate,decoration=brace,yshift=-24cm]
   (1.9,1.4)--(1.9,-.7);
  \draw[yshift=-24cm] (-1.9,1.9)--(1.9,1.9);
  
  \draw[xshift=8cm,decorate,decoration=brace,yshift=-24cm]
   (-1.9,-0.7)--(-1.9,1.4);
  \draw[xshift=8cm,decorate,decoration=brace,yshift=-24cm]
   (1.9,1.4)--(1.9,-.7);
  \draw[xshift=8cm,yshift=-24cm] (-1.9,1.9)--(1.9,1.9);
  
  \draw[xshift=16cm,decorate,decoration=brace,yshift=-24cm]
   (-1.9,-0.7)--(-1.9,1.4);
  \draw[xshift=16cm,decorate,decoration=brace,yshift=-24cm]
   (1.9,1.4)--(1.9,-.7);
  \draw[xshift=16cm,yshift=-24cm] (-1.9,1.9)--(1.9,1.9);
  
 \draw[->,>=stealth, very thick,yshift=-24cm]
  (3,0)->(5.1,0);
 \draw[->,>=stealth, very thick,yshift=-24cm]
  (11,0)->(13.1,0);
 \draw[->,>=stealth, very thick,yshift=-24cm]
  (18.9,0)->(27.4,0);
 
 \draw[->,>=stealth,very thick, yshift=-18cm, rounded corners]
  (30.7,0)--(31.4,0)--(31.4,-3.1)--(-3.9,-3.1)--(-3.9,-6)--(-2.7,-6);
 
 \filldraw [xshift=28.5cm,yshift=-24cm]
  (0,0) circle (2pt) 
  (.5,0) circle (2pt) 
  (1,0) circle (2pt);
 
 \draw [xshift=12cm, yshift=-28cm] 
 (0,0) rectangle (1,.6)
 (3.5,.3) node {\tiny{:\;$\tn{weight}\!=\!0$}}; 
 \shadedraw [shading=radial,xshift=19cm, yshift=-28cm] 
 (0,0) rectangle (1,.6);
 \draw [xshift=19cm, yshift=-28cm] 
 (3.5,.3) node {\tiny{:\;$\tn{weight}\!=\!1$}};
 \draw [fill=black!42,xshift=26cm, yshift=-28cm] 
 (0,0) rectangle (1,.6)
 (3.5,.3) node {\tiny{:\;$\tn{weight}\!\ge\!1$}};
\end{tikzpicture}
\end{center}
\caption{The first round ($\rd_1$) of the modular blowups}\label{figBlowup1}
\end{figure}

Here, we comment that in $(\rd_1\ph_5$), we treat the case when the core of a curve has weight 1,
even though it does not correspond to any point on the main 
component $\MPdd^{\rm mc}$. 
The reason is twofold.
First,
the structural homomorphism (\ref{PHI}) is not diagonalized near the corresponding points of $\fdd$.
Since we have reduced the problem of resolving $\MPdd$ to locally diagonalizing (\ref{PHI}),
we shall not specify which points of $\fdd$
miss the image of $\MPdd$ or its main component under
(\ref{toP0}).
Second,
the underlying weighted curves of the blowup center of $(\rd_1\ph_5)$ 
correspond to some boundary 
components of $\MPdd$.
As presented in Corollary 2,
we aim to resolve the entire moduli so that all the components are smooth and meet transversely.

When $(\rd_1)$ terminates,
we denote the final stack by $\ti\fM^{\rd_1}$.
Every point of $\Mw$ has a neighborhood~$\cV$ so that the
pullback $\ti\varphi^{\rd_1}$ of the structural homomorphism 
$\varphi$ of~\eqref{PHI} to $\ti\cV^{\rd_1}=\cV\times_{\Mw}\ti\fM^{\rd_1}$ 
has one row that contains an element that divides all other entries of $\ti\varphi^{\rd_1}$
in that row (in such a case, we say that this row is  locally $``$diagonalized$"$).
However, $\ti\varphi^{\rd_1}|_{\ti\cV^{\rd_1}}$ may
{\it not} be locally diagonalizable (c.f.~Definition~\ref{DfnDiag}) due to the existence of 
\noindent (1): some distinctive directions in the exceptional divisors obtained in $(\rd_1)$ and/or
\noindent (2): the Weierstrass and conjugate points on the smallest  subcurves of genus two.

The second round of blowups, $(\rd_2)$, solves the former case;
the third round of blowups, $(\rd_3)$, solves the latter.

\vsp
\noindent
{\it The second round of blowups $(\rd_2)$.} 
\vsp
After the first round  is done,
one of the two rows, say, the first row,  of  the structural homomorphism 
$\varphi$ can be  locally $``$diagonalized$"$, as explained above. We then
 turn our attention to the second row. But, in the second row, on any affine chart,
some  terms of the proper transform of $\varphi$  may acquire exceptional parameters.
Note that locally, the blowup centers in the first rounds are always some intersections of the (local) divisors
corresponding to the loci where  nodes are not smoothed (we may call them \ts{nodal divisors}). 
The blowup centers in the second rounds are similar in the sense
that, locally, they are always some intersections of the proper transforms of
the exceptional divisors created in ($\rd_1\ph_1$) as well as
the proper transforms of nodal divisors.
In particular, the blowup centers $X_{i,j}$ in the second rounds are always contained in the exceptional
divisors of ($\rd_1\ph_1$), as  illustrated in Figure~\ref{figBlowup2}.
But, locally, the process in $(\rd_2)$ is parallel to that in $(\rd_1)$;
the difference is just that some exceptional parameters in $(\rd_2)$ will take the roles
of node-smoothing parameters in $(\rd_1)$. 
In this narrow sense, we may say that
the second round of blowups in essence is totally analogous to the first round.
The blowup centers $X_{i,j}$ in $(\rd_2)$ are defined in \S\ref{rd2},  specifically in \eqref{e_Xkk'}. 

\begin{figure}[htp]
\begin{center}
\begin{tikzpicture}[scale=0.3] 
 \def\g2left{
       (0,0.8) arc (90:270:1.6 and 0.8)
       (-1.04,0.08)--(-0.88,0)
       ..controls (-0.64,-0.12)..(-0.4,0)
       --(-0.24,0.08)
       (-0.88,0)..controls (-0.64,0.12)..(-0.4,0)
       }
 \draw[yshift=-13cm] \g2left
       [xscale=-1] \g2left
       (-.8,1.9)--(.8,1.9)
       (-1.9,3)--(1.9,3);
 \draw[yshift=-13cm,fill=black!42] 
       (-.8cm,2.3cm) circle (0.4cm)
       (.8cm,2.3cm) circle (0.4cm);
 \filldraw[yshift=-13cm]
       (-.8cm,1.9cm) circle (3pt)
       (.8cm,1.9cm) circle (3pt)
       (-.8,.73) circle (3pt)
       (.8,.73) circle (3pt);
 \draw[yshift=-13cm] 
 	   (-.8,1.1) node {\tiny $\mathfrak{{}_1}$}
       (.8,1.1) node {\tiny $\mathfrak{{}_2}$}
       (0,1.38) node {\tiny $\times$}
       (-.8,2.3) node {\tiny $\mathfrak{{}_1}$}
       (.8,2.3) node {\tiny $\mathfrak{{}_2}$}
       (-4.5,.95) node {\tiny{$\wh\fM_{\ord}^\bullet\eq\tn{PT}$}}
       (-4,4.5) node {\tiny{$\P(L_{2;1}\!\oplus 0)$}}
       (-2.1,-1.8) node {\scriptsize{$X_{2,2}$}};
 \draw [->,>=stealth,yshift=-13cm] 
       (-6,3.9)--(-6,1.8);
 \draw[decorate,decoration=brace,yshift=-13cm]
       (-1.9,-0.7)--(-1.9,2.7);
 \draw[decorate,decoration=brace,yshift=-13cm]
       (1.9,2.7)--(1.9,-.7);
 \draw[yshift=-13cm] (-6.8,5.2) rectangle (2.5,-1.1);
 
 \draw[xshift=11.5cm,yshift=-13cm] \g2left
       [xscale=-1] \g2left
       (-.8,1.9)--(.8,1.9)
       (-1.9,3.3)--(1.9,3.3);
 \draw[yshift=-13cm,xshift=11.5cm,fill=black!42] 
       (-.8cm,2.3cm) circle (0.4cm)
       (.5cm,2.82cm) circle (0.2cm)
       (1.1cm,2.82cm) circle (0.2cm);
 \draw[yshift=-13cm,xshift=11.5cm] (.8cm,2.3cm) circle (0.4cm);
 \filldraw[yshift=-13cm,xshift=11.5cm]
       (-.8cm,1.9cm) circle (3pt)
       (.8cm,1.9cm) circle (3pt)
       (-.8,.73) circle (3pt)
       (.8,.73) circle (3pt);
 \draw[yshift=-13cm,xshift=11.5cm]
       (-.8,1.1) node {\tiny $\mathfrak{{}_1}$}
       (.8,1.1) node {\tiny $\mathfrak{{}_2}$}
       (0,1.38) node {\tiny $\times$}
       (-.8,2.3) node {\tiny $\mathfrak{{}_1}$}
       (.8,2.3) node {\tiny $\mathfrak{{}_2}$}
       (-4.5,.95) node {\tiny{$\wh\fM_{\ord}^\bullet\eq\tn{PT}$}}
       (-3.4,4.5) node {\tiny{$\P(L_{2;1}\!\oplus\!L_{2;2})$}}
       (-2.1,-1.8) node {\scriptsize{$X_{2,3}$}};
 \draw [->,>=stealth,xshift=11.5cm,yshift=-13cm] 
       (-6,3.9)--(-6,1.8);
 \draw[xshift=11.5cm,yshift=-13cm,decorate,decoration=brace]
       (-1.9,-0.7)--(-1.9,2.8);
 \draw[xshift=11.5cm,yshift=-13cm,decorate,decoration=brace]
       (1.9,2.8)--(1.9,-.7);
 \draw[xshift=11.5cm,yshift=-13cm] (-6.8,5.2) rectangle (2.5,-1.1);
       
 \draw[xshift=23.1cm,yshift=-13cm] \g2left
       [xscale=-1] \g2left
       (-.8,1.9)--(.8,1.9)
       (-1.9,3.4)--(1.9,3.4);
 \draw[xshift=23.1cm,yshift=-13cm,fill=black!42] 
       (-.8cm,2.3cm) circle (0.4cm)
       (.8cm,2.9cm) circle (0.2cm)
       (.28cm,2.6cm) circle (0.2cm)
       (1.32cm,2.6cm) circle (0.2cm);
 \draw[xshift=23.1cm,yshift=-13cm] (.8cm,2.3cm) circle (0.4cm);
 \filldraw[xshift=23.1cm,yshift=-13cm]
       (-.8cm,1.9cm) circle (3pt)
       (.8cm,1.9cm) circle (3pt)
       (-.8,.73) circle (3pt)
       (.8,.73) circle (3pt);
 \draw[xshift=23.1cm,yshift=-13cm]
       (-.8,1.1) node {\tiny $\mathfrak{{}_1}$}
       (.8,1.1) node {\tiny $\mathfrak{{}_2}$}
       (0,1.38) node {\tiny $\times$}
       (-.8,2.3) node {\tiny $\mathfrak{{}_1}$}
       (.8,2.3) node {\tiny $\mathfrak{{}_2}$}
       (-4.5,.95) node {\tiny{$\wh\fM_{\ord}^\bullet\eq\tn{PT}$}}
       (-3.4,4.5) node {\tiny{$\P(L_{2;1}\!\oplus\!L_{2;2})$}}
       (-2.1,-1.8) node {\scriptsize{$X_{2,4}$}};
 \draw [->,>=stealth,xshift=23.1cm,yshift=-13cm] 
       (-6,3.9)--(-6,1.8);
 \draw[xshift=23.1cm,yshift=-13cm,decorate,decoration=brace]
       (-1.9,-0.7)--(-1.9,2.9);
 \draw[xshift=23.1cm,yshift=-13cm,decorate,decoration=brace]
       (1.9,2.9)--(1.9,-.7);
 \draw[xshift=23.1cm,yshift=-13cm] (-6.8,5.2) rectangle (2.5,-1.1); 
       
 \draw \g2left
       [xscale=-1] \g2left
       (-1.2,2.1)--(1.2,2.1)
       (-1.9,3.2)--(1.9,3.2);
 \draw[fill=black!42] 
       (-1.2cm,2.5cm) circle (0.4cm)       
       (0,2.5cm) circle (0.4cm)
       (1.2cm,2.5cm) circle (0.4cm);
 \filldraw
       (-1.2cm,2.1cm) circle (3pt)
       (0,2.1cm) circle (3pt)
       (1.2cm,2.1cm) circle (3pt)
       (-1.2,.54) circle (3pt)
       (0,.8cm) circle (3pt)
       (1.2,.54) circle (3pt);
 \draw
       (-1.2,1.0) node {\tiny $\mathfrak{{}_1}$}
       (0,.45) node {\tiny $\mathfrak{{}_2}$}
       (1.2,1.0) node {\tiny $\mathfrak{{}_3}$}
       (0,1.43) node {\tiny $\times$}
       (-1.2,2.5) node {\tiny $\mathfrak{{}_1}$}
       (0,2.5) node {\tiny $\mathfrak{{}_2}$}
       (1.2,2.5) node {\tiny $\mathfrak{{}_3}$}
       (-4.5,.95) node {\tiny{$\wh\fM_{\ord}^\bullet\eq\tn{PT}$}}
       (-3.4,4.5) node {\tiny{$\P(L_{3;1}\!\oplus\!0\!\oplus\!0)$}}
       (-2.1,-1.8) node {\scriptsize{$X_{3,3}$}};
 \draw [->,>=stealth] 
       (-6,3.9)--(-6,1.8);
 \draw[decorate,decoration=brace]
       (-1.9,-0.7)--(-1.9,2.8);
 \draw[decorate,decoration=brace]
       (1.9,2.8)--(1.9,-.7);
 \draw (-6.8,5.2) rectangle (2.5,-1.1); 
  
 \draw[xshift=11.5cm] \g2left
       [xscale=-1] \g2left
       (-1.2,2.1)--(1.2,2.1)
       (-1.9,3.5)--(1.9,3.5);
 \draw[xshift=11.5cm,fill=black!42] 
       (-1.2cm,2.5cm) circle (0.4cm)
       (1.2cm,2.5cm) circle (0.4cm)
       (-.3cm,3.02cm) circle (0.2cm)
       (.3cm,3.02cm) circle (0.2cm);
 \draw[xshift=11.5cm]
       (0cm,2.5cm) circle (0.4cm);
 \filldraw[xshift=11.5cm]
       (-1.2cm,2.1cm) circle (3pt)
       (0,2.1cm) circle (3pt)
       (1.2cm,2.1cm) circle (3pt)
       (-1.2,.54) circle (3pt)
       (0,.8cm) circle (3pt)
       (1.2,.54) circle (3pt);
 \draw[xshift=11.5cm]
       (-1.2,1) node {\tiny $\mathfrak{{}_1}$}
       (0,.45) node {\tiny $\mathfrak{{}_2}$}
       (1.2,1) node {\tiny $\mathfrak{{}_3}$}
       (0,1.43) node {\tiny $\times$}
       (-1.2,2.5) node {\tiny $\mathfrak{{}_1}$}
       (0,2.5) node {\tiny $\mathfrak{{}_2}$}
       (1.2,2.5) node {\tiny $\mathfrak{{}_3}$}
       (-4.5,.95) node {\tiny{$\wh\fM_{\ord}^\bullet\eq\tn{PT}$}}
       (-2.7,4.5) node {\tiny{$\P(L_{3;1}\!\oplus\!L_{3;2}\!\oplus\!0)$}}
       (-2.1,-1.8) node {\scriptsize{$X_{3,4}$}};
 \draw [->,>=stealth,xshift=11.5cm] 
       (-6,3.9)--(-6,1.8);
 \draw[xshift=11.5cm,decorate,decoration=brace]
       (-1.9,-0.7)--(-1.9,2.8);
 \draw[xshift=11.5cm,decorate,decoration=brace]
       (1.9,2.8)--(1.9,-.7);
 \draw[xshift=11.5cm] (-6.8,5.2) rectangle (2.5,-1.1); 
       
 \draw[xshift=23.1cm] \g2left
       [xscale=-1] \g2left
       (-1.2,2.1)--(1.2,2.1)
       (-1.9,3.5)--(1.9,3.5);
 \draw[xshift=23.1cm,fill=black!42] 
       (-1.2cm,2.5cm) circle (0.4cm)
       (1.2cm,2.5cm) circle (0.4cm)
       (0cm,3.1cm) circle (0.2cm)
       (-.52cm,2.8cm) circle (0.2cm)
       (.52cm,2.8cm) circle (0.2cm);
 \draw[xshift=23.1cm]
       (0cm,2.5cm) circle (0.4cm);
 \filldraw[xshift=23.1cm]
       (-1.2cm,2.1cm) circle (3pt)
       (0,2.1cm) circle (3pt)
       (1.2cm,2.1cm) circle (3pt)
       (-1.2,.54) circle (3pt)
       (0,.8cm) circle (3pt)
       (1.2,.54) circle (3pt);
 \draw[xshift=23.1cm]
       (-1.2,1) node {\tiny $\mathfrak{{}_1}$}
       (0,.45) node {\tiny $\mathfrak{{}_2}$}
       (1.2,1) node {\tiny $\mathfrak{{}_3}$}
       (0,1.43) node {\tiny $\times$}
       (-1.2,2.5) node {\tiny $\mathfrak{{}_1}$}
       (0,2.5) node {\tiny $\mathfrak{{}_2}$}
       (1.2,2.5) node {\tiny $\mathfrak{{}_3}$}
       (-4.5,.95) node {\tiny{$\wh\fM_{\ord}^\bullet\eq\tn{PT}$}}
       (-2.7,4.5) node {\tiny{$\P(L_{3;1}\!\oplus\!L_{3;2}\!\oplus\!0)$}}
       (2.9,-1.8) node {\scriptsize{$X_{3,5}$}}
       (2.9,2) node {\scriptsize{$\cup$}};
 \draw [->,>=stealth,xshift=23.1cm] 
       (-6,3.9)--(-6,1.8);
 \draw[xshift=23.1cm,decorate,decoration=brace]
       (-1.9,-0.7)--(-1.9,2.8);
 \draw[xshift=23.1cm,decorate,decoration=brace]
       (1.9,2.8)--(1.9,-.7);
 \draw[xshift=23.1cm] (-6.8,5.2) rectangle (12.8,-1.1);  
       
 \draw[xshift=33.3cm] \g2left
       [xscale=-1] \g2left
       (-1.2,2.1)--(1.2,2.1)
       (-1.9,3.5)--(1.9,3.5);
 \draw[xshift=33.3cm,fill=black!42] 
       (-1.2cm,2.5cm) circle (0.4cm)
       (-.3cm,3.02cm) circle (0.2cm)
       (.3cm,3.02cm) circle (0.2cm)
       (.9cm,3.02cm) circle (0.2cm)
       (1.5cm,3.02cm) circle (0.2cm);
 \draw[xshift=33.3cm]
       (0cm,2.5cm) circle (0.4cm)
       (1.2cm,2.5cm) circle (0.4cm);
 \filldraw[xshift=33.3cm]
       (-1.2cm,2.1cm) circle (3pt)
       (0,2.1cm) circle (3pt)
       (1.2cm,2.1cm) circle (3pt)
       (-1.2,.54) circle (3pt)
       (0,.8cm) circle (3pt)
       (1.2,.54) circle (3pt);
 \draw[xshift=33.3cm]
       (-1.2,1) node {\tiny $\mathfrak{{}_1}$}
       (0,.45) node {\tiny $\mathfrak{{}_2}$}
       (1.2,1) node {\tiny $\mathfrak{{}_3}$}
       (0,1.43) node {\tiny $\times$}
       (-1.2,2.5) node {\tiny $\mathfrak{{}_1}$}
       (0,2.5) node {\tiny $\mathfrak{{}_2}$}
       (1.2,2.5) node {\tiny $\mathfrak{{}_3}$}
       (-4.5,.95) node {\tiny{$\wh\fM_{\ord}^\bullet\eq\tn{PT}$}}
       (-2,4.5) node {\tiny{$\P(L_{3;1}\!\oplus\!L_{3;2}\!\oplus\!L_{3;3})$}};
 \draw [->,>=stealth,xshift=33.3cm] 
       (-6,3.9)--(-6,1.8);
 \draw[xshift=33.3cm,decorate,decoration=brace]
       (-1.9,-0.7)--(-1.9,2.8);
 \draw[xshift=33.3cm,decorate,decoration=brace]
       (1.9,2.8)--(1.9,-.7);
 
 \filldraw [xshift=28cm, yshift=-11cm]
  (0,0) circle (2pt) 
  (.4,0) circle (2pt) 
  (.8,0) circle (2pt);
  
 \filldraw [xshift=38.2cm, yshift=2cm]
  (0,0) circle (2pt) 
  (.4,0) circle (2pt) 
  (.8,0) circle (2pt);
    
 \draw[yshift=-13cm,color=black!42,thick,decorate,decoration=brace]
  (28.8,-3.1)--(-6.7,-3.1);
 \draw[color=black!42,thick,decorate,decoration=brace]
  (39,-3.1)--(-6.7,-3.1);
 
 \draw[yshift=-13cm,->,>=stealth, very thick]
  (3.2,2)->(4.2,2);
 \draw[yshift=-13cm,->,>=stealth, very thick]
  (14.7,2)->(15.7,2);
 \draw[yshift=-13cm,->,>=stealth, very thick]
  (26.3,2)->(27.3,2);
  
 \draw[->,>=stealth, very thick]
  (3.2,2)->(4.2,2);
 \draw[->,>=stealth, very thick]
  (14.7,2)->(15.7,2);
 \draw[->,>=stealth, very thick]
  (36.5,2)->(37.5,2);
  
 \draw[->,>=stealth,very thick, rounded corners]
  (39.7,2)--(40.4,2)--(40.4,-6)--(-2.15,-6)--(-2.15,-7.6);
 \draw[yshift=13cm,->,>=stealth,very thick, rounded corners]
  (8.1,-4)--(8.8,-4)--(8.8,-6)--(-2.15,-6)--(-2.15,-7.6);
 
 \draw[color=black!42] 
  (12.2,-17.5) node{\small{images $\subset\cE_{2}$}}
  (14.35,-4.5) node{\small{images $\subset\cE_{3}$}};
 \draw 
  (-6.8,9) node[right] {\scriptsize{$X_{k,k'}$ with $k\ge 4$ (not illustrated)}};
\end{tikzpicture}
\end{center}
\caption{The second round $(\rd_2)$ of the modular blowups}\label{figBlowup2}
\end{figure}

On $\ti\fM^{\rd_1}$, let $\cE_k$ be the proper transform of the exceptional divisor created
in the $k$-th step of $(\rd_1\ph_1)$, which may meet the proper transforms 
of some nodal divisors. 
Locally, each connected component of $X_{k,k'}$ is the intersection of $\cE_k$ with 
(the proper transforms of) some $k'\!-\!1$ many nodal divisors. Then 
$$\ti\fM^{\rd_2} \lra \ti\fM^{\rd_1}$$
decomposes as the sequential blowups along the proper transforms of $\{X_{k,k'}\}$,   with respect to the lexicographic order on $(k,k')$ so that the order on $k'$ is the usual order of natural numbers, whereas that on $k$ is in reverse order of natural numbers; c.f.~Figure~\ref{figBlowup2}.

After the second round, $(\rd_2)$,
the only obstacle to the
pullback of $\varphi$ being locally diagonalizable is the existence of the Weierstrass and conjugate points.
They will be dealt with in the next (i.e.~the last) round of blowups.

\vsp 
\noindent
{\it The third round of blowups $(\rd_3)$.}
\vsp
There are four phases of sequential blowups in this round. 

The blowup centers in the 
first phase $(\rd_3\ph_1)$ locally lie in the proper transforms of intersections of exceptional divisors of $(\rd_1\ph_1)$ and some nodal divisors.
Each of such intersections,
after $(\rd_2)$,
gives rise to two distinctive directions in the corresponding exceptional divisor of $(\rd_1\ph_1)$.
For the underlying curve of 
a general point of the blowup center of any step in $(\rd_3\ph_1)$, its core is of weight zero, and there are two rational
tails attached to conjugate points of the core that correspond to the two distinctive directions aforementioned.

The blowup centers in the 
second phase $(\rd_3\ph_2)$ locally lie in the proper transforms of intersections of exceptional divisors of $(\rd_2)$ and some nodal divisors.
Each of such intersections
gives rise to a distinctive direction in the corresponding exceptional divisor of $(\rd_2)$.
For the underlying curve of 
a general point of the blowup center of any step in $(\rd_3\ph_2)$, its core is of weight zero, and there is a rational
tail attached to a Weierstrass point of the core that corresponds to the distinctive direction aforementioned.

All the previous blowups are performed over the base stack $\fM_2^{\rm wt}$.
The remaining blowups are to be performed over the base stack $\fdd$.

The general points of the blowup centers in the 
third phase $(\rd_3\ph_3)$ are related to the pairs $(C, D)$
 such that the core $F$ of each $C$ is smooth and satisfies $D\!\cap\!F$ consists of two conjugate points.
More precisely, let~$\cH_k$ ($k\!\ge\!0$) be the closed substack of $\fdd$ whose general points
 are the pairs $(C,D)$ 
 such that there are $k$ rational tails attached to the smooth core $F$,
 satisfying $D\!\cap\!F$ consists of two conjugate points.
We blow up the stack $\fdd\!\times_\Mw\!\ti\fM^{\rd_3\ph_2}$ along the proper transforms of 
 $\cH_0\!\times_\Mw\!\ti\fM^{\rd_3\ph_2}, \cdots, \cH_k\!\times_\Mw\!\ti\fM^{\rd_3\ph_2}, \cdots$  and obtain $\ti\fM^{\rd_3\ph_3}$.

The blowup centers in the 
fourth phase $(\rd_3\ph_4)$ 
locally lie in the proper transforms of intersections of exceptional divisors of $(\rd_1\ph_5)$ and some nodal divisors.
Each of such intersections
gives rise to a distinctive direction in the corresponding exceptional divisor of $(\rd_1\ph_5)$.
For the underlying point $(C,D)$ of 
a general point of the blowup center of any step in $(\rd_3\ph_4)$, its core $F$ contains exactly one point of $D$, and this point is conjugate to a point of the core where a rational tail is attached; this tail corresponds to the distinctive direction of the exceptional divisor aforementioned.
We remark that $(\rd_3\ph_4)$ solely treats the boundary  components of $\MPdd$, just as  in $(\rd_1\ph_5)$.

A set of examples that illustrate the above blowups are given in \S \ref{Subsec:blowup_examples}.

The key fact to verify is that
the pullback of $\varphi$ becomes locally diagonalizable when $(\rd_3)$ terminates; see Proposition~\ref{PrpChangeofPhi}.
To this end, 
it is crucial to understand how each round or phase of (the globally defined)
modular blowups affects the moduli space {\it locally},
so as to study the local behavior of the structural homomorphism $\varphi$ in \S\ref{SecChangeOfPhi}.

Once  Proposition~\ref{PrpChangeofPhi} is established in \S\ref{SecChangeOfPhi},
we apply the techniques in \cite{HL10}  and \cite{HN2} to obtain the local structure of  $\widetilde{M}_2(\PP^n,d)$ in \S\ref{localEquations}.

Throughout the paper, we work over a fixed algebraically closed field
$\kk$ of characteristic zero.

\vsp
\noindent
{\bf Acknowledgement.}
The first named author would like to thank Dawei Chen for useful discussions on
 Weierstrass and conjugate loci. 
The third named author would like to thank Aleksey~Zinger for the valuable discussions on the blowup loci.
We also thank  Sanghyeon Lee and Mulin Li for their questions and for pointing out some errors.
During the revision of this paper, the first named author was visiting Great Bay University, whose generous hospitality and support are gratefully acknowledged.

The anonymous referee supplied an extensive report containing numerous insightful comments, valuable suggestions, and meticulous corrections. The referee’s generous contribution is hereby formally acknowledged and deeply appreciated.

\section{Local defining equations of $\ov M_2(\PP^n,d)$}\label{SecPhi}

In this section, we study in details the local defining equations of $\ov M_2(\PP^n,d)$.
Our resolution process is guided by these local defining equations. Some crucial examples
will be provided to illustrate the simple ideas before introducing the heavy machinery.

As explained in \cite{HL10}, the local structure of $\ov M_2(\PP^n,d)$,
indeed, $\ov M_g(\PP^n,d)$ for any $g$,  is governed by
the derived object $\bR\pi_*\ff^* \sO_{\Pn}(1)$. Our aim of this section is
two-fold: (1)
to describe $\bR\pi_*\ff^* \sO_{\Pn}(k)$
as explicit as possible; (2) to explain how to locally diagonalize  (in the sense of Definition~\ref{dObject})
the object $\bR\pi_*\ff^* \sO_{\Pn}(k)$, the $k=1$ case of 
which will imply a desingularization of $\ov M_2(\PP^n,d)$,
as desired.

The local defining equations of $\ov M_2(\PP^n,d)$ are summarized in Proposition~\ref{Prp:phi_key}.

\subsection{Basic setup}
\label{sheafStructures}

For the aforesaid purpose, we 
consider two stacks throughout this paper:
$\mwt$ and $\fM_2^{\rm div}$.

First, let $\mwt$ be the Artin stack of stable pairs $(C,\textbf{w})$ of genus~2 nodal curves $C$ with non-negative weights $\textbf{w}\inn H^2(C,\mathbb Z)$, meaning that $\textbf{w}(\Si)\!\ge\!0$ for all irreducible components $\Si$ of $C$.
Here $(C,\bw)$ is said to be \ts{stable} if every rational irreducible component of weight 0 contains at least three nodal points;
c.f.~\cite[\S2.1]{HL10}.
The stability requirement guarantees that each connected component of $\mwt$,
determined by the total weight on the curve,
is of finite type.

Next,
let $\fM_2^{\rm div}$ be the Artin stack of stable pairs $(C,D)$ where $C$ are genus 2 nodal curves and $D\!\subset\!C$ are simple and effective divisors on $C$ away from nodes.
Here $(C,D)$ is said to be  \ts{pre-stable} if every rational irreducible component disjoint from $D$ contains at least three nodal points;
c.f.~\cite[\S2.3]{HL10}. 
The stability requirement guarantees that each connected component of $\fM_2^{\rm div}$,
determined by the degree of $D$,
is also of finite type.

Between the above two stacks,
there is a morphism
\begin{align}\label{Eqn:div_to_wt}
	\fM^{\rm div}_2\lra\Mw,\qquad
	&(C,D)\mapsto \big(C,c_1(D)\big).
\end{align}
We point out that this is a smooth morphism.

In studying $\MPdd$ (along with $\pi_*\ff^* \sO_{\Pn}(k)$),
we cover it by sufficiently small \'etale opens $\{U \to \MPdd\}$ such that for each $U\inn\{U\}$,
the set $\bH_U$ as in (\ref{Eqn:H}) is nonempty.
We pick an affine \'etale $\{\cV\to \fM_2^{\rm div}\}$ so that for each $U\inn\{U\}$ and $H\inn\bH_U$,
the morphism 
\begin{align*}
f_{U,H}:\,
 U\lra \fM_2^{\rm div},\qquad
 [C,u]\mapsto\big(C,u^{-1}(H)\big)
\end{align*}
introduced in (\ref{toP0}) factors through
$U\to \cV$ for some $\cV\inn\{\cV\}$.
This way, we  cover $\MPdd$  by the charts $\{U/\cV\}$.

Using the (local) morphism \eqref{toP0}, 
it suffices  to study a parallel problem on $\fM_2^{\rm div}$.
That is, for a fixed $(C,D)\in\fM^{\rm div}_2$ lying over $(C,\bw) \in \Mw$ 
 with $\deg D=m$ for some positive integer $m$,
we pick an affine smooth chart $\cV\to\fM^{\rm div}_2$
containing $(C,D)$. Here, $m$ can be any positive integer;
it corresponds to the case of ${\bf R}\pi_*\ff^* \sO_{\Pn}(k)$ when $m=dk$.

Let 
$ \rho\!: \cC\!\to\!\cV$ and $\cD\!\sub\!\cC$
be the universal family on $\cV$.
We find disjoint sections
$\cA_1, \cA_2, \cB$ of $\cC/\cV$, disjoint from $\cD$, called the \ts{auxiliary divisors}.
We list this package as
\beq\label{package}
\rho: \cC\to \cV \, (\to\fM_2^{\rm div}),
\quad\cD\sub \cC; \quad \text{plus}\  \cA=\cA_1\!+\!\cA_2, \, \cB\sub \cC.
\eeq

Let
$$\sL=\sO_\cC (\cD),\quad 
\sM
=\sO_\cC(\cA-\cB),\and \sO_\cA(\cA)=\sO_\cV(\cA)\otimes_{\sO_\cC}\sO_\cA.
$$
Note that $\sL(\cA-\cB)=\sM(\cD)$.
We consider the two-term complex
\beq \label{intro-varphi} 
 \rho_* \sL(\cA) \lra \rho_* \sO_\cA(\cA)
\eeq
via the evaluation homomorphism. These data will be fixed throughout the paper.

\begin{defi}\label{Dfn:nodal_curve}
Let $C$ be a nodal curve.
We call a node $q$ (resp. an irreducible component $C'$) of $C$ 
\ts{separating}
if $C\bsl q$ (resp. $C\bsl C'$) is
disconnected.
We call other nodes (resp. irreducible components) of $C$ \ts{non-separating}.
A curve $C$ is said to be \ts{separable} (resp.~\ts{inseparable}) if it contains at least one separating node (resp.~contains no separating node).
 A \ts{genus-one
inseparable component}
of $C$ is a subcurve of $C$ that itself is an inseparable curve of (arithmetic) genus one.
If $C$ is of positive (arithmetic) genus $g$,
we call the smallest connected genus $g$ subcurve $F$ of $C$ the \ts{core} of~$C$.
Removing $F$ from $C$ results in disjoint connected trees of rational curves,
called the \ts{tails} of $C$.
Each tail is attached to the core $F$ at a separating node,
called a \ts{pivotal node}.
\end{defi}

 We remark it is possible that a non-separating node of $C$ is a separating node of a subcurve of $C$,
hence for a genus-one inseparable component $C'$ of $C$, its nodes must be non-separating in $C'$.
It may also be worth pointing out here that
an inseparable $F$ can be reducible: for example, a smooth rational curve attached to
two distinct points of a smooth genus one curve.

\vsp
\begin{assu}\label{Ass:basic} 
Possibly after an \'etale base change and  shrinking  $\cV$,
we assume
\begin{enumerate}
[label=(\Alph*),leftmargin=*]
\item \label{AssI}  $a_i=\cA_i\cap C$ and $b=\cB\cap C$ lie in the core $F$ of $C$ and are in general position; in case $F$ is
separable, then $a_1$ and $a_2$ lie on different genus-one inseparable components of $F$;
\item \label{AssII}
 $\cD=\cD_1+\cdots+\cD_m$ is a union of $m$ disjoint sections $\cD_i$ of $\cC\to\cV$;
\item \label{AssIII}
Assumptions~\ref{AssI} and~\ref{AssII} hold for all fibers
of $\cC\to\cV$.
\end{enumerate}
\end{assu}
\vsp

We emphasize that the general position requirement in
Assumption~\ref{AssI} implies none of $a_1$, $a_2$, and $b$ is a Weierstrass point in the sense of Definition~\ref{Dfn:Weierstrass},
and any two of these three points are not conjugate in the sense of Definition~\ref{Dfn:conjugate}.
Particularly, $a_1$, $a_2$, and $b$
cannot be on any non-separating bridge (should it exist;
 see Definition~\ref{Dfn:bridge} for terminology),
because by Part~\ref{Cond:W_non_sep_bri} of Lemma~\ref{Lm:W_smoothness} (in~\S\ref{Subsec:WC}),
any point on a non-separating bridge is a Weierstrass point.

The lemma below follows from direct verification.

\begin{lemm}\label{Lm:L(A)}
Under Assumption~\ref{Ass:basic}, we have
$R^1 \rho_* \sL(\cA)\eq 0$, $\rho_* \sL(\cA)$ is locally
free, and $\bR\rho_*\sL$ is quasi-isomorphic to the complex \eqref{intro-varphi}.
\end{lemm}

\subsection{The first reduction}
\label{firstReduction}


%


The standard inclusion $\sO_{\cC}\sub \sO_{\cC}(\cD)$
determines the obvious section $1\in\Gamma(\rho\lsta\sL)$. For other sections, let
$\sM= \sO_{\cC}(\cA-\cB)$
and consider the inclusion
$$ \sM (\cD_i)=\sO_{\cC}(\cD_i+\cA-\cB)\mapright{\sub}
\sM (\cD) =\sO_{\cC}(\cD+\cA-\cB),
$$
and the induced inclusions
$$\eta_i: \rho_*\sM(\cD_i) \mapright{\sub}
\rho_*\sM (\cD).
$$
Because of Assumption~\ref{Ass:basic}, the above terms are locally free. By
Riemann-Roch Theorem, $\rho\lsta\sM (\cD_i)$ is invertible and
$\rho\lsta\sM (\cD)$ has rank $m$. Let \beq\label{varphi}\varphi:
\rho\lsta\sM (\cD)\lra
\rho\lsta\sO_{\cA}(\cA)\and
\varphi_i: \rho\lsta\sM (\cD_i)\lra
\rho\lsta \sO_{\cA}(\cA)
\eeq
be the evaluation homomorphisms; then $\varphi_i\eq\varphi
\!\circ\!\eta_i$. We denote their sum by
\begin{align*}
(\eta_i)&: \oplus_{i=1}^m \rho_*\sM (\cD_i)\lra \rho_*\sM (\cD),\\
(\varphi_i)&: \oplus _{i=1}^m\rho\lsta\sM (\cD_i)\lra\rho\lsta\sO_{\cA}(\cA).
\end{align*}

\begin{lemm} \label{lemm:usefulFacts}
We have
(1). $\rho\lsta\sL\cong\sO_\cV\oplus\rho\lsta \sL(-\cB)$; (2). $\rho\lsta\sL(-\cB)\cong \ker\varphi$;
(3). $(\eta_i)$ is an isomorphism, and (4).
$(\varphi_i)=\varphi \circ(\eta_i)$.
Consequently,
$$\rho\lsta\sL\cong
\sO_\cV\oplus \ker\{(\varphi_i)\}.
$$
\end{lemm}

\begin{proof}
The proofs of (1) and (2) are the same as for the genus 1 case, \cite[Lemma
4.10]{HL10}; we omit the details.
The proof of (3) is also similar. We now 
elaborate it below.
Since both
$\rho_*\sM (\cD_i)$ and $\rho_*\sM (\cD)$ are locally free, we
only need to show that for any closed $z\in \cV$,
$\oplus_{i=1}^m \rho_*\sM (\cD_i)|_z\to \rho_*\sM (\cD)|_z$
is an
isomorphism. By base change, it suffices to show that the tautological
homomorphism
$$\bigoplus_{i=1}^m \eta_{i}(z):\bigoplus_{i=1}^m H^0\bl
\sO_{\cC_z}(\cD_i+\cA-\cB)\br \mapright{} H^0\bl\sO_{\cC_z}(\cD+\cA-\cB)\br
$$
is an isomorphism. Because both sides are of equal
dimensions, it suffices to show that
it is injective. 
For this, we look at the composite
of $\oplus \eta_i(z)$ with 
the restriction
$$\phi_j(z):H^0\bl
\cC_z,\sO_{\cC_{{z}}}(\cD+\cA-\cB)\br \lra H^0\bl\cC_z\cap\cD_j,
\sO_{\cC_z\cap\cD_j}(\cD+\cA-\cB)\br.$$
For any $i$, when $i=j$, $\phi_j(z)\circ\eta_i(z)$ is an
isomorphism for all the points $z \in \cV$, and when $i\ne j$,
$\phi_j(z)\circ\eta_i(z)=0$ holds for general point
$z \in \cV$ and hence also holds for all the points. 
Thus, if $(\oplus_{i=1}^m\eta_i(z))(\oplus v_i)
=\oplus_{i=1}^m\eta_i(z)(v_i)=0$ where
$v_i \in H^0\big(
\sO_{\cC_z}(\cD_i+\cA-\cB)\big)$, then 
$\oplus_{i=1}^m \phi_j(z)\circ\eta_i(z)( v_i)=\phi_j(z)\circ\eta_j(z)( v_j)=0$.
Hence, $v_j=0$ for all $1 \le j \le m$.
This shows that $\oplus_{i=1}^m\eta_i(z)$ is injective.
The item (4) is a tautology. This proves the lemma.
\end{proof}

This proves that the homomorphism $\varphi$ in \eqref{varphi}
is equivalent to the homomorphism $(\varphi_i)$, via the isomorphism
$(\eta_i)$.

We will call either $\varphi$ or $(\varphi_i)$ a \ts{structural homomorphism}.
Note that each $\varphi_i$ has the form
\beq\label{varphiA}
\varphi_i = \varphi_{i}^1 \oplus
\varphi_{i}^2: \rho_*\sM (\cD_i)\lra \rho_* \sO_{\cA_1}
(\cA_1) \oplus \rho_* \sO_{\cA_2} (\cA_2).
 \eeq
These  homomorphisms  will be our focus in this section.

\subsection{Local open immersions}\label{subsect:localEq}

To obtain local equations of $\MPdd$, we consider 
the case of $\pi_*\ff^* \sO_{\Pn}(1)$, and 
 let $m=d$. Then,
as in \S\ref{sheafStructures},  we  cover $\MPdd$  by charts $\{U/\cV\}$.
 Let $\sE_{\cV}$ be the total space of the vector
bundle $\rho_* \sL(\cA)^{\oplus n}$ and let $$p: \sE_{\cV} \to
\cV$$ be the projection.  We also set
$$\rho_* (\sL(\cA)^{\oplus n}|_\cA)=\rho_* (\sL(\cA)^{\oplus
n}|_{\cA_1}) \oplus \rho_* (\sL(\cA)^{\oplus n}|_{\cA_2}).$$  Then
the tautological restriction homomorphism 
\beq \label{restHom-0}
\text{rest}: \rho_* \sL(\cA)^{\oplus n}\lra \rho_*
(\sL(\cA)^{\oplus n}|_\cA) \eeq lifts to a section 
\beq \label{sec1-0}
\Phi \in \Gamma (\sE_{\cV}\,,\; p^*\rho_* (\sL(\cA)^{\oplus
n}|_{\cA})). \eeq

By choosing suitable trivialization, the homomorphism
$``$rest$"$ of \eqref{restHom-0} can be identified with the
homomorphism
 \beq \label{0plusphi-0} (0 \oplus \varphi)^{\oplus n}:
(\sO_\cV \oplus \rho_*\sM(\cD))^{\oplus n} \lra (\rho_*
\sO_{\cA}(\cA))^{\oplus n}\,, \eeq where $\varphi$ is the structural
homomorphism in \eqref{varphi}.  Then, the section
$\Phi$ of  \eqref{sec1-0} is the lift of $(0 \oplus \varphi)^{\oplus
n}$ in~\eref{0plusphi-0}.

\vsp

\begin{theo}\label{thm:immersion} 
	Let $\cU= \cV\! \times_{\fM_2^{\rm div};f_{U,H}} \!U$.
Then there is a canonical open immersion \beq\label{immersion}
\cU
\lra \;{\rm (} \Phi=0 {\rm )} \sub \sE_{\cV}.\eeq
\end{theo}
\begin{proof} This theorem is the natural  extension of \cite[Theorem
2.17]{HL10}; the proof is also parallel.
\end{proof}


In what follows, we will focus on 
the  structures of the  objects 
\beq \nonumber \varphi:
\rho\lsta\sM (\cD)\lra
\rho\lsta\sO_{\cA}(\cA). \eeq
We begin with  the notion of locally diagonalizable derived objects as introduced in \cite{HL11}.
It is a key to this article.

\subsection{Locally diagonalizable homomorphisms}

Let $\varphi: \sE_1\to \sE_2$ be a homomorphism of locally free sheaves of $\sO_U$-modules on a
DM stack $U$.

\begin{defi}\label{diagonalizable}\label{DfnDiag} We say  $\varphi$ is diagonalizable if
there are integers $r$, $l_1$ and $l_2\in\ZZ_{\ge 0}$, $p_i\in\Gamma(\sO_U)$ such that
the ideals $(p_{i})  \supset ( p_{i+1})$, and
isomorphisms $\sE_i\cong \sO_U^{\oplus r}\oplus \sO_U^{\oplus l_i}$
such that
$$\varphi=\mathrm{diag}[p_1,\cdots, p_r]\oplus 0: \sO_U^{\oplus r}\oplus \sO_U^{\oplus l_1}\lra
\sO_U^{\oplus r}\oplus \sO_U^{\oplus l_2},
$$
where $\mathrm{diag}[p_1,\cdots, p_r]: \sO_U^{\oplus r}\to \sO_U^{\oplus r}$
is the diagonal homomorphism and 
\hbox{$0: \sO_U^{\oplus l_1}\to \sO_U^{\oplus l_2}$} is the zero homomorphism.
We say $\varphi$ is locally diagonalizable if there is an \'etale cover $U\lalp$ of $U$
such that $\varphi$ is diagonalizable over~$U\lalp$.
\end{defi}

\begin{defi}\label{dObject}
Let $\bf E$ be a perfect derived object on a DM stack $M$. 
We say $\bf E$ is locally diagonalizable if
there is an  \'etale cover $\{U\}$ of $M$ such that ${\bf E}|_U$ is quasi-isomorphic to  $[\varphi: \sE_1\to \sE_2]$, where
$\sE_1$ and $\sE_2$ are locally free sheaves and $\varphi$ is diagonalizable.
\end{defi}

Using~\cite[Proposition 3.2]{HL11} (the universality of diagonalization),  one checks directly that the definition does not
depend on the local representation of the object $\bf E$. Here it is worth to point out that in Definition \ref{diagonalizable},
if the requirement  $(p_{i+1})  \sub( p_{i})$ is removed, then the ``diagonalization" thus defined
 is not a property of the derived object
but merely a property of a presentation of the object $\bf E$.  

The following result is a conclusion of~\cite[Corollary~4.4]{HL11}.
It is needed to deduce Corollary 3 from Theorem 1 in \S\ref{SecIntro}.

\begin{prop}[{\cite{HL11}}]
\label{PrpLocallyFree}
Let $M$ and $\ti M$ be two DM stacks and $\bf E$ be a perfect derived object on $M$.
Assume there is a birational morphism 
$\psi\!:\ti M\!\to\!M$ such that the pullback of $\bf E$ is locally diagonalizable in the sense of Definition~\ref{DfnDiag}.
Then, 
for any irreducible component $N$ of $\ti M$ endowed with the reduced stack structure,
$\sH^0\big(L\psi^*{\bf E}|_N\big)$ is locally free.
\end{prop}

As local diagonalization is functorial, a local diagonalization of 
$\bR \rho_* \sO_\cC (\cD)$ will 
imply a resolution of  $\ov M_2(\PP^n,d)$.
The remainder of \S2 is devoted to provide
detailed analysis of the local structures of the object $\bR \rho_* \sO_\cC (\cD)$.

\subsection{Notation on nodal curves and the initial forms of structural homomorphisms}\label{mpara}
\label{SubsecPhi}

By the deformation theory of nodal curves, for
each node $q\in C$ there is a regular function
$\zeta_q\in\Gamma(\sO_\cV)$ so that $\{\zeta_q=0\}$ is
the locus where the node $q$ is not smoothed; the divisor
$(\ze_q\eq 0)$ is smooth.
We will refer to $\zeta_q$ as a \ts{modular parameter}  or a \ts{node-smoothing parameter}.

To proceed, we follow the notation and terminology of nodal curves
 in Definition~\ref{Dfn:nodal_curve}.
Let $C$ be a nodal curve of genus 2.
We denote by $N(C)$ the set of the nodes of $C$.
A node $q\inn N(C)$ is called a \ts{core node}
if it is a node (i.e.~a singular point) of the core $F$ of $C$;
particularly, any pivotal node cannot be a core node.

We say a separating  node $q$ \ts{separates} (or \ts{lies between}) two  smooth
points $\de$ and $\de'$ of $C$
if  $\de$ and $\de'$ lie in different connected components of $C\bsl\{q\}$.
If we replace $\de'$ by a
connected subcurve  
$C'$ of $C$,
we say $q$ \ts{lies between} $\de$ and $C'$ if  $\de'$ and $C'\bsl\{q\}$ lie in different connected components of $C\bsl\{q\}$.
We denote by
$$
N_{[\de,\de']}\qquad \big(\tn{resp}.~N_{[\de,C']}\big)
$$ the set of the {\it separating} nodes lying between $\de$
and $\de'$ (resp.~$C'$).
Particularly, 
for $C'\eq F$, we write
\begin{align*}
	N_{[\de]}:=N_{[\de,F]}.
\end{align*}

Below, we introduce an important type of rational subcurves of the core $F$ of $C$.

\begin{defi}
	\label{Dfn:bridge}
If $F$ is a union of two connected subcurves: $\tn B$ of arithmetic genus 0 and $F_1$ of arithmetic genus 1,
such that $\tn B$ and $F_1$ do not share any common irreducible component,
then we call $\tn B$
a \ts{non-separating bridge} of 
$C$.
In other words, $\tn B$ is a non-separating bridge if there exist two non-separating core nodes $\fp$ and $\fq$ such that
\begin{align*}
	F= \tn B\cup F_1,\qquad
	\tn B\cap F_1=\{\fp,\fq\},\qquad
	p_a(\tn B)=0,\qquad
	p_a(F_1)=1,
\end{align*}
where $p_a$ denotes the arithmetic genus.

If $F$ is a union of three connected subcurves: $\tn B$ of arithmetic genus 0, and $F_1$ and $F_2$ both of arithmetic genus 1,
such that any two of $\tn B$, $F_1$ and $F_2$ do not share any common irreducible component,
then we call $\tn B$
a \ts{separating bridge} of 
$C$.
In other words,
$\tn B$ is a separating bridge if there exist two separating core nodes $\fp$ and $\fq$ such that
\begin{align*}
	F= \tn B\cup F_1\cup F_2,\quad
	\tn B\cap F_1=\{\fp\},\quad
	\tn B\cap F_2=\{\fq\},\quad
	p_a(\tn B)=0,\quad
	p_a(F_1)=p_a(F_2)=1
\end{align*}
(which implies $F_1\!\cap\!F_2\eq \emptyset$ because $p_a(F)\eq 2$).

A bridge, non-separating or separating, is \ts{maximal} if it is not a strict subset of any other bridge.

When needed, we denote a bridge $\tn B$ by $\tn B[\fp,\fq]$ to emphasize the two nodes.
\end{defi}

It is straightforward that all the core nodes meeting a non-separating (resp.~separating) bridge $\tn B[\fp,\fq]$, including $\fp$ and $\fq$, are non-separating (resp.~separating).
By the topology of nodal curves,
a genus-2 nodal curve $C$ may contain at most three non-separating bridges, or at most one separating bridge.
Indeed, if $C$ contains three non-separating bridges,
then it cannot have any separating bridge;
if $C$ contains a separating bridge,
then it may contain at most two non-separating bridges.

For example, the shaded component of the leftmost picture in the third row (resp.~the second row) of Figure~\ref{figBlowup1} is a non-separating bridge (resp.~a separating bridge).

If there exists  a non-separating bridge $B\eq \tn B[\fp,\fq]$ on $F$,
for every smooth point $\de\inn B$,
we define 
\begin{align}\label{Eqn:N_bridge}
	N_{[\de,\fp]}\qquad
	\big(\tn{resp}.~N_{[\de,\fq]}\big)
\end{align}
to be set of all the core nodes of $B$ between $\de$ and $\fp$ (resp.~$\fq$), including $\fp$ (resp.~$\fq$).
We emphasize  both $N_{[\de,\fp]}$ and $N_{[\de,\fq]}$ are nonempty as long as $B$ exists, because $\fp\inn N_{[\de,\fp]}$ and $\fq\inn N_{[\de,\fq]}$.

If $T$ is a tail and $\de, \de'\inn T$,
among all the  nodes of $N_{[\de]}\!\cap\!N_{[\de']}$,
	i.e.~all the common nodes lying between $\de$ and $F$ as well as between $\de'$ and $F$,
	we denote by $\de\!\wedge\! \de'$ the one that is the farthest from $F$.
	In other words,  $\de\!\wedge\! \de'$ is the node satisfying
\begin{equation}\label{Eqn:wedge}
	N_{[\de\wedge \de']}= N_{[\de]}\cap N_{[\de']}\,;
\end{equation}
in particular, $\de\!\wedge\! \de'$ cannot be any core node.
If $\de$ and $\de'$ are on different tails, or at least one of them is on the core,
we define $N_{[\de\wedge \de']}\!:= \!\emptyset$, which is in accordance with (\ref{Eqn:wedge}). However, $\de\!\wedge\!\de'$ is not defined in these cases.

Next, consider the family $\cC\to\cV$ with $C=\cC\times_{\cV}0$ the central fiber mentioned before.
For $1\le i \le m$ and $1 \le  s \le 2$, let
$$\delta_i = \cD_i \cap C \and a_s = \cA_s \cap C.$$
For smooth points $\de,\de'\inn C$,
we define
\beq\label{product}
\zeta_{[\de,
\de']}=\prod_{q\in N_{[\de, \de']}}\zeta_q\and
\zeta_{[\de]}=\prod_{q\in
N_{[\de]}}\zeta_q.
 \eeq
If there exists  a non-separating bridge $B\eq \tn B[\fp,\fq]$ on $F$,
then for every smooth point $\de\inn B$,
we set
\begin{align}\label{Eqn:bridge_product}
	\zeta_{[\de,
	\fp]}=\prod_{q\in N_{[\de, \fp]}}\zeta_q
	\and	
	\zeta_{[\de,
	\fq]}=\prod_{q\in N_{[\de, \fq]}}\zeta_q.
\end{align}
In case $N_{[\de, \de']}$ (resp.~$N_{[\de]}$, $N_{[\de,\fp]}$, $N_{[\de,\fp]}$) is empty, we set $\zeta_{[\de, y]}$ (resp.~$\zeta_{[\de]}$, $\zeta_{[\de, \fp]}$, $\zeta_{[\de, \fq]}$) to be 1.
Similarly, if $\de\!\wedge\! \de'$ does not exist (i.e.~$N_{[\de\wedge\de']}\eq\emptyset$),
we set $\ze_{[\de\wedge \de']}\eq 1$.

The above notation on nodal curves are illustrated in Figure~\ref{Fig:nodal} below.

\vsp
We fix trivializations
\beq\label{tri-A}
\rho_*\sM (\cD_i) \cong\sO_\cV,\quad \rho_* \sO_{\cA_s}(\cA_s)\cong \sO_\cV;\quad
1 \le i \le m
,\quad 1 \le s \le 2,
\eeq
and the induced trivialization $\rho_* \sO_\cA(\cA)\cong\sO_\cV^{\oplus 2}$, throughout 
this section, unless stated otherwise.

\begin{prop}\label{HomFirstOrder}
\label{PrpPhi} After fixing \eqref{tri-A}, 
the evaluation homomorphism 
\beq\label{explicitHom} \varphi_{si }: 
\rho_* \sM (\cD_i) \lra \rho_* \sO_{\cA_s}(\cA_s)
\eeq
is given by the multiplication
$$\varphi_{si}= c_{si}\cdot \zeta_{[\delta_i,a_s]}\ : \sO_\cV\lra\sO_\cV,
\quad c_{si}\in \Gamma(\sO_\cV\sta).
$$
\end{prop}

\begin{proof}
The proof is parallel to that of
\cite[Prop. 4.13]{HL10} and is thus omitted.
\end{proof}

By abuse of notation,
we still denote by~$\varphi$ the $2\!\times\!m$ matrix 
$(c_{si}\cdot \zeta_{[\delta_i,a_s]} )$
with entries as in~(\ref{explicitHom}) 
and still call it the \ts{structural homomorphism} when the context is clear.

\subsection{Weierstrass and conjugate loci}
\label{Subsec:WC}

In the genus 1 case,
the analogue of Proposition~\ref{HomFirstOrder} contains all the necessary data for locally diagonalizing the structural homomorphism in the sense of Definition~\ref{DfnDiag}.
In the genus 2 case, however,
Proposition~\ref{HomFirstOrder} alone is not sufficient for studying the $2\!\times\!2$ minors of $\varphi$,
which is crucial in diagonalizing $\varphi$.
To this end,	 
we need to first study the loci of $\fM^{\rm div}_{2}$ consisting of $(C,D)$ when some points of $D$ are in the special positions of $C$.
For conciseness, in this subsection we deal with the loci of $\bigsqcup_{\ell>0}\ov\cM_{2,\ell}$ rather than $\fM^{\rm div}_{2}$,
because {\it locally} neither the order of the marked points nor the subtle difference in the stability conditions
of the two stacks affects the local equations of the Weierstrass or conjugate loci;
see Remark~\ref{Rmk:D2_vs_M2l} for more details.

\begin{defi}\label{Dfn:Weierstrass}
A point $\de$ on a smooth genus 2 curve $C$ is called a \ts{Weierstrass} point if
$$
 h^0\big(C,\sO_C(2\de)\big)=2.
$$
We denote by $\cW\!\subset\!\cM_{2,1}$ the locus of smooth genus 2 curves with one marked Weierstrass point and
by $\ov\cW$ the closure of $\cW$ in $\ov\cM_{2,1}$.
For every one-marked stable curve $(C,\de)\inn\ov\cW$, we call the marked point $\de$ a \ts{Weierstrass point}.
\end{defi}

It is well known that $\cW$ is a smooth Cartier divisor of $\cM_{2,1}$;
see~\cite[Proposition~3.2.1]{Cu} for example.
We show the same result holds for $\ov\cW\inn\ov\cM_{2,1}$ as well.
In fact,
using admissible double covers (c.f.~\cite[\S 3G]{HM98}), we can identify $\overline{\cM}_2$ with 
 $\overline{\cM}_{0,6}/S_6$ as coarse moduli spaces, where $S_6$ is the symmetry group on 6 letters. Here,
 the 6 unordered markings on the target rational curve 
 are the branch points of the hyperelliptic admissible double covers. 
 With this identification, choosing a Weierstrass point in the domain curve is the same as specifying a distinctive marking out of the 6 unordered markings in the target rational curve. Hence, $\ov\cW$ is isomorphic 
 to $\overline{\cM}_{0,6}/S_5$ as coarse moduli spaces, where the last marking is declared to be distinctive from the others. 
 In particular, $\ov\cW$ is a smooth Cartier divisor in $\overline{\cM}_{2,1}$.

For $1\!\le\!i\!\le\!\ell$, let $\ov\cW_{\ell;i}\!\subset\!\ov\cM_{2,\ell}$ be the pullback of $\ov\cW$ in $\ov\cM_{2,\ell}$ via the forgetful morphism that forgets all but the $i$-th marked point.
If $x\eq(C,\de_1,\ldots,\de_\ell)\inn\ov\cW_{\ell;i}$,
we say $\de_i$ is \ts{Weierstrass}.

Below, we study the local equation of $\ov\cW_{\ell;i}$.
For convenience, for every point $\de\inn C$,
hereafter we write
\begin{align}
	\label{Eqn:lr}
	\lr \de:= 
	\begin{cases}
		\de, & \tn{if}~\de~\tn{is~on~the~core~of}~C,\\
		\tn{the~pivotal~node~of~the~tail~containing}~\de, & \tn{if}~\de~\tn{is~on~a~tail~of}~C;
	\end{cases}
	\end{align}
see Figure~\ref{Fig:nodal} for illustration.
\begin{figure}[tb]
	\begin{center}
		\begin{tikzpicture}[scale=1.1]
			\def\g1{
				(-1,0) ellipse (1 and 0.5)
				(-1.4,0)..controls(-1,0.1)..(-0.6,0)
				(-0.6,0)..controls(-1,-0.1)..(-1.4,0)
				(-0.5,0.05)--(-0.6,0)
				(-1.4,0)--(-1.5,0.05)
			}
			\def\crs{
				(-.03,-.03)--(.03,.03)
				(-.03,.03)--(.03,-.03)
			}
			\draw\g1;
			\draw[xshift=-1cm]
			(-1.3,0) circle (.3cm)
			(-2.749,0) circle (.3cm)
			(-2.149,.829) circle (.3cm)
			(0,.8) circle (.3cm)
			(-.424,1.224) circle (.3cm)
			(.424,1.224) circle (.3cm)
			(-1.024,1.224) circle (.3cm)
			(-.424,1.824) circle (.3cm)
			;
			
			\draw[xshift=-1cm,thick]
			(-1.724,.424) circle (.3cm)
			(-1.724,-.424) circle (.3cm)
			(-2.149,0) circle (.3cm)
			;
			
			\draw[xshift=-3.5cm,thick]
			(-1.724,.424) circle (.3cm)
			(-1.724,-.424) circle (.3cm)
			(-2.149,0) circle (.3cm)
			;
			
			\draw[xshift=-2cm,yshift=1cm,thick]
			\crs;
			\draw[xshift=-2.15cm,yshift=1.4cm,thick]
			\crs;
			\draw[xshift=-.8cm,yshift=.75cm,thick]
			\crs;
			\draw[xshift=-.4cm,yshift=1.35cm,thick]
			\crs;
			\draw[xshift=-1.2cm,yshift=1.8cm,thick]
			\crs;
			\draw[xshift=-3.34cm,yshift=.829cm,thick]
			\crs;
			\draw[xshift=-2.94cm,yshift=.88cm,thick]
			\crs;
			\draw[xshift=-3.94cm,yshift=.1cm,thick]
			\crs;
			\draw[xshift=-2.724cm,yshift=-.66cm,thick]
			\crs;
			
			\draw
			(-2.45,.15) node {\tiny{$\fp$}}
			(-2.45,-.15) node {\tiny{$\fq$}}
			(-2.9,.525) node {\tiny{$c$}}
			(-2.88,.32) node {\tiny{$d$}}
			(-2.88,-.32) node {\tiny{$g$}}
			(-1,.39) node {\tiny{$h$}}
			(-.7,1.1) node {\tiny{$i$}}
			(-1.2,.9) node {\tiny{$j$}}
			(-1.63,1.23) node {\tiny{$k$}}
			(-1.4,1.4) node {\tiny{$\ell$}}
			(-2.05,0) node[right] {\tiny{$s$}}
			(-2.724,-.66) node[below] {\tiny{$\de_6$}}
			(-3.35,0) node {\tiny{$f$}}
			(-3.94,.1) node[left] {\tiny{$\de_5$}}
			(-3.34,.829) node[left] {\tiny{$\de_3$}}
			(-2.94,.88) node[right] {\tiny{$\de_4$}}
			(-2,1) node[below] {\tiny{$\de_2$}}
			(-2.2,1.4) node[above] {\tiny{$\de_1$}}
			(-.8,.7) node[right] {\tiny{$\de_7$}}
			(-.4,1.4) node[right] {\tiny{$\de_8$}}
			(-1.2,1.85) node[right] {\tiny{$\de_9$}}
			(-5.3,-1.3) node {\scriptsize{$\tn B[\fp,\fq]$}}
			(-2,-1.3) node {\scriptsize{$(C,D=\de_1+\cdots+\de_9)$}}
			(1,2) node[right] {\scriptsize{core nodes: $s,\fp,d,g,\fq$}}
			(1,1.6) node[right] {\scriptsize{separating nodes: $s,c,f,h,i,j,k,\ell$}}
			(1,1.2) node[right] {\scriptsize{non-separating nodes: $\fp,d,g,\fq$}}
			(1,.8) node[right] {\scriptsize{$\de_1\!\wedge\!\de_2\eq k$}}
			(1,.4) node[right] {\scriptsize{$\de_1\!\wedge\!\de_9\eq j$}}
			(1,0) node[right] {\scriptsize{$\de_1\!\wedge\!\de_7\eq\de_1\!\wedge\!\de_8\eq h$}}
			(1,-.4) node[right] {\scriptsize{$\ze_{[\de_1\wedge\de_3]}\eq 1$}}
			(1,-.8) node[right] {\scriptsize{$N_{[\de_1]}\eq \{h,j,k\}$}}
			(1,-1.2) node[right] {\scriptsize{$N_{[\de_1,\de_5]}\eq \{h,j,k,s,f\}$}}
			(5.5,2) node[right] {\scriptsize{$N_{[\de_6,\fp]}\eq \{\fp,d,g\}$}}
			(5.5,1.6) node[right] {\scriptsize{$N_{[\de_6,\fq]}\eq \{\fq\}$}}
			(5.5,1.2) node[right] {\scriptsize{$\lr{\de_1}\eq\lr{\de_2}\eq h$}}
			(5.5,.8) node[right] {\scriptsize{$\lr{\de_5}\eq f$}}
			(5.5,.4) node[right] {\scriptsize{$\lr{\de_6}\eq\de_6$}}
			(5.5,0) node[right] {\scriptsize{$\ze_{[z_{\lr{\de_5}},\fp]}\eq\ze_{\fp}\ze_{d}$}}
			(5.5,-.4) node[right] {\scriptsize{$\ze_{[z_{\lr{\de_5}},\fq]}\eq\ze_{\fq}\ze_{g}$}}
			(5.5,-.8) node[right] {\scriptsize{$\ze_{[z_{\lr{\de_6}},\fp]}\eq\ze_{\fp}\ze_{d}\ze_{g}$}}
			(5.5,-1.2) node[right] {\scriptsize{$\ze_{[z_{\lr{\de_6}},\fq]}\eq\ze_{\fq}$}}
			;			
		\end{tikzpicture}
	\end{center}
	\caption{An example of the notation on nodal curves}\label{Fig:nodal}
\end{figure}

\begin{lemm}\label{Lm:W_smoothness}
Let $1\!\le\!i\!\le\!\ell$, $x\eq(C,\de_1,\ldots,\de_\ell) \inn\ov\cM_{2,\ell}$, and $\cV\!\lra\!\ov\cM_{2,\ell}$ be a sufficiently small smooth chart containing $x$.
Let $\{\ze_q\}$ be a set of modular parameters (corresponding to all the nodes of~$C$) on $\cV$ as in~\S\ref{SubsecPhi}.
We then have the following.
\begin{enumerate}
[leftmargin=*,label=(\arabic*)]
\item 
There exists $\ka\inn\Ga(\sO_\cV)$ such that
$$\ov\cW_{\ell;i}\cap\cV=(\ka= 0).$$
In other words,
$\ov\cW_{\ell;i}$ is a Cartier divisor of $\ov\cM_{2,\ell}$.

\item 
If $\de_i$ is on a separating bridge of the core of $C$,
then $x\!\not\in\!\ov\cW_{\ell;i}$.
In other words,
a Weierstrass marked point cannot be on any separating bridge (of the core).

\item \label{Cond:W_sep_bri}
Assume $x\inn\ov\cW_{\ell;i}$, and $\lr{\de_i}$ as in (\ref{Eqn:lr}) does not lie on any non-separating bridge.
Then, $\ov\cW_{\ell;i}$ is smooth at~$x$.
In this case, $\ka\!\sqcup\!\{\ze_q\}$ form a subset of independent local parameters on $\cV$.

\item \label{Cond:W_non_sep_bri}
If $\lr{\de_i}$ lies on a maximal non-separating bridge $\tn B[\fp,\fq]$, then $x\inn \ov\cW_{\ell;i}$.
In this case,
with $z_{\lr{\de_i}}$ denoting a smooth point on $\tn B[\fp,\fq]$ that is sufficiently close to $\lr{\de_i}$,
there exist $f,g\inn\Ga(\sO^*_\cV)$ such that
$$
 \ka=
 f\ze_{[z_{\lr{\de_i}}\,,\,\fp]}+g\ze_{[z_{\lr{\de_i}}\,,\,\fq]}.
$$
\end{enumerate} 
\end{lemm}

We remark in the right-hand side of the above identity,
the products $\ze_{[z_{\lr{\de_i}}\,,\,\fp]}$ and $\ze_{[z_{\lr{\de_i}}\,,\,\fq]}$ as defined in (\ref{Eqn:bridge_product}) only involve the core nodes on $\tn B[\fp,\fq]$;
see Figure~\ref{Fig:nodal} for example.
We further point out that the reason to
use $z_{\lr{\de_i}}$
instead of ${\lr{\de_i}}$ in $\ze_{[z_{\lr{\de_i}}\,,\,\fp]}$ and $\ze_{[z_{\lr{\de_i}}\,,\,\fq]}$
is to make it clear that $\ze_{\lr{\de_i}}$ does not divide either of the products when
$\lr{\de_i}$ is a pivotal node.

When $\lr{\de_i}$ lies on a maximal non-separating bridge $\tn B[\fp,\fq]$, Lemma~\ref{Lm:W_smoothness} implies $\ov\cW_{\ell;i}$ is singular unless the irreducible component containing $\lr{\de_i}$ also contains $\fp$ or~$\fq$.
For instance, suppose the maximal non-separating bridge $\tn B[\fp,\fq]$ consists of one smooth rational subcurve. There then exist $f,g\inn\Ga(\sO^*_\cV)$
$$
 \ov\cW_{\ell;i}\cap\cV=\big\{
 f\ze_{\fp}+g\ze_{\fq}=0
 \big\}.
$$
Hence, in this case,  $\ov\cW_{\ell;i}$ is smooth at $x$.
Consider another example.
Suppose the maximal non-separating bridge $\tn B[\fp,\fq]$ is made of three smooth rational subcurves $R_a$, $R_b$ and $R_c$ with
$\fp \in R_a$,  $\fq \in R_c$, $\de_i \in R_b$. We let $q_a=R_a \cap R_b$ and $q_b= R_b \cap R_c$.
There then exist $f,g\inn\Ga(\sO^*_\cV)$
$$
 \ov\cW_{\ell;i}\cap\cV=\big\{
 f\zeta_{q_a}\ze_{\fp}+g \zeta_{q_b}\ze_{\fq}= 0
 \big\}.
$$
Hence, in this case,  $\ov\cW_{\ell;i}$ is singular at $x$.

To prove Lemma~\ref{Lm:W_smoothness},
we need the following local property of the forgetful maps.
Consider the forgetful map
$\wh\pr\!:\ov\cM_{2,\ell}\!\lra\!\ov\cM_{2,\ell-1}$
that forgets the last marked point.
Given $\wh x\eq(\wh C,\wh\de_1,\ldots,\wh\de_\ell)$,
we write $\wh\pr(\wh x)\eq(\wh C',\wh\de_1',\ldots,\wh\de_{\ell-1}')$.
Since $\wh\pr$ can be identified with the universal curve $\ov\cC_{2,\ell-1}\!\lra\!\ov\cM_{2,\ell-1}$,
the $\ell$-marked curve $\wh x$ can be identified with a pair $(\wh\pr(\wh x),\wh\de_\ell')$,
where $\wh\de_\ell'$ can be either a smooth point or a node of $\wh C'$.
We take smooth charts $\wh\cU\!\lra\!\ov\cM_{2,\ell}$ and $\wh\cV\!\lra\!\ov\cM_{2,\ell-1}$ centered at $\wh x$ and $\wh\pr(\wh x)$, respectively.

\begin{lemm}\label{Lm:Forgetful_local}
Shrinking $\wh \cU$ and $\wh \cV$ if necessary,
we have the following.
\begin{itemize}
[leftmargin=*]
\item If $\wh\de_\ell'$ is a smooth point of $\wh C'$,
every subset $\ts S$ of independent local parameters on $\wh\cV$ pulls back to a subset of independent local parameters on $\wh\cU$ via $\wh\pr$.
\item If $\wh\de_\ell'$ is a node $q$ of $\wh C'$,
then $\wh\de_\ell$ is the marked point on a smooth rational component of $\wh C$ that contains exactly two nodal points $q_1$ and $q_2$.
With $\ze_q$ denoting the modular parameter corresponding to the smoothing of the node $q$,
if $\ts S\!\sqcup\!\{\ze_q\}$ is a subset of independent local parameters on $\wh\cV$,
then $\ti{\ts S}\!\sqcup\!\{\ze_{q_1},\ze_{q_2}\}$ forms a subset of independent local parameters of $\wh\cU$, where $\ti{\ts S}$ is the pullback of $\ts S$ via $\wh\pr$.
In addition, $\ze_q$ pulls back to $\ze_{q_1}\ze_{q_2}$ up to a unit.
\end{itemize}
\end{lemm}

The statements of Lemma~\ref{Lm:Forgetful_local} are standard by the deformation theory of nodal curves with marked points,
hence the proof is omitted.

\begin{proof}[Proof of Lemma~\ref{Lm:W_smoothness}]
The first statement holds because $\ov\cW$ is a Cartier divisor of $\ov\cM_{2,1}$.

To see the second statement,
just notice that for every $(C',\de')\inn\ov\cW$,
$\de'$ cannot be on any separating bridge $\tn B$ of $C'$ (should $\tn B$ exist).
This can be verified directly from the admissible covers,
but for the convenience of showing the last statement,
we cite~\cite[Proposition~3.12]{BCGGM},
which lists all the boundary components of $\ov\cW$; see~\cite[Figure~11]{BCGGM} particularly.

It remains to check the smoothness of $\ov\cW_{\ell;i}$.
Let $\ov x\eq(\ov C,\ov\de_i)\inn\ov\cM_{2,1}$ be the image of $x$  by forgetting all but the $i$-th marked point.
Choose a smooth chart $\ov\cV\!\subset\!\ov\cM_{2,1}$ centered at $\ov x$ and a local parameter
$\ov\ka$ on $\ov\cV$ such that
$$
 \ov\cW\cap\ov\cV=\{\ov\ka\eq 0\}.
$$

If $\ov x\inn\ov\cW$ and $\lr{\de_i}$ is not on any non-separating bridge,
then $\ov x$ is either in the main component of $\ov\cW$ and/or in the boundary components (I), (II) of~\cite[Proposition~3.12]{BCGGM}.
In any situation,
the local parameter $\ov\ka$,
along with all the modular parameters corresponding to the nodes of $\ov C$,
forms a subset of independent local parameters on $\ov\cV$.
Applying Lemma~\ref{Lm:Forgetful_local} inductively,
we see that $\ov\ka$ pulls back to a local parameter $\ka$ on $\cV$,
which, by the definition of $\ov\cW_{\ell;i}$,
locally defines $\ov\cW_{\ell;i}$:
$$
 \ov\cW_{\ell;i}\cap\cV=\{\ka\eq 0\}.
$$
In addition, Lemma~\ref{Lm:Forgetful_local} also implies that $\ka\!\sqcup\!\{\ze_q\}$ form a subset of independent local parameters on $\cV$.

If $\lr{\de_i}$ is on a maximal  non-separating bridge $\tn B[\fp,\fq]$,
then $\ov x$ belongs to the boundary component (III) of~\cite[Proposition~3.12]{BCGGM};
i.e.~$\ov C$ is the union of a connected subcurve $\ov C_1$ of arithmetic genus~1 and a smooth rational subcurve $\ov R$ containing $\ov\de_i$ such that $\ov C_1\!\cap\!\ov R\eq\{\ov\fp,\ov\fq\}$.
Let $\ze_{\ov\fp}, \ze_{\ov\fp}$ be corresponding modular parameters on $\ov\cV$.
The description of (III) of~\cite[Proposition~3.12]{BCGGM} implies that 
$$
 \ov\cW\cap\ov\cV\supset\{\ze_{\ov\fp}\eq\ze_{\ov\fq}\eq 0\}.
$$
This implies $\ov\ka\eq 0$ whenever $\ze_{\ov\fp}\eq\ze_{\ov\fq}\eq 0$,
so there exist $f,g\inn\Ga(\sO_{\ov\cV})$ such that
$$
 \ov\ka=f\ze_{\ov\fp}+g\ze_{\ov\fq}.
$$
Since $\ov\ka$ is a local parameter, shrinking $\ov\cV$ if necessary,
we see that one of $f$ and $g$ must be a unit,
so are both of them by symmetry.
Applying Lemma~\ref{Lm:Forgetful_local} inductively,
we see that $\ze_{\ov\fp}$ and $\ze_{\ov\fq}$ pull back to $\ze_{[\lr{\de_i},\fp]}$ and $\ze_{[\lr{\de_i},\fq]}$, respectively, each up to a unit.
This completes the last statement of Lemma~\ref{Lm:W_smoothness}.
\end{proof}

	We remark that if $\lr{\de_i}$ is on a maximal non-separating bridge $\tn B[\fp,\fq]$, sometimes it is convenient to use $f\ze_{\ov\fp}$ and $g\ze_{\ov\fq}$ as new modular parameters instead of $\ze_{\fp}$ and $\ze_{\fq}$, so that $\ka$ can be rewritten as
	\begin{equation*}
		\ka=\ze_{[z_{\lr{\de_i}}\,,\,\fp]}+\ze_{[z_{\lr{\de_i}}\,,\,\fq]}
	\end{equation*}
	in this situation.

\begin{defi}\label{Dfn:conjugate}
Two distinct points $\de_1$ and $\de_2$ on a smooth genus 2 curve $C$ are said to be \ts{conjugate} if
$$
 h^0\big(C,\sO_C(\de_1+\de_2)\big)=2.
$$
We denote by $\cK$ the locus of smooth genus 2 curves with a pair of conjugate marked points and
by $\ov\cK$ the closure of $\cK$ in $\ov\cM_{2,2}$.
For every two-marked stable curve $(C,\de_1,\de_2)\inn\ov\cK$, we say the marked points $\de_1$ and $\de_2$ are \ts{conjugate}.
\end{defi}

Similar to the above discussions of $\ov\cW$, $\ov\cK$ is isomorphic to $\overline{\cM}_{0,7}/S_6$ as coarse moduli spaces, 
where the 7-th marking corresponds to the image of a pair of conjugate points. 
In particular, $\ov\cK$ is a smooth divisor in $\overline{\cM}_{2,2}$.
Moreover, given $(C,\de_1,\de_2)\inn\ov\cK$,
the possible locations of the two marked points are as follows.

\begin{lemm}
	\label{Lm:K_position}
	Given $(C,\de_1,\de_2)\inn\ov\cM_{2,2}$, we have the following:
	\begin{itemize}[leftmargin=*]
		\item Assume $\de_1$ is on a tail,
		which implies $\de_2$ is on the same tail by stability.
		Then, $(C,\de_1,\de_2)\inn\ov\cK$ if and only if  $\lr{\de_1}$ ($=\!\lr{\de_2}$) is a Weierstrass point of the core $F$ of $C$.
		\item Assume $\de_1$ is on a genus 1  inseparable component of $C$, and  $(C,\de_1,\de_2)\inn\ov\cK$. 
		Then, $\de_2$ is on the same genus 1 inseparable component.
		\item Assume $\de_1\inn F$ is a Weierstrass point. 
		Then, $(C,\de_1,\de_2)\inn\ov\cK$ if and only if there exists a  non-separating bridge (of $F$) containing both $\de_1$ and $\de_2$.
		\item 
		Assume $\de_1$ is on a separating bridge (of $F$), and  $(C,\de_1,\de_2)\inn\ov\cK$. 
		Then, $\de_1$ and $\de_2$ are on the same irreducible component.
	\end{itemize}
\end{lemm}
\begin{proof}
All the statements of Lemma~\ref{Lm:K_position} follow  immediately from \cite[Theorem 1.3]{BCGGM}.

Consider the first case when $\de_1$ and  $\de_2$ are on a tail $R$,
which implies $R$ is smooth.
In this case,
we further assume $F$ is smooth;
the sub-case when $F$ is singular follows from the closedness of $\ov\cK$ and $\ov\cW$.
By~\cite[Corollary~1.4]{BCGGM},
we see
$(C,\de_1,\de_2)\inn\ov\cK$
if and only if
there exists a pointed stable differential of type $(1,1)$ over $(C,\de_1,\de_2)$, satisfying the conditions in~\cite[Theorem 1.3]{BCGGM}.
Notice that on $R$, there always exists a (non-zero) meromorphic differential with simple zeros at $\de_1$ and $\de_2$,
as well as a pole at $q\!:=\!\lr{\de_1}$ whose order and residue are respectively $-4$ and 0.
Analyzing the five conditions of~\cite[Definitions 1.1 and 1.2]{BCGGM},
we conclude that $(C,\de_1,\de_2)\inn\ov\cK$
if and only if  there exists a (non-zero) holomorphic differential on $F$ with a zero of order 2 at $q$ (because $F$ is assumed to be smooth),
i.e.~$h^0(F,K_F\!-\!2q)\eq 1$,
which is equivalent to $(F,q)\inn\ov\cW$.

The proofs of the remaining cases of Lemma~\ref{Lm:K_position} are similar, hence are omitted.
\end{proof}

For $1\!\le\!i,j\!\le\!\ell$ with $i\!\ne\!j$, we set $\ov\cK_{\ell;i,j}\!\subset\!\ov\cM_{2,\ell}$ to be the pullback of $\ov\cK$ in $\ov\cM_{2,\ell}$ via the forgetful morphism that forgets all but the $i$-th and $j$-th marked points. 
If $x\eq(C,\de_1,\ldots,\de_\ell)\inn\ov\cK_{\ell;i,j}$,
we say $\de_i$ and $\de_j$ are \ts{conjugate}.

Parallel to Lemma~\ref{Lm:W_smoothness},
we have the following lemma for $\ov\cK_{\ell;i,j}$.

\begin{lemm}\label{Lm:K_smoothness}
Let $1\!\le\!i,j\!\le\!\ell$ with $i\!\ne\!j$, $x\eq(C,\de_1,\ldots,\de_\ell)\inn\ov\cM_{2,\ell}$, and $\cV\!\lra\!\ov\cM_{2,\ell}$ be a sufficiently small smooth chart containing $x$.
Let $\{\ze_q\}$ be a set of modular parameters (corresponding to all the nodes of~$C$) on $\cV$ as in~\S\ref{SubsecPhi},
and $\lr{\de_i}\inn F$ be as in (\ref{Eqn:lr}) for all $i$.
We then have the following.
\begin{enumerate}
[leftmargin=*]
\item 
	There exists $\ka\inn\Ga(\sO_\cV)$ such that
	$$\ov\cK_{\ell;i,j}\cap\cV=(\ka= 0).$$
	In other words,
$\ov\cK_{\ell;i,j}$ is a Cartier divisor of $\ov\cM_{2,\ell}$.
\item \label{Cond:K_param}
If $x\inn\ov\cK_{\ell;i,j}$ and $\lr{\de_i}$ and $\lr{\de_j}$ do not lie on any non-separating bridge,
then $\ov\cK_{\ell;i,j}$ is smooth at $x$.
In this case, $\ka\!\sqcup\!\{\ze_q\}$ form a subset of independent local parameters on $\cV$.
\item\label{Cond:K_nonsep_brid}
If $\lr{\de_i}$ lies on a maximal non-separating bridge $\tn B[\fp,\fq]$,
then $x\inn\ov\cK_{\ell;i,j}$ if and only if $\lr{\de_j}$ lies on $\tn B[\fp,\fq]$.
Moreover,
if 
$\lr{\de_i}, \lr{\de_j}\in \tn B[\fp,\fq]$,
then by taking two arbitrary smooth points $z_{i},z_{j}\inn\tn B[\fp,\fq]$ that are sufficiently close to $\lr{\de_i}$ and $\lr{\de_j}$, respectively,
	there then exist $f,g\inn\Ga(\sO^*_\cV)$ such that
	$$
	\ka=
	f\cdot\prod_{q\in N_{[z_i,\fp]}\cap N_{[z_j,\fp]}}\!\!\!\!\!\!\!\!\!\!\ze_q
	\ \ +\ g\cdot\prod_{q\in N_{[z_i,\fq]}\cap N_{[z_j,\fq]}}\!\!\!\!\!\!\!\!\!\!\ze_q.
	$$
\end{enumerate} 
\end{lemm}

The proof of Lemma~\ref{Lm:K_smoothness} is parallel to that of Lemma~\ref{Lm:W_smoothness}, hence is omitted.
Although
$\lr{\de_i}$ and $\lr{\de_j}$ are possibly on a separating bridge of $C$ by Lemma~\ref{Lm:K_position},
this fact does not affect the proof of the second statement of Lemma~\ref{Lm:K_smoothness}. 
What actually matters is that the local parameter (locally) defining $\ov\cK$  is independent of the modular parameters on $\ov\cM_{2,2}$,
unless the two marked points are on a non-separating bridge,
which is dealt with in the last statement of Lemma~\ref{Lm:K_smoothness}.

We conclude this subsection with a relation between $\ov\cW_{\ell;i}$'s and $\ov\cK_{\ell;i,j}$'s.

\begin{lemm}\label{Lm:KW}
Let $1\!\le\!i,j\!\le\!\ell$, $i\!\ne\!j$,
and $x\eq(C,\de_1,\ldots,\de_\ell)\inn\ov\cM_{2,\ell}$.
If $\de_i,\de_j$ are on the same (rational) tail,
then 
$$
 x\in\ov\cK_{\ell;i,j}
 \quad\Longleftrightarrow\quad
 x\in\ov\cW_{\ell;i}
 \quad\Longleftrightarrow\quad
 x\in\ov\cW_{\ell;j}.
$$
\end{lemm}

\begin{proof} 
	The statements follow from  the first statement of Lemma~\ref{Lm:K_position} via pullback.
\end{proof}

\begin{rema}\label{Rmk:D2_vs_M2l}
	As mentioned at the beginning of this subsection,
	all the statements for $\ov\cW_{\ell;i}$ and $\ov\cK_{\ell;i,j}$ can be applied to $\fM^{\rm div}_2$ verbatim.
	
	More precisely, a point
	$(C,D)\inn \fM^{\rm div}_{2}$ is allowed to have  smooth (tail) rational subcurves that each contains exactly one nodal point of $C$ and exactly one element of $D$,
	hence may not be a stable curve with marked points.
%
	By writing $D\eq\{\de_1,\ldots,\de_m\}$ (recall $D$ is simple and effective), w.l.o.g.~we assume there exists $\ell\!\le\!m$ such that each $\de_i$, $1\!\le\!i\!\le\!\ell$, is on a smooth rational subcurve $R_i$ that contains exactly one nodal point $q_i$ of $C$ and satisfies $D\!\cap\!R_i\eq\{\de_i\}$.
	We denote by $(\ov C,\ov D\eq\{\ov\de_1,\ldots,\ov\de_m\})$ the image of $(C,D)$ in $\ov\cM_{2,m}$ after
	applying stabilization.
	Then, the aforementioned $R_i$'s are all contracted in $(\ov C,\ov D)$.
	
	For distinct $1\!\le\!j,k\!\le\!m$,
	we say $\de_j$ is \ts{Weierstrass} (resp.~$\de_j$ and $\de_k$ are \ts{conjugate})
	if so is $\ov\de_j$ (resp.~so are $\ov\de_j$ and $\ov\de_k$).
	It is straightforward that
	Lemmas~\ref{Lm:W_smoothness}, \ref{Lm:K_smoothness} and~\ref{Lm:KW} still hold for $\fM^{\rm div}_2$;
	c.f.~the proof of \cite[Lemma~4.2]{HN2} for details.
	
	In the remainder of this paper, when a property for the Weierstrass or conjugate loci on $\fM^{\rm div}_2$ is presented (e.g.~Corollary~\ref{Crl:str_homom_one_tail_chain_minor_12vs23} and Proposition~\ref{Prp:phi_key}~\ref{Part:kappa}),
	it suffices to prove the same property on $\ov\cM_{2,m}$.
\end{rema}

\subsection{Structural homomorphisms: general results}
\label{Subsec:phi_summary}

At this stage, we are ready to present the general results on the structural homomorphism 
	\begin{align*}
		\varphi:\;
		\rho_*\sM(\cD)\;
		\big(=\rho_*\sO_\cC(\cD+\cA-\cB)\big)
		\lra
		\rho_*\sO_\cA(\cA)
	\end{align*}
	in Proposition~\ref{Prp:phi_key}.
	This key proposition,
	along with Lemma~\ref{Lm:K_smoothness}, contains everything we need for diagonalizing $\varphi$ (considered as a $2\!\times\!m$ matrix):
	its entries, $2\!\times\!2$ minors, and the relations between them. 
	
	We continue with the notation from previous subsections of \S\ref{SecPhi}. 
	We still fix an arbitrary point of $\fM_2^{\rm div}$ and an affine smooth chart containing it:
	$$ z_0=\big(C\,,\,D=\de_1+\cdots+\de_m\big)\in\cV\subset\fM_2^{\rm div}.$$
	We always assume $\cV$ is sufficiently small.
	Recall that
	the zero locus of $\ka_{ij}\inn\Ga(\sO_\cV)$ consists of $z\inn\cV$ with $\cD_i(z)$ and $\cD_j(z)$ being a pair of conjugate points on $\cC_z$.
	Also recall that $\de_i\!\wedge\!\de_j$ refers to the node in~(\ref{product}), which
	intuitively is a node between the core $F$ of $C$ and both $\de_i$ and $\de_j$, and it is the ``farthest away'' from the core among such nodes.
	Finally, for each $\de_i$, recall $\lr{\de_i}$ is the ``projection'' of $\de_i$ onto the core $F$ of $C$ as in (\ref{Eqn:lr}).
	
	\begin{prop}\label{Prp:phi_key}
		With notation as above, we have the following.
		\begin{enumerate}[label=($\rm{SH}_{\arabic*}$),leftmargin=*]
			\item \label{Part:phi}
			Under suitable trivialization of $\rho\lsta\sM(\cD)$ and $\rho\lsta\sO_\cA(\cA)$,
			$\varphi$ takes the form
			\begin{align*}
				\varphi=\left[\begin{matrix}
					c_{11}\ze_{[\de_1,a_1]} & c_{12}\ze_{[\de_2,a_1]} & \cdots
					& c_{1i}\ze_{[\de_i,a_1]} & \cdots\\
					c_{21}\ze_{[\de_1,a_2]} & c_{22}\ze_{[\de_2,a_2]} & \cdots
					& c_{2i}\ze_{[\de_i,a_2]} & \cdots
				\end{matrix}\right].
			\end{align*}
			\item \label{Part:theta}
			For any distinct $1\!\le\!i,j\!\le\!m$,
			there exists $u\inn\Ga(\sO_\cV^*)$ such that
			the determinant
			$\dt{ij}\!:=\!\det\left[\begin{matrix}
				c_{1i} & c_{1j}\\ c_{2i} & c_{2j}\end{matrix}\right]$ satisfies
			\begin{align*}
				\dt{ij}=u\cdot\ka_{ij}\cdot\ze_{[\de_i\wedge\de_j]}.
			\end{align*}
			\item \label{Part:kappa}
			For any pairwise distinct $1\!\le\!i,j,k\!\le\!m$,
			there exist $v\inn\Ga(\sO_\cV^*)$ 
			and $w\inn\Ga(\sO_\cV)$ so that
			$$
			\ka_{ik}=v\cdot\ka_{ij}+w,
			$$ and the following holds.
			\begin{enumerate}[label=$\bullet$,leftmargin=*]
				\item 
				If $\lr{\de_i}$, $\lr{\de_j}$ and $\lr{\de_k}$ all lie on a maximal non-separating bridge $\tn B[\fp,\fq]$ of the core of $C$,
				let $x_{i}$ and $x_j$ be smooth points  on the core $F$ that are sufficiently close to $\lr{\de_i}$ and $\lr{\de_j}$, respectively. 
				Then, there exist $f, g\inn\Ga(\sO^*_\cV)$ and $f',g'\inn\Ga(\sO_\cV)$ so that
				\begin{align*}
					\ka_{ij}
					&=f\cdot\prod_{q\in N_{[x_i,\fp]}\cap N_{[x_j,\fp]}}\!\!\!\!\!\!\!\!\!\!\ze_q
					\ \ +\ g\cdot\prod_{q\in N_{[x_i,\fq]}\cap N_{[x_j,\fq]}}\!\!\!\!\!\!\!\!\!\!\ze_q\quad,
					\\
					w
					&=f'\cdot\prod_{q\in N_{[\de_j,\fp]}\cap N_{[\de_k,\fp]}}\!\!\!\!\!\!\!\!\!\!\ze_q
					\ \ +\ 
					g'\cdot\prod_{q\in N_{[\de_j,\fq]}\cap N_{[\de_k,\fq]}}\!\!\!\!\!\!\!\!\!\!\ze_q\quad;
				\end{align*}
				moreover, 
				$\det\left[\begin{matrix}
					f & f'\\ g & g'
				\end{matrix}\right]\inn\Ga(\sO^*_\cV)$.
				
				\item 
				Otherwise, there exists $u'\inn\Ga(\sO^*_\cV)$ such that $w=u'\cdot\ze_{[\de_j\wedge\de_k]}$.
			\end{enumerate} 
		\end{enumerate}
	\end{prop}
	
	Here, \ref{Part:phi} is a restatement of Proposition~\ref{HomFirstOrder}.
	To establish~\ref{Part:theta} and~\ref{Part:kappa},
	we shall first analyze the entries and $2\!\times\!2$ minors of $\varphi$ along a maximal chain of smooth rational subcurves on a given tail in \S\ref{Subsec:using-ss} and \S\ref{Subsec:construct-rational-tails}, respectively;
	we will then complete the proof of Proposition~\ref{Prp:phi_key} in \S\ref{Subsec:phi_proof}.
	
	Before delving into the technical proofs in \S\ref{Subsec:using-ss}-\ref{Subsec:phi_proof},
	we present a simple application of Proposition~\ref{Prp:phi_key} in this subsection,
which reveals where $\ov M_2(\P^n,d)$ is smooth.
We follow the notation in \S \ref{subsect:localEq}.
\begin{prop}\label{Prp:genericSmoothness} 
	Assume $(C,D) \in \fM_2^{\rm div}$ satisfies one of the following: 
	\begin{itemize}[leftmargin=*]
		\item the core $F$ of $C$ is inseparable and $D\!\cap\!F$ contains a pair of non-conjugate points,
		\item $F$ is separable, hence contains two genus-one inseparable components $F_1$ and $F_2$ (see Definition~\ref{Dfn:nodal_curve}), and 
		$D\!\cap\!F_i \ne \emptyset$ for $i=1,2$.
	\end{itemize}
	Then by shrinking $\cV$ if necessary, we see that  $\cU\eq \cV \times_{\fM_{2}^{\rm div}} U$ is smooth.
	In particular, the main component of $\MPdd$ is generically smooth.
	When $d\!>\!2$,
	the main component is also of expected dimension.
\end{prop}
\begin{proof}  We can trivialize $
	\rho_*\sM(\cD)$ and $ \rho_* \sO_{\cA}(\cA)$ so that
	$$\sE_\cV \cong \cV \times (\AA^{d+1})^n.$$
	Let $w_j^i \in \Ao$ for all $1 \le i \le n$ and $0 \le j \le d$.
	Recall that by \eqref{0plusphi-0}, the homomorphism
	$``$rest$"$ of \eqref{restHom-0} can be identified with the
	homomorphism $$(0 \oplus \varphi)^{\oplus n}: (\sO_\cV \oplus \rho_*\sM(\cD))^{\oplus n} \lra (\rho_*
	\sO_{\cA}(\cA))^{\oplus n}. $$ 
	
	Suppose $F$ is inseparable and $D\!\cap\!F$ contains a pair of non-conjugate points.
	By Proposition \ref{Prp:phi_key}~\ref{Part:phi} and~\ref{Part:theta},
	$0 \oplus \varphi$ can be represented by
	$$0\oplus \varphi=[0, I_2, 0 \cdots, 0].$$  
	
	Suppose $F$ contains two genus-one inseparable components $F_1$ and $F_2$, and 
	$D\!\cap\!F_i \ne \emptyset$ for $i=1,2$.
	W.l.o.g.~we assume
	$$
	\de_1,\,a_1\in F_1,\qquad
	\de_2,\,a_2\in F_2.
	$$
	Therefore,
	$\ze_{[\de_1,a_1]}\eq\ze_{[\de_2,a_2]}\eq 1$ (see the sentence below (\ref{product})),
	whereas $\ze_{[\de_1,a_2]}$ and $\ze_{[\de_2,a_1]}$ vanish at $(C,D)$.
	By Proposition~\ref{Prp:phi_key}~\ref{Part:phi}, $0 \oplus \varphi$ can thus be represented by
	$$0\oplus \varphi=[0, I_2, 0 \cdots, 0].$$  Hence, in either case,  the equation $\Phi =0$ is equivalent to the following equations
	\beq\label{finalEq}
	w_1^i = w_2^i=0 \quad \hbox{for all $1 \le i \le n$},
	\eeq
	which obviously implies that $\cU$ is isomorphic to an open subset of $\cV \times (\AA^{d-1})^n$. Note that
	$\dim (\cV \times (\AA^{d-1})^n)= 3 + d + (d-1)n = d(n+1) -n+3$, which is the expected dimension of $\MPdd$ when $d\!>\!2$.
	The assertion then follows.
\end{proof}

All the sequential blowups  in~\S\ref{SecGlobal} will be performed over the loci of $\mwt$ or $\fdd$ {\it outside} the images
of the loci described in Proposition~\ref{Prp:genericSmoothness}
under the morphism $\MPdd \!\lra\! \mwt$ as in (\ref{MPddToWeight0}) and the local morphisms $\MPdd\!\supset\!U \!\lra\! \fdd$ as in (\ref{toP0}).

\subsection{Structural homomorhpisms: applying semi-stabilization}\label{Subsec:using-ss}

In this subsection, 
we aim to extract information of  the matrix~(\ref{explicitHom}) of the homomorphism $\varphi$.
The main result of this subsection,
	Proposition~\ref{Prp:key-tail_chain}, 
	provides a strengthened form (compared with Proposition~\ref{HomFirstOrder}) of  each row of the restriction $\hmr$ of the homomorphism~$\varphi$.
	This is the first step towards the proof of Proposition~\ref{Prp:phi_key}.

We fix $(C,D)\inn\fM_2^{\rm div}$ and an affine smooth chart $\cV\lra\fM_2^{\rm div}$ containing $(C,D)$.
We also fix a rational tail $T$ of $C$.
We  focus on a maximal chain $R \subset {T}$ of smooth rational
subcurves
$$R=R_1 \cup \cdots \cup R_h$$ of $C$
attached to a smooth point of the core $F \subset C$.
Recall that a bridge is always a part of the core $F$ of $C$,
whilst a rational tail meets the core $F$ only at one point (indeed a pivotal node) per definition.
In particular, the chain $R$ is not contained in any bridge.

We fix some 
notation.
First, we index the irreducible components of $R$ so that $R_1$ is attached to $F$, and $R_{{\mu}+1}$ is attached
to $R_{\mu}$. 
Let
\begin{equation}\label{e_q}
	q_1=F\cap R_1,\qquad
	q_{{\mu}+1}=R_{\mu}\cap R_{{\mu}+1}\quad
	\forall~1\!\le\!{\mu}\!\le\!h\!-\!1.
\end{equation}
In particular,
the node $q_1$ is a pivotal node.

Next, recall that $D$ is simple and effective.
We fix a sub-divisor of $D$ (i.e.~a subset of $D$):
$$
 \{\de_1,\cdots,\de_r\} \subset D.
$$
Since $\deg D\eq m$, we have $1\!\le\!r\!\le\!m$.
Throughout this subsection, we assume 
\begin{equation}\begin{split}\label{linearOrder}
	&
	\{\delta_1, \cdots,\delta_r\}\subset R,\\
	&
	\{\delta_1, \cdots,\delta_{r_1}\}\subset R_1\,,\quad
	\{\delta_{r_1+1},\cdots,\delta_{r_1+r_2}\}\subset R_2\,,\ \ 
	\cdots,\\
	&
	\{\delta_{r_1+\cdots+r_{h-1}+1}, \cdots,\delta_{r_1+\cdots+r_{h-1}+r_h}\eq\de_r\}\subset R_h\,.
\end{split}\end{equation}
This implies $r\eq r_1\!+\!\cdots\!+\!r_h$.
We emphasize that every $r_\mu$ is possibly zero.
Intuitively, the assumption (\ref{linearOrder}) means $\{\delta_1, \cdots,\delta_r\}$ is a subset of $D\!\cap\!R$,
and the larger the index of an element of this subset is,
the ``farther away'' it is from the core $F$.

Let
\begin{equation}\label{e_index}
	\begin{CD}
		e\!:\{1,\cdots, r\}\lra\{1,\cdots,h\}\quad \\
		e(i)=\mu \;\; \tn{if $\delta_i \in R_\mu$ for some $1\le \mu \le h$},
		\;\; \forall \;1 \le i \le r.
	\end{CD}
\end{equation}
Thus, 
\begin{align}\label{Eqn:e(i)}
	\delta_i \in R_{e(i)}\ \ 
	\forall~1\!\le\!i\!\le\!r\quad\tn{and}\quad
	e(i)\le e(j)\ \ \forall~1\!\le\!i\!<\!j\!\le\!r.
\end{align}
We denote by
\beq\label{phi-R}
 \hmr: \bigoplus_{1\le i \le r}
  \rho_* \sM (\cD_i) \lra \rho_* \sO_\cA(\cA)
\eeq
the restriction homomorphism.


\begin{rema}
	To avoid some possible confusion, we mention firmly here that 
	we typically index a member of
    $D$ or $\cD$ by an integer $1\le i$ (or $j$) $\le m$, e.g.~$\delta_i$ or $\cD_j$;
	we typically index a member of nodes on $R$ or its corresponding smooth rational component 
	by an integer $1 \le \mu \le h$,  e.g.~$q_\mu$ or $R_\mu$.
\end{rema}

Recall the modular parameter associated with the node $q_{\mu}$ is denoted by 
$\zeta_{q_{\mu}}$ {for all $1 \le \mu \le h$.}
Since $q_{\mu}$ is separating,
$$\cC\times_\cV(\zeta_{q_{\mu}}=0)=\cG_{\mu}\cup {\cR_{\ge {\mu}}}
$$
is a union of two families of nodal curves over $(\zeta_{q_{\mu}}\eq 0)$ intersecting at (i.e. glued along)
a section of nodes 
$$
\Sigma_{\mu}=\cG_{\mu}\cap {\cR_{\ge {\mu}}}
\qquad\tn{with}\quad
\Sigma_{\mu}\cap C=q_{\mu},\quad
\cR_{\ge {\mu}}\cap C= R_{\mu}\cup\cdots\cup R_h.$$
In other words,
$\cG_\mu\!\cap\!C$ and $\cR_{\ge\mu}\!\cap\!C$ are the two components obtained by normalizing $C$ at the node $q_\mu$,
and $\cR_{\ge\mu}$ is the component that, before the normalization, is contained in the tail $T$.

Let $\fM_{2,\sss}$ be the Artin stack of semi-stable genus two nodal curves. (A nodal curve is semi-stable if
its smooth rational subcurves contain at least two nodes of the nodal curve.) By (successively) contracting rational subcurves
that contain only one node of the ambient curves, we obtain a semi-stabilization
$\fM_2^{\rm div}\lra \fM_{2,\sss}$.
Possibly after shrinking and base changing $\cV$, we can assume that there is a smooth chart
$\underline{\cV}\to \fM_{2,\sss}$ with $\cY\to \underline{\cV}$ its universal family
so that the composite $\cV\to \fM_{2}^{\rm div}\to \fM_{2,\sss}$
{factors through a morphism $g: \cV \to  \underline{\cV}$}
together with a semi-stable contraction $\ti g$ as shown
\beq\label{sq2}
\begin{CD}
	\cV @>>> \fM_{2}^{\rm div} \\
	@VV{g}V  @VV{}V \\
	\underline{\cV} @>>> \fM_{2,\sss}
\end{CD}
, \quad\and  \ti g: \cC\lra \overline{\cC}=\cY\times_{\underline{\cV}}\cV.
\eeq
Note that $\ti g$ contracts all rational tails of the curves in the family $\cC\to\cV$.
Let 
$
\overline{\cD}_i=\ti g(\cD_i)$,
which are family of smooth divisors of $\overline{\cC}\to \cV$, and
let
$$
{{\cD_{[j]}}=\sum_{i=1}^j \cD_i}\sub \cC, \quad {\bcD_{[j]}}=\sum_{i=1}^j \bcD_i\sub \bcC\quad
\forall~1\le j\le r.
$$

Let $\bar\rho: \bcC\eq\cY\!\times_{\ud\cV}\!\cV\to\cV$ be the projection to the second component.
By construction, the divisors $\cR_{\ge 1},\cdots,\cR_{\ge h}$
are contracted under the morphism $\ti g$, and
\begin{align*}
	\label{lift-D}
\ti g\upmo(\bcD_i)=\sum_{\mu=1}^{e(i)}\cR_{\ge\mu}+\cD_i\,,
\end{align*}
where $e$ is the index function~\eref{e_index}.

Recall that each $r_\mu\!\ge\!0$ is the number of points in $R_\mu \cap D$, satisfying
$\sum_{\mu=1}^h  {r_\mu}\eq r$.
For convenience, we set $$r_{[\mu]}:=r_1+\cdots+r_\mu\quad
	\forall~1\le\mu\le h,\qquad
	r_{[0]}:=0.$$
For every $1\!\le\!j\!\le\!r$,
let
\begin{equation}
	\label{Eqn:R_j}	
	\cR_{[j]}=
	\sum_{i=1}^j\sum_{\mu=1}^{e(i)}\cR_{\ge \mu}
	=
	\sum_{\mu=1}^{e(j)}\big(j-r_{[\mu-1]}\big)\cR_{\ge \mu}\,.
\end{equation}
As a particular case, we write
$$\cR:=\cR_{[r]}\quad\bigg(=\sum_{\mu=1}^h (r_\mu+\cdots+ r_h) \cR_{\ge \mu}\bigg).$$
Then, one calculates that for each $1\!\le\!i\!\le\!r$,
\beq\label{equal1}
\ti g\sta \sO_{\bcC}({\bcD_{[i]}}+\bcA-\bcB)=\sO_{\cC}(\cR_{[i]}+{\cD_{[i]}}+\cA-\cB).
\eeq
Thus, we have
\beq\label{basic-inc}
\rho\lsta \sO_{\cC}({\cD_{[i]}}+\cA-\cB)\mapright{\sub} \rho\lsta \sO_{\cC}(\cR_{[i]}+{\cD_{[i]}}+\cA-\cB)=\bar\rho\lsta \sO_{\bcC}({\bcD_{[i]}}+\bcA-\bcB).
\eeq
Here the identity follows from combining \eqref{equal1} and the identities
$$\rho\lsta \ti g\sta  \sO_{\bcC}({\bcD_{[i]}}+\bcA-\bcB)=
\bar\rho\lsta \ti g\lsta \ti g\sta  \sO_{\bcC}({\bcD_{[i]}}+\bcA-\bcB)=\bar\rho\lsta  \sO_{\bcC}({\bcD_{[i]}}+\bcA-\bcB),
$$
where the last equality holds since both $\cC$ and $\bcC$ are smooth, $\ti g: \cC\to\bcC$ is a
proper divisorial contraction.

As before, we write $a_s=F\cap \cA_s$, $a=a_1+a_2$, and $b=F\cap \cB$. Since they are in general position
(see Assumption~\ref{Ass:basic}), we have $H^1(\sO_F(rq_1+a-b))=0$,
where $q_1$ is as in (\ref{e_q}).
Thus $\bar\rho\lsta \sO_{\bcC}({\bcD_{[r]}}+\bcA-\bcB)$ is a free
$\Gamma(\sO_\cV)$-module. For convenience, we denote
$${\bf M}\defeq \Gamma(\rho\lsta \sO_{\cC}({\cD_{[r]}}+\cA-\cB))\ \sub \ \overline{{\bf M}}\defeq\Gamma(\bar\rho\lsta \sO_{\bcC}({\bcD_{[r]}}+\bcA-\bcB)).
$$

We now pick a basis of $\overline{{\bf M}}$.
We start with the point $(C,D)\inn\cV$.
Since $a$ and $b$ are in general positions,
by  a simple vanishing argument, we see 
that
$ H^0(\sO_F(a-b))\eq H^1(\sO_F(a-b))\eq 0$. Hence, the restriction homomorphism
\begin{align}
	\label{Eqn:H0_isom}
H^0\big(\sO_F(r q_1 +a-b)\big) \lra
H^0\big(\sO_F(r  q_1 ) |_{r q_1}\big)
\end{align}
is an isomorphism. 
Moreover, for every $1\!\le\!j\!\le\!r$,
since $H^1(\sO_F(jq_1\!+\!a\!-\!b))\eq 0$,
we have the short exact sequence
\begin{align}\label{Eqn:H0_ses}
	0\to
	H^0\big(\sO_F(jq_1\!+\!a\!-\!b)\big)
	\to
	H^0\big(\sO_F(rq_1\!+\!a\!-\!b)\big)
	\to
	H^0\big(\sO_F(rq_1\!+\!a\!-\!b)|_{(r-j)q_1}\big)
	\to 0.
\end{align}
This determines a complete flag
\begin{align*}
	0=H^0\big(\sO_F(a\!-\!b)\big)
	\subset
	H^0\big(\sO_F(q_1\!+\!a\!-\!b)\big)
	\subset
	H^0\big(\sO_F(2q_1\!+\!a\!-\!b)\big)
	\subset \cdots\subset
	H^0\big(\sO_F(rq_1\!+\!a\!-\!b)\big)\,.
\end{align*}
By choosing
\begin{align*}
	s_i\in 
	H^0\big(\sO_F(iq_1\!+\!a\!-\!b)\big)\big\bsl H^0\big(\sO_F((i\!-\!1)q_1\!+\!a\!-\!b)\big)\qquad
	\forall~1\!\le\! i\!\le\! r,
\end{align*}
it is routine to show they form a basis $$s_1, \,\cdots, s_r$$
of the vector space $H^0\big(\sO_F(rq_1\!+\!a\!-\!b)\big)$,
satisfying for every $1\!\le\!j\!\le\!r$,
\begin{itemize}[leftmargin=*]
	\item $s_j$,  as an element of $H^0\big(\sO_F(rq_1\!+\!a\!-\!b)\big)$, has vanishing order $(r\!-\!j)$ at $q_1$,
	and
	\item 
	$s_1,\ldots,s_j$ forms a basis of 
	$H^0\big(\sO_F(jq_1\!+\!a\!-\!b)\big)$.
\end{itemize}

\begin{rema}\label{Rmk:van_ord}
Notice that 
for every $1\!\le\!i\!\le\!j\!\le\!r$,
$s_i$, as an element of $H^0\big(\sO_F(jq_1\!+\!a\!-\!b)\big)$,
has vanishing order $(j\!-\!i)$ at $q_1$.
Therefore, when studying the vanishing order of a given element (e.g.~$s_i$), we need to specify the module (e.g.~$H^0(\sO_F(jq_1\!+\!a\!-\!b))$ versus $H^0(\sO_F(rq_1\!+\!a\!-\!b))$) containing it. 
The same remark applies to the extensions of $s_1,\ldots,s_r$ to $\cV$ as follows.
\end{rema}

Next, for any  $1 \le \mu \le h$,
let $\cH_\mu=(\zeta_{q_\mu}=0)\sub \cV$ and $\cH=\bigcup_{\mu=1}^h\! \cH_\mu$.
Let
$$\bar\rho_\cH: \bcC_\cH\defeq \bcC\times_\cV\cH\lra \cH
$$
be the projection {and} $\bar \Sigma_\mu=\ti g(\Sigma_\mu)$, where
$\Sigma_\mu=\cG_\mu\cap {\cR_{\ge \mu}}$. For the same reason {as earlier},
$\overline{{\bf M}}_\cH\defeq \Ga\big((\bar\rho_\cH)\lsta\bl \sO_{\bcC_\cH}({\bcD_{[r]}}+\bcA-\bcB)\br\big)$ 
is a free $\Gamma(\sO_{\cH})$-module and has
a basis $\bar S_{\cH,1},\ldots,\bar S_{\cH, r}$ 
extending the basis $s_1,\ldots,s_r$,
satisfying for every $1\!\le\!j\!\le\!r$,
\begin{itemize}[leftmargin=*]
\item $\bar S_{\cH, j}$,
as an element of $\overline{{\bf M}}_\cH$,
has vanishing order 
$(r\!-\!j)$ along $\bar \Sigma_{\mu}$ for all $1\!\le\!\mu\!\le\!e(j)$, and 
\item $\bar S_{\cH,1},\ldots,\bar S_{\cH, j}$ forms a basis of $\Ga\big((\bar\rho_\cH)\lsta\bl \sO_{\bcC_\cH}({\bcD_{[j]}}+\bcA-\bcB)\br\big)$.
\end{itemize}

By shrinking $\cV$ if necessary, we can extend $\bar S_{\cH,i}$ to $\bar S_i\in \overline{{\bf M}}$ so that each $$\ov{\bf M}_j:=\Gamma\big(\bar\rho\lsta \sO_{\bcC}({\bcD_{[j]}}+\bcA-\bcB)\big),$$ as a free $\Gamma(\sO_\cV)$-module, is generated freely by $\bar S_1,\cdots,\bar S_j$.
Thus by \eqref{basic-inc}, this implies  $\ti g\sta\bar S_1,\cdots,\ti g\sta \bar S_j$ freely generate
the free $\Gamma(\sO_\cV)$-module
$\Gamma(\rho\lsta\sO_\cC(\cR_{[j]}+{\cD_{[j]}}+\cA-\cB))$ for all $1\!\le\!j\!\le\!r$.

\begin{lemm}\label{sublemma}
	Let 
	$$
	\zeta_{[\![\delta_i]\!]}=\prod_{\ell=1}^i\zeta_{[\delta_\ell]},\quad
	S_i=\zeta_{[\![{\delta_i}]\!]}\cdot \ti g\sta\bar S_i,\qquad
	1 \le i\le r.
	$$
	Then, each submodule 
	$${\bf M}_j:=\Ga\big(\rho_*\sO_{\cC}\big(\cD_{[j]}+\cA-\cB\big)\big)\subset\overline{{\bf M}}\ \big(=\ov{\bf M}_r\big) $$
	is a free  $\Gamma(\sO_\cV)$-module
	{of rank $j$} generated by
	$S_1,\cdots, S_j$.
	
	Particularly,
	${\bf M}$ ($=\!{\bf M}_r$) is a free  $\Gamma(\sO_\cV)$-module
	{of rank $r$} generated by
	$S_1,\cdots, S_r$.
\end{lemm}

\begin{proof}
	Throughout this proof, we fix an arbitrary $1\!\le\!j\!\le\!r$.
	
	We begin with more convention. For any $1 \le \mu \le h$,
	let $\cU_{\mu} \sub\cC$ be an affine open subset that contains the generic point $\xi_\mu$ of ${\cR_{\ge \mu}}$. Let
	$u_{\mu}\in \Gamma(\sO_{\cU_{\mu}})$ be so that ${\cR_{\ge \mu}}\cap \cU_{\mu}=(u_{\mu}=0)\sub \cU_{\mu}$.
	Since ${\cR_{\ge \mu}}\to\cH_{\mu}$
	is a family of rational curves and has general fibers isomorphic to $\Po$, we can
	find a rational function $v_{\mu}$ on $\cU_{\mu}\cap {\cR_{\ge \mu}}$ so that $v_{\mu}$ restricts to general fibers of
	${\cR_{\ge \mu}}\cap \cU_{\mu}\to \cH_{\mu}$ are birational maps to $\Ao$.
	We next let $\xi_{\mu}'$ be the generic point of $\bar\Sigma_{\mu}$. Since $\ti g(\xi_{\mu})=\xi_{\mu}'$, the field
	$K_{\mu}=\sO_{\xi_{\mu}}$ contains $K_{\mu}'=\sO_{\xi_{\mu}'}$ as its subfield. Because of our choice of $v_{\mu}$,
	we have $K_{\mu}=K_{\mu}'(v_{\mu})$.

	Let $\hat \sO_{\cU_{\mu}}$ be the formal completion of $\sO_{\cU_{\mu}}$ along
	$\xi_{\mu}$. Then $\hat \sO_{\cU_{\mu}}=K_{\mu}[\![ \hat u_{\mu}]\!]$,
	where $\hat u_{\mu}$ is the image of $u_{\mu}$ in $\hat \sO_{\cU_{\mu}}$.
	For any $f\in\Gamma(\sO_\cV)$, we denote by
	$\hat f$ (for ease of notation) its image in $\hat \sO_{\cU_{\mu}}$ via the pullback $\sO_\cV\to\sO_{\cU_{\mu}}$ and the completion map.
	
	We next pick an open $\cU_{\mu}'\sub \bcC$ that contains the generic point $\xi_{\mu}'$ of $\bar\Sigma_{\mu}$.
	By shrinking $\cU_{\mu}$ if necessary, we can assume $\ti g(\cU_{\mu})\sub \cU_{\mu}'$. 
	
	For every $1\!\le\!i\!\le\!j$ and $1\!\le\!\mu\!\le\!e(i)$, fix a trivialization
	$\sO_{\cU_{\mu}'}({\bcD_{[i]}}+\bcA-\bcB)\cong \sO_{\cU_{\mu}'}$, and form the induced trivialization
	\beq\label{tri-3}
	\sO_{\cU_{\mu}}(\cR_{[i]}+{\cD_{[i]}}+\cA-\cB)\cong \ti g\sta \sO_{\cU_{\mu}'}({\bcD_{[i]}}+\bcA-\bcB)
	\cong \sO_{\cU_{\mu}}.
	\eeq
	Using this trivialization, we can identify any section in $\Gamma(\sO_{\cU_{\mu}}(\cR_{[i]}+{\bcD_{[i]}}+\bcA-\bcB))$ with a
	regular function in $\Gamma(\sO_{\cU_{\mu}})$, thus obtaining its image in $\hat \sO_{\cU_{\mu}}$.
	We denote this process by
	\begin{align*}
		&\gamma\in \Gamma(\sO_{\cU_{\mu}}(\cR_{[i]}+ {\cD_{[i]}}+\cA-\cB))\\
		&\hspace{1in}\mapsto [\gamma]\in \sO_{\cU_{\mu}}(\cR_{[i]}+{{\cD_{[i]}}+\cA-\cB})\otimes_{\sO_{\cU_{\mu}}}\hat\sO_{\cU_{\mu}}\cong \hat\sO_{\cU_{\mu}} .
	\end{align*}
	For $1\!\le\!i\!\le\!j$ and $1\!\le\!\mu\!\le\!e(i)$ as above,
	by  Remark~\ref{Rmk:van_ord} and our choice of the vanishing order of 
	$\bar S_i$ along $\bar \Sigma_{\mu}$,
	we see that $\ti g\sta \bar S_i|{\cU_{\mu}}$, as an element of $\Gamma(\rho\lsta\sO_{\cU_\mu}(\cR_{[j]}+{\cD_{[j]}}+\cA-\cB))$, 
	has vanishing order $(j\!-\!i)$ along $\cR_{\ge \mu}$.
	So, there exists (nonzero) $\la_i\inn \hat\sO_{\cU_{\mu}}\eq K_{\mu}[\![  \hat u_{\mu}]\!]$ such that
	\begin{align}
		\label{Eqn:g*barS_i}
		[\ti g\sta \bar S_i]=\la_i\, (\hat u_{\mu})^{j-i}\qquad
		\tn{and}\qquad
		\hat u_{\mu}\not\mid \la_i.
	\end{align}
	

	We are ready to prove the statement of Lemma~\ref{sublemma}.
				
	First, for every $1\!\le\!i\!\le\!j$,
	we show $S_i\inn {\bf M}_j$.
	Since ${\bf M}_i$ is naturally included in ${\bf M}_j$,
	it suffices to show $S_i\inn{\bf M}_i$.
	By \eqref{basic-inc}, this is equivalent to show
	\beq\label{F0}
	S_i|_{\cR_{[i]}}=0\in \Gamma\big(\rho\lsta\sO_{\cR_{[i]}}(\cR_{[i]}+{\cD_{[i]}}+\cA-\cB)\big).
	\eeq
		
	To establish (\ref{F0}),
	recall that {for every $1 \!\le\! \mu\! \le\! e(i)$,}  $\zeta_{[\![\delta_i]\!]}=\prod_{\ell=1}^i \zeta_{[\delta_\ell]}$ and $r_{[\mu-1]}\eq \sum_{\nu=1}^{\mu-1}r_\nu$ (which is 0 if $\mu\eq 1$).
	Thus,
	$\zeta_{[\![\delta_i]\!]}$ has vanishing order 
	$$
	\#\{1\le\ell\le i:e(\ell)\ge \mu\}=
	i-r_{[\mu-1]}\qquad
	\tn{along}~\cR_{\ge\mu}\,.
	$$
	In addition,
	by Remark~\ref{Rmk:van_ord} and our choice of the vanishing order of 
	$\bar S_i$ along $\bar \Sigma_{\mu}$,
	we see $\ti g\sta \bar S_i$, as an element of $\Gamma(\rho\lsta\sO_\cC(\cR_{[i]}+{\cD_{[i]}}+\cA-\cB))$,
	has vanishing order $0$ along $\cR_{\ge\mu}$ for all $1\!\le\!\mu\!\le\!e(i)$.
	Therefore,
	$S_i=\zeta_{[\![\delta_i]\!]}\cdot \ti g\sta \bar S_i$, as an element of $\Gamma(\rho\lsta\sO_\cC(\cR_{[i]}+{\cD_{[i]}}+\cA-\cB))$,
	has vanishing order $\big(i-r_{[\mu-1]}\big)$ along ${\cR_{\ge\mu}}$ for all $1\!\le\!\mu\!\le\!e(i)$.
	Taking (\ref{Eqn:R_j}) into consideration,
	we obtain (\ref{F0}).	
	
	We now prove that ${\bf M}_j$ is generated by $S_1,\ldots,S_j$. Take an arbitrary $S'\inn {\bf M}_j$, which is identified with its image in $\Gamma(\rho\lsta\sO_{\cC}(\cR_{[j]}+{\cD_{[j]}}+\cA-\cB))$ via the inclusion in~\eqref{basic-inc}.
	Using the identity in \eqref{basic-inc} and the fact that $\ov{\bf M}_j$ is a free $\Gamma(\sO_\cV)$-module
	generated freely by $\bar S_1,\ldots,\bar S_j$,  $S'$ can be uniquely expressed as
	$$S'=f_1\cdot\ti g\sta \bar S_1+\cdots+ f_j\cdot \ti g\sta \bar S_j, \quad f_i\in \Gamma(\sO_\cV).
	$$
	Since $S'\inn {\bf M}_j$, $S'|_{\cR_{[j]}}\eq 0$. In particular,  $(u_{\mu})^{j- r_{[\mu-1]}}$ 
	divides $S'|_{\cU_{\mu}}$ for all $1 \!\le \!\mu \!\le\! e(j)$.
	
	%
	%
	
	We claim that for  every $1\!\le \!i\!\le\! j$,  every $1\!\le\!\mu\!\le\!e(i)$, and every positive integer $\alpha$,
	 $u^\alpha_{\mu}$ 
	divides $S'|_{\cU_{\mu}}$
	if and only if $f_i |_{\cU_{\mu}}$ has vanishing order at least $\max(0,\alpha\!-\!(j\!-\!i))$ along ${\cR_{\ge\mu}}$.
	The ``if'' direction holds because 
	each $f_i\cdot\ti g\sta \bar S_i$,
	as an element of $\Gamma(\rho\lsta\sO_\cC(\cR_{[i]}+{\cD_{[i]}}+\cA-\cB))$,
	has vanishing order 
	$$
	 \max\big(0,\alpha\!-\!(j\!-\!i)\big)+(j\!-\!i)
	 \ge \alpha\!-\!(j\!-\!i)+(j\!-\!i) =\al
	$$
	along $\cR_{\ge\mu}$ for any $1\!\le\!\mu\!\le\!e(i)$.
	
	Next, we prove the ``only if'' part.
	Taking their respective images in $\hat \sO_{\cU_{\mu}}$, 
	$u^\alpha_{\mu}$ dividing $S'|_{\cU_{\mu}}$
	implies that
	\beq\label{div}
	(\hat u_{\mu})^\alpha\quad \text{divides}\quad [S']=\hat f_1\cdot[\ti g\sta \bar S_1]+\cdots +\hat f_j\cdot[\ti g\sta \bar S_j]\in K_{\mu}[\![\hat u_{\mu}]\!].
	\eeq
	By replacing $v_{\mu}$ by $v_{\mu}\upmo$ if necessary, we can assume that
	$\hat \zeta_{q_{\mu}}=c_{\mu} \hat u_{\mu}v_{\mu}$, where $\zeta_{q_{\mu}}\inn\Ga(\sO_\cV)$ is the modular parameter corresponding to the node $q_{\mu}$, $\hat \zeta_{q_{\mu}}$ is its image in $\hat \sO_{\cU_{\mu}}$, and $c_{\mu}\ne 0\in K'_{\mu}$.
	Let $\alpha_i$ be the order of which $f_i$ is divisible by $\zeta_{q_{\mu}}$,
	then $\hat f_i=c_i'(\hat \zeta_{q_{\mu}})^{\alpha_i}$, where $c_i'\ne 0\in K'_{\mu}$.
	(The case $f_i=0$
	is trivially true and is ignored.)
	
	By (\ref{Eqn:g*barS_i}), we see
	\eqref{div} translates to that $(\hat u_{\mu})^\alpha$ divides
	$$\sum_{i=1}^j \hat f_i\cdot [\ti g\sta \bar S_i]=\sum_{i=1}^j c_i'(c_{\mu}\hat u_{\mu} v_{\mu})^{\alpha_i}\cdot \la_i (\hat u_{\mu})^{j-i}
	=\sum_{k\ge 0} v_{\mu}^k\sum_{1\le i\le j~\tn{s.t.}~\alpha_i=k}\!\!\!\!\!\!\!
	(c_i' (c_{\mu})^{k}\la_i)(\hat u_{\mu})^{k+j-i}
	$$
	in $K_{\mu}[\![\hat u_{\mu}]\!]=K_{\mu}'(v_{\mu})[\![{\hat u_{\mu}}]\!]$.  
	As $\hat u_{\mu}$ does not divide $c_i' (c_{\mu})^{k}\la_i$,
	this divisibility holds if and only if $(\hat u_{\mu})^\alpha$ divides each individual $v_{\mu}^k(\hat u_{\mu})^{k+j-i}$, which holds if
	and only if
	$\alpha\le \alpha_i+j-i$ for each $i$ where $f_i\ne 0$. This establishes the ``only if'' part, so the claim is proved.
	
	For every $1 \!\le \!i \!\le\! j$ and $1 \!\le \!\mu \!\le\! e(i)$, 
	recall we have shown that $(u_{\mu})^{j- r_{[\mu-1]}}$ 
	divides $S'|_{\cU_{\mu}}$.
	By applying the above claim to $\alpha=j- r_{[\mu-1]}$, we conclude that 
	$$\alpha_i\geq \big(j- r_{[\mu-1]}\big)-(j\!-\!i)\,=\,i- r_{[\mu-1]}.
	$$ 
	This proves that $f_i/\zeta_{[\![\delta_i]\!]}$, which is a meromorphic
	function on $\cV$, is regular over
	an open subset of $\cV$ containing the generic point of $\cH_{\mu}$, for any
	$1\!\le \!{\mu}\!\le\! e(i)$.
	Since $f_i/\zeta_{[\![\delta_i]\!]}$ is regular over $\cV-\bigcup_{\mu=1}^{e(i)}\! \cH_{\mu}$, by Hartogs Lemma, it is regular over $\cV$.
	
	This implies that $S'$ lies in the $\Gamma(\sO_\cV)$ span of $S_1,\cdots, S_j$.
	Since ${\bf M}_j$ is a rank~$j$ free $\Gamma(\sO_\cV)$-modules, it is freely generated by $S_1,\cdots, S_j$.
	This completes the proof of Lemma~\ref{sublemma}.
\end{proof}

	We are ready to present a strengthened form of the restriction homomorphism 
	$$
	 \hmr: \bigoplus_{1\le i \le r}
	 \rho_* \sM (\cD_i) \lra \rho_* \sO_\cA(\cA)
	$$ given in~(\ref{phi-R}).
By Proposition \ref{HomFirstOrder}, possibly under  basis changes of ${\bf M}$ and
\eqref{tri-A}, we see that $\hmr$ takes the form
\begin{equation}\label{cij-2}
	A=\left[\begin{matrix} c_{11} \zeta_{[\delta_1,a_1]} &  c_{12}  \zeta_{[\delta_2,a_1]} & c_{13} \zeta_{[\delta_3,a_1]}& \cdots\\
		c_{21}  \zeta_{[\delta_1,a_2]} & c_{22}  \zeta_{[\delta_2,a_2]}
		& c_{23}\zeta_{[\delta_3,a_2]} & \cdots\\
	\end{matrix} \right],
\end{equation}
with $c_{ij} \in \Ga(\sO^*_\cV)$.

\begin{prop}\label{Prp:key-tail_chain}
	Assume $\de_1,\ldots,\de_r$ are on $R$ and ordered according to~(\ref{linearOrder}).
	After choosing suitable trivialization $$\bigoplus_{1\le i\le r}   \rho_* \sM (\cD_i) \cong \sO_\cV^{\,\oplus\,r}
	\qquad\tn{and}\qquad
	\rho_*\sO_{\cA_s}(\cA_s)\cong\sO_\cV,\quad
	s\eq 1,2,
	$$ 
	the restriction homomorphism $\hmr$ also takes the form
	\begin{equation}\label{matrix0} 
		\hmr=
		\left[\begin{matrix}  \wh c_{11} \zeta_{[\delta_1,a_1]}&  \wh c_{12}
			\zeta_{[\delta_1,a_1]}\zeta_{[\delta_2]}
			& \wh c_{13}\zeta_{[\delta_1,a_1]}\zeta_{[\delta_2]}\zeta_{[\delta_3]}&\cdots\\
			\wh c_{21} \zeta_{[\delta_1,a_2]}&
			\wh c_{22}\zeta_{[\delta_1,a_2]}\zeta_{[\delta_2]}
			& \wh c_{23}\zeta_{[\delta_1,a_2]}\zeta_{[\delta_2]}\zeta_{[\delta_3]}&\cdots\\
		\end{matrix} \right] 
	\end{equation}
	where  $\wh c_{11} =c _{11}, \wh c_{21} = c_{21} \in \Gamma (\sO_\cV^*),$
	and $\wh c_{ij} \in \Gamma (\sO_\cV),  \forall  \;  j > 1$.
\end{prop}

\begin{proof}
	We follow the notation of Lemma \ref{sublemma}.
	Let $$b_{ij}= \bar S_j|_{\bar \cA_i} \in \Ga \big(\rho_* \sO_{\cA_i}(\cA_i)\big) \cong\Ga (\sO_\cV).$$
	Then under $\{S_1, \cdots, S_r\} $ and \eqref{tri-A}, $ \varphi^R$ takes the form
	\begin{equation}\label{cij-1}
		B=\left[\begin{matrix} b_{11} \zeta_{[\delta_1]}& b_{12}
			\zeta_{[\delta_1]}\zeta_{[\delta_2]}
			&b_{13}\zeta_{[\delta_1]}\zeta_{[\delta_2]}\zeta_{[\delta_3]}&\cdots\\
			b_{21} \zeta_{[\delta_1]}&
			b_{22}\zeta_{[\delta_1]}\zeta_{[\delta_2]}
			&b_{23}\zeta_{[\delta_1]}\zeta_{[\delta_2]}\zeta_{[\delta_3]}&\cdots\\
		\end{matrix} \right].
	\end{equation}
	Therefore, $B=A(\mu_{ij})$ for a matrix $(\mu_{ij}) \in GL(r,\sO_\cV)$.
	
	 Moreover, the choices of the sections $S_i$, $1\!\le\!i\!\le\!r$, in Lemma~\ref{sublemma} implies $(\mu_{ij})$ is upper triangular,
	 because with $\si_1,\ldots,\si_r$ denoting a basis of $\bf M$ that gives rise to the matrix $A$
	 (i.e.~each $\si_i$ is an nonzero element of $\Ga(\rho_*\sO_\cC(\cD_{i}\!+\!\cA\!-\!\cB))$),
	 we conclude from Lemma~\ref{sublemma} that for each $1\!\le\!j\!\le\!r$,
	 both $\si_1,\ldots,\si_j$ and $S_1,\ldots,S_j$ are bases of
	 ${\bf M}_j\!=\!
	 \Ga\big(\rho_*\sO_\cC(\cD_{[j]}\!+\!\cA\!-\!\cB)\big)$.
		Multiplying each $b_{ij}$ by a unit if necessary,
		we may further assume the entries on the diagonal of $(\mu_{ij})$ are all~1.
	
	Let $x_{q_1} \in F$ be a smooth point of the curve $C$ that is sufficiently close to the node
	$q_1$.
	Since $\zeta_{[x_{q_1},a_1]} $ divides the first row of \eqref{cij-2}, it divides the first row of \eqref{cij-1} as well.
	Since $\zeta_{[x_{q_1},a_1]} $
	is coprime with $\zeta_{[\delta_j]}$ for all $1\!\le\!j\!\le\!r$, we conclude $\zeta_{[x_{q_1},a_1]} $ divides $b_{1j}$,
	thus $b_{1j}\eq \wh c_{1j} \zeta_{[x_{q_1},a_1]}$
    for some $\wh c_{1j}\inn\Ga(\sO^*_\cV)$.
    
    Similar arguments can be applied to the second row, 
	so we obtain
	\begin{equation}\label{e_bc}
		b_{sj} = \wh c_{sj} \zeta_{[x_{q_1},a_s]}
		\qquad
		\forall\ s\eq 1,2,\ \ 1\!\le\!j\!\le\!r.
	\end{equation}
	Note that  $\zeta_{[\delta_j,a_s]} = \zeta_{[x_{q_1},a_s]}  \zeta_{[\delta_j]}$.
	This brings both rows of
	\eqref{cij-1} to the desired form.
	
	It remains to check $\wh c_{11}$ and $\wh c_{21}$.
	Let $(\nu_{ij})=(\mu_{ij})^{-1}$. 
	Then,
	$(\nu_{ij})$ is also upper triangular with all diagonal entries equal to 1, and for $s\eq 1,2$ and $1\!\le\!i\!\le\!r$, we have
		\begin{align}\label{Eqn:c_vs_wh_c}
			c_{si} \zeta_{[\delta_i,a_s]} =
			\nu_{1i} \wh c_{s1}  \zeta_{[\delta_1,a_s]} + \nu_{2i} \wh c_{s2} \zeta_{[\delta_1,a_s]} \zeta_{[\delta_2]}+\cdots+
			\nu_{ii}\wh c_{si} \zeta_{[\delta_1,a_s]} \zeta_{[\delta_2]}\!\cdots\!\ze_{[\de_i]}.
		\end{align}
	Hence for $s\eq 1,2$, we have $c_{s1}= \nu_{11}\!\cdot\!\wh c_{s1} =\wh c_{s1}$  and $\wh c_{s1} \in \Gamma(\sO_\cV^*)$.
\end{proof}

\subsection{Structural homomorphisms: $2\times 2$ minors}\label{Subsec:construct-rational-tails}

We continue with the notation of \S\ref{Subsec:using-ss}. 
For conciseness, we let
$$ z_0=(C,D)\in\cV\subset\fM_2^{\rm div}$$
and always assume $\cV$ is sufficiently small.
	In this subsection,
	we still focus on $\de_i$, $1\!\le\!i\!\le\!r$, that are on a maximal chain of tail rational subcurves $R$,
	attached to the core $F$ of $C$ at the node $q_1$.
	
	Below we analyze the $2\!\times\!2$ minors of $\hmr$,
	which is the second step towards the proof of Proposition~\ref{Prp:phi_key}.
	
	For $s\eq 1,2$ and $1\!\le\!i\!\le\!r$, let
	$c_{si}$ and $\wh c_{si}$ be as in~(\ref{cij-2}) and (\ref{matrix0}),
	respectively.
	For $1\!\le\!i\!<\!j\!\le\!r$, we set
	\begin{align}\label{Eqn:minors}
		\dt{ij}= \det
		\left[\begin{matrix}
			c_{1i} &   c_{1j} \\
			c_{2i} &   c_{2j} \end{matrix} \right], \quad
		\wdt{ij}= \det
		\left[\begin{matrix}
			\wh c_{1i} &  \wh c_{1j} \\
			\wh c_{2i} &  \wh c_{2j} \end{matrix} \right]
		\quad\in \Ga(\sO_\cV) .
	\end{align}

	For distinct $1\!\le\!i\!,\!j\!\le\!r$, let $\cK_{ij} \subset \fM_2^{\rm div}$ be the locus
	of points $z$ such that $\cD_i(z)$ ($=\!\cD_i\!\cap\!\cC_z$) and $\cD_j(z)$ ($=\!\cD_j\!\cap\!\cC_z$) are
	a pair of
	conjugate points. 
	By Lemma \ref{Lm:K_smoothness}, $\cK_{ij}$ is a Cartier divisor locally given by $\ka_{ij}\eq 0$ for some $\ka_{ij}\in\Ga(\sO_\cV)$.
	All possible expressions of $\ka_{ij}$ are listed in Lemma~\ref{Lm:K_smoothness}.

	\begin{prop}\label{keyProp1}
		With notation as above, we have
		\begin{enumerate}[label=(R${}_{\arabic*}$),leftmargin=*]
			\item 
			\label{Cond:eta_12}
			there exists $u\inn\Ga(\sO^*_\cV)$ such that $$
			\dt{12}= \ze_{[\de_1]}\cdot\wdt{12}
			\qquad\tn{and}\qquad
			\wdt{12}=u\cdot\ka_{12};$$
			
			\item 
			\label{Cond:eta_13}
			there exists $v\inn\Ga(\sO_\cV)$ such that $$
			\dt{13}= \ze_{[\de_1]}\ze_{[\de_2]}
			\wdt{13}+v\!\cdot\!\dt{12};$$
			
			\item 
			\label{Cond:eta_13_ns}
			if $q_1$ is not on any non-separating bridge of $C$, then
			$$ \rank \left[\begin{matrix}
				\wh c_{11} & \wh c_{12}& \wh c_{13} \\
				\wh c_{21} & \wh c_{22} & \wh c_{23}\end{matrix} \right](z_0) =2,$$
			or equivalently,
			$\wdt{13}$ is invertible along $\cK_{12}$;
			
			\item 
			\label{Cond:eta_13_ns'}
			if $q_1$ is on a maximal non-separating bridge $\tn B[\fp,\fq]$ of~$C$, 
			with $x_{q_1}\inn F$ denoting a smooth point of the core $F$ that is sufficiently close to the node~$q_1$,
			then there exist $f_2,  g_2\inn\Ga(\sO^*_\cV)$ and $f_3, g_3\inn\Ga(\sO_\cV)$ such that
			$$
			\wdt{12}=f_2\ze_{[x_{q_1},\fp]}\!+\!g_2\ze_{[x_{q_1},\fq]},\quad
			\wdt{13}=f_3\ze_{[x_{q_1},\fp]}\!+\!g_3\ze_{[x_{q_1},\fq]},\quad
			\det\left[
			\begin{matrix}
				f_2 & f_3\\
				g_2 & g_3
			\end{matrix}
			\right]\in\Ga(\sO^*_\cV).
			$$
		\end{enumerate}
	\end{prop}
	
	To establish Proposition~\ref{keyProp1},
	we need the following lemma.

\begin{lemm}\label{Lm:str_homom_one_tail_chain_minor}
	With notation as in Proposition~\ref{keyProp1}, we have  
	\begin{enumerate}[label=(\arabic*),leftmargin=*]
		\item if the core $F$ of the curve $C$  contains at most one node and that node, if exists, is non-separating,  then
		\begin{enumerate}[label=(1.\alph*),leftmargin=*]
			\item 
			\label{Cond:eta_12_lm}
			$\wdt{12} (z_0)=0$  if and only if $q_1$ is Weierstrass, or equivalently, 
			if and only if $\delta_1$ and $\delta_2$ are conjugate;
			\item \label{Cond:eta_13_lm}
			$ \rank \left[\begin{matrix}
				\wh c_{11} & \wh c_{12}& \wh c_{13} \\
				\wh c_{21} & \wh c_{22} & \wh c_{23}\end{matrix} \right](z_0) =2$;
		\end{enumerate}
		
		\item \label{Cond:th_12} for any $z \inn\cV$ such that $\cC_z$ contains at most one node and that node, if exists, is non-separating, 
		\begin{enumerate}[label=(2.\alph*),leftmargin=*]
		\item \label{Cond:th_12_ns_lm}
		$\dt{12}(z)\eq 0$ if and only if $z \inn \cK_{12}\!\cap \!\cV$,
		that is, $\cD_1(z)$ and $\cD_2(z)$
		are a pair of conjugate points;
		\item \label{Cond:th_13_ns_lm}
		$ \rank \left[\begin{matrix}
		c_{11} &  c_{12} &  c_{13} \\
		c_{21} &  c_{22} &  c_{23}\end{matrix} \right](z_0) =2$;
		\end{enumerate}
		
		\item \label{Cond:th_12_order} $\dt{12}$ vanishes  to the first order along $\cK_{12} \cap \cV$;
		
		\item \label{Cond:sep_nodal_lm}
			for any $z\inn\cV$ such that $\cC_z$ consists of two smooth genus-1 subcurves attached at a (separating) node,
		\begin{enumerate}[label=(4.\alph*),leftmargin=*]
			\item  \label{Cond:th_12_sep_nodal}
			if $z\!\not\in\!\cK_{12}$, then $\dt{12}(z)\!\ne\!0$;
			\item \label{Cond:th_13_sep_nodal}
			if $z\!\in\!\cK_{12}$, then $\dt{13}(z)\!\ne\!0$.
		\end{enumerate}
	\end{enumerate}
\end{lemm}
\begin{proof}
	We continue to use the notation of \S\ref{Subsec:using-ss}.
	
	\vsp
	\ref{Cond:eta_12_lm}.
	Since $s_1, s_2 \in H^0(\sO_F(rq_1+a-b)) $ have vanishing order $r-1$ and $r-2$ respectively at $q_1$,
	one verifies directly that the image of $H^0(\sO_F(2q_1+a-b)) $ in $H^0(\sO_F(rq_1+a-b))$ (via the inclusion)
	is the subspace
	spanned by $s_1$ and $s_2$. By identifying the image of $H^0(\sO_F(2q_1+a-b)) $ with this subspace,  we obtain a homomorphism
	\begin{align*}
		&\begin{CD}H^0(\sO_F(2q_1+a-b)) @>{s_{12}}>> \kk^2
		\end{CD},\qquad\tn{where}
		\\
		&
		s_{12}=
		\left[\begin{matrix}
			s_1(a_1)\! &\! s_2(a_1) \\
			s_1(a_2)\! &\! s_2(a_2)
		\end{matrix} \right]
		=\left[\begin{matrix}
			\wh c_{11} & \wh c_{12} \\
			\wh c_{21} & \wh c_{22} \end{matrix} \right](z_0).
	\end{align*}
	The last equality above holds because $F$ is assumed inseparable.
	Thus,  the matrix $\left[\begin{matrix}
		\wh c_{11} & \wh c_{12} \\
		\wh c_{21} & \wh c_{22} \end{matrix} \right](z_0)$ is of rank 1 if and only if $h^0(\sO_F(2q_1-b))\eq 1$.
	Since $b$ is a general point, the evaluation homomorphism  at the point $b$,
	$H^0(\sO_F(2q_1)) \to \kk$, is surjective. By the exact sequence
	$$ 0 \lra H^0(\sO_F(2q_1-b)) \lra H^0(\sO_F(2q_1)) \lra \kk \lra 0,
	$$
	$h^0(\sO_F(2q_1-b))\eq 1$ is equivalent to $h^0(\sO_F(2q_1))\eq 2$, which is further equivalent to that $q_1$ is a Weierstrass point because $F$ either is smooth, or has  exactly one self-intersection and no other nodes; see the proof of~\cite[Proposition~3.12]{BCGGM} for the details of the latter case.
	Thus,
	$\wdt{12}(z_0)\eq 0$ if and only if $q_1$ is Weierstrass, 
	which by Lemma~\ref{Lm:KW}  is equivalent to  $\de_1$ and $\de_2$ being conjugate.
	This proves Part~\ref{Cond:eta_12_lm}.
	
	\vsp
	\ref{Cond:eta_13_lm}.
	Likewise, one sees that
	the natural image of $H^0(\sO_F(3q_1+a-b)) $ in $H^0(\sO_F(rq_1+a-b)) $ is the subspace
	spanned by $s_1, s_2$ and $s_3$. Assume  $ \rank \left[\begin{matrix}
		\wh c_{11} & \wh c_{12}& \wh c_{13} \\
		\wh c_{21} & \wh c_{22} & \wh c_{23}\end{matrix} \right](z_0) \le 1$.
	By  arguments similar to the ones used in~\ref{Cond:eta_12_lm}, we obtain $h^0(\sO_F(3q_1-b))\ge 2$,
	and consequently $h^0(\sO_F(3q_1))\ge 3$, which is impossible by Riemann-Roch because $F$ either is smooth, or has  exactly one self-intersection and no other nodes.
	This completes the proof of Part~\ref{Cond:eta_13_lm}.
	
	\vsp
	\ref{Cond:th_12}.
	Let $z \in \cV$ be such that the curve $\cC_z$ either is smooth, or has  exactly one self-intersection and no other nodes. Then, 
	the evaluation homomorphisms are given by
	\begin{align*}
		&
		\begin{CD}H^0(\sO_{\cC_z}( \cD_1(z)+\cD_2 (z)+\cA(z)-\cB(z))) @>{\Th_{12}(z)}>> \kk^2,
		\end{CD}\\
		&
		\begin{CD}H^0(\sO_{\cC_z}( \cD_1(z)+\cD_2(z)+\cD_3 (z)+\cA(z)-\cB(z))) @>{\Th_{13}(z)}>> \kk^2,
		\end{CD}
		\\
		&\tn{where}\qquad
		\Th_{12}(z)=\left[\begin{matrix}
			c_{11} &  c_{12} \\  c_{21} &  c_{22} \end{matrix} \right]({z}),\quad
		\Th_{13}(z)=\left[\begin{matrix}
			c_{11} &  c_{12} &  c_{13} \\  c_{21} &  c_{22} &  c_{23} \end{matrix} \right]({z}),
	\end{align*}
	respectively.
	Thus,
	\begin{align*}
		\dt{12} (z)=0\quad
		&\Longleftrightarrow\quad
		\tn{rank}\,\Th_{12}(z)= 1\\
		&\Longleftrightarrow\quad
		h^0(\sO_{\cC_z}(\cD_1(z)+\cD_2(z)-\cB(z)))=1;\\
		\tn{rank}\,\Th_{13}(z)\le 1\quad
	 	&\Longleftrightarrow\quad
	 	h^0(\sO_{\cC_z}(\cD_1(z)+\cD_2(z)+\cD_3(z)-\cB(z)))\ge 2.
	\end{align*}
	
	Since $\cB(z)$ is a general point on the  curve $\cC_z$, 
	the evaluation homomorphism  at the point $\cB(z)$,
	$H^0(\sO_{\cC_z}(\cD_1(z)+\cD_2(z))) \to \kk$, is surjective. By the exact sequence
	$$ 0 \to H^0(\sO_{\cC_z}(\cD_1(z)\!+\!\cD_2(z)\!-\!\cB(z))) \to H^0(\sO_{\cC_z}(\cD_1(z)\!+\!\cD_2(z))) \to \kk \to 0,
	$$
	we see that
	$h^0(\sO_{\cC_z}(\cD_1(z)+\cD_2(z)-\cB(z)))=1$ if and only if
	$h^0(\sO_{\cC_z}(\cD_1(z)+\cD_2(z))=2$,
	which is  equivalent to that $\cD_1(z)$ and $\cD_2(z)$ are a pair of conjugate points, again,
	due to the same reason as in the proof of Part~\ref{Cond:eta_12_lm}.
	 
	Similarly,
	$h^0(\sO_{\cC_z}(\cD_1(z)+\cD_2(z)+\cD_3(z)-\cB(z)))\ge 2$ if and only if
	$h^0(\sO_{\cC_z}(\cD_1(z)+\cD_2(z)+\cD_3(z))\ge 3$,
	which is impossible
	due to the same reason as in the proof of Part~\ref{Cond:eta_13_lm}.
	This proves Part~\ref{Cond:th_12}.
	
	\vsp
	\ref{Cond:th_12_order}. 
	By Part~\ref{Cond:th_12},
	$\dt{12}/\ka_{12}$ is regular on the open subset of $\cV$ consisting of $z\inn\cV$ such that $\cC_z$ is smooth.
	Notice that
	Lemma~\ref{Lm:K_smoothness} implies $\ka_{12}$ is irreducible and coprime to any node-smoothing parameter,
	so $\dt{12}/\ka_{12}$ is regular on the entire $\cV$.
	To show Part~\ref{Cond:th_12_order},
	it thus suffices to prove that $\dt{12}$ vanishes  at the point $z$ to the first order for any $z \in \cK_{12}\cap \cV$ such that $\cC_z$ is smooth.

	Let $\Delta \subset \cV$ be a small curve through $z$, meeting $\cK_{12}\cap\cV$ transversely
	at $z$.
	Let $\cC_\Delta=\cC\times_\cV \Delta$, $\cD_{\Delta,i}=\cD_i \times_\cV \Delta$ for $i=1,2$,
	$\cA_\Delta=\cA\times_\cV \Delta$, and $\cB_\Delta=\cB\times_\cV \Delta$.
	Then, we obtain the evaluation homomorphism
	\begin{align*}
		&\begin{CD}\rho\lsta \sO_{\cC_\Delta}(\cD_{\Delta,1}+\cD_{\Delta,2}+\cA_\Delta-\cB_\Delta) @>{ \Th_{12}|_\Delta}>>\rho\lsta \sO_{\cA_\Delta}(\cA_\Delta) \\
		\end{CD},\\
		&\tn{where}\qquad\Th_{12}|_\Delta=
		\left.\left[\begin{matrix}
			c_{11} &  c_{12} \\  c_{21} &  c_{22} \end{matrix} \right]\right|_\Delta.
	\end{align*}
	By long exact sequence of cohomology and Riemann-Roch, the co-kernel of $\Th_{12}|_\Delta$ is a skyscraper sheaf with one-dimensional stalk at $z$. This implies that $\dt{12}|_\Delta$ vanishes at the point $z$
	at the first order,
	hence $\dt{12}$ vanishes along $\cK_{12}$ to the first order.
	
	\vsp
	\ref{Cond:sep_nodal_lm}.
		The proofs of~\ref{Cond:th_12_sep_nodal} and~\ref{Cond:th_13_sep_nodal} are both similar to that of Part~\ref{Cond:th_12_order}. 
		Given $z\inn \cV$ such that $\cC_z$ consists of two smooth genus 1 subcurves attached at a node $q(z)$,
		The node smoothing parameter on $\cV$ corresponding to $q(z)$ is denoted by $\ze_q$.
			
	\vsp	
		\ref{Cond:th_12_sep_nodal}.
		Assume that $z\!\not\in\!\cK_{12}$.
		Let $\Delta \subset \cV$ be a small curve through $z$ and disjoint from $\cK_{12}$, meeting $(\ze_q\eq 0)$ transversely at $z$.
		Mimicking the proof of Part~\ref{Cond:th_12_order}, we define  $\cC_\Delta$, $\cD_{\Delta,1}$, $\cD_{\Delta,2}$,
		$\cA_\Delta$, and $\cB_\Delta$ likewise.
		Then, we obtain the evaluation homomorphism
		\begin{align*}
			&\begin{CD}\rho\lsta \sO_{\cC_\Delta}(\cD_{\Delta,1}+\cD_{\Delta,2}+\cA_\Delta-\cB_\Delta) @>{ \ti\Th_{12}|_\Delta}>>\rho\lsta \sO_{\cA_\Delta}(\cA_\Delta) \\
			\end{CD},\\
			&\tn{where}\qquad\ti\Th_{12}|_\Delta=
			\left.\left[\begin{matrix}
				c_{11} &  c_{12} \\  c_{21}\ze_q &  c_{22}\ze_q \end{matrix} \right]\right|_\Delta;
		\end{align*}
		see Proposition~\ref{HomFirstOrder} for the factor $\ze_q$ in the second row of $\ti\Th_{12}|_\Delta$.
		By long exact sequence of cohomology and Riemann-Roch, the co-kernel of $\ti\Th_{12}|_\Delta$ is still a skyscraper sheaf with one-dimensional stalk at $z$. This implies that $\det\ti\Th_{12}|_\Delta$ vanishes at the point $z$
		to the first order,
		hence 
		$\det\ti\Th_{12}$ vanishes along $(\ze_q\eq 0)$ to the first order. 
		Notice that $\det\ti\Th_{12}=\ze_q\dt{12}$,
		which implies $\dt{12}$ is invertible on a neighborhood of $z$.
		
	\vsp
		\ref{Cond:th_13_sep_nodal}.
		Assume that $z\!\in\!\cK_{12}$.
		Since $q(z)$ is the only node of $\cC_z$,
		we see from Lemma~\ref{Lm:K_smoothness} that
		near $z$ the locus $\cK_{12}$ is a smooth Cartier divisor, meeting $(\ze_q\eq 0)$ transversely at $z$.
		We can thus take $\Delta \subset \cK_{12}$ to be a small curve through $z$, meeting $(\ze_q\eq 0)$ transversely at $z$.
		Let $\cC_\Delta$, $\cD_{\Delta,1}$, $\cD_{\Delta,2}$, $\cD_{\Delta,3}$,
		$\cA_\Delta$, and $\cB_\Delta$ be analogous to Part~\ref{Cond:th_12_sep_nodal}.
		Then, we obtain the evaluation homomorphism
		\begin{align*}
			&\begin{CD}\rho\lsta \sO_{\cC_\Delta}(\cD_{\Delta,1}+\cD_{\Delta,2}+\cD_{\Delta,3}+\cA_\Delta-\cB_\Delta) @>{ \ti\Th_{13}|_\Delta}>>\rho\lsta \sO_{\cA_\Delta}(\cA_\Delta) \\
			\end{CD},\\
			&\tn{where}\qquad\ti\Th_{13}|_\Delta=
			\left.\left[\begin{matrix}
				c_{11} &  c_{12} & c_{13} \\  c_{21}\ze_q &  c_{22}\ze_q & c_{23}\ze_q \end{matrix} \right]\right|_\Delta;
		\end{align*}
		once again see Proposition~\ref{HomFirstOrder} for the factor $\ze_q$ in the second row of $\ti\Th_{13}|_\Delta$.
		By long exact sequence of cohomology and Riemann-Roch, the co-kernel of $\ti\Th_{13}|_\Delta$ is still a skyscraper sheaf with one-dimensional stalk at $z$. 
		
		Meanwhile,
		notice that after an elementary row operation,
		$\ti\Th_{13}|_\Delta$ can be rewritten as
		$$
		\left.\left[\begin{matrix}
			c_{11} &  c_{12} & c_{13} \\  0 &  \dt{12}\ze_q & \dt{13}\ze_q \end{matrix} \right]\right|_\Delta.
		$$
		By Part~\ref{Cond:th_12_order} and the assumption $\De\!\subset\!\cK_{12}$,
		we have $\dt{12}|_\De$ is identically zero.
		Taking into account that $c_{11}$ is invertible,
		we conclude that the entry $\dt{13}\ze_q|_\Delta$ vanishes at the point $z$
		to the first order,
		hence 
		$\dt{13}\ze_q$ vanishes along $(\ze_q\eq 0)$ to the first order,
		which implies $\dt{13}$ is invertible on a neighborhood of $z$.
\end{proof}

	\begin{proof}[Proof of Proposition~\ref{keyProp1}]
	\ref{Cond:eta_12}.
	Taking $i\eq 2$ in~(\ref{Eqn:c_vs_wh_c}),
		we have
		\begin{align*}
			c_{s2} \zeta_{[\delta_2,a_s]} =
			\nu_{12} \wh c_{s1}  \zeta_{[\delta_1,a_s]} +  \wh c_{s2} \zeta_{[\delta_1,a_s]} \zeta_{[\delta_2]}=
			\nu_{12} \wh c_{s1}  \zeta_{[\delta_1,a_s]} +  \wh c_{s2} \zeta_{[\delta_2,a_s]} \zeta_{[\delta_1]}.
		\end{align*}
		The last equality holds because 
		$\ze_{[\de_2,a_s]}\eq \ze_{[\de_1,a_2]}\ze_{[\de_2,\de_1]}$ and $\ze_{[\de_2]}\eq \ze_{[\de_1]}\ze_{[\de_2,\de_1]}$.
		Consequently, we have
		\begin{align*}
			c_{s2}=
			\ze_{[\de_1]}\wh c_{s2}+
			\frac{\nu_{12}}{\ze_{[\de_1,\de_2]}}\wh c_{s1}=
			\ze_{[\de_1]}\wh c_{s2}+
			\frac{\nu_{12}}{\ze_{[\de_1,\de_2]}} c_{s1},\qquad
			s\eq 1,2.
		\end{align*}
		The last equality holds because $\wh c_{s1}\eq c_{s1}$, $s\eq 1,2$, by Proposition~\ref{Prp:key-tail_chain}.
		Then, by computing $\dt{12}$ using the above, we obtain $\dt{12}\eq\ze_{[\de_1]}\wdt{12}$.
		
		To show $\wdt{12}/\ka_{12}\inn\Ga(\sO^*_\cV)$, we observe that Lemma~\ref{Lm:str_homom_one_tail_chain_minor}~\ref{Cond:th_12_order} implies
		$\dt{12}/\ka_{12}$ is regular on $\cV$.
	    In other words,
	    $\ze_{[\de_1]}\wdt{12}/\ka_{12}$ is regular on $\cV$.
		However, $\ze_{[\de_1]}$ and $\ka_{12}$ are coprime,
		hence we can write $\wdt{12}\eq u\!\cdot\!\ka_{12}$ for some $u\inn\Ga(\sO_\cV)$.
		
		To show $u\inn\Ga(\sO^*_\cV)$,
		we fix $z\inn\cV$.
		If $\cC_z$ contains at most one node and that node, if exists, is non-separating,
		then Lemma~\ref{Lm:str_homom_one_tail_chain_minor}~\ref{Cond:th_12} implies that $u(z)\!=\!0$ can possibly be true only if $z\inn\cK_{12}$.
		If $\cC_z$ contains one separating node $q(z)$ and no other nodes,
		there are two possibilities:
		either the core $F_z$ of $\cC_z$ is smooth, or $\cC_z$ consists of two smooth genus 1 subcurves attached at $q(z)$.
		By Lemma~\ref{Lm:str_homom_one_tail_chain_minor}~\ref{Cond:eta_12_lm} and~\ref{Cond:th_12_sep_nodal},
		we always conclude that $u(z)$ can possibly be 0 only if $z\inn\cK_{12}$.
		
		In sum, let $\cZ\!\subset\!\cV$ consist of all $z\inn\cV$ satisfying $\cC_z$ contains at least two nodes.
		The above paragraph ensures that
		$$
		u(z)\ne 0\qquad\forall~z\in\cV\bsl(\cK_{12}\cup\cZ).
		$$
		Since $\cZ$ is of codimension 2 and $\ka_{12}$ is irreducible,
		we have $\wdt{12}\eq \ti u\cdot \ka_{12}^\ell$  for some $\ti u\inn\Ga(\sO_\cV^*)$ and $\ell\inn\mathbb Z_{>0}$.
		By Lemma~\ref{Lm:str_homom_one_tail_chain_minor}~\ref{Cond:th_12_order},
		we then see $\ell\eq 1$.
		This completes the proof of Part~\ref{Cond:eta_12}.
	
	\vsp
	\ref{Cond:eta_13}.
	Taking $i\eq 1,2,3$ in~(\ref{Eqn:c_vs_wh_c}) and considering~(\ref{Eqn:e(i)}),
	we see
	\begin{align*}
		c_{s3}
		&=
		\ze_{[\de_1]}\ze_{[\de_2]}\wh c_{s3}+
		\frac{\nu_{23}\ze_{[\de_1]}}{\ze_{[\de_2,\de_3]}}\wh c_{s2}+
		\frac{\nu_{13}}{\ze_{[\de_1,\de_3]}}\wh c_{s1}
		\\
		&=
		\ze_{[\de_1]}\ze_{[\de_2]}\wh c_{s3}+
		\frac{\nu_{23}\ze_{[\de_1]}}{\ze_{[\de_2,\de_3]}}\wh c_{s2}+
		\frac{\nu_{13}}{\ze_{[\de_1,\de_3]}}c_{s1},\qquad
		s\eq 1,2.
	\end{align*}
	The last equality above holds because $\wh c_{11}\eq c_{11}$ and $\wh c_{11}\eq c_{11}$;
	c.f.~Proposition~\ref{Prp:key-tail_chain}.
	Therefore,
	$$\dt{13}= \ze_{[\de_1]}\ze_{[\de_2]}
	\wdt{13}+\frac{\nu_{23}\ze_{[\de_1]}}{\ze_{[\de_2,\de_3]}}\wdt{12}=
	\ze_{[\de_1]}\ze_{[\de_2]}
	\wdt{13}+\frac{\nu_{23}}{\ze_{[\de_2,\de_3]}}\dt{12}.$$
	The last equality above follows from Part~\ref{Cond:eta_12}.
	
	Notice that $\ze_{[\de_1]}\ka_{12}$, which equals $\dt{12}$ up to a unit by Part~\ref{Cond:eta_12}, and $\ze_{[\de_2,\de_3]}$ are coprime,
	hence $\nu_{23}/\ze_{[\de_2,\de_3]}\inn\Ga(\sO_\cV)$.
	This proves Part~\ref{Cond:eta_13}.
	
	\vsp
	\ref{Cond:eta_13_ns}.
	To see the two statements of~\ref{Cond:eta_13_ns} are equivalent,
	just observe that 
	\begin{align*}
	\rank\left[
	\begin{matrix}
		\wh c_{11} & \wh c_{12} & \wh c_{13}\\
		\wh c_{21} & \wh c_{22} & \wh c_{23}
	\end{matrix}\right](z)
	&=
	\rank\left[
	\begin{matrix}
		\wh c_{11} & \wh c_{12} & \wh c_{13}\\
		0 & \wdt{12} & \wdt{13}
	\end{matrix}\right](z)\\
	&=
	\rank\left[
	\begin{matrix}
		\wh c_{11} & \wh c_{12} & \wh c_{13}\\
		0 & \ka_{12} & \wdt{13}
	\end{matrix}\right](z).
	\end{align*}
	The last equality above follows from Part~\ref{Cond:eta_12}.
	
	When $q_1$ is not on any non-separating bridge of $C$, we conclude from Lemma~\ref{Lm:W_smoothness} that $\cK_{12}\!\cap\!\cV$ is a smooth Cartier divisor of $\cV$ (recall that $\cV$ is assumed to be sufficiently small).
	Our goal is to show $\wdt{13}(z)\!\ne\!0$ for all $z\inn\cK_{12}\!\cap\!\cV$.
	Below we fix $z\inn\cK_{12}\!\cap\!\cV$.
	
	If $\cC_z$ either is smooth, or contains exactly one node and that node is non-separating,
	we consider the evaluation homomorphism of
	$H^0\big(\sO_{\cC_z}(\cD_1(z)\!+\!\cD_2(z)\!+\!\cD_3(z)+\cA(z)\!-\!\cB(z)))$ to $\bk^2$. 
	Following an argument parallel to the proof of Lemma~\ref{Lm:str_homom_one_tail_chain_minor}, Part~\ref{Cond:eta_13_lm},
	we conclude that 
	$$\rank\left[
	\begin{matrix}
		c_{11} & c_{12} & c_{13}\\
		c_{21} & c_{22} & c_{23}
	\end{matrix}\right](z)=
	\rank\left[
	\begin{matrix}
		c_{11} & c_{12} & c_{13}\\
		0 & \dt{12} & \dt{13}
	\end{matrix}\right](z)=2.
	$$
	Since $\ka_{12}(z)\eq 0$,
	we have $\dt{12}(z)\eq 0$ and thus $\dt{13}(z)\!\ne\!0$.
	Taking Part~\ref{Cond:eta_13} into account,
	we see $\wdt{13}(z)\!\ne\!0$.
	
	If $\cC_z$ contains exactly one node $q(z)$ and that node is separating, once again
	there are two possibilities:
	either the core $F_z$ of $\cC_z$ is smooth, or $\cC_z$ consists of two smooth genus 1 subcurves attached at $q(z)$.
	In the former case,
	$\wdt{13}(z)\!\ne\!0$ because of
	Lemma~\ref{Lm:str_homom_one_tail_chain_minor}~\ref{Cond:eta_13_lm}.
	In the latter case,
	Lemma~\ref{Lm:str_homom_one_tail_chain_minor}~\ref{Cond:th_13_sep_nodal} implies $\dt{13}(z)\!\ne\!0$,
	which by Part~\ref{Cond:eta_13} further implies $\wdt{13}(z)\!\ne\!0$.
	
	In sum, let $\cZ\!\subset\!\cV$ consist of all $z\inn\cV$ such that $\cC_z$ contains at least two nodes.
	The previous two paragraphs ensure that
	$$
	\wdt{13}(z)\ne 0\qquad\forall~z\in\cK_{12}\bsl\cZ.
	$$
	Since $\ka_{12}$ is coprime to any node-smoothing parameter, 
	we see $\cZ\!\cap\!\cK_{12}$ is of codimension 2 in $\cK_{12}$,
	hence $\wdt{13}$ is invertible on $\cK_{12}$. 
	This establishes Part~\ref{Cond:eta_13_ns}.
	
	\vsp
	\ref{Cond:eta_13_ns'}.
	Assume $q_1$ is on a maximal non-separating bridge $\tn B[\fp,\fq]$. 
	By~Part~\ref{Cond:eta_12}, there exists $u\inn\Ga(\sO_\cV^*)$ such that
	$\wdt{12}\eq u\cdot\ka_{12}$.
	Since $\lr{\de_1}\eq\lr{\de_2}\eq q_1$ in this case,
	we choose a smooth point $x_{q_1}\inn\tn B[\fp,\fq]$ near $q_1$ and conclude from Part~\ref{Cond:K_nonsep_brid} of Lemma~\ref{Lm:K_smoothness} that there exist $f_2,g_2\inn \Ga(\sO_\cV^*)$ such that 
	\begin{align}
		\label{Eqn:whth12}
		\wdt{12}= f_2\ze_{[x_{q_1},\fp]}+g_2\ze_{[x_{q_1},\fq]}.
	\end{align}
	
	Below, we show the remaining statements of Part~\ref{Cond:eta_13_ns'} in the following steps.
	\begin{enumerate}[label=(${\tn R}_4$.\alph*),leftmargin=*]
		\item\label{Cond:R4-1}
		We first show~\ref{Cond:eta_13_ns'} can be reduced to the case when
		$\tn B[\fp,\fq]$
		is smooth.
		
		\item \label{Cond:R4-2}
		Under the assumption that $\tn B[\fp,\fq]$ is smooth, we then show  $\wdt{13}\eq u'\wdt{12}\!+\!v'\ze_\fp$ for some $u'\inn\Ga(\sO^*_\cV)$ and $v'\inn\Ga(\sO_\cV)$.
		
		\item \label{Cond:R4-3}
		Finally, we show $v'$ in \ref{Cond:R4-2} is invertible.
	\end{enumerate}
	The above three steps together imply \ref{Cond:eta_13_ns'}.
	
	\ref{Cond:R4-1}.
	Let $C_1\!\subset\!\tn B[\fp,\fq]$ be the irreducible component containing $q_1$. Reordering the points of $D\eq\de_1\!+\!\cdots\!\de_m$ if necessary, we assume $\de_{\ell+1},\ldots,\de_m$, for some $r\!\le\!\ell\!\le\!m$, be all the points of $D$ such that either
	they themselves or the pivotal nodes of the rational tails containing them lie on $\tn B[\fp,\fq]\bsl C_1$.
	If $m\eq\ell$, then the stability of $(C,D)$ implies $\tn B[\fp,\fq]\eq C_1$,
	so we assume $m\!>\!\ell$ below.
	
	Shrinking $\cV$ if necessary,
	we can find a small chart $\ov\cV$ of $\fM_2^{\rm div}(m\!-\!1)$,
	the connected component of $\fM_2^{\rm div}$ whose divisors are of degree $m\!-\!1$,
	such that the forgetful morphism $\varpi:\cV\!\lra\!\ov\cV$ that forgets the last marked point (corresponding to $\de_m$) is well-defined.
	With $\ov\rho\!:\ov\cC\to\ov\cV$ denoting the universal family,
	and $\ov\cD\!\subset\!\ov\cC$ denoting the tautological divisor,
	the forgetful morphism $\varpi$ leads to a contraction
	$\ti \varpi: \cC\!\lra\!\ov \cC$
	so that the diagram
	\begin{align}\label{Eqn:forgetful_diagram}
	\begin{CD}
		\cC @>{\ti\varpi}>> \ov\cC \\
		@VV{\rho}V  @VV{\ov\rho}V \\
		\cV @>{\varpi}>> \ov\cV
	\end{CD}
	\end{align}
	commutes.
	By writing $\overline{\cD}_i\eq\ti \varpi(\cD_i)$ for all $1\!\le\!i\!\le\!\ell$,
	the tautological divisor $\ov\cD\!\subset\!\ov\cC$ can then be written as $\ov\cD=\overline{\cD}_1+\cdots+\ov\cD_\ell.$
	
	Since none of $\cD_1,\ldots,\cD_r$, $\cA$, and $\cB$ meets the contracted locus of $\ti\varpi$, we have 
	$$\ti \varpi\upmo(\bcD_{i})\eq\cD_{i}\ \ \forall~1\!\le\!i\!\le\!r,\quad
	\ti \varpi\upmo(\ov\cA)\eq\cA,\quad\tn{and}\quad
	\ti \varpi\upmo(\ov\cB)\eq\cB.
	$$
	Hence
	$$
	\ti \varpi\sta \sO_{\bcC}({\bcD_{i}}\!+\!\bcA\!-\!\bcB)=\sO_{\cC}({\cD_{i}}\!+\!\cA\!-\!\cB)\ \ \forall~1\!\le\!i\!\le\!r,\quad
	\ti \varpi\sta \sO_{\bcA}(\bcA)=\sO_{\cA}(\cA).
	$$
	Therefore,
	\begin{align*}
	&
	\rho\lsta \sO_{\cC}({\cD_{i}}+\cA-\cB)=
	\rho\lsta \ti\varpi\sta\sO_{\ov\cC}(\ov{\cD}_{i}+\ov\cA-\ov\cB)=
	\varpi\sta\ov\rho\lsta \sO_{\ov\cC}(\ov{\cD}_{i}+\ov\cA-\ov\cB),\\
	&
	\rho\lsta \sO_{\cA}(\cA)=
	\rho\lsta\ti \varpi\sta \sO_{\bcA}(\bcA)=
	\varpi\sta\ov\rho\lsta \sO_{\bcA}(\bcA),
	\end{align*}
	where the last equality of each line follows from direct computation
	(for instance, see~\cite[Theorem~4.4]{SS}) since $\rho$ and $\ov\rho$ are both separated and locally proper.
	Consequently, the restriction homomorphism on the target
	$$
	 \ohmr:\bigoplus_{1\le i\le r}
	 \ov\rho_*\sO_{\ov\cC}(\ov\cD_i\!+\!\ov\cA\!-\!\ov\cB)
	 \lra \ov\rho_*\sO_{\bcA}(\bcA)
	$$
	exactly pulls back to the restriction homomorphism on the source
	$$
	\hmr:\bigoplus_{1\le i\le r}
	\rho_*\sO_{\cC}(\cD_i\!+\!\cA\!-\!\cB)
	\lra \rho_*\sO_{\cA}(\cA).
	$$
	
	Then,
	we apply the above forgetful morphism inductively so as to drop all the marked points labeled by $\ell\!+\!1,\ldots,m$ (corresponding to $\de_{\ell+1},\ldots,\de_{m}$),
	so particularly  $\tn B[\fp,\fq]\bsl C_1$ is contracted.
	Slightly abusing the notation,
	we still denote by $\ohmr$
	the resulting restriction homomorphism.
	If the statements of Part~\ref{Cond:eta_13_ns'} holds for $\ohmr$, then by applying Lemma~\ref{Lm:Forgetful_local} inductively,
	we see the statements of Part~\ref{Cond:eta_13_ns'} holds for the original $\hmr$.
	
	\ref{Cond:R4-2}.
	Assume $B[\fp,\fq]$ is smooth.
	Then, (\ref{Eqn:whth12}) can be rewritten as
	\begin{align}\label{Eqn:wdt12_ns}\wdt{12}=f_2\ze_\fp+g_2\ze_\fq.
	\end{align}
	With $\varphi_i$ as in~(\ref{varphi}),
	we consider
	$\varphi_1\!\oplus\!\varphi_3$.
	On the one hand, the determinant $\th_{13}$ remains unchanged because $c_{si}$ are determined by Proposition~\ref{HomFirstOrder} (hence independent of the choice of $S_i$'s in Lemma~\ref{sublemma});
	on the other hand,
	by applying Part~\ref{Cond:eta_12} of Proposition~\ref{keyProp1} to $\varphi_1\!\oplus\!\varphi_3$,
	we observe $\th_{13}$ should take the same form as $\th_{12}$ does in $\varphi_{\le r}$,
	so particularly it contains $\ka_{13}$ as a factor.
	Thus by Lemma~\ref{Lm:K_smoothness}, one can write
	\begin{align}\label{Eqn:th13_ns}
	\dt{13}=f'\ze_\fp+g'\ze_\fq
	\end{align}
	for some $f',g'\inn\Ga(\sO_\cV)$.
	Then, by~\ref{Cond:eta_13} and the fact that the node smoothing parameters are independent,
	we have
	$$
	\wdt{13}=f_3\ze_\fp+g_3\ze_\fq
	$$
	for some $f_3,g_3\inn\Ga(\sO_\cV)$.
	By setting $$\eta:=\det\left[\begin{matrix}
		f_3 & g_3\\ f_2 & g_2
	\end{matrix}\right],$$
	we can further write $\wdt{13}$ as 
	\begin{align}
		\label{Eqn:wdt13_ns}
		\wdt{13}=\frac{\eta}{g_2}\ze_\fp+\frac{g_3}{g_2}\wdt{12}.
	\end{align}
	
	\ref{Cond:R4-3}.
	It suffices to show~$\eta(z_0)\!\ne\!0$ because $\cV$ is assumed small.
	Consider the locus $$\cZ:=(\ze_{q_1}= 0,\,\ka_{12}= 0)\quad\subset\cV,$$
	which is smooth and contains $z_0$.
	Indeed, we claim
	$\eta|_\cZ$ is nowhere zero.
	
	To establish the above claim,
	consider the sub-locus $\cZ'\!\subset\!\cZ$ consisting of $z\inn\cZ$ such that the core $F_z$ of $\cC_z$ contains at most one node.
	Then,
	$F_z$ does not have any non-separating bridge (so w.l.o.g.~we assume $\ze_{\fp}(z)\!\ne\!0$),
	and the marked points corresponding to $\de_1,
	\ldots,\de_r$ are on a rational tail whose pivotal node (corresponding to $q_1$) is a Weierstrass point on $F_z$.
	By~\ref{Cond:eta_13_ns} and (\ref{Eqn:wdt13_ns}),
	we have $\eta(z)\!\ne\!0$.
	Therefore,
	$\eta|_\cZ$ is nowhere zero outside the codimension~2 subset $\cZ\bsl\cZ'$,
	hence is nowhere zero.
\end{proof}
	
We conclude this subsection with the following implications on the relations between
$\ka_{12}$, $\ka_{13}$ and $\ka_{23}$,
provided $\{\de_1,\de_2,\de_3\}$ is on a chain $R$ of tail rational subcurves of the center $z_0\eq(C,D)$ of $\cV$, ordered according to~(\ref{linearOrder}).

\begin{lemm}\label{Lm:str_homom_one_tail_chain_minor_12vs13}
	With notation as in Proposition~\ref{keyProp1}, there exist $v\inn\Ga(\sO_\cV^*)$ 
	and $w\inn\Ga(\sO_\cV)$ so that
	$$
	\ka_{13}=v\cdot\ka_{12}+w,
	$$ and the following holds.
	\begin{enumerate}[label=(\alph*),leftmargin=*]
		\item \label{Part:k1213_general} 
		If $q_1$ does not lie on any non-separating bridge of the core of $C$, then there exists $u\inn\Ga(\sO^*_\cV)$ such that $w=u\cdot\ze_{[\de_2]}$.
		
		\item \label{Part:k1213_nsb}
		If $q_1$ lies on a maximal non-separating bridge $\tn B[\fp,\fq]$ of the core $F$ of $C$,
		let $x_{q_1}\inn F$ be an arbitrary smooth point  that is sufficiently close to the node~$q_1$.
		Then, there exist $f, g\inn\Ga(\sO^*_\cV)$ and $f', g'\inn\Ga(\sO_\cV)$ so that
		\begin{align*}
			\ka_{12}
			&=f\ze_{[x_{q_1},\fp]}+g\ze_{[x_{q_1},\fq]},\qquad
			w
			=f'\cdot\!\!\!\!\prod_{q\in N_{[\de_2,\fp]}}\!\!\!\!\!\ze_q
			\ +\ 
			g'\cdot\!\!\!\!\prod_{q\in N_{[\de_2,\fq]}}\!\!\!\!\!\ze_q;
		\end{align*}
		moreover,
		$\det\left[\begin{matrix}
			f & f'\\ g & g'
		\end{matrix}\right]\inn\Ga(\sO^*_\cV)$.
	\end{enumerate} 
\end{lemm}

\begin{proof}
W.l.o.g.~we assume $\ze_{[q_1,a_1]}$ divides $\ze_{[q_1,a_2]}$ evenly in this proof.
Moreover,
we assume $z_0\inn\cK_{12}$ 
and consequently $z_0\inn\cK_{13}$ (i.e.~$\ka_{12}(z_0)\eq \ka_{13}(z_0)\eq 0$),
for otherwise we may shrink $\cV$ if necessary so that $\ka_{12}$ and $\ka_{13}$ are both units on $\cV$, thus all statements of the corollary trivially hold.

We first fix trivialization $\rho\lsta\sO_{\cA_s}(\cA_s)\cong\sO_\cV$, $s\eq 1,2$.
Taking $r\eq 2$ in Proposition~\ref{keyProp1} and choosing suitable trivialization $\rho\lsta\cM(\cD_1)\oplus \rho\lsta\cM(\cD_2)\cong \sO_\cV^{\oplus 2}$ so that the first component of $\sO_\cV^{\oplus 2}$ corresponds to $\rho\lsta\cM(\cD_1)$,
we have
$$
 \varphi_{\le 2}=
 \left[
 \begin{matrix}
 	\wh c_{11}\ze_{[\de_1,a_1]} & 0\\
 	\wh c_{21}\ze_{[\de_1,a_2]} & \ka_{12}\ze_{[\de_1,a_2]}\ze_{[\de_2]}
 \end{matrix}\right]
 :\rho\lsta\cM(\cD_1)\oplus \rho\lsta\cM(\cD_2)\lra \rho\lsta\sO_{\cA}(\cA).
$$
Similarly,
by applying Proposition~\ref{keyProp1} to $\{\de_1,\de_3\}$ (i.e.~dropping $\de_2$ and all $\de_j$ with $j\!>\!3$) and choosing suitable trivialization $\rho\lsta\cM(\cD_1)\oplus \rho\lsta\cM(\cD_3)\cong \sO_\cV^{\oplus 2}$ so that the first component of $\sO_\cV^{\oplus 2}$ corresponds to $\rho\lsta\cM(\cD_1)$,
we can write the restriction of the structural homomorphism as
$$
\varphi_{1,3}=
\left[
\begin{matrix}
	\wh c_{11}\ze_{[\de_1,a_1]} & 0\\
	\wh c_{21}\ze_{[\de_1,a_2]} & \ka_{13}\ze_{[\de_1,a_2]}\ze_{[\de_3]}
\end{matrix}\right]
:\rho\lsta\cM(\cD_1)\oplus \rho\lsta\cM(\cD_3)\lra \rho\lsta\sO_{\cA}(\cA).
$$
We emphasize the coefficients $\wh c_{11}$ and $\wh c_{21}$ in $\varphi_{1,3}$ are the same as those in $\varphi_{\le 2}$.
Therefore, one can choose suitable trivialization
$\bigoplus_{i=1}^3\rho\lsta\cM(\cD_i)\cong \sO_\cV^{\oplus 3}$ and 
$\rho\lsta\sO_{\cA}(\cA)\cong\sO_\cV^{\oplus 2}$ so that
\begin{align*}
&\varphi_{\le 3}:
\bigoplus_{1\le i\le 3}\rho\lsta\cM(\cD_i)\lra \rho\lsta\sO_{\cA}(\cA),\\
&\varphi_{\le 3}=
\left[
\begin{matrix}
	\wh c_{11}\ze_{[\de_1,a_1]} & 0 & 0\\
	0 & \ka_{12}\ze_{[\de_1,a_2]}\ze_{[\de_2]} & \ka_{13}\ze_{[\de_1,a_2]}\ze_{[\de_3]}
\end{matrix}\right].
\end{align*}

Alternatively,
we can take  $r\eq 2$ in Proposition~\ref{keyProp1} and choose suitable trivialization
$\bigoplus_{i=1}^3\rho\lsta\cM(\cD_i)\cong \sO_\cV^{\oplus 3}$ and 
$\rho\lsta\sO_{\cA}(\cA)\cong\sO_\cV^{\oplus 2}$ so that
\begin{align*}
	&\varphi_{\le 3}:
	\bigoplus_{1\le i\le 3}\rho\lsta\cM(\cD_i)\lra \rho\lsta\sO_{\cA}(\cA),\\
	&\varphi_{\le 3}=
	\left[
	\begin{matrix}
		\wh c_{11}\ze_{[\de_1,a_1]} & 0 & 0\\
		0 & \ka_{12}\ze_{[\de_1,a_2]}\ze_{[\de_2]} & \wdt{13}\ze_{[\de_1,a_2]}\ze_{[\de_2]}\ze_{[\de_3]}
	\end{matrix}\right].
\end{align*}
The above two expressions of $\varphi_{\le 3}$ together imply there exist $v'\inn\Ga(\sO_\cV)$ and  $u'\inn\Ga(\sO^*_\cV)$ such that
$$
 \ka_{13}\ze_{[\de_1,a_2]}\ze_{[\de_3]}=
 v'\cdot\ka_{12}\ze_{[\de_1,a_2]}\ze_{[\de_2]}+
 u'\cdot\wdt{13}\ze_{[\de_1,a_2]}\ze_{[\de_2]}\ze_{[\de_3]}.
$$
By Lemma~\ref{Lm:K_smoothness},
$\ka_{12}$ is coprime with $\ze_q$ for any $q\inn N_{[\de_2,\de_3]}$ regardless of the position of the pivotal node $q_1$.
Taking into account that the node-smoothing parameters are independent,
we conclude that
\begin{align}\label{Eqn:ka_12_vs_13_ns}
	v:=\frac{v'}{\ze_{[\de_2,\de_3]}}\in\Ga(\sO_\cV)\qquad\tn{and}\qquad
	\ka_{13}=
	v\cdot\ka_{12}+
	u'\cdot\wdt{13}\ze_{[\de_2]}.
\end{align}

If $q_1$ does not lie on any non-separating bridge,
then Part \ref{Cond:eta_13_ns} of Proposition \ref{keyProp1} implies that $\wdt{13}$ is invertible.
Hence by writing $u\eq u'\!\cdot\!\wdt{13}\inn\Ga(\sO^*_\cV)$,
we have $\ka_{13}\eq 
v\!\cdot\!\ka_{12}\!+\!
u\!\cdot\!\ze_{[\de_2]}$.
Moreover,
Lemma~\ref{Lm:K_smoothness} implies $\ka_{12}$ and $\ka_{13}$, as well as their pullbacks under any forgetful morphism, are always smooth local parameters,
hence $v$ must be invertible.
This completes the proof of Part~\ref{Part:k1213_general}.

If $q_1$ is on a maximal non-separating bridge $\tn B[\fp,\fq]$,
then
Part \ref{Cond:eta_13_ns} of Proposition \ref{keyProp1} implies that
\begin{align*}
	&
	\ka_{12}=f\!\cdot\!\ze_{[x_{q_1},\fp]}+g\!\cdot\!\ze_{[x_{q_1},\fq]},\qquad
	\wdt{13}=f'\!\cdot\!\ze_{[x_{q_1},\fp]}+g'\!\cdot\!\ze_{[x_{q_1},\fq]},\\
	&
	\det\left[\begin{matrix}
		f & f'\\g & g'
	\end{matrix}\right]\in\Ga(\sO^*_\cV).
\end{align*}
Hence by~(\ref{Eqn:ka_12_vs_13_ns}),
we have
\begin{align*}
	&
	\ka_{13}=
	v\!\cdot\!f\!\cdot\!\ze_{[x_{q_1},\fp]}+
	v\!\cdot\!g\!\cdot\!\ze_{[x_{q_1},\fq]}+
	u'\cdot\!f'\!\cdot\!\ze_{[\de_2,\fp]}+
	u'\cdot\!g'\!\cdot\!\ze_{[\de_2,\fq]},\\
	&
	\det\left[\begin{matrix}
		f & u'\!\cdot\!f'\\g & u'\!\cdot\!g'
	\end{matrix}\right]\in\Ga(\sO^*_\cV).
\end{align*}
Finally,
by applying Lemma~\ref{Lm:K_smoothness} to $\ka_{13}$,
we see that $v$ must be invertible.
This completes the proof of Part~\ref{Part:k1213_nsb}.
\end{proof}

Lemma~\ref{Lm:str_homom_one_tail_chain_minor_12vs13} reveals the relation between $\ka_{12}$ and $\ka_{13}$,
provided $\{\de_1,\de_2,\de_3\}$ is on a chain $R$ of tail rational subcurves, ordered according to~(\ref{linearOrder}).
Indeed,
analogous relations between $\ka_{12}$ and $\ka_{23}$ and between  $\ka_{13}$ and $\ka_{23}$ also hold.
Such relations are summarized below.

Recall that for any smooth points  $\de_i$ and $\de_j$ on $R$, $\de_i\!\wedge\!\de_j$ refers to the node in~(\ref{product});
intuitively it is a node between the core and both $\de_i$ and $\de_j$, and it is the ``farthest away'' from the core among such nodes.

\begin{coro}\label{Crl:str_homom_one_tail_chain_minor_12vs23}
With notation as in Proposition~\ref{keyProp1},
let $(i,j,k)$ be an arbitrary permutation of $\{1,2,3\}$.
There then exist $v\inn\Ga(\sO_\cV^*)$ 
and $g\inn\Ga(\sO_\cV)$ so that
$$
\ka_{ik}=v\cdot\ka_{ij}+w,
$$ and the following holds.
\begin{enumerate}[label=$\bullet$,leftmargin=*]
	\item 
	If $q_1$ does not lie on any non-separating bridge of the core of $C$, then there exists $u\inn\Ga(\sO^*_\cV)$ such that $w=u\cdot\ze_{[\de_j\wedge\de_k]}$.
	
	\item
	If $q_1$ lies on a maximal non-separating bridge $\tn B[\fp,\fq]$ of the core $F$ of $C$,
	let $x_{q_1}\inn F$ be an arbitrary smooth point  that is sufficiently close to the node~$q_1$.
	Then, there exist $f, g\inn\Ga(\sO^*_\cV)$ and $f',g'\inn\Ga(\sO_\cV)$ so that
	\begin{align*}
		\ka_{ij}
		&=f\cdot\prod_{q\in N_{[x_{q_1},\fp]}}\!\!\!\!\!\ze_q
		\ \ +\ g\cdot\prod_{q\in N_{[x_{q_1},\fq]}}\!\!\!\!\!\ze_q\ ,
		\\
		w
		&=f'\cdot\prod_{q\in N_{[\de_j,\fp]}\cap N_{[\de_k,\fp]}}\!\!\!\!\!\!\!\!\!\!\ze_q
		\ \ +\ 
		g'\cdot\prod_{q\in N_{[\de_j,\fq]}\cap N_{[\de_k,\fq]}}\!\!\!\!\!\!\!\!\!\!\ze_q\quad;
	\end{align*}
	moreover, 
	$\det\left[\begin{matrix}
		f & f'\\ g & g'
	\end{matrix}\right]\inn\Ga(\sO^*_\cV)$.
\end{enumerate} 
\end{coro}

\begin{proof}
Since $\ka_{ij}\eq\ka_{ji}$ for all $1\!\le\!i,j\!\le\!r$,
it suffices to show Corollary~\ref{Crl:str_homom_one_tail_chain_minor_12vs23}for the permutations $(1,2,3)$,
$(2,3,1)$ and $(3,1,2)$.
In addition,
the $(i,j,k)\eq(1,2,3)$ case is covered in Lemma~\ref{Lm:str_homom_one_tail_chain_minor_12vs13}.

If $(i,j,k)\eq(2,3,1)$ (i.e.~to find relations between $\ka_{12}$ and $\ka_{23}$),
notice that
by applying Proposition~\ref{keyProp1} to $\{\de_2,\de_3\}$ (i.e.~dropping $\de_1$ and all $\de_j$ with $j\!>\!3$) and choosing suitable trivialization $\rho\lsta\cM(\cD_2)\oplus \rho\lsta\cM(\cD_3)\cong \sO_\cV^{\oplus 2}$ so that the first component of $\sO_\cV^{\oplus 2}$ corresponds to $\rho\lsta\cM(\cD_2)$,
we can write the restriction of the structural homomorphism as
$$
\varphi_{2,3}=
\left[
\begin{matrix}
	 c_{12}\ze_{[\de_2,a_1]} & 0\\
	 c_{22}\ze_{[\de_2,a_2]} & \ka_{23}\ze_{[\de_2,a_2]}\ze_{[\de_3]}
\end{matrix}\right]
:\rho\lsta\cM(\cD_2)\oplus \rho\lsta\cM(\cD_3)\lra \rho\lsta\sO_{\cA}(\cA).
$$
The coefficients $c_{12}$ and $c_{12}$ in $\varphi_{2,3}$ are the same as in Proposition~\ref{HomFirstOrder}.	
Therefore, one can choose suitable trivialization
$\bigoplus_{i=1}^3\rho\lsta\cM(\cD_i)\cong \sO_\cV^{\oplus 3}$ and 
$\rho\lsta\sO_{\cA}(\cA)\cong\sO_\cV^{\oplus 2}$ so that
\begin{align*}
	&\varphi_{\le 3}:
	\bigoplus_{1\le i\le 3}\rho\lsta\cM(\cD_i)\lra \rho\lsta\sO_{\cA}(\cA),\\
	&\varphi_{\le 3}=
	\left[
	\begin{matrix}
		\wh c_{11}\ze_{[\de_1,a_1]} & 0 & 0\\
		0 & \ka_{12}\ze_{[\de_1,a_2]}\ze_{[\de_2]} & \ka_{23}\ze_{[\de_2,a_2]}\ze_{[\de_3]}
	\end{matrix}\right].
\end{align*}
Parallel to~(\ref{Eqn:ka_12_vs_13_ns}),
we can thus find $v\inn\Ga(\sO_\cV)$ and $u'\inn\Ga(\sO^*_\cV)$ so that
\begin{align*}
	\ka_{23}=
	v\cdot\ka_{12}+
	u'\cdot\wdt{13}\ze_{[\de_1]}.
\end{align*}
The remainder of the proof of the $(i,j,k)\eq(2,3,1)$ case is analogous to the proof of Lemma~\ref{Lm:str_homom_one_tail_chain_minor_12vs13} after~(\ref{Eqn:ka_12_vs_13_ns}),
hence is omitted.

If $(i,j,k)\eq(3,1,2)$ (i.e.~to find relations between $\ka_{13}$ and $\ka_{23}$),
we observe that it suffices to verify the statements after forgetting all the marked points except $\de_1,$ $\de_2$ and $\de_3$. (Recall locally we identify $\cV$ with a smooth chart of $\ov\cM_{2,m}$.)
Then the fact that $\ka_{ij}$'s are defined via pullbacks and Lemma~\ref{Lm:Forgetful_local} lead to the statements for the $(i,j,k)\eq(3,1,2)$ case.

To show the $(i,j,k)\eq(3,1,2)$ case after forgetting all the marked points except $\de_1,$ $\de_2$ and $\de_3$,
just notice that under the forgetful map,
either the images of $\de_1,$ $\de_2$ and $\de_3$ are on a smooth rational subcurve $C'$,
or the images of  $\de_2$ and $\de_3$ are on a smooth rational subcurve $C'$, which attaches to another smooth rational subcurve $C''$ that contains the image of $\de_1$ and attaches to the core at a pivotal node.
In either case,
however, the symmetry between $\de_2$ and $\de_3$,
as well as the statements for $\ka_{12}$ and $\ka_{23}$ that has already been established,
give rise to the relations between  $\ka_{13}$ and $\ka_{23}$.
\end{proof}

\subsection{Proof of Proposition~\ref{Prp:phi_key}}
\label{Subsec:phi_proof}

In this subsection, we complete the proof of Proposition~\ref{Prp:phi_key} using Propositions~\ref{Prp:key-tail_chain},~\ref{keyProp1} and Corollary~\ref{Crl:str_homom_one_tail_chain_minor_12vs23}.

\ref{Part:phi} is a restatement of Proposition~\ref{HomFirstOrder}.

To show \ref{Part:theta}, we divide the arguments into three cases, depending on the location of $\de_i$ and $\de_j$.

{\bf Case~2.a:} {\it $\de_i$ and $\de_j$ belong to a chain of tail rational subcurves.}
This case simply follows from replacing $\de_1$ and $\de_2$ by $\de_i$ and $\de_j$ in Proposition~\ref{keyProp1}~\ref{Cond:eta_12}.

{\bf Case~2.b:} {\it$\de_i$ and $\de_j$ belong to the same tail, but lie on different maximal chains of tail rational subcurves.}
We write the maximal chains as 
$$
R^i:=R^i_1\cup\cdots\cup R^i_{h_i}\qquad\tn{and}\qquad
R^j:=R^j_1\cup\cdots\cup R^j_{h_j},
$$ 
respectively,
whose rational subcurves are ordered in the same way as in \S\ref{Subsec:using-ss}.
Obviously, $R^i_1\eq R^j_1$.
Let
$$
 \wh\mu:=\max\{\mu\inn \mathbb Z_{>0}:\; 1\!\le\!\mu\!\le\!\min\{h_i,h_j\},~R^i_\mu\eq R^j_\mu\},
 \qquad
 q_{\wh\mu}:=q^i_{\wh\mu}=q^j_{\wh\mu},
$$
where the nodes $q^i_{\wh\mu}$ and $q^j_{\wh\mu}$ are as in~(\ref{e_q}).
It is straightforward that $q_{\wh\mu}\eq \de_i\!\wedge\!\de_j$.

Next, we apply an argument parallel to the proof of Proposition~\ref{keyProp1} \ref{Cond:eta_13_ns'}.
W.l.o.g.~we assume that 
\begin{align*}
	D\cap(R^i_{\wh\mu}\cup R^i_{\wh\mu+1}\cup\cdots\cup R^i_{h_i})
	=\de_i+\big(\de_{\ell+1}+\de_{\ell+2}+\cdots+\de_m\big)
\end{align*}
for some $\ell$ satisfying $\max\{i,j\}\!\le\!\ell\!\le\!m$.
We then forget the marked points $\de_{\ell+1},\ldots,\de_m$, take stabilization, and define the points $\big(\ov C(z),\ov D(z)\big)\inn\fM^{\rm div}_2$, the universal family $\ov\rho:\ov\cC\!\to\!\ov\cV$, the tautological divisor $\ov\cD\!\subset\!\ov\cC$, the forgetful map $\varpi:\cV\!\to\!\ov\cV$,
and the contraction $\ti\varpi:\cC\!\to\!\ov\cC$ analogously so that (\ref{Eqn:forgetful_diagram}) commutes.

In the current scenario,
notice that we still have 
$$
\varpi^{-1}(\ov\cD_k)=\cD_k\ \ \forall~1\!\le\!k\!\le\!\ell,\quad
\varpi^{-1}(\ov\cA)=\cA,\quad\tn{and}\quad
\varpi^{-1}(\ov\cB)=\cB.
$$
As in the proof of Proposition~\ref{keyProp1} \ref{Cond:eta_13_ns'},
we conclude that $\ov\varphi_{\le\ell}$ pulls back to $\varphi_{\le\ell}$.
Particularly, $\ov\th_{ij}$ pulls back to $\dt{ij}$.

Meanwhile, observe that on $\ov C(z_0)$,
the rational subcurves of the maximal chain $R^j$ are not contracted by $\ti\varpi$,
and the images $\ov\de_i$ and $\ov\de_j$ of $\de_i$ and $\de_j$ both belong this chain.
So by replacing $\varphi_{\le r}$, $\de_1$, and $\de_2$ in Proposition~\ref{keyProp1} with $\ov\varphi_{\le\ell}$, $\ov\de_i$, and $\ov\de_j$, respectively,
we have $$\ov\th_{ij}=\ov u\cdot\ze_{[\ov\de_i]}\cdot\ov\ka_{ij}
=\ov u\cdot\ze_{[q_{\wh\mu}]}\cdot\ov\ka_{ij}
\qquad\tn{with}\quad\ov u\in\Ga(\sO^*_{\ov\cV}).$$
Since $\ze_{[q_{\wh\mu}]}$ and $\ov\ka_{ij}$ pull back to $\ze_{[\de_i\wedge\de_j]}$ and $\ka_{ij}$, respectively,
Case B of~\ref{Part:theta} is thus established.

{\bf Case~2.c:} {\it either $\de_i$ and $\de_j$ belong to different tails, or at least one of them is on the core of $C$.}
In this case,
notice that $\ze_{[\de_i\wedge\de_j]}$ is set to be 1 (see~the sentence below~(\ref{Eqn:wedge})).
Also notice the statements~\ref{Cond:th_12_ns_lm}, \ref{Cond:th_12_order}, and~\ref{Cond:th_12_sep_nodal} of Lemma~\ref{Lm:str_homom_one_tail_chain_minor} remain valid by replacing $\de_1$ and $\de_j$ with $\de_i$ and $\de_j$,
whereas~\ref{Cond:eta_12_lm} can be replaced by the following:
$\dt{ij}(z_0)\eq 0$ if and only if $\de_i$ and $\de_j$ are conjugate.
Mimicking the proof of Part~\ref{Cond:eta_12} of Proposition~\ref{keyProp1},
we then observe that $\dt{ij}$ still takes the form $u\!\cdot\!\ka_{ij}$ for some $u\inn\Ga(\sO_\cV)$,
and such $u$ is still nowhere vanishing on $\cV$ except possibly on a codimension 2 subset,
hence $u$ is invertible.
This completes the proof of~\ref{Part:theta}.

To establish the statements in~\ref{Part:kappa},
first observe that if one of $\ka_{ik}$ and $\ka_{ij}$ is invertible on $\cV$,
then $\lr{\de_i}$, $\lr{\de_j}$, and $\lr{\de_k}$ cannot be on a common non-separating bridge,
so one can always choose invertible functions $v$ and $u'$ on $\cV$ so that
$$
 \ka_{ik}=v\cdot\ka_{ij}+u'\cdot\ze_{[\de_j\wedge\de_k]}
$$
holds. So hereafter we assume \begin{align}\label{Eqn:ka_assump}
	\ka_{ik}(z_0)=\ka_{ij}(z_0)= 0.
\end{align}
Moreover,
since $\ov\cK_{ik}$ and $\ov\cK_{ij}$ are defined using pullbacks via the forgetful morphisms,
it suffices to justify~\ref{Part:kappa} by forgetting all the remaining elements of $D$ (recall we consider $D$ a finite set of points).
In other words,
we assume 
$$
D=\{\de_i,\de_j,\de_k\}
$$
and treat $(C,D)$ as a point in $\ov{\cM}_{2,3}$.

Below we divide the argument into five cases.

{\bf Case 3.a:} {\it $\de_i$, $\de_j$ and $\de_k$ belong to the same tail.}
Notice that in this scenario, there exists a (unique) maximal chain of tail rational subcurves containing these three points.
This case simply follows from replacing $\de_1$, $\de_2$ and $\de_3$ by $\de_i$, $\de_j$ and $\de_k$ in Corollary~\ref{Crl:str_homom_one_tail_chain_minor_12vs23}.

{\bf Case 3.b:} {\it $\de_k$ is on the core $F$ of $C$, whereas $\de_i$ and $\de_j$ belong to the same tail.}
In this scenario,
we have 
\begin{align*}
	\de_k=\lr{\de_k}\qquad\tn{and}\qquad
	\lr{\de_i}=\lr{\de_j}\ =:q_i\,.
\end{align*}
Let $x_i\inn F$ be a smooth point that is sufficiently close to the node $q_i$.
Notice that the assumption~(\ref{Eqn:ka_assump}) only hold when $q_i$ and $\de_k$ are on a maximal non-separating bridge, denoted by $\tn B[\fp,\fq]$.
Notice that $\tn B[\fp,\fq]$ consists of either one or two (rational) subcurves.

If $\tn B[\fp,\fq]$ consists of two rational subcurves $R_\fp$ and $R_\fq$, respectively containing $\fp$ and $\fq$,
we denote by $\fn$ the node $R_\fp\!\cap\!R_\fq$.
W.l.o.g.~we assume $\de_k\inn R_\fp$ and $q_i\inn R_\fq$.
Then by Lemma~\ref{Lm:K_smoothness},
there exist $f,g,\ti f,\ti g\inn\Ga(\sO^*_\cV)$ such that
\begin{align*}
 \ka_{ij}=f\ze_\fn\ze_\fp+g\ze_\fq,\qquad
 \ka_{ik}=\ti f\ze_\fp+\ti g\ze_\fq.
\end{align*}
Therefore,
$$
 \ka_{ik}=\tfrac{\ti g}{g}\ka_{ij}+
 f'\cdot\ze_\fq,\qquad
 \tn{where}\quad
 f':=\ti f-\tfrac{f\cdot\ti g}{g}\ze_{\fn}.
$$
Since $\ti f$ is invertible, $\ze_{\fn}(z_0)\eq 0$ and $\cV$ is sufficiently small,
we see
$w\eq f'\!\cdot\!\ze_\fp+ 0\!\cdot\ze_\fq$,
$f'$ is invertible, and so is $\det\left[\begin{matrix}
	f & f'\\ g & 0
\end{matrix}\right]$.

If $\tn B[\fp,\fq]$ consists of only one rational subcurve $R_0$ (thus $\fp, \fq, q_i\inn R_0$),
there exist $f,g,\ti f,\ti g\inn\Ga(\sO^*_\cV)$ such that
\begin{align*}
	\ka_{ij}=f\ze_\fp+g\ze_\fq,\qquad
	\ka_{ik}=\ti f\ze_\fp+\ti g\ze_\fq.
\end{align*}
Therefore,
$$
\ka_{ik}=\tfrac{\ti g}{g}\ka_{ij}+
f'\cdot\ze_\fq,\qquad
\tn{where}\quad
f':=\ti f-\tfrac{f\cdot\ti g}{g}.
$$
It remains to show $f'$ is invertible.

To this end,
we mimic the last two paragraphs of the proof of Proposition~\ref{keyProp1} \ref{Cond:eta_13_ns'}.
More precisely,
if the core $F$ is inseparable,
by Lemma~\ref{Lm:K_smoothness}, $\cK_{ij}$ is still a smooth Cartier divisor of $\cV$,
so we can take a small curve $\De\!\subset\!\cK_{ij}$ through $z_0$,
meeting $(\ze_\fq\eq 0)$ (and hence $(\ze_\fp\eq 0)$ because of the expression of $\ka_{ij}$ above) transversely at $z_0$
so that the node $q_i$ is not smoothed out along $\De$.
Let $\cC_\Delta$, $\cQ_{\Delta,i}$, 
$\cD_{\De,k}$,
$\cA_\Delta$, and $\cB_\Delta$ be the analogues of those in the proof of Proposition~\ref{keyProp1} \ref{Cond:eta_13_ns'}.
Then, we obtain the evaluation homomorphism
\begin{align*}
	&\begin{CD}\rho\lsta \sO_{\cC_\Delta}(2\cQ_{\De,i}+\cD_{\De,k}+\cA_\Delta-\cB_\Delta) @>{ \Th_{i,k}|_\Delta}>>\rho\lsta \sO_{\cA_\Delta}(\cA_\Delta) \\
	\end{CD},\\
	&\tn{where}\qquad\Th_{i,k}|_\Delta=
	\left.\left[\begin{matrix}
		c_{1i} & \wh c_{1j} & c_{1k} \\  c_{2i} & \wh c_{2j} & c_{2k} \end{matrix} \right]\right|_\Delta.
\end{align*}
which after an elementary row operation can be rewritten as
$$
\left.\left[\begin{matrix}
	c_{11} & \wh c_{12} & c_{13} \\  0 &  0 & f'\ze_\fq \end{matrix} \right]\right|_\Delta.
$$
Following the argument in the penultimate paragraph of Proposition~\ref{keyProp1} \ref{Cond:eta_13_ns'},
we see $f'\ze_\fq|_\Delta$ vanishes at $z_0$
to the first order, hence $f'(z_0)\!\ne\!0$ and $f'$ is invertible on $\cV$ (which is assumed small).

If the core $F$ contains a separating node,
we consider the locus $\cZ:=(\ze_{q_i}\eq \ka_{ij}\eq 0)$,
which is smooth and contains $z_0$.
From the last paragraph we see $\eta(z)\!\ne\!0$ for all $z\inn\cZ$ satisfying the core of $\cC_z$ is inseparable.
For any $z\inn\cZ$ satisfying that the core of $\cC_z$ contains only one node and that node is separating,
$\cQ_i(z)$ is a Weierstrass point by Lemma~\ref{Lm:KW},
so $\cD_k(z)$ cannot be conjugate to $\cD_i(z)$.
Thus, $\ka_{ik}(z)\!\ne\!0$, so $\eta(z)\!\ne\!0$.
This implies $\eta$ is nowhere zero on $\cZ$ except possibly on a codimension~2 subset of $\cZ$,
hence is invertible on $\cZ$ and thus on $\cV$.
This completes the proof of Case 3.b.

{\bf Case 3.c:} {\it $\de_j$ is on the core $F$ of $C$, whereas $\de_i$ and $\de_k$ belong to the same tail.}
This case is parallel to Case 3.b due to the symmetry between $\de_j$ and $\de_k$ in the statements of~\ref{Part:kappa}.

{\bf Case 3.d:} {\it $\de_i$ is on the core $F$ of $C$, whereas $\de_j$ and $\de_k$ belong to the same tail.}
In this case, notice that
\begin{align*}
	\de_i=\lr{\de_i}\qquad\tn{and}\qquad
    \lr{\de_j}=\lr{\de_k}\ =:q_j.
\end{align*}
Let $x_j\inn F$ be a smooth point sufficiently close to the node $q_j$.

\begin{itemize}[leftmargin=*]
	\item If $\de_i$ and $q_j$ are on a maximal non-separating bridge $\tn B[\fp,\fq]$,
	then by Lemma~\ref{Lm:K_smoothness}, there exist $f,g,\ti f,\ti g\inn\Ga(\sO^*_\cV)$ such that
	\begin{align}\label{Eqn:ka_ij_ka_ik}
		\ka_{ij}= f\ze_\fp+g\ze_\fq,\qquad\ka_{ik}= \ti f\ze_\fp+\ti g\ze_\fq,
	\end{align}
	regardless of the number of irreducible components of $\tn B[\fp,\fq]$.
	Mimicking the $(i,j,k)\eq(2,3,1)$ case of the proof of Corollary~\ref{Crl:str_homom_one_tail_chain_minor_12vs23},
	we apply~\ref{Part:theta} and choose suitable trivialization
	$\rho\lsta\cM(\cD_i)\!\oplus\!\rho\lsta\cM(\cD_j)\!\oplus\!\rho\lsta\cM(\cD_k)\cong \sO_\cV^{\oplus 3}$ and 
	$\rho\lsta\sO_{\cA}(\cA)\cong\sO_\cV^{\oplus 2}$ so that the structural homomorphism 
	\begin{align*}
		&\varphi:\;
		\rho\lsta\cM(\cD_i)\oplus\rho\lsta\cM(\cD_j)\oplus\rho\lsta\cM(\cD_k)\lra \rho\lsta\sO_{\cA}(\cA),
	\end{align*}
	can be written as
	\begin{align*}
		&\varphi=
		\left[
		\begin{matrix}
			c_{1i}\ze_{[\de_i,a_1]} & 0 & 0\\
			0 & \ka_{ij}\ze_{[\de_j,a_2]} & \ka_{jk}\ze_{q_j}\ze_{[\de_k,a_2]}
		\end{matrix}\right].
	\end{align*}
	Alternatively,
	still by~\ref{Part:theta},
	we can choose other suitable trivialization so that $\varphi$ takes the form
	\begin{align*}
		\varphi=
		\left[
		\begin{matrix}
			c_{1i}\ze_{[\de_i,a_1]} & 0 & 0\\
			0 & \ka_{ij}\ze_{[\de_j,a_2]} & \ka_{ik}\ze_{[\de_k,a_2]}
		\end{matrix}\right].
	\end{align*}
	In addition, we have $\ze_{[\de_j,a_2]}\eq \ze_{[\de_k,a_2]}$.
	Consequently, there exist $\nu\inn\Ga(\sO_\cV)$ and $\la\inn\Ga(\sO^*_\cV)$ so that
	\begin{align}\label{Eqn:ka_ij_ik_jk}
		\ka_{ik}=
		\la\cdot\ze_{q_j}\ka_{jk}+\nu\cdot\ka_{ij}.
	\end{align}
	By~(\ref{Eqn:ka_ij_ka_ik}) and the fact that $\ze_{q_j}(z_0)\eq 0$,
	we see $\nu\inn\Ga(\sO^*_\cV)$.
	Moreover,
	Case 3.b above implies there exist $\la',\nu'\inn\Ga(\sO^*_\cV)$ so that
	$$
	 \ka_{jk}=\la'\ka_{ik}+\nu'\ze_\fq.
	$$
	Combining the above two equalities,
	we conclude that there exist $v,g'\inn\Ga(\sO^*_\cV)$ so that
	$$
	 \ka_{ik}=v\cdot\ka_{ij}+g'\ze_{q_j}\ze_\fq.
	$$
	
	\item 
	If $\de_i$ is not on any non-separating bridge,
	then the assumption~(\ref{Eqn:ka_assump}) implies $q_j$ is not on any non-separating bridge either,
	which further implies $\ka_{jk}(z_0)\!\ne\!0$.
	Therefore,
	(\ref{Eqn:ka_ij_ik_jk}) takes a simpler form:
	$$
	 \ka_{ik}=\la\cdot\ze_{q_j}+\nu\cdot\ka_{ij}
	$$ 
	for some $\la\inn\Ga(\sO^*_\cV)$ and $\nu\inn\Ga(\sO_\cV)$.
	The symmetry between $\de_j$ and $\de_k$ in Case 3.d also guarantees that $\nu\inn\Ga(\sO^*_\cV)$.
\end{itemize}

{\bf Case 3.e:} {\it $\de_i$, $\de_j$ and $\de_k$ are all on the core $F$ of $C$.}
In this scenario, we have
$$
\de_i=\lr{\de_i},\quad
\de_j=\lr{\de_j},\quad
\de_k=\lr{\de_k}.
$$
Once again
the assumption~(\ref{Eqn:ka_assump}) only hold when the three points are on a common maximal non-separating bridge $\tn B[\fp,\fq]$.
Depending on the number of the smooth rational subcurve(s) of $\tn B[\fp,\fq]$ and the distribution of $\de_i$, $\de_j$ and $\de_k$,
the proof of Case 3.e is analogous to, but simpler than, the proof of either Case 3.b or Case 3.d.
We omit further details.

\section{Sequential modular blowups of $\fM_2^{\rm wt}$ and  $\fdd$}\label{global}
\label{SecGlobal}

The key idea to resolve $\ov M_2(\PP^n,d)$ is to globally blow up $\fdd$  to obtain
$\widetilde \fM_2^{\rm div}$ so that
the derived object ${\bf R} \pi_*\ff^* \sO_{\Pn}(k)$ becomes locally diagonalizable 
in the sense of Definition \ref{dObject} upon pulling back to any
$$\ti \cU_{U,H}:=U\times_{f_{U,H};\,\fdd} \widetilde \fM_2^{\rm div}.$$
The fact that these $\ti \cU_{U,H}$ glue to form  $\widetilde M_2(\PP^n,d)/\ov M_2(\PP^n,d)$ is relatively straightforward,
which will be handled in \S\ref{SecLocalEqns}.
As soon as the global sequential blowups of $\fdd$
are rigorously described, the verification of the local
diagonalizability of the object ${\bf R} \pi_*\ff^* \sO_{\Pn}(k)$, or equivalently,
the diagonalizability of the structural homomorphisms $\varphi$ 
 is local.

As mentioned in the introduction, there are three rounds of sequential blowups,
denoted by $(\rd_1), (\rd_2)$, and $(\rd_3)$.
The first and third rounds each contain several phases, which are written as $(\rd_i\ph_j)$.

The first two rounds $(\rd_1)$ and $(\rd_2)$, as well as the first two phases $(\rd_3\ph_1)$ and $(\rd_3\ph_2)$ of the third round, can be introduced and performed on $\fM_2^{\rm wt}$.
After $(\rd_3\ph_2)$,
we pull them back to $\fdd$, and then describe the remaining phases of sequential blowups.

\vsp
The global blowup process is honestly guided by the  desire to  make the 
local structural homomorphism 
\begin{equation}\label{phi-adhoc}
\varphi=\left[\begin{matrix}
	c_{11}\ze_{[\de_1,a_1]} & c_{12}\ze_{[\de_2,a_1]} & \cdots
	& c_{1i}\ze_{[\de_i,a_1]} & \cdots\\
	c_{21}\ze_{[\de_1,a_2]} & c_{22}\ze_{[\de_2,a_2]} & \cdots
	& c_{2i}\ze_{[\de_i,a_2]} & \cdots
\end{matrix}\right].
\end{equation}
as in Proposition \ref{Prp:phi_key}~\ref{Part:phi} diagonalizable in the sense of Definition \ref{DfnDiag}.

We focus on  one of the two rows of $\varphi$ first, although any blowup will inevitably affect the other row.
But in any case, the purpose of the first round $(\rd_1)$ of blowups is to make  one row,
say the first row,
diagonalized. 
This is parallel to the blowups in \cite{VZ08} and \cite{HL10},
although the blowup centers are somewhat more involved.
Then, the task of the remaining rounds of the sequential blowups is to make the other 
(say, the second) row
and hence the matrix~\eqref{phi-adhoc} diagonalized.

More precisely,
when $(\rd_1)$ terminates,
we denote the final stack by $\ti\fM^{\rd_1}$.
Every point of $\Mw$ has a neighborhood~$\cV$ so that the
pullback $\ti\varphi^{\rd_1}$ of  \eqref{phi-adhoc} to $\ti\cV^{\rd_1}=\cV\times_{\Mw}\ti\fM^{\rd_1}$ 
has one row that contains an element that divides all other entries of $\ti\varphi^{\rd_1}$.
However, $\ti\varphi^{\rd_1}|_{\ti\cV^{\rd_1}}$ may
{\it not} be locally diagonalizable (c.f.~Definition~\ref{DfnDiag}) due to the existence of 
\noindent (1): some distinctive directions in the exceptional divisors obtained in $(\rd_1)$ and/or
\noindent (2): the Weierstrass and conjugate points of genus 2 curves.

The second round, $(\rd_2)$, solves the former case;
the third round, $(\rd_3)$, solves the latter.
In the following subsections, we precisely describe the three rounds of sequential blowups.

\begin{rema}\label{Rmk:m}
As mentioned in \S\ref{SecIntro},
throughout this paper,
we fix an arbitrary positive integer~$m$. For any point $(C, \bw) \in \mwt$,
$m$ corresponds to the total weight of $(C, \bw)$;
for any point $(C, D) \in \fM^{\rm div}_2$,
$m=\deg D$. Further, when applying to ${\bf R}\pi_*\ff^* \sO_{\Pn}(k)$ over $\MPdd$,
$m=kd$.
\end{rema}

	In \S\ref{SecGlobal},
	we describe the global blowup centers of $\ti\fM_2^{\rm div}/\fM_2^{\rm div}$ phase by phase.
	Later in \S\ref{SecChangeOfPhi},
	we will show each step of each phase is a smooth blowup on a case-by-case basis.

\subsection{The first round $(\rd_1)$}\label{rd1}

The blowup centers in $(\rd_1)$ are the closures the following  strata
\begin{equation}\label{strata}
 \fM_{(1,k)},~
 \fM_{(2,k)},~
 \fM_{(3,k)},~
 \fM_{(4,k)},~
 \fM_{(5,k)},\qquad
 k\inn\mathbb Z_{>0},
\end{equation}
of $\mwt$.
Recall we denote by $F$ the core of a curve $C$.
For each $k$, 
\begin{itemize}[leftmargin=*]
\item
$(C,\textbf{w})\inn\fM_{(1,k)}$ 
if $F$ is a smooth genus 2 curve of weight 0 and the tails of $C$ are $k$ smooth rational curves of positive weights.

\item
$(C,\textbf{w})\inn\fM_{(2,k)}$ 
if $F$ consists of two smooth genus 1 curves of weight 0 and one smooth rational curve of positive weight, and the tails of $C$ are $(k\!-\!1)$ smooth rational curves of positive weights;

\item
$(C,\textbf{w})\inn\fM_{(3,k)}$ 
if $F$ consists of one smooth genus 1 curve of weight 0 and one smooth rational curve of positive weight, and the tails of $C$ are $(k\!-\!1)$ smooth rational curves of positive weights;

\item 
$(C,\textbf{w})\inn\fM_{(4,k)}$ 
if $F$ consists of one smooth genus 1 curve of weight 0 and one smooth genus 1 curve of positive weight, and the tails of $C$ are $(k\!-\!1)$ smooth rational curves of positive weights;

\item
$(C,\textbf{w})\inn\fM_{(5,k)}$ 
if $F$ is a smooth genus 2 curve of weight 1 and the tails of $C$ are $k$ smooth rational curves of positive weights.
\end{itemize}

In Figure~\ref{figBlowup1},
we illustrate $\fM_{(i,k)}$ with $k\!\le\!3$.
An unshaded irreducible component indicates the weight is~0.

The closure $\ov\fM_{(i,k)}$ of $\fM_{(i,k)}$ in $\fM_2^{\rm wt}$ can be described as
the images of  natural node-identifying surjective immersions from smooth domains.
Given a finite set $S$,
let $\ov\cM_{g;S}$ be the moduli space of  genus $g$ stable curves whose marked points are indexed by $S$. 
Let 
$${\fM}_{g,S}^{\tn{wt}}$$ be the Artin stack of stable pairs $(C,\textbf{w})$ of genus $g$ nodal curves $C$ whose marked points are indexed by $S$ and weights $\textbf{w}\inn H^2(C;\mathbb Z)$ satisfying  $\textbf{w}(\Si)\!\ge\!0$ for all irreducible $\Si\!\subset\!C$.
Here $(C,\textbf{w})$ is said to be \ts{stable} if any smooth rational (resp.~genus 1)  irreducible component of weight 0 contains at least three (resp.~one) nodal and/or marked points.
Particularly, $\mwt\eq\fM^{\tn{wt}}_{2,\emptyset}$. 
Each connected component $${\fM}_{g,S}^{\tn{wt}}(\ell)$$
of ${\fM}_{g,S}^{\tn{wt}}$ is determined by the total weight $\ell$ of any element.

For conciseness, throughout this article,
we write
\begin{align}
	\label{e_lrbr}
	\lrbr{k}:=[1,k]\cap\mathbb Z\qquad
	\forall~k\in\mathbb Z_{>0}.
\end{align}
With notation as above, we define
\begin{equation*}\begin{split}
 \wh{\fM}_{(1,k)}
 &:=
 \ov{\cM}_{2,k}\times
 \prod_{i=1}^k{\fM}_{0,\{i\}}^{\tn{wt}},
 \\
 \wh{\fM}_{(2,k)}
 &:=
 \bigsqcup_{\ell=1}^k\Big(
 \ov{\cM}_{1,\lrbr{\ell-1}\sqcup\{k+1\}}\times\ov\cM_{1,(\lrbr{k}\bsl\lrbr{\ell})\sqcup\{k+2\}}\times
 {\fM}_{0,\{k+1,k+2\}}^{\tn{wt}}\times
 \!\!\prod_{i\in\lrbr{k}\bsl\{\ell\}}\!\!\!\!{\fM}_{0,\{i\}}^{\tn{wt}}\Big),
 \\
 \wh{\fM}_{(3,k)}
 &:=
 \ov{\cM}_{1,k+1}
 \times
 {\fM}_{0,\{k,k+1\}}^{\tn{wt}}
 \times
 \prod_{i=1}^{k-1}{\fM}_{0,\{i\}}^{\tn{wt}},
 \\
 \wh{\fM}_{(4,k)}
 &:=
 \ov{\cM}_{1,k}
 \times
 {\fM}_{1,\{k\}}^{\tn{wt}}
 \times
 \prod_{i=1}^{k-1}{\fM}_{0,\{i\}}^{\tn{wt}},
 \\
 \wh{\fM}_{(5,k)}
 &:=
 {\fM}_{2,\lrbr{k}}^{\tn{wt}}(1)\times
 \prod_{i=1}^k{\fM}_{0,\{i\}}^{\tn{wt}}.
\end{split}
\end{equation*}
Here, each $\{i\}$ is the set with $i$ as its unique element; it does not mean $i$ unordered marked points.
For each $(i,k)\inn\lrbr{5}\!\times\!\mathbb Z_{>0}$,
there is a natural node-identifying immersion
$$
 \iota_{(i,k)}:
 \wh\fM_{(i,k)}\lra\ov\fM_{(i,k)}
 \subset\mwt
$$
obtained by identifying the marked points that share the same index.
Such $\iota_{(i,k)}$ is surjective onto $\ov\fM_{(i,k)}$ and descends to the quotient 
\begin{equation}\label{e_oviota}
 \ov\iota_{(i,k)}:\,
 \wh{\fM}_{(i,k)}/S_{(i,k)}
 \lra\ov\fM_{(i,k)}\subset\mwt,
\end{equation}
where $S_{(i,k)}$ is the symmetric group of the dual graph of $\fM_{(i,k)}$.
Notice that $\ov\iota_{(i,k)}$ is generally not an isomorphism to its image;
c.f.~\cite[Figure~3]{VZ08} with the genus~1 component replaced by a genus~2 curve, as well as the example at the end of Example~\ref{Eg:G1}.
Thus $\ov\fM_{(i,k)}$ is generally not smooth.

Let
\begin{equation}\label{e_row}
\begin{split}
 \ov\fM_{(i)}:=
 \bigcup_{k\ge 1}\ov\fM_{(i,k)},
 \quad
 &\fM_{(i)}^{\mn}:=
 \ov\fM_{(i)}\big\bsl\big(
 \bigcup_{h=1}^{i-1}\ov\fM_{(h)}
 \big),~~
 1\!\le\!i\!\le\!4;
 \quad
 \fM^{\mn}:=
 \mwt\,\big\bsl\big(\bigcup_{i=1}^4\ov\fM_{(i)}\big).
\end{split}
\end{equation}
In particular, $\fM_{(1)}^\mn\eq\ov\fM_{(1)}$. 
Notice that $\ov\fM_{(i,k)}\!\cap\!\fM_{(j,\ell)}\eq\emptyset$	whenever $(j,\ell)\!>\!(i,k)$ (w.r.t.~the lexicographical order),
so for each $1\!\le\!i\!\le\! 4$,
we have $\fM^\mn_{(i)}\!\supset\!\bigsqcup_{k\ge 1}\fM_{(i,k)}$.
Then we obtain a partition of $\mwt$:
$$
 \mwt=\fM_{(1)}^\mn\sqcup\fM_{(2)}^\mn\sqcup \fM_{(3)}^\mn\sqcup\fM_{(4)}^\mn\sqcup \fM^\mn.
$$
Particularly, $\fM_{(5,k)}\!\subset\!\fM^{\mn}$ for all $k\!\ge\!1$.

In ($\rd_1$), we blow up 
$\mwt$ along (the proper transforms of) $\ov\fM_{(i,k)}$ with respect to the lexicographical order and obtain
$$\widetilde\fM^{\rd_1} \lra \mwt.$$
More concretely, we start with $\ti\fM^{\rd_1\ph_1\st_0}=\mwt$. Inductively, suppose
$\ti\fM^{\rd_1\ph_i\st_{k-1}}$ is constructed for some $i, k > 0$,  we then take
$$\ti\fM^{\rd_1\ph_i\st_k} \lra \ti\fM^{\rd_1\ph_i\st_{k-1}}$$
to be the blowup of $\ti\fM^{\rd_1\ph_i\st_{k-1}}$ along the proper transform of $\ov\fM_{(i,k)}$.
We call $\ti\fM^{\rd_1\ph_i\st_k}$ the resulting blowup stack in \ts{Round 1}
$(\rd_1)$, \ts{Phase} $i$ $(\ph_i)$, and \ts{Step} $k$ $(\st_k)$.
Here, to reconcile notation, 
for $i\!>\!1$, we set
$\ti\fM^{\rd_1\ph_i\st_{0}}$ to be the stack obtained when Phase $i\!-\!1$ is complete.

Recall the connected components of $\Mw$ are determined by the total weights of the curves,
thus the stability of $\Mw$ (c.f.~\S\ref{sheafStructures}) implies the 
process of sequential blowups $(\rd_1)$ terminates after finitely many steps.
 
\begin{rema}[The order of the blowups in ($\rd_1$)]\label{Rmk:r1_order}
The reason we choose the blowup centers of ($\rd_1$) to be~the proper transforms of $\ov\fM_{(i,k)}$ is based on the way the node-smoothing parameters appear in the structural homomorphism (\ref{phi-adhoc}).
As mentioned earlier in this subsection,
$\ov\fM_{(i,k)}$ are generally not smooth in $\mwt$.
The order of the blowups
are thus chosen so that in each step, 
we always have a {\it smooth} blowup,
on which the pullbacks of~(\ref{phi-adhoc}) can be systematically analyzed.

As mentioned at the beginning of \S\ref{SecGlobal}, we will prove
the smoothness of the blowups of ($\rd_1$) in \S\ref{SecChangeOfPhi}.
Here, we provide the general idea;
see Example~\ref{Eg:r1_intersection} below for a concrete example.

Within each phase of ($\rd_1$),
the order of the blowups resembles that of \cite{VZ08} and \cite{HL10}.
The order among the phases of ($\rd_1$) are based on the following observations.
\begin{itemize}[leftmargin=*]
\item 
First, observe 
\begin{align}
\label{e_M2M3}
 \ov\fM_{(2)}\!\cap\!\ov\fM_{(3)}\subset\ov\fM_{(1)},
\end{align}
hence ($\rd_1\ph_1$) should be exerted before ($\rd_1\ph_2$) and ($\rd_1\ph_3$).
When ($\rd_1\ph_1$) terminates, the proper transforms of $\ov\fM_{(2)}$ and $\ov\fM_{(3)}$ are disjoint.
Then, it will make no difference if we perform ($\rd_1\ph_3$) prior to ($\rd_1\ph_2$).

\item 
Next, notice 
\begin{align*}
	\ov\fM_{(2)}\subset\ov\fM_{(4)}.
\end{align*}
In particular, some singularities of $\ov\fM_{(4)}$ are contained in $\ov\fM_{(2)}$;
for instance, $\ov\fM_{(4,1)}$ has a normal crossing singularity at a general point of $\ov\fM_{(2,1)}$.
So ($\rd_1\ph_2$) should  be carried out prior to ($\rd_1\ph_4$).

\item 
Similarly,
($\rd_1\ph_5$) should be performed after $(\rd_1\ph_1)$.
For instance, $\ov\fM_{(5,1)}$ contains a general point of $\ov\fM_{(1,2)}$ whose tails are both of weight 1 as a normal crossing singularity.

\item 
It is worth mentioning that it is feasible to put ($\rd_1\ph_4$) before ($\rd_1\ph_3$),
which is indeed the order chosen in \cite{HN2}, although the resulting resolution of $\ov M_2(\P^n,d)$ will be different from the one constructed in this paper.
It is also possible to perform $(\rd_1\ph_5)$ before any $(\rd_1\ph_i)$, $2\!\le\!i\!\le\!4$,
and obtain a different resolution of $\ov M_2(\P^n,d)$.
\end{itemize}
\end{rema}

\begin{rema}[The reason for ($\rd_1\ph_5$)]\label{rem:r1p5}
As mentioned in the introduction,
the last phase of the first round, $(\rd_1\ph_5)$, is  solely to treat the boundary components of $\MPdd$.
As all the modular blowups are
performed on $\mwt$ and $\fdd$,   more or less, we are forced to include the loci whose general points have smooth core subcurves of weight~1,
even though the images of  the {\it primary} component of $\MPdd$, under either the morphism $\MPdd \!\lra\! \mwt$ as in (\ref{MPddToWeight0}) or the local morphisms $\MPdd\!\supset\!U \!\lra\! \fdd$ as in (\ref{toP0}) miss such loci.
\end{rema}

The following example suggests how various $\ov\fM_{(i,k)}$ intersect,
and how the blowups of ($\rd_1$) affect them locally.

\begin{figure}[htb]
	\begin{center}
		\begin{tikzpicture}
			\def\g1{
				(-1,0) ellipse (1 and 0.5)
				(-1.4,0)..controls(-1,0.1)..(-0.6,0)
				(-0.6,0)..controls(-1,-0.1)..(-1.4,0)
				(-0.5,0.05)--(-0.6,0)
				(-1.4,0)--(-1.5,0.05)
			}
			\def\halfg1{
				(-.32,0)..controls(-.32,-.1) and (-.4,-.2)..(-.6,-.2)
				(-.6,-.2)..controls(-.8,-.2) and (-1,-.25)..(-1,-.35)
				(-.6,-.5)..controls(-.8,-.5) and (-1,-.45)..(-1,-.35)
				(-.6,-.5)..controls(-.25,-.5) and (0,-.35)..(0,0)
				(-.32,0)..controls(-.32,.1) and (-.4,.2)..(-.6,.2)
				(-.6,.2)..controls(-.8,.2) and (-1,.25)..(-1,.35)
				(-.6,.5)..controls(-.8,.5) and (-1,.45)..(-1,.35)
				(-.6,.5)..controls(-.25,.5) and (0,.35)..(0,0)
			}
			\draw\g1;
			\draw[xshift=-1cm]
			(-1.4,0) circle (.4cm)
			;
			\draw[xshift=-2.8cm]
			\halfg1;
			\draw[xshift=-4.8cm,xscale=-1]
			\halfg1;
			
			\draw[fill=black!26] 
			(-2.4,.8) circle (.4cm)
			(-5.2,0) circle (.4cm)
			(-1,.9) circle (.4cm)
			;
			
			\draw
			(1,-.4) node[right] {{$x\eq (C,\textbf w)$}}
			(-2.05,0) node[right] {\tiny{$p$}}
			(-2.85,0) node[right] {\tiny{$q$}}
			(-3.8,.4) node[above] {\tiny{$r$}}
			(-3.8,-.4) node[below] {\tiny{$s$}}
			(-4.75,0) node[left] {\tiny{$c$}}
			(-1,.45) node[above] {\tiny{$a$}}
			(-2.4,.35) node[above] {\tiny{$b$}}
			;			
		\end{tikzpicture}
	\end{center}
	\caption{An example of the intersection of various $\ov\fM_{(i,k)}$}\label{Fig:r1_intersection}
\end{figure}

\begin{exam}\label{Eg:r1_intersection}
Let $x\eq (C,\textbf w)\inn\mwt$ be as in Figure~\ref{Fig:r1_intersection},
where the shaded components are of positive weights, and the not shaded ones are of weight 0.
It is a direct check that
\begin{align*}
	x\in 
	\ov\fM_{(1,3)}\cap 
	\ov\fM_{(2,3)}\cap 
	\ov\fM_{(3,3)}\cap 
	\ov\fM_{(4,2)}\cap 
	\ov\fM_{(4,3)}\,;
\end{align*}
see Figure~\ref{figBlowup1} for illustration of the $\ov\fM_{(i,k)}$ above.
In fact, if any of the shaded component is of weight 1, then $x\inn\ov\fM_{(5,2)}$ as well;
for simplicity, we assume no component is of weight 1 in this example.

Let $\cV\!\lra\!\mwt$ be a sufficiently small smooth chart containing $x$, and $\ze_e$, $e\eq a,b,c,p,q,r,s$, be the corresponding node-smoothing parameters.
Then,
\begin{align*}
	&
	\ov\fM_{(1,3)}\cap \cV=
	\{\ze_a\eq\ze_b\eq\ze_c\eq 0\},\qquad
	\ov\fM_{(2,3)}\cap \cV=
	\{\ze_p\eq\ze_q\eq\ze_a\eq\ze_c\eq 0\},\\
	&
	\ov\fM_{(3,3)}\cap \cV=
	\{\ze_r\eq\ze_s\eq\ze_a\eq\ze_b\eq 0\},\qquad
	\ov\fM_{(4,2)}\cap \cV=
	\{\ze_p\eq\ze_a\eq 0\}\cup 
	\{\ze_q\eq\ze_c\eq 0\},\\
	&
	\ov\fM_{(4,3)}\cap \cV=
	\{\ze_q\eq\ze_a\eq\ze_b\eq 0\}\cup 
	\{\ze_p\eq\ze_b\eq\ze_c\eq 0\};
\end{align*}
any other $\ov\fM_{(i,k)}$ does not meet $x$, hence is disjoint from $\cV$ (for $\cV$ is assumed small).
By direct computation,
we obtain
\begin{align*}
	\big(\ov\fM_{(2,3)}\!\cap\!\ov\fM_{(3,3)}\!\cap\! \cV\big)\subset 
	\big(\ov\fM_{(1,3)}\!\cap\! \cV\big),\qquad
	\big(\ov\fM_{(2,3)}\!\cap\! \cV\big)
	\subset 
	\big(\ov\fM_{(4,2)}\!\cap\!\cV\big).
\end{align*}
Moreover,
\begin{itemize}[leftmargin=*]
	\item after blowing up $\cV$ along $\ov\fM_{(1,3)}\!\cap\!\cV$, the proper transforms of $\ov\fM_{(2,3)}\!\cap\!\cV$ and $\ov\fM_{(2,3)}\!\cap\!\cV$ become disjoint;
	\item after blowing up $\cV$ along $\ov\fM_{(1,3)}\!\cap\!\cV$ then along the proper transform of $\ov\fM_{(2,3)}\!\cap\!\cV$, the proper transforms of the two irreducible components of $\ov\fM_{(4,2)}\!\cap\!\cV$ become disjoint, hence the proper transform of $\ov\fM_{(4,2)}\!\cap\!\cV$ becomes smooth;
	\item after blowing up $\cV$ along $\ov\fM_{(1,3)}\!\cap\!\cV$, the proper transforms of the two irreducible components of $\ov\fM_{(4,3)}\!\cap\!\cV$ become disjoint, hence the proper transform of $\ov\fM_{(4,3)}\!\cap\!\cV$ becomes smooth.
\end{itemize}
\end{exam}

\subsection{The second round $(\rd_2)$}\label{rd2} 

In this subsection, we describe the blowups of $(\rd_2)$,
whose blowup centers lie in the proper transforms
of the exceptional divisors obtained in $(\rd_1\ph_1)$.
Therefore, we start with analyzing these exceptional divisors.
The key observation is in any step ($\rd_1\ph_1\st_{k}$) of ($\rd_1\ph_1$),
the blowup center $\ov\fM_{(1,k)}^{\rd_1\ph_1\st_{k-1}}$,
which is the proper transform of $\ov\fM_{(1,k)}$ in $\ti\fM^{\rd_1\ph_1\st_{k-1}}$,
has its normal bundle equal to a direct sum of $k$ line bundles,
each corresponds to the smoothing of the corresponding node.

To proceed,
we use the following terminology from~\cite[Section 3.1]{VZ08}.
For an arbitrary stack $\ov\fM$,
denote by $TC\ov\fM$ its tangent cone,
which may be related to the intrinsic normal cone of Behrend-Fantechi \cite{BF}.
If $X$ is a smooth stack,
a morphism $\iota_X\!:X\!\lra\!\ov\fM$ is an \ts{immersion} if its differential
$$
\tn d\iota_X:\,
TX\lra\iota_X^*TC\ov\fM
$$
is injective at every point of $X$. Let
$$
\cN_{\iota_X}:=
\iota_X^*TC\ov\fM\big/
\tn{Im}\tn d\iota_X
$$
be the \ts{normal cone} of $\iota_X$ in $\ov\fM$. 

For $k\!\ge\!1$,
the normal cone of the immersion $\iota_{(1,k)}\!:\wh{\fM}_{(1,k)}\!\lra\!\mwt$ is the direct sum of~$k$ line bundles:
\begin{equation}
	\label{e_whLki}
	\cN_{\iota_{(1,k)}}=\bigoplus_{i=1}^{k}
	\Big(\pr_1^*\mathcal{L}_{k;i}\!\otimes\!\pr_{2;i}^*\mathfrak{L}\Big)
	:=
	\bigoplus_{i=1}^{k} \wh L_{k;i}
	\lra
	\wh\fM_{(1,k)},
\end{equation}
where $\mathcal{L}_{k;i}$ and $\mathfrak L$ are respectively the universal tangent line bundle at the marked point of $\ov{\cM}_{2,k}$ labeled by $i$ and the marked point of ${\fM}_{0,\{i\}}^{\tn{wt}}$, 
and $\pr_1\!:\wh\fM_{(1,k)}\!\to\!\ov\cM_{2,k}$ and $\pr_{2;i}\!:\wh\fM_{(1,k)}\!\to\!{\fM}_{0,\{i\}}^{\tn{wt}}$, $1\!\le\!i\!\le\!k$, are the projections onto components.

The immersion
$\ov\iota_{(i,k)}$ in~\eref{e_oviota} induced by $\iota_{(i,k)}$ is generally not an isomorphism to its image.
For $j\!<\!k$, the immersion $\iota_{(1,k)}$ induces an immersion and the corresponding quotient:
\begin{equation}\label{e_iota}
	\begin{split}
		\iota_{(1,k)}^{\rd_1\ph_1\st_j}
		&:\,
		\wh\fM_{(1,k)}^{\rd_1\ph_1\st_j}\lra
		\ov\fM_{(1,k)}^{\rd_1\ph_1\st_j}\subset\ti\fM^{\rd_1\ph_1\st_j}\,,
		\\
		\ov\iota_{(1,k)}^{\rd_1\ph_1\st_j}
		&:\,
		\wh\fM_{(1,k)}^{\rd_1\ph_1\st_j}\big/S_{(1,k)}\lra
		\ov\fM_{(1,k)}^{\rd_1\ph_1\st_j}\subset\ti\fM^{\rd_1\ph_1\st_j}\,.
	\end{split}
\end{equation}
The domain $\wh\fM_{(1,k)}^{\rd_1\ph_1\st_j}$ is constructed by applying~\cite[Lemma 3.3.(1)]{VZ08} repeatedly.
It is analogous to the smooth variety introduced in \cite[Section 2.3]{VZ08}.
To be precise, for an arbitrary nonempty finite set $S$,
let $\cM_{2,S;j}\!\subset\!\ov\cM_{2,S}$ be comprised of marked curves that have exactly $j$ special points (pivotal nodes and/or elements of $S$) on its core,
and $\ov\cM_{2,S;j}$ be its closure in $\ov\cM_{2,S}$.
We blow up $\ov\cM_{2,S}$ successively along the proper transforms of 
	$$
	\ov\cM_{2,S;1},\ldots,\ov\cM_{2,S;|S|-1}.
	$$ 
Let $\ov\cM_{2,S}^{\,\st_j}$ be the blowup of $\ov\cM_{2,S}$ after the $j$-th step.
For $S\eq\lrbr{k}$ and $1\!\le\!j\!\le\!{k\!-\!1}$,
let
\begin{equation}
	\label{e_PTfM}
	\wh{\fM}_{(1,k)}^{\rd_1\ph_1\st_j}
	:=
	\ov{\cM}_{2,k}^{\,\st_j}\times
	\prod_{i=1}^k{\fM}_{0,\{i\}}^{\tn{wt}},
\end{equation}
(which will turn out to be smooth by Lemma~\ref{LmLocalLoci}).
For $j\eq 0$, we set $\iota_{(i,k)}^{\rd_1\ph_1\st_0}\!:=\!\iota_{(i,k)}.$

The immersion
$\ov\iota_{(1,k)}^{\,\rd_1\ph_1\st_{k-1}}$ in~\eref{e_iota} is an embedding.
This follows from~\cite[Lemma~3.3.(2)]{VZ08} repeatedly; c.f.~\cite[\S4.3 (I13)]{VZ08}.
The normal cone of
$
\iota_{(1,k)}^{\rd_1\ph_1\st_{k-1}}
$
is still a direct sum of $k$ line bundles:
\begin{equation}\label{e_nsplit}
	\cN_{\iota_{(1,k)}^{\rd_1\ph_1\st_{k-1}}}
	=\bigoplus_{j=1}^k L_{k;j} 
	\lra
	\wh\fM_{(1,k)}^{\,\rd_1\ph_1\st_{k-1}}\,;
\end{equation}
c.f.~\cite[Lemma~3.5]{VZ08}.
Each line bundle $L_{k;\ell}$ above is the pullback of $\wh L_{k;\ell}$ twisted by some exceptional divisors of previous steps.
The immersion $\iota_{(1,k)}^{\rd_1\ph_1\st_{k-1}}$ induces an immersion
\begin{equation}\label{e_iotanormalcone}
	\iota_{(1,k)}^{\varepsilon}\!:
	\P\cN_{\iota_{(1,k)}^{\rd_1\ph_1\st_{k-1}}}
	\!\lra\! \cE_{(1,k)}:=\P\cN_{\ti\fM^{\,\rd_1\ph_1\st_{k-1}}}{\ov\fM_{(1,k)}^{\,\rd_1\ph_1\st_{k-1}}}\!\subset\!\ti\fM^{\,\rd_1\ph_1\st_{k}},
\end{equation}
which determines an isomorphism
\begin{equation}\label{e_oviotanormalcone}
	\ov\iota_{(1,k)}^{\,\ve}\!:
	\P\cN_{\ov\iota_{(1,k)}^{\,\rd_1\ph_1\st_{k-1}}}
	\!\lra \!\cE_{(1,k)}.
\end{equation}
Notice that $\cE_{(1,k)}$ is exactly the exceptional divisor obtained in ($\rd_1\ph_1\st_{k}$).

We are ready to describe the blowup centers of ($\rd_2$),
which lie in the proper transforms of $\cE_{(1,k)}$.
Let $m$ be as in Remark~\ref{Rmk:m}.
Consider the index set
\begin{equation*}
	\begin{split}
		\fD_1
		&=
		\big\{\,
		\big(k,j,\si \big):\,
		(k,j)\inn\mathbb Z\!\times\!\mathbb Z,~
		1\!\le\!j\!\le\!k,~
		\si\!:\{2,\cdots,j\}\!\lra\!\mathbb Z\!\cap\![2,+\infty)
		\,\big\}.
	\end{split}
\end{equation*}
Here, if $j\eq 1$, then the domain of $\si$ is empty.
The subscript of $\fD_1$ indicates there is one special direction, given by $1\inn\lrbr k$, in $\cE_{(1,k)}$.

Fix $\ord\eq(k,j,\si)\inn\fD_1$.
For every $\ell\inn\mathbb Z_{>0}$,
let
\begin{align}
	\label{Eqn:Mwt0_ell_bubbles}
\fM_{0,S;\ud\ell}^{\tn{wt}}\subset\fM_{0,S}^{\tn{wt}}
\end{align}
be the closed subset whose general point consists of a smooth rational subcurve $C_0$ of weight 0 that contains all the marked points, as well as $\ell$ positively weighted unmarked smooth rational subcurves all attached to $C_0$.
Then, we set
\begin{gather}\label{e_tnMord}
	\wh\nM_\ord = 
	\ov\cM_{2,k}\times
	\fM_{0,\{1\}}^{\tn{wt}}\times
	\prod_{i=2}^j\,
	\fM_{0,\{i\};\ud{\si(i)}}^{\tn{wt}}\times
	\prod_{i=j+1}^k\!
	\fM_{0,\{i\}}^{\tn{wt}}\,.
\end{gather}
We denote by
\begin{equation}\label{e_Mord}
	\wh\fM_\ord
	\subset
	\wh\fM_{(1,k)}
\end{equation}
the image of $\wh\nM_\ord$ in $\wh\fM_{(1,k)}$
obtained by 
\begin{itemize}[leftmargin=*]
	\item
	identifying the $i$-th marked point of $\ov\cM_{2,k}$ with
	the unique marked point of $\fM_{0,\{i\};\ud{\si(i)}}^{\tn{wt}}$ for every $2\!\le\!i\!\le\!j$, and
	\item 
	identifying the $i$-th marked point of $\ov\cM_{2,k}$ with the unique marked point of $\fM_{0,\{i\}}^{\tn{wt}}$ for $i\eq 1$ and every $j\!+\!1\!\le\!i\!\le\!k$.
\end{itemize}

For  $\ord\eq(k,j,\si)$ as above,
with $\wh\fM_{(1,k)}^{\,\rd_1\ph_1\st_{k-1}}$ as in~\eref{e_iota},
we denote by
$$
\wh\fM_\ord^{\,\rd_1\ph_1\st_{k-1}}
\subset \wh\fM_{(1,k)}^{\,\rd_1\ph_1\st_{k-1}}
$$ 
the proper transform of $\wh\fM_\ord$ in $\wh\fM_{(1,k)}^{\,\rd_1\ph_1\st_{k-1}}$.
Set
\begin{equation} \label{e_zk}
	\begin{split}
		&\wh X_\ord:=
		\P\Big(\bigoplus_{i=1}^j L_{k;i}\Big)\Big|_{\wh\fM_\ord^{\rd_1\ph_1\st_{k-1}}}
		\subset
		\P\cN_{\iota_{(1,k)}^{\rd_1\ph_1\st_{k-1}}}\,,
		\qquad 
		X_{\ord}:=
		\iota_{(1,k)}^{\ve}\big(\wh X_\ord\big)
		\subset \cE_{(1,k)}\,.
	\end{split}
\end{equation}
If $j\eq 1$,
then $\wh X_\ord$ is isomorphic to $\wh\fM_\ord^{\,\rd_1\ph_1\st_{k-1}}$.

Intuitively, a general point of
$\wh X_\ord$ satisfies the following criteria:
\begin{itemize}[leftmargin=*]
	\item 
	for $j\!+\!1\!\le\!i\!\le\!k$, the $i$-th component of the fibers of $\P\cN_{\iota_{(1,k)}^{\rd_1\ph_1\st_{k-1}}}$ is identically 0;
	\item for $2\!\le\!i\!\le\!j$,
	the $i$-th tail of the underlying curve $C$ consists of a smooth rational subcurve $C_i$ of weight 0, attaching to the core of $C$, as well as $\si(i)$ positively weighted smooth rational subcurves attaching to $C_i$.
\end{itemize}

Notice that $\sum_{i=2}^j\si(i)\!\ge\!(j\!-\!1)$ for every $(k,j,\si)\inn\fD_1$.
So, for every $1\!\le\!k\!\le\!k'$,
we set
\begin{equation}\label{e_Xkk'}
	\fD_1(k,k')=\big\{\,(k,j,\si)\inn\fD_1:
	\,
	k\!-\!(j\!-\!1)\!+\!\sum_{i=2}^j\si(i)\eq k'\,\big\},\qquad
	X_{k,k'}\eq\!\!
	\bigcup_{\ord\in\fD_1(k,k')}\!\!\!\!\!\!
	X_\ord\,.
\end{equation}
(Particularly, each
$\fD_1(k,k)$ consists of a unique element $(k,1,\si\!:\emptyset\!\to\!\mathbb Z\!\cap\![2,+\infty)$.)
The proper transforms of the above $X_{k,k'}$ will be the blowup centers of $(\rd_2)$.
Some of these loci are illustrated in Figure~\ref{figBlowup2},
where for conciseness, $\cE_{(1,k)}$ is written as $\cE_k$,
and
the proper transform $\wh\fM_\ord^{\,\rd_1\ph_1\st_{k-1}}$ is abbreviated as $\wh\fM_\ord^\bullet$.

\begin{rema}
		The reason to introduce the index sets $\fD_1$ and $\fD_1(k,k')$, particularly the map~$\si$, so as to describe blowup centers of $(\rd_2)$, solely comes from the expressions of structural homomorphisms $\varphi$.
		In Example~\ref{Eg:abcd} of \S\ref{Subsec:blowup_examples}, for instance,
		the second locus in (\ref{Eqn:S_ne_empty}) corresponds to 
		$X_{\rho}$ with $\rho\eq(2, 2, \si(2)\eq 2)$.
\end{rema}

Next, we describe the order of the sequential blowups within $(\rd_2)$.
Let the set
$$
\{\,
(k,k')\in\mathbb Z\!\times\!\mathbb Z:\,
1\!\le\!k\!\le\!k'\,\}
$$
be endowed with the following  order:
\begin{align}\label{Eqn:Ad_1_order}
	(k,k')<(\ell,\ell')\qquad\iff\qquad
	\tn{either}\quad k\!>\!\ell,\quad\tn{or}\quad k\eq\ell~\tn{and}~k'\!<\!\ell'.
\end{align}
In other words,
(\ref{Eqn:Ad_1_order}) is  inherited from the lexicographic order on $\mathbb Z^{\rm op}\!\times\!\mathbb Z$,
where $\mathbb Z^{\rm op}$ is $\mathbb Z$ with the opposite order.

We blow up 
$\widetilde\fM^{\rd_1}$ along the proper transforms of $X_{k,k'}$, $1\!\le\!k\!\le\!k'$, with respect to the  order (\ref{Eqn:Ad_1_order}).
More precisely,
for every $\ord\inn\fD_1(k,k')$,
notice that the total weight of any point of $\wh{\tn M}_\ord$ is at least $k'$,
hence on the connected component of $\ti\fM^{\rd_1}$ whose underlying weighted curves are of weight $m$,
the relevant blowup centers of $(\rd_2)$ are indexed by $1\!\le\!k\!\le\!k'\!\le\!m$.
For every $(k,k')$ with $1\!\le\!k\!\le\!k'\!\le\!m$, 
we denote by $(k,k')\!-\!1\eq(h,h')$ its immediate predecessor
with respect to (\ref{Eqn:Ad_1_order}) that still satisfies $1\!\le\!h\!\le\!h'\!\le\!m$.
The stack obtained after Step $(k,k')\!-\!1$ is written as $\widetilde\fM^{\rd_2\st_{(k,k')-1}}$.
Then, inductively,
we blow up $\widetilde\fM^{\rd_2\st_{(k,k')-1}}$ along the proper transform of $X_{k,k'}$ and obtain
$\widetilde\fM^{\rd_2\st_{(k,k')}}$.
In this way, we obtain
$$\widetilde\fM^{\rd_2} \lra \widetilde\fM^{\rd_1}.$$ 
These are the sequential blowups in $(\rd_2)$.
This procedure is illustrated in Figure~\ref{figBlowup2}.
The proper transforms of $X_{1,k'}$ are not included in Figure~\ref{figBlowup2} because $X_{1,1}$ is isomorphic to the proper transform of $\ov\fM_{(1,1)}$,
which is a Cartier divisor,
while the remaining $X_{1,k'}$ are all empty.

Since on each connected component of $\ti\fM^{\rd_1}$,
the blowups in $(\rd_2)$ is indexed by a finite set $\{(k,k'):1\!\le\!k\!\le\!k'\!\le\!m\}$,
we see $(\rd_2)$ terminates after finitely many steps.

\begin{rema}
There are two reasons for choosing the order (\ref{Eqn:Ad_1_order}).
The first is to make sure the blowup centers of $(\rd_2)$ are all smooth.
Indeed,
if we fix $k$ and consider $k'\!\ge\! k$,
then the blowup centers, i.e.~the proper transforms of $X_{k,k'}$, intersect in a way similar to those of ($\rd_1\ph_1$).
Second, the reverse order on the first index $k$
is  due to the expressions of the structural homomorphisms~$\varphi$.
In Example~\ref{Eg:abcd''} of \S\ref{Subsec:blowup_examples}, one can see why the lexicographic order inherited from $\mathbb Z\!\times\!\mathbb Z$ would fail for diagonalizing~$\varphi$.

Alternatively, it is possible to use the order inherited from $\mathbb Z^{\rm op}\!\times\!\mathbb Z^{\rm op}$ in $(\rd_2)$,
which is equivalent to the order used in~\cite{HN2}.
\end{rema}

The following example shows how $X_{k,k'}$ with various $k'$ intersect, and how ($\rd_2$) affect them locally;
also see Examples~\ref{Eg:abcd} and~\ref{Eg:abcd'} in \S\ref{Subsec:blowup_examples} for an example of $X_{2,2}\!\cap\!X_{2,3}$.
In addition, Example~\ref{Eg:abcd'} includes an example of $X_{2,3}\!\cap\!X_{3,3}$.

\begin{exam}
	\label{Eg:r2_inters}
Let $x=(C,\textbf w)$ be as in Figure~\ref{Fig:r2_inters},
where the weight of each shaded component is no less than 2, while the not shaded components are of weight 0.

\begin{figure}[htp]
	\begin{center}
		\begin{tikzpicture}
			\def\g2left{
				(0,0.8) arc (90:270:1.6 and 0.8)
				(-1.04,0.08)--(-0.88,0)
				..controls (-0.64,-0.12)..(-0.4,0)
				--(-0.24,0.08)
				(-0.88,0)..controls (-0.64,0.12)..(-0.4,0)
			}
			
			\draw \g2left
			[xscale=-1] \g2left;
			
			\draw
			(.8,1.13) circle (0.4)
			(-.8,1.13) circle (0.4)
			(-.8,1.93) circle (0.4);
			
			\draw[fill=black!42]
			(-.8,2.63) circle (0.3)
			(-1.5,1.13) circle (0.3)
			(-1.5,1.93) circle (0.3)
			(.8,1.83) circle (0.3)
			(1.5,1.13) circle (0.3)
			(-1.9,0) circle (0.3);
			\draw 
			(-.8,1.4) node {\tiny{$e$}}
			(-1.05,1.13) node {\tiny{$d$}}
			(-.8,2.2) node {\tiny{$g$}}
			(-1.05,1.93) node {\tiny{$f$}}
			(0,-.5) node {\tiny{$F$}}
			(-1.45,0) node {\tiny{$a$}}
			(-.64,.56) node {\tiny{$b$}}
			(.7,.56) node {\tiny{$c$}}
			(.8,1.38) node {\tiny{$p$}}
			(1.03,1.13) node {\tiny{$q$}};
		\end{tikzpicture}
	\end{center}
	\caption{The underlying curve $C$ in Example~\ref{Eg:r2_inters}}\label{Fig:r2_inters}
\end{figure}

Let $\cV\!\lra\!\mwt$ be a small affine chart containing $x$, and $\ze_\bullet$,
$\bullet\eq a,b,c,d,e,f,g,p,q$, be respectively the node smoothing parameters.
In ($\rd_1\ph_1$),
since $\ov\fM_{(1,k)}\!\cap\!\cV\eq\emptyset$ for $k\eq 1,2$,
the locus
\begin{align*}
	\ov\fM_{(1,3)}\!\cap\!\cV=
	\{\ze_a\eq\ze_b\eq\ze_c\eq 0\}
\end{align*}
is blown up first. After ($\rd_1\ph_1\st_3$),
consider a lift $\ti x$ of $x$ lying in the chart $\ti\cV_{\ti x}$ given by 
$$
 \ze_a=\ve_a,\qquad
 \ze_b=u_b\ve_a,\qquad
 \ze_c=\ve_a\wc\ze_c
$$
where $u_b$ is invertible, while $\wc\ze_c$, the proper transform of $\ze_c$, is not invertible.

It is direct to check the remaining blowups of $(\rd_1\ph_1)$ does not affect $\ti\cV_{\ti x}$,
and 
\begin{alignat*}{2}
	&\ti x\in 
	X_{\ord_1}\!\cap\! X_{\ord_2}\!\cap\! X_{\ord_3}\!\cap\! X_{\ord_4},\qquad
	\tn{where}&&\\
	&
	\ord_1=(3,2,\si(2)\eq 2)\in \fD_1(3,4),\qquad&&
	\ord_2=(3,2,\si(2)\eq 3)\in \fD_1(3,5),
	\\&
	\ord_3=(3,3,\si(2)\eq 2, \si(3)\eq 2)\in \fD_1(3,5),\qquad&&
	\ord_4=(3,3,\si(2)\eq 3, \si(3)\eq 2)\in \fD_1(3,6).
\end{alignat*}
So $\ti x$ lies in the intersection of $X_{3,k'}$, $k'\eq 4,5,6$;
the remaining $X_{k,k'}$ does not contain $\ti x$, hence does not meet $\ti\cV_{\ti x}$.

Locally on $\ti\cV_{\ti x}$,
we have
\begin{alignat*}{2}
	&
	\ti x\in 
	\{\ve_a\eq \wc\ze_c\eq \ze_d\eq\ze_e\eq\ze_f\eq\ze_g\eq\ze_p\eq\ze_q\eq 0\},
	&&\\
	&
	X_{\ord_1}\cap\ti\cV_{\ti x}=
	\{\ve_a\eq \wc\ze_c\eq\ze_d\eq\ze_e\eq 0\},
	\qquad&&
	X_{\ord_2}\cap\ti\cV_{\ti x}=
	\{\ve_a\eq \wc\ze_c\eq\ze_d\eq\ze_f\eq\ze_g\eq 0\},\\
	&
	X_{\ord_3}\cap\ti\cV_{\ti x}=
	\{\ve_a\eq \ze_d\eq\ze_e\eq\ze_p\eq\ze_q\eq 0\},\qquad
	&&X_{\ord_4}\cap\ti\cV_{\ti x}=
	\{\ve_a\eq \ze_d\eq\ze_f\eq\ze_g\eq\ze_p\eq\ze_q\eq 0\}.
\end{alignat*}
The blowups of $(\rd_2)$ first blow up $\ti\cV_{\ti x}$ along $X_{\ord_1}\!\cap\!\ti\cV_{\ti x}$,
then along the proper transform of 
$(X_{\ord_2}\!\cup\!X_{\ord_3})\!\cap\!\ti\cV_{\ti x}$,
and finally along the proper transform of 
$X_{\ord_4}\!\cap\!\ti\cV_{\ti x}$.
By direct computation,
we have
\begin{align*}
	\big(X_{\ord_2}\cap X_{\ord_3}\cap\ti\cV_{\ti x}\big)
	\,\subset\, \big(X_{\ord_1}\cap\ti\cV_{\ti x}\big),
\end{align*}
and after blowing up $\ti\cV_{\ti x}$ along $X_{\ord_1}\!\cap\!\ti\cV_{\ti x}$,
the proper transforms of $X_{\ord_1}\!\cap\!\ti\cV_{\ti x}$ and $X_{\ord_1}\!\cap\!\ti\cV_{\ti x}$ become disjoint from each other and are both smooth.
It is also a direct check that the proper transform of $X_{\ord_4}\!\cap\!\ti\cV_{\ti x}$ is smooth after any of the above blowups.
\end{exam}

After $(\rd_2)$ is complete, the only issue to achieve
the diagonalizability of the structural homomorphism lies on the  second row 
which involves the proper transforms of  node-smoothing  and/or exceptional parameters and 
the local parameters for the proper transforms of  Weierstrass or conjugate loci:
the latter are the proper transforms of
$\ka_{ij}$
as characterized in Propositions \ref{keyProp1} and \ref{Prp:phi_key} 
	(also see Lemmas 
	\ref{Lm:W_smoothness} and \ref{Lm:K_smoothness}).

There are four phases of sequential blowups in $(\rd_3)$,
which we will describe in the next four subsections, separately. 

\subsection{The first phase of the third round $(\rd_3\ph_1)$}\label{rd3ph1} 

In ($\rd_3\ph_1$), we consider the index set
\begin{equation*}
	\begin{split}
		\fD_2
		=
		\big\{\,
		\big(k,j,\si \big):\,
		(k,j)\inn\mathbb Z\!\times\!\mathbb Z,~
		2\!\le\!j\!\le\!k,~
		\si\!:\{3,\cdots,j\}\!\lra\!\mathbb Z\!\cap\![2,+\infty)
		\,\big\}.
	\end{split}
\end{equation*}
Here, if $j\eq 2$, then the domain of $\si$ is empty.
The subscript of $\fD_2$ indicates there are two special directions, given by $1,2\inn\lrbr k$, in $\cE_{(1,k)}$.

For every $\ord\eq (k,j,\si)\inn\fD_2$,
recall
$\ov\cK_{k;1,2}$ denotes the locus in $\ov\cM_{2,k}$ whose first two marked points are conjugate.
By Lemma \ref{Lm:K_smoothness},
it is a Cartier divisor of $\ov\cM_{2,k}$.
Also, for every $\ell\inn\mathbb Z_{>0}$ and every finite set $S$, recall $\ov\fM_{0,S;\ud\ell}^{\tn{wt};+}$ is given in (\ref{Eqn:Mwt0_ell_bubbles}).
We set
$$
\wh{\tn{K}}_\ord
:=
\ov\cK_{k;1,2}\times
\fM_{0,\{1\}}^{\tn{wt}}\times
\fM_{0,\{2\}}^{\tn{wt}}\times
\prod_{i=3}^j\,
\fM_{0,\{i\};\ud{\si(i)}}^{\tn{wt}}\times
\prod_{i=j+1}^k\!
\fM_{0,\{i\}}^{\tn{wt}}\,,
$$
Mimicking the construction of $\wh\fM_\ord$ in (\ref{e_Mord}), we denote by
\begin{equation}\label{e_Kord}
	\wh\fK_\ord
	\subset
	\wh\fM_{(1,k)}
\end{equation}
the image of $\wh{\tn K}_\ord$ in $\wh\fM_{(1,k)}$ obtained
by  $\wh{\tn K}_\ord$ as follows.
\begin{itemize}[leftmargin=*]
	\item 
	identifying the $i$-th marked point of $\ov\cK_{k;1,2}$ with the unique marked point of $\fM_{0,\{i\};\ud{\si(i)}}^{\tn{wt}}$ for every $3\!\le\!i\!\le\!j$,
	\item 
	identifying the $i$-th marked point of $\ov\cK_{k;1,2}$ with the unique marked point of $\fM_{0,\{i\}}^{\tn{wt}}$ for $i\eq 1,2$ and every $j\!+\!1\!\le\!i\!\le\!k$.
\end{itemize}

For $\ord\eq (k,j,\si)$ as above, with  $\wh\fM_{(1,k)}^{\,\rd_1\ph_1\st_{k-1}}$ as in~\eref{e_iota}, we denote by
$$
\wh\fK_\ord^{\,\rd_1\ph_1\st_{k-1}}\,
\subset\, \wh\fM_{(1,k)}^{\,\rd_1\ph_1\st_{k-1}}
$$ 
the proper transform of $\wh\fK_\ord$ in $\wh\fM_{(1,k)}^{\,\rd_1\ph_1\st_{k-1}}$.
Let
\begin{equation*}
	\begin{split}
		&
		\wh K_\ord:=
		\P\Big(\bigoplus_{i=1}^j
		L_{k;i}\Big)\Big|_{\wh\fK_\ord^{\rd_1\ph_1\st_{k-1}}}\,
		\subset \P\cN_{\iota_{(1,k)}^{\rd_1\ph_1\st_{k-1}}},
		\qquad
		K_\ord:= 
		\iota_{(1,k)}^{\ve}\big(\wh K_\ord\big)\,
		\subset\cE_{(1,k)}.
	\end{split}
\end{equation*} 
Intuitively, a general point of
$\wh K_\ord$ satisfies the following criteria:
\begin{itemize}[leftmargin=*]
	\item for $j\!+\!1\!\le\!i\!\le\!k$, the $i$-th component of the fibers of $\P\cN_{\iota_{(1,k)}^{\rd_1\ph_1\st_{k-1}}}$ is identically 0;
	\item for $3\!\le\!i\!\le\!j$,
	the $i$-th tail of the underlying curve $C$ consists of a smooth rational subcurve $C_i$ of weight 0, attaching to the core of $C$, as well as $\si(i)$ positively weighted smooth rational subcurves attaching to $C_i$;
	\item 
	the first two tails of the underlying curve $C$ are attached to the core of $C$ at conjugate points.
\end{itemize}

Notice that $\sum_{i=3}^j\si(i)\!>\!(j\!-\!2)$ for every $(k,j,\si)\inn\fD_2$,
so for every $2\!\le\!k\!\le\!k'$,
we set
\begin{align*}
	\fD_2(k,k')=\big\{\,(k,j,\si)\inn\fD_2:
	\,
	k\!-\!(j\!-\!2)\!+\!\sum_{i=3}^j\si(i)\eq k'\,\big\},\qquad
	K_{k,k'}\eq\!\!
	\bigcup_{\ord\in\fD_2(k,k')}\!\!\!\!\!\!
	K_\ord\,.
\end{align*}
In $(\rd_3\ph_1)$,
we blow up $\ti\fM^{\rd_2}$ successively along the proper transforms of 
$K_{k,k'}$, $2\!\le\!k\!\le\!k'$, with respect to the order inherited from the {\it usual} lexicographic order on $\mathbb Z\!\times\!\mathbb Z$, 
and obtain
$$\ti\fM^{\rd_3\ph_1} \lra \ti\fM^{\rd_2}.$$
For every $\ord\inn\fD_2(k,k')$, the total weight of every point of $\wh{\tn K}_\ord$ is at least $k'$,
hence on the connected component of $\ti\fM^{\rd_2}$ whose underlying weighted curves are of weight $m$, the relevant blowup centers of $(\rd_3\ph_1)$ are indexed by $2\!\le\!k\!\le\!k'\!\le\!m$,
hence $(\rd_3\ph_1)$ terminates after finitely many steps on each connected component of $\ti\fM^{\rd_2}$.


\subsection{The second phase of the third round $(\rd_3\ph_2)$}\label{rd3ph2}


The blowup centers of $(\rd_3\ph_2)$ are within the proper transforms of the exceptional divisors $\cE_{k,k'}$, $1\!\le\!k\!\le\!k'$, obtained in $(\rd_2)$,
so we begin with analyzing these $\cE_{k,k'}$.

First, given $\ord\eq(k,j,\si)\inn\fD_1$,
for every $2\!\le\!i\!\le\!j$,
notice there is a natural immersion
\begin{align*}
	\ov\cM_{0,\si(i)+1}\times\prod_{h=1}^{\si(i)}\fM_{0,\{(i,h)\}}^{\tn{wt}}
	\lra \fM_{0,\{i\};}^{\tn{wt}}
\end{align*}
(where each $\{(i,h)\}$ is the set containing the pair $(i,h)$ as its unique element),
given by identifying the $h$-th marked point of $\ov\cM_{0,\si(i)+1}$ with the unique marked point of $\fM_{0,\{(i,h)\}}^{\tn{wt}}$ for all $1\!\le\!h\!\le\!\si(i)$,
and then labeling the last marked point of $\ov\cM_{0,\si(i)+1}$ by the number $i$.
Such immersions together give rise to an immersion
\begin{align}\label{Eqn:tnM_to_fM}
	\ov{\tn M}_\ord :=
	\ov\cM_{2,k}\times
	\fM_{0,\{1\}}^{\tn{wt}}\times
	\prod_{i=2}^j
	\Big(\ov\cM_{0,\si(i)+1}\!\times\!\prod_{h=1}^{\si(i)}\fM_{0,\{(i,h)\}}^{\tn{wt}}\Big)\!\times\!
	\prod_{i=j+1}^k\!\!
	\fM_{0,\{i\}}^{\tn{wt}}
	\lra \wh{\tn M}_\ord
	\lra \wh\fM_\ord,
\end{align}
where $\wh{\tn M}_\ord$ and $\wh\fM_\ord$ are respectively as in (\ref{e_tnMord}) and (\ref{e_Mord}), and the last arrow is the node-identifying by the sentence containing~(\ref{e_Mord}).
Mimicking (\ref{e_PTfM}),
we replace $\ov\cM_{(2,k)}$ with $\ov\cM_{(2,k)}^{\st_{k-1}}$ in the above construction of $\ov{\tn M}_\ord$ to obtain $\ov{\tn M}_\ord^{\rd_1\ph_1\st_{k-1}}$,
along with the immersion
\begin{align}\label{Eqn:tnM_to_fM_ord}
	\ov{\tn M}_\ord^{\rd_1\ph_1\st_{k-1}}
	\lra 
	\wh\fM_\ord^{\rd_1\ph_1\st_{k-1}}.
\end{align}
With $\wh X_\ord$ and $X_\ord$ as in (\ref{e_zk}), this gives rise to 
\begin{align}\label{Eqn:tnX_ord}
	\ov{\tn X}_\ord\stackrel{\nu_\ord}{\lra}
	\wh X_\ord\quad \Big(\stackrel{\iota_{(1,k)}^{\ve}}{\lra}
	X_\ord\Big),
\end{align}
where $\ov{\tn X}_\ord$ denotes the pullback of $\wh X_\ord$ via (\ref{Eqn:tnM_to_fM_ord}),
and $\nu_\ord$ is the natural morphism.

With notation as above, let
\begin{equation}\label{e_iotak}
	\begin{split}
		\iota_\ord:=
		\iota_{(1,k)}^{\ve}\circ\nu_\ord
		&:\,\ov{\tn X}_\ord
		\lra X_\ord
		\subset \ti\fM^{\,\rd_1\ph_1\st_{k}},
		\\
		\ov\iota_\ord:=
		\ov\iota_{(1,k)}^{\ve}\circ\nu_\ord
		&:\,\ov{\tn X}_\ord/G_\ord
		\lra X_\ord
		\subset \ti\fM^{\,\rd_1\ph_1\st_{k}},
	\end{split}
\end{equation}
where $\iota_{(1,k)}^{\ve}$ and $\ov\iota_{(1,k)}^{\ve}$ are as in~\eref{e_iotanormalcone} and~\eref{e_oviotanormalcone}, respectively, and $G_\ord$ is the symmetry group of the dual graph of a general point of $\ov{\tn M}_\ord$.
With the line bundles $L_{k;i}$ as in~\eref{e_nsplit},
for every $j\!\le\!k$,
let
$$
{\gamma}_{k,j}
\lra
\P\Big(\!\bigoplus_{i=1}^j L_{k;i}\Big)
$$
be the tautological line bundle.
By~\eref{e_nsplit} and~\eref{e_zk},
the normal cone of $
\iota_{\ord}
$ can be written explicitly as a sum of line bundles:
\begin{align}\label{Eqn:N_r2}
	\cN_{\iota_{\ord}}
	\simeq
	\nu_\ord^*\gamma_{k,j}
	\oplus
	\bigoplus_{i=2}^j
	\bigoplus_{h=1}^{\si(i)}
	\wh L_{\ord;(i,h)}
	\oplus
	\bigoplus_{
		i=j+1}^k\!\!
	\nu_\ord^*\big(
	\gamma_{k, j}^{\chk}
	\otimes
	L_{k;i}
	\big)\,,
\end{align}
where each $\wh L_{\ord;(i,h)}$ is the line bundle corresponding to the smoothing of the node labeled by $(i,h)$; these line bundles are analogous to those in~\eref{e_whLki}.

For every $k\!\le\!k'\!\le\!m$,
recall $(k,k')\!-\!1$ denotes the immediate predecessor of $(k,k')$ with respect to the order~(\ref{Eqn:Ad_1_order}).
For every $\ord\inn\fD_1(k,k')$
(see (\ref{e_Xkk'}) for notation),
we observe after $(\rd_2\st_{(k,k')-1})$,
the immersion $\iota_{\ord}$ as in (\ref{e_iotak}) induces the immersion
$$
\iota_{\ord}^{\rd_2\st_{(k,k')-1}}:\,
\ov{\tn X}_{\ord}^{\rd_2\st_{(k,k')-1}}
\lra
X_\ord^{\rd_2\st_{(k,k')-1}}
\subset\cE_{(1,k)}^{\rd_2\st_{(k,k')-1}} 
\subset\ti\fM^{\rd_2\st_{(k,k')-1}},
$$
where $\ov{\tn X}_{\ord}^{\rd_2\st_{(k,k')-1}}$ is obtained by blowing up $\ov{\tn X}_{\ord}$ successively along some smooth substacks analogously to the construction of $\wh\fM^{\rd_1\ph_1\st_j}_{(1,k)}$ in~\eref{e_PTfM}.
The key fact is that the normal cone of $\iota_{\ord}^{\rd_2\st_{(k,k')-1}}$ is still a sum of line bundles:
\begin{equation}\label{e_nstillsplit}
	\cN_{\iota_{\ord}^{\rd_2\st_{(k,k')-1}}}\!
	=
	L_{\ord;1}
	\oplus
	\bigoplus_{i=2}^j
	\bigoplus_{h=1}^{\si(i)}
	L_{\ord;(i,h)}
	\oplus
	\bigoplus_{i=j+1}^k\!\!
	L_{\ord;i}\,,
\end{equation}
where each line bundle is the pullback of the corresponding line bundle in (\ref{Eqn:N_r2}) twisted by some exceptional divisors of previous steps.
Furthermore,
the immersion $\iota_{\ord}^{\rd_2\st_{(k,k')-1}}$ determines
$$
\ov\iota_{\ord}^{\rd_2\st_{(k,k')-1}}:\,
\ov{\tn X}_{\ord}^{\rd_2\st_{(k,k')-1}}/G_\ord
\lra
X_\ord^{\rd_2\st_{(k,k')-1}}
\subset\ti\fM^{\rd_2\st_{(k,k')-1}},
$$
which is isomorphic to its image.
Thus,
the immersion
$$
\iota_{\ord}^\ve:
\P\cN_{\iota_{\ord}^{\rd_2\st_{(k,k')-1}}}
\lra
\cE_{k,k'}
$$
determines an isomorphism 
$$
\ov\iota_{\ord}^\ve:
\P\cN_{\ov\iota_{\ord}^{\rd_2\st_{(k,k')-1}}}
\lra
\cE_{k,k'}.
$$
In sum, we obtain the desired form of $\cE_{k,k'}.$

Next, we describe the blowup center of $(\rd_3\ph_2)$ within $\cE_{k,k'}$'s.
Notice every $\ord\eq(k,j,\si)\inn\fD_1(k,k')$ determines a set
\begin{align*}
	\aleph(\ord):=
	\big\{\,
	(j',\al,\be):~
	&j'\inn\{j,\cdots,k\},~
	\al:\{2,\cdots,j\}\!\lra\!\mathbb Z_{\ge 0}~\tn{s.t.}~\al\!\le\!\si,\\
	&
	\be:
	\{j\!+\!1,\cdots,j'\}\!\sqcup\!
	\{(i,h):2\!\le\!i\!\le\!j,\,1\!\le\!h\!\le\!\al(i)\}\!\lra\![2,+\infty)\!\cap\!\mathbb Z\,
	\big\}\,.
\end{align*} 
For every $\vt\eq(j',\al,\be)\inn\aleph(\ord)$ and $2\!\le\!i\!\le\!j$,
by 
\begin{itemize}[leftmargin=*]
	\item identifying the $h$-th marked point of $\ov\cM_{0,\si(i)+1}$ with the unique marked point of $\fM_{0,\{(i,h)\};\ud{\be(i,h)}}^{\tn{wt}}$ for each $1\!\le\!h\!\le\!\al(i)$,
	\item then identifying the $h$-th marked point of $\ov\cM_{0,\si(i)+1}$ with the unique marked point of $\fM_{0,\{(i,h)\}}^{\tn{wt}}$ for each $\al(i)\!+\!1\!\le\!\ell\!\le\!\si(i)$,
	\item and finally labeling the last marked point of $\ov\cM_{0,\si(i)+1}$ by the number $i$,
\end{itemize} we obtain a substack 
\begin{align*}
	\fM_{0,\{i\};\vt}^{\tn{wt}}\subset
	\fM_{0,\{i\};\ud{\si(i)}}^{\tn{wt}},\qquad
	2\!\le\!i\!\le\!j.
\end{align*}
Intuitively,
a general point of 
$\fM_{0,\{i\};\vt}^{\tn{wt}}$ has a unique smooth rational subcurve $C_0$ of weight 0 that contains the unique marked point,
$\si(i)$ smooth rational subcurves $C_1,\ldots,C_{\si(i)}$ that all attach to $C_0$,
among which the first $\al(i)$ are of weight 0 and the last $\si(i)\!-\!\al(i)$ are positively weighted,
as well as positively weighted smooth rational subcurves $C_{(h,\ell)}$, $1\!\le\!h\!\le\!\al(i)$, $1\!\le\!\ell\!\le\!\be(i,h)$,
such that for each $1\!\le\!h\!\le\!\al(i)$, $C_{(h,1)},\ldots,C_{(h,\be(i,h))}$ are attached to $C_i$.

For $\ord\eq(k,j,\si)$ and $\vt\eq (j',\al,\be)$ as above,
recall
$\ov\cW_{k;1}$ denotes the locus in $\ov\cM_{2,k}$ whose first marked point is Weierstrass, which by Lemma \ref{Lm:W_smoothness}
is a Cartier divisor of $\ov\cM_{2,k}$.
We then set
$$
\wh{\tn W}_{\ord,\vt}=
\ov\cW_{k;1}\times
\fM_{0,\{1\}}^{\tn{wt}}\times
\prod_{i=2}^j
\fM_{0,\{i\};\vt}^{\tn{wt}}\times
\prod_{i=j+1}^{j'}\!\!
\fM_{0,\{i\};\ud{\be(i)}}^{\tn{wt}}\times
\prod_{i=j'+1}^k\!\!
\fM_{0,\{i\}}^{\tn{wt}}\quad
\subset \wh{\tn M}_\ord,
$$
where $\wh{\tn M}_\ord$ is as in~\eref{e_tnMord}.
Via the immersion (\ref{Eqn:tnM_to_fM}),
we denote the preimage of $\wh{\tn W}_{\ord,\vt}$ in $\ov{\tn M}_\ord$ by $$\ov{\tn W}_{\ord,\vt}\subset\ov{\tn M}_\ord\,.$$
The proper transform of $\ov{\tn W}_{\ord,\vt}$ after $(\rd_1\ph_1\st_{k-1})$ is written as
$\ov{\tn W}_{\ord,\vt}^{\rd_1\ph_1\st_{k-1}}$.

For the above $\ord$ and $\vt$,
with $\ov{\tn X}_\ord$ as in (\ref{Eqn:tnX_ord}),
let
$$
\ov{\tn X}_{\tn w;\ord,\vt}:=\ov{\tn X}_\ord|_{\ov{\tn W}_{\ord,\vt}^{\rd_1\ph_1\st_{k-1}}};
$$
its proper transform in $\ov{\tn X}_{\ord}^{\rd_2\st_{(k,k')-1}}$ is denoted by 
$\ov{\tn X}_{\tn w;\ord,\vt}^{\rd_2\st_{(k,k')-1}}$.
With the line bundles $L_{\ord;i}$ and $L_{\ord;(i,h)}$ as in (\ref{e_nstillsplit}),
we set
\begin{equation}
	\begin{split}
		\ov{\tn W}_{\ord,\vt}
		&:=
		\P\big(
		L_{\ord;1}
		\oplus\,
		\bigoplus_{i=2}^j
		\bigoplus_{h=1}^{\al(i)}
		L_{\ord;(i,h)}
		\oplus\!
		\bigoplus_{i=j+1}^{j'}\!\!
		L_{\ord;i}\,
		\big)
		\Big|_{\ov{\tn X}_{\tn w;\ord,\vt}^{\rd_2\st_{(k,k')-1}}}
		\subset\P\cN_{\iota_{\ord}^{\rd_2\st_{(k,k')-1}}},\\
		W_{\ord,\vt}
		&:=
		\iota_\ord^\ve\big(
		\ov{\tn W}_{\ord,\vt}
		\big)=
		\ov\iota_\ord^\ve\big(
		\ov{\tn W}_{\ord,\vt}/G_\ord
		\big)
		\subset\cE_{k,k'}.
	\end{split}
\end{equation}

Notice that 
for every $\ord\eq(k,j,\si)\inn\fD_1(k,k')$ and $\vt\eq(j',\al,\be)\inn\aleph(\ord)$,
we have
\begin{align*}
	\Big(\sum_{i=2}^j\sum_{h=1}^{\al(i)}\be(i,h)\Big)
	+\Big(\sum_{i=j+1}^{j'}\!\!\be(i)\Big)
	\ge 
	\Big(\sum_{i=2}^j {\al(i)}\Big)+
	(j'-j ).
\end{align*}
Therefore,
for every $k''\!\ge\!k'$ $(\ge\!k\!\ge\!1)$,
we set
\begin{align*}
	&\fD_{\tn w}(k,k',k'')=
	\Big\{\,
	(\ord,\vt)\!:~
	\ord\inn\fD_1(k,k'),\,
	\vt\inn\aleph(\ord),\\
	&\hspace{1.66in}
	k-(j'\!-\!1)+
	\sum_{i=2}^j\Big(\si(i)\!-\!\al(i)\!+\!\sum_{h=1}^{\al(i)}\be(i,h)\Big)
	+
	\sum_{i=j+1}^{j'}\!\!\be(i)
	= k''\,\Big\},
	\\
	&
	W_{k,k',k''}=\!\!
	\bigcup_{(\ord,\vt)\in\fD_{\tn w}(k,k',k'')}\!\!\!\!\!\!\!\!\!\!\!
	W_{\ord,\vt}.
\end{align*}
The proper transforms of these $W_{k,k',k''}$ will be the blowup centers of $(\rd_3\ph_2)$.

\begin{rema}
	The reason to introduce the index set  $\fD_{\tn w}(k,k',k'')$ once again comes from the expressions of structural homomorphisms $\varphi$.
	In Example~\ref{Eg:abcd'} of \S\ref{Subsec:blowup_examples}, for instance,
	the second locus in (\ref{Eqn:S^dot_nonempty}) corresponds to 
	$W_{\ord,\vt}$ with $\big(\ord\eq(2, 1, \emptyset),\vt\eq(2,\emptyset,\be(2)\eq 2)\big)\inn
	\fD_{\tn w}(2,2,3)$
	(where $\emptyset$ refers to the empty map).
\end{rema}

In $(\rd_3\ph_2)$,
we blow up $\ti\fM^{\rd_3\ph_1}$ successively along the proper transforms of
$
W_{k,k',k''},
$
$1\!\le\!k\!\le\!k'\!\le\!k''$,
with respect to the order inherited from the {\it usual} lexicographic order on $\mathbb Z\!\times\!\mathbb Z\!\times\!\mathbb Z$, and obtain
$$\ti\fM^{\rd_3\ph_2} \lra \ti\fM^{\rd_3\ph_1}.$$
On the connected component of $\ti\fM^{\rd_3\ph_1}$ whose underlying weighted curves are of weight $m$, the relevant blowup centers of $(\rd_3\ph_2)$ are indexed by the finite set $1\!\le\!k\!\le\!k'\!\le\!k''\!\le\!m$,
hence the process of 
the sequential blowups $(\rd_3\ph_2)$ terminates after finitely many steps.

\subsection{The third phase of the third round $(\rd_3\ph_3)$}\label{rd3ph3}

The blowup loci of $(\rd_3\ph_3)$ lie in the stack~$\fdd$ over $\mwt$ 
instead of the stack $\mwt$ itself;
see \S\ref{sheafStructures} for definition.
Let
\begin{align*}
	\mathfrak H\subset \fdd
\end{align*}
be the closed substack consisting of all $(C,D)\inn\fdd$ such that $\deg D\eq 2$, and the two points of $D$ are conjugate.
Particularly, $\mathfrak H$ is a Cartier divisor of $\fdd$ by
Lemma \ref{Lm:K_smoothness}.

For  any $k\!\ge\!0$,
let~$\cH_k$ be the closed substack of $\fdd$ whose {\it general} points
are pairs $(C,D)$ such that with $F$ denoting the core of $C$, $(F,D\!\cap\!F)\inn\mathfrak H$ and $C$ has $k$ tails attached to $F$. 
After $(\rd_3\ph_2)$, let
\begin{align}\label{Eqn:r3p3_blowup_center}
	\ti\fM^{{\rm div},\rd_3\ph_2}:=\fdd\times_{\mwt}\ti\fM^{\rd_3\ph_2},\qquad
	\cH_k^{\rd_3\ph_2}:=\PT_{\ti\fM^{{\rm div},\rd_3\ph_2}}(\cH_k),\quad k\!\ge \!0,
\end{align} 
i.e.~$\cH_k^{\rd_3\ph_2}$ is the proper transform of $\cH_k$ in $\ti\fM^{{\rm div},\rd_3\ph_2}$ for every $k\!\ge\!0$.
In $(\rd_3\ph_3)$,
we blow up the stack $\ti\fM^{{\rm div},\rd_3\ph_2}$ successively along the proper transforms of 
$
\cH_k^{\rd_3\ph_2}
$
with $k\!\ge\!0$, starting from $k\eq 0$, and obtain
$$ 
\ti\fM^{{\rm div},\rd_3\ph_3} \lra \ti\fM^{{\rm div},\rd_3\ph_2}.$$
The stability requirement of $\fdd$ implies that on each connected component of $\ti\fM^{{\rm div},\rd_3\ph_2}$,
the process of 
the sequential blowups $(\rd_3\ph_3)$ terminates after finitely many steps.

\subsection{The fourth phase of the third round $(\rd_3\ph_4)$}\label{rd3ph4} 

The process of $(\rd_3\ph_4)$ is analogous to the round $(\rd_2)$,
but the blowup loci lie in the proper transforms of~$\fdd$ instead of~$\mwt$.

Notice that the main difference between $\ov\fM_{(1)}$ and $\ov\fM_{(5)}^{\rd_1\ph_4}$ in~\eref{e_row} is the weight of the core curve.
Thus, the exceptional divisors obtained in the phases $(\rd_1\ph_1)$ and~$(\rd_1\ph_5)$ share similar properties.
In particular,
we define
$$
X'_\ord\subset\cE_{(5,k)},
\qquad
\ord=(k,j,\si)\in\fD_1,
$$
to be parallel to $X_\ord$ in~\eref{e_zk}.
For each $\ord\eq(k,j,\si)\inn\fD_1$ and every $x\inn X'_\ord,$
we denote by $F_x$ the core of the underlying curve $C_x$, by $q_x$ the node labeled by~$1$, and by $p_x\eq\lr{q_x}$ the pivotal node corresponding to $q_x$.
Let 
\begin{equation*}
	\begin{split}
		\cQ_\ord\subset
		\ti \fM^{{\rm div},\rd_1\ph_5\st_{k}}
		:=
		\fdd\!\times_{\fM_2^{\wt}}\!
		\ti\fM^{\rd_1\ph_5\st_{k}} ,\qquad
		\ord=(k,j,\si)\in\fD_1,
	\end{split}
\end{equation*}
be the closed substack of $\ti \fM^{{\rm div},\rd_1\ph_5\st_{k}}$ whose {\it general} points are pairs $(x,D)$ satisfying
$$
x\in X'_\ord,\qquad
\big(\,F_x,\,(D\!\cap\!{F_x})+p_x\,\big)\in\mathfrak H,$$
where $\mathfrak H$ is as in \S\ref{rd3ph3}.

In $(\rd_3\ph_4)$,
we blow up $\ti\fM^{{\rm div},\rd_3\ph_3}$ successively along the proper transforms of \begin{equation}\label{rd3ph4-pi++}
	\cQ_{k,k'}=\!\!\bigcup_{\ord\in\fD_1(k,k')}
	\!\!\!\!\!\!\cQ_\ord\ \ \big(\subset\ti \fM^{{\rm div},\rd_1\ph_5\st_{k}}\big),
	\qquad
	1\!\le\!k\!\le\!k',
\end{equation}
with respect to the order inherited from the {\it usual} lexicographic order on $\mathbb Z\!\times\!\mathbb Z$,
and obtain
$$ \ti\fM_2^{\rm div}:=\ti\fM^{{\rm div},\rd_3\ph_4} \lra \ti\fM^{{\rm div},\rd_3\ph_3}.$$
The stability of $\fdd$ once again implies that on each connected component of $\ti\fM^{{\rm div},\rd_3\ph_3}$,
$(\rd_3\ph_4)$ terminates after finitely many steps.

\subsection{Examples of modular blowups}\label{Subsec:blowup_examples}

We conclude \S\ref{SecGlobal} with a set of examples illustrating the sequential modular blowups.

\begin{exam}\label{Eg:ab}    
	We begin with an example that appears simple enough yet still possibly involves all three rounds of the blowups.
	
	Let $z_0=(C,D)\inn\fM^{\rm div}_2$ be such that $C$ consists of a smooth core  $F$ and two smooth rational subcurves $R_a$ and $R_b$ attached to $F$ at
	$q_a$ and $q_b$, respectively,
	and $D$ satisfies
	$$
	 D\cap R_a=\{\de_1,\ldots,\de_{\ell-1}\},\quad
	 D\cap R_b=\{\de_\ell,\ldots,\de_m\},\quad
	 D\cap F=\emptyset,
	$$
	where $1\!<\!\ell\!\le\!m$.
	Particularly, $D\cap R_a$ and $D\cap R_b$ are both nonempty.	
	An illustration of $C$ is provided in Figure~\ref{Fig:ab},
	where we follow the convention of Figure~\ref{figBlowup2} to shade the irreducible components of $C$ that are positively weighted (i.e.~the intersections with $D$ are nonempty).
	The same convention applies to the remaining figures in \S\ref{Subsec:blowup_examples}.
	
	\begin{figure}[htp]
		\begin{center}
			\begin{tikzpicture}
			\def\g2left{
				(0,0.8) arc (90:270:1.6 and 0.8)
				(-1.04,0.08)--(-0.88,0)
				..controls (-0.64,-0.12)..(-0.4,0)
				--(-0.24,0.08)
				(-0.88,0)..controls (-0.64,0.12)..(-0.4,0)
			}
			
			\draw \g2left
			[xscale=-1] \g2left;
			\draw[fill=black!42]
			(-.8,1.13) circle (0.4)
			(.8,1.13) circle (0.4);
			\draw 
			(-.8,1.13) node {$R_a$}
			(.8,1.13) node {$R_b$}
			(0,-.5) node {$F$}
			(-.64,.5) node {$q_a$}
			(.7,.5) node {$q_b$};
			\end{tikzpicture}
		\end{center}
		\caption{The curve $C$ in Example~\ref{Eg:ab}}\label{Fig:ab}
	\end{figure}
	
	By Proposition~\ref{Prp:phi_key}~\ref{Part:phi},
	locally on a small affine chart $\cV\!\lra\!\fM^{\rm div}_2$ containing the point $z_0$,
	we have
	$$\varphi =\left[\begin{matrix}  
		c_{11} \zeta_a &\cdots & c_{1,\ell-1} \zeta_a &  c_{1\ell} \zeta_b &\cdots & c_{1m} \zeta_b
		\\
		c_{21} \zeta_a &\cdots & c_{2,\ell-1} \zeta_a &  c_{2\ell} \zeta_b &\cdots & c_{2m} \zeta_b
	\end{matrix} \right],
	$$
	where $\ze_a$ and $\ze_b$ are the node smoothing parameters (i.e.~modular parameters) corresponding to $q_a$ and $q_b$, respectively, and $c_{ij}$ are invertible functions on $\cV$.
	
	In $(\rd_1\ph_1)$,
	$\cV$ is blown up along the locus $(\ze_a\eq \ze_b\eq 0)$.
	After $(\rd_1\ph_1)$,
	up to switching the symbols $a$ and $b$,
	the pullback $\ti\varphi^{\rd_1\ph_1}$ of $\varphi$ locally takes the form
	$$
	\ve
	\left[\begin{matrix}  
		c_{11}  &\cdots & c_{1,\ell-1} &  c_{1\ell} \wc\zeta_b &\cdots & c_{1m} \wc\zeta_b
		\\
		c_{21} &\cdots & c_{2,\ell-1} &  c_{2\ell} \wc\zeta_b &\cdots & c_{2m} \wc\zeta_b
	\end{matrix} \right],
	$$
	where $\ve$ is the local parameter corresponding to the exceptional divisor, and $\wc\ze_b$ is the proper transform of $\ze_b$.
	In other words, up to a multiplicative unit, $\ze_a$ (resp.~$\ze_b$) pulls back to $\ve$ (resp.~$\ve\wc\ze_b$).
	
	Recall that $\ka_{ij}$, $1\!\le\!i\!<\!j\!\le\!m$, are the local parameters on $\cV$ whose zero loci consists of $z\inn\cV$ such that $\cD_i(z)$ and $\cD_j(z)$ are conjugate points on $\cC_z$.
	By Proposition~\ref{Prp:phi_key}~\ref{Part:theta},
	one can apply elementary row and column operations to the above matrix of $\ti\varphi^{\rd_1\ph_1}$ and obtain
	\begin{align*}
	\ve
	\left[\begin{matrix}  
		1 & 0 & 0 &\cdots & 0  & 0 &  0 &\cdots & 0
		\\
		0 & \wc\ka_{12}\!\cdot\!\ve & \wc\ka_{13}\!\cdot\!\ve & \cdots & \wc\ka_{1,\ell-1}\!\cdot\!\ve &  \wc\ka_{1\ell}\!\cdot\! \wc\zeta_b & \wc\ka_{\ell,\ell+1}\!\cdot\!\ve\!\cdot\!\wc\zeta_b^2 &\cdots & \wc\ka_{\ell m}\!\cdot\!\ve\!\cdot\!\wc\zeta_b^2
	\end{matrix} \right],
	\end{align*}
	where $\wc\ka_{ij}$ refers to the proper transform of $\ka_{ij}$,
	which is equivalent to its pullback in this case.
	Notice that depending on the value of $\ell$ and $m$,
	some columns above may not exist.
	
	Below, we analyze $\ti\varphi^{\rd_1\ph_1}$ depending on the location of $q_a$ and $q_b$ on the core $F$ of $C$.
	
	\vsp
	
	{\it Case A: $q_a$ is in general position.} This means $q_a$ is neither Weierstrass nor conjugate to $q_b$. 
	Hence $\ka_{1i}$ with $2\!\le\!i\!\le\!\ell$ are all invertible and so are their pullbacks.
	Therefore, the matrix of $\ti\varphi^{\rd_1\ph_1}$, after further elementary column operations, can be rewritten as 
	$$
	\ve
	\left[\begin{matrix}  
		1 & 0 & 0 &\cdots & 0  & 0 &  0 &\cdots & 0
		\\
		0 & \ve & 0 & \cdots & 0 &  \wc\zeta_b & 0 &\cdots & 0
	\end{matrix} \right].
	$$
	
	If $\ell\eq 2$ (i.e.~$|D\!\cap\!R_a|\eq 1$), then the second column $[0,\ve]^T$ does not exist and the matrix is already diagonalized.
	
	If $\ell\!\ge\!3$ (i.e.~$|D\!\cap\!R_a|\!\ge\! 2$),
	we observe that the remaining phases of $(\rd_1)$ does not affect this chart,
	and the blowups in $(\rd_2)$ that may affect this chart has the center $X_{2,2}$ (c.f.~Figure~\ref{figBlowup2}), which locally is given by	$(\ve\eq \wc\ze_b\eq 0)$.
	Therefore, after $(\rd_2)$, the pullback $\ti\varphi^{\rd_2}$ of $\varphi$ becomes diagonalized.
	
	\vsp 
	
	{\it Case B: $q_a$ and $q_b$ are a pair of conjugate points on the core $F$.}
	Since $F$ is smooth, neither $q_a$ nor $q_b$ is a Weierstrass point of $F$.
	Lemma~\ref{Lm:KW} then implies $\ka_{1i}$ with $2\!\le\!i\!\le\!\ell\!-\!1$ and $\ka_{\ell k}$ with $\ell\!+\!1\!\le\!k\!\le\!m$ are all invertible and so are their proper transforms.
	Therefore, the matrix of $\ti\varphi^{\rd_1\ph_1}$, after further elementary column operations, can be rewritten as 
	$$
	\ve
	\left[\begin{matrix}  
		1 & 0 & 0 &\cdots & 0  & 0 & 0 & 0 &\cdots & 0
		\\
		0 & \ve & 0  & \cdots & 0 &  \wc\ka_{1\ell}\wc\zeta_b & \ve\wc\ze_b^2 & 0 &\cdots & 0
	\end{matrix} \right].
	$$
	
	If $\ell\eq m\eq 2$ (i.e.~$|D\!\cap\!R_a|\eq|D\!\cap\!R_b|\eq 1$),
	then only the first and the $\ell$-th columns exist, and the matrix takes the form
	$$
	\left[\begin{matrix}  
		1 & 0
		\\
		0 &  \wc\ka_{1\ell} \wc\ze_b
	\end{matrix} \right],
	$$ 
	which is already diagonalized.
	
	If $2\eq \ell\!<\! m$ (i.e.~$1\eq|D\!\cap\!R_a|\!<\!|D\!\cap\!R_b|$),
	then the second column $[0,\ve]^T$ does not exist and the matrix takes the form
	$$
	\left[\begin{matrix}
		\ve & 0\\
		0 & \ve\wc\ze_b
	\end{matrix}\right]
	\left[\begin{matrix}  
		1 & 0 &  0& 0 &\cdots & 0
		\\
		0 &  \wc\ka_{1\ell} & \ve\wc\ze_b & 0 &\cdots & 0
	\end{matrix} \right].
	$$
	The blowups in $(\rd_1\ph_2)$-$(\rd_1\ph_4)$ do not affect the chart.
	In $(\rd_1\ph_5)$, the locus $(\wc\ze_b\eq 0)$ is blown up, which does not yield anything new because the blowup center is a smooth Cartier divisor.
	The blowups in $(\rd_2)$, $(\rd_3\ph_2)$ and $(\rd_3\ph_3)$ also do not affect the chart either,
	whereas in $(\rd_3\ph_1)$ and $(\rd_3\ph_4)$,
	the blowup centers are locally given by
	$$(\wc\ka_{1\ell}=\ve= 0)\qquad\tn{and}\qquad
	(\wc\ka_{1\ell}=\wc\ze_b= 0),
	$$
	respectively.
	It is easy to check that after $(\rd_3\ph_4)$,
	the pullback $\ti\varphi^{\rd_3\ph_4}$ of $\varphi$ becomes diagonalized.
		
	If $\ell\!\ge\!3$ (i.e.~$|D\!\cap\!R_a|\!\ge\!2$),
	then the second column $[0,\ve]^T$ exists and the matrix of $\varphi^{\rd_1\ph_1}$, after another column addition, takes the form
	$$
	\ve
	\left[\begin{matrix}  
		1 & 0 & 0 &\cdots & 0  & 0 & 0 &\cdots & 0
		\\
		0 & \ve & 0  & \cdots & 0 &  \wc\ka_{1\ell}\wc\zeta_b & 0 &\cdots & 0
	\end{matrix} \right].
	$$
	The blowups in $(\rd_1\ph_2)$-$(\rd_1\ph_5)$ and $(\rd_3\ph_2)$-$(\rd_3\ph_4)$ do not affect (the pullback of) the chart.
	In $(\rd_2)$ and $(\rd_3\ph_1)$,
	the blowup centers are locally given by
	$$(\ve= \wc\ze_b=0)\qquad\tn{and}\qquad
	\tn{PT}(\ve=\wc\ka_{1\ell}= 0),
	$$
	respectively,
	where $\tn{PT}$ refers to the proper transform.
	So after $(\rd_3\ph_1)$,
	the pullback $\ti\varphi^{\rd_3\ph_1}$ of $\varphi$ becomes diagonalized.
	
	\vsp 
	
	{\it Case C: $q_a$ is a Weierstrass point on the core $F$.}
	Since $F$ is smooth, $q_a$ and $q_b$ cannot be conjugate.
	Lemma~\ref{Lm:KW} then implies $\ka_{1\ell}$ and $\ka_{\ell k}$ with $\ell\!+\!1\!\le\!k\!\le\!m$ are all invertible and so are their proper transforms.
	Taking Proposition~\ref{Prp:phi_key} \ref{Part:kappa} into consideration, we see that the matrix of $\ti\varphi^{\rd_1\ph_1}$, after further elementary column operations, can be rewritten as 
	$$
	\ve
	\left[\begin{matrix}  
		1 & 0 & 0 & 0 &\cdots & 0  & 0 & 0 &\cdots & 0
		\\
		0 & \wc\ka_{12}\ve & \ve^2 & 0  & \cdots & 0 &  \wc\zeta_b & 0 &\cdots & 0
	\end{matrix} \right].
	$$
	
	If $\ell\eq 2$ (i.e.~$|D\!\cap\!R_a|\eq 1$), then the second and third columns $[0,\wc\ka_{12}\ve]^T$ and $[0,\ve^2]^T$ do not exist, so the matrix is already diagonalized.
	
	If $\ell\!\ge\! 3$ (i.e.~$|D\!\cap\!R_a|\!\ge\! 2$),
	notice that the blowups in $(\rd_1\ph_2)$-$(\rd_1\ph_5)$ and $(\rd_3\ph_3)$-$(\rd_3\ph_4)$ do not affect (the pullback of) the chart. 
	The blowups in $(\rd_2)$ may affect the chart only when $\wc\ze_b$ is not invertible (otherwise the matrix is already diagonalized),
	in which case the relevant blowup center is
	$(\ve\eq \wc\ze_b\eq 0)$.
	After $(\rd_2)$,
	on the new chart given by $\wc\ze_b\eq\ve'$ and $\ve\eq\ve'\wc\ve$ (where $\ve'\eq 0$ defines the exceptional divisor obtained in $(\rd_2)$),
	the matrix becomes diagonalized.
	On the new chart given by $\ve\eq\ve'$ and $\wc\ze_b\eq\ve'\wc\ze_b'$,
	the matrix can be written as
	$$
	\left[\begin{matrix}
		\ve' & 0\\
		0 & (\ve')^2
	\end{matrix}\right]
	\left[\begin{matrix}  
		1 & 0 & 0 & 0 &\cdots & 0  & 0 & 0 &\cdots & 0
		\\
		0 & \wc\ka_{12} & \ve' & 0  & \cdots & 0 &  \wc\zeta_b' & 0 &\cdots & 0
	\end{matrix} \right].
	$$
	We remark that the third column $[0,\ve']^T$ of the right matrix above does not exist if $\ell\eq 3$ (i.e.~$|D\!\cap\!R_a|\!=\! 2$).
		
	The blowups in $(\rd_3\ph_1)$ do not affect this new chart locally.
	In $(\rd_3\ph_2)$ (and $(\rd_3\ph_3)$ if $\ell\eq 3$),
	the blowup center is locally given by
	$$(\wc\ka_{12}=\ve'= \wc\ze_b'=0)\qquad\big(\tn{and}\quad
	\tn{PT}(\wc\ka_{12}=\wc\ze_b'= 0)\quad\tn{if}~\ell\eq 3\big),
	$$
	where $\tn{PT}$ refers to the proper transform.
	So after $(\rd_3\ph_2)$ \big($(\rd_3\ph_3)$ if $\ell\eq 3$\big),
	the pullback of $\varphi$ becomes diagonalized.
\end{exam}

\begin{exam}\label{Eg:abcd}  
	Let $z_0=(C,D)$, where $C$ is made of a smooth core curve $F$
	with weight zero and two smooth rational tails $T_1$ and $T_2$ with positive weights attached to $F$ at
	$q_a$ and $q_b$, respectively.  Further, $T_1$ is made of  a single smooth rational curve $R_a$;
	$T_2$ is made of three rational curves $R_b,$ $R_c$ and $R_d$ such that $R_b$ is attached to $F$ and
	$R_c$ and $R_d$ are attached to $R_b$ at two distinct points $q_c$ and $q_d$, respectively;
	see Figure~\ref{Fig:abcd}.
	
	\begin{figure}[htp]
		\begin{center}
			\begin{tikzpicture}
				\def\g2left{
					(0,0.8) arc (90:270:1.6 and 0.8)
					(-1.04,0.08)--(-0.88,0)
					..controls (-0.64,-0.12)..(-0.4,0)
					--(-0.24,0.08)
					(-0.88,0)..controls (-0.64,0.12)..(-0.4,0)
				}
				
				\draw \g2left
				[xscale=-1] \g2left;
				\draw
				(.8,1.13) circle (0.4);
				\draw[fill=black!42]
				(-.8,1.13) circle (0.4)
				(.8,1.93) circle (0.4)
				(1.6,1.13) circle (0.4);
				\draw 
				(-.8,1.13) node {\tiny{$R_a$}}
				(.66,1) node {\tiny{$R_b$}}
				(.8,1.93) node {\tiny{$R_c$}}
				(1.6,1.13) node {\tiny{$R_d$}}
				(0,-.5) node {\tiny{$F$}}
				(-.64,.5) node {\tiny{$q_a$}}
				(.7,.5) node {\tiny{$q_b$}}
				(.8,1.38) node {\tiny{$q_c$}}
				(1.03,1.13) node {\tiny{$q_d$}};
			\end{tikzpicture}
		\end{center}
		\caption{The curve $C$ in Example~\ref{Eg:abcd}}\label{Fig:abcd}
	\end{figure}
	
	We assume that {\it $q_a$ and $q_b$ are in general positions on $F$}. In other words, they are not conjugate points,
	and neither is Weierstrass.
	
	By Proposition~\ref{Prp:phi_key}~\ref{Part:phi}, locally on some affine chart $\cV$ containing the point $z_0$, we have
	$$\varphi =\left[\begin{matrix}  
		a_{11} \zeta_a & c_{11} \zeta_b\zeta_c &  d_{11} \zeta_b\zeta_d &\cdots
		\\
		a_{21} \zeta_a& c_{21} \zeta_b\zeta_c &  d_{21} \zeta_b\zeta_d  &\cdots\\
	\end{matrix} \right]\,,
	$$  
	where $\zeta_a$ (resp.~$\zeta_b, \zeta_c, \zeta_d$) is the node-smoothing parameter for $q_a$ (resp.~$q_b,$ $q_c,$ $q_d$), and $a_{ij}$, $c_{ij}$ and $d_{ij}$ are invertible functions on $\cV$. 
	The omitted entries each can also be divided by $\ze_a$, $\ze_b\ze_c$, or $\ze_b\ze_d$.
	Moreover, by Proposition~\ref{Prp:phi_key}~\ref{Part:theta},
	the vanishing orders of $2\!\times\!2$ minors of $\varphi$ can be higher:
	under the assumption of this example, we have
	\begin{align*}
		\det\left[\begin{matrix}
			a_{11} & c_{11}\\
			a_{21} & c_{21}
		\end{matrix}\right],~
		\det\left[\begin{matrix}
			a_{11} & d_{11}\\
			a_{21} & d_{21}
		\end{matrix}\right]
		\in\Ga(\sO_\cV^*),\qquad
		\det\left[\begin{matrix}
			c_{11} & d_{11}\\
			c_{21} & d_{21}
		\end{matrix}\right]\Big/\ze_{b}
		\in\Ga(\sO_\cV^*)
		.
	\end{align*}
	In addition, if $
	R_a$ has weight greater than one, the $2\!\times\!2$ minors corresponding to the points of $D\!\cap\!R_a$ may need to be taken into consideration,
	which contains $\zeta_a^2$ as a factor. Similarly, the term $\zeta_b^2\zeta_c^2$ 
	(resp.~$\zeta_b^2\zeta_d^2$) may need to be taken into consideration if  $R_c$ 
	(resp.~$R_d$) has  weight greater than one.

Here, we point out the following. 
In $(\rd_1\ph_1)$, for any $z_0=(C,\bw) \in \mwt$ whose core is of weight 0,
the farther the separating nodes are away from the core, the later the corresponding parameters 
appear in the local equations of (the proper transforms of) the blowup centers. For instance, in the above example,
for the nodes $q_a$ and $q_b$ that are closest to the core, locally at $z_0$,
$$(\zeta_a=\zeta_b=0)$$ defines the first blowup center $\ov \fM_{(1,2)} \cap \cV$.
After this blowup, on the chart given by $\ze_b\eq\ve_b$ and $\ze_a\eq\ve_b\wc\ze_a$, the locus
$$(\wc\zeta_a= \zeta_c=\zeta_d=0)$$ locally defines the 
next blowup center, the proper transform of $\ov \fM_{(1,3)}$.
This generalizes to all separating nodes that are farther from the core, if any.

	Back to the current example, the rooted tree of this example is identical to that of
	Example 5.17 of \cite{HL10}.
	The calculation within each phase of $(\rd_1)$ is also similar to  \cite{HL10}.
	 
	More precisely,
	assume that $R_a$, $R_c$ and $R_d$ all have 
	weight one.
	Then, using similar calculations of Example 5.17 of \cite{HL10} 
	and then applying Proposition \ref{Prp:phi_key},
	one finds that after $(\rd_1)$, 
	up to switching the symbols  $c$ and $d$,
	the pullback $\ti\varphi^{\rd_1}$ of $\varphi$ may take one of the following forms:
	\begin{align*}
		&
		\left[\begin{matrix}
			\ve_a^{\rd_1\ph_1} & 0\\
			0 & \ve_a^{\rd_1\ph_1}\ve_b^{\rd_1\ph_5}\ve_c^{\rd_1\ph_5}
		\end{matrix}\right]
		\left[\begin{matrix}  
			1 & 0 &  0
			\\
			0 &  1  &  \wc\zeta_d \\
		\end{matrix} \right],\ \ 
		&&
		\left[\begin{matrix}
			\ve_a^{\rd_1\ph_1} & 0\\
			0 & \ve_a^{\rd_1\ph_1}\ve_c^{\rd_1\ph_5}
		\end{matrix}\right]
		\left[\begin{matrix}  
			1 & 0 &  0
			\\
			0 &  1  &  \wc\zeta_d \\
		\end{matrix} \right],
		\\
		&
		\left[\begin{matrix}
			\ve_b^{\rd_1\ph_1}\ve_a^{\rd_1\ph_1}  & 0\\
			0 &  \ve_b^{\rd_1\ph_1}\ve_a^{\rd_1\ph_1}\ve_c^{\rd_1\ph_5}
		\end{matrix}\right]
		\left[\begin{matrix}  
			1& 0 & 0
			\\
			0 &  1 &  \wc\zeta_d \\
		\end{matrix} \right],
		&
		&
		\ve_b^{\rd_1\ph_1}  \ve_c ^{\rd_1\ph_1}
		\left[\begin{matrix}  
			0 & 1 &  0  \\
			1& 0 & \ve_b^{\rd_1\ph_1}\wc\zeta_d
		\end{matrix} \right],
		\\
		&
		\ve_b^{\rd_1\ph_1}  \ve_c ^{\rd_1\ph_1}
		\left[\begin{matrix}  
			0 & 1 &  0 \\
			\wc\zeta_a& 0 & \ve_b^{\rd_1\ph_1}
		\end{matrix} \right],
		&&
		\left[ 
		\begin{matrix}
			\ve_b^{\rd_1\ph_1}  \ve_c ^{\rd_1\ph_1} & 0\\
			0 & \ve_b^{\rd_1\ph_1}  \ve_c ^{\rd_1\ph_1}\ve_a^{\rd_1\ph_5}
		\end{matrix}
		\right]
		\left[\begin{matrix}  
			0 & 1 &  0 \\
			1& 0 &  \ve_b^{\rd_1\ph_1}\wc\zeta_d 
		\end{matrix} \right],
		\\
		&\tn{or}\quad 
		\left[ 
		\begin{matrix}
			\ve_b^{\rd_1\ph_1}  \ve_c ^{\rd_1\ph_1} & 0\\
			0 & \ve_b^{\rd_1\ph_1}  \ve_c ^{\rd_1\ph_1}\ve_d^{\rd_1\ph_5}
		\end{matrix}
		\right]
		\left[\begin{matrix}  
			0 & 1 &  0 \\
			\wc\zeta_a& 0 &  \ve_b^{\rd_1\ph_1}
		\end{matrix} \right],
		&&
	\end{align*}
	where $\ve$ are local parameters corresponding to the proper transforms of the exceptional divisors, the superscripts of $\ve$ indicate the corresponding phases in ($\rd_1$),
	and
	$\wc\zeta_a$ and $\wc\zeta_d$ are proper transforms of 
	$\zeta_a$ and $\zeta_d$, respectively.
	Among the above seven possibilities of $\ti\varphi^{\rd_1}$,
	five are already diagonalized;
	in either of the remaining two cases,
	the only blowup center of ($\rd_2$) that may  affect  $\ti\varphi^{\rd_1}$ is locally defined by
	\begin{align*}
		\big(\ve_b^{\rd_1\ph_1}=\wc\ze_a=0\big),
	\end{align*}
	hence the pullback of $\varphi$ after ($\rd_2$) becomes diagonalized.
	
	Next, consider the case when at least one of $R_a$, $R_b$ and $R_c$, say $R_a$, has 
	weight greater than one. In such a case, by Proposition~\ref{Prp:phi_key}~\ref{Part:theta}, entries containing $\zeta_a^2$ may occur 
	in $\varphi$.
	For instance, 
	one finds that after $(\rd_1\ph_1)$ (hence after $(\rd_1)$), 
	on the chart given by $\ze_a\eq\ve_a$ and $\ze_b\eq\ve_a\wc\ze_b$,
	the pullback of $\varphi$ takes the following form:
	$$\ve_a \left[\begin{matrix}  
		1 & 0 & 0 &  0 & 0& \cdots
		\\
		0 & \ve_a &  \wc\zeta_b  \zeta_c  &  \wc\zeta_b  \zeta_d  & 0 & \cdots \\
	\end{matrix} \right].$$
	The blowup in ($\rd_2$) that may affect this chart has the centers locally defined by
	\begin{align}\label{Eqn:S_ne_empty}
	(\ve_a=\wc\ze_b=0)\qquad
	\tn{and}\qquad\tn{PT}(\ve_a=\ze_c=\ze_d=0),
	\end{align}
	where $\tn{PT}$ refers to the proper transform. 
	Globally the two loci correspond to $X_{2,2}$ and $X_{2,3}$, respectively; see Figure~\ref{figBlowup2}.
	(The derived tree in this case is labeled by (3) in Figure~\ref{figDerivedTFMRs}.)
	Then, following similar calculations as in $(\rd_1)$,
	one finds that after $(\rd_2)$, 
	the pullback of $\varphi$ becomes diagonalized.
	
	Similarly, one can find that after $(\rd_2)$, 
	the pullback of $\varphi$ becomes diagonalized on other charts.
\end{exam}

\begin{exam}\label{Eg:abcd'}  
	Let $z_0=(C,D)$ be the same as in Example~\ref{Eg:abcd} with only one difference: {\it we assume that $q_a$ is a Weierstrass point of the core $F$ of $C$.} Since $F$ is assumed smooth, $q_a$ and $q_b$ cannot be conjugate.
We further assume $D\!\cap\!R_a\eq\{\de_1,\ldots,\de_{k}\}$ with $k\!\ge\!3$.

After $(\rd_1)$, 
on the chart given by $\ze_a\eq\ve_a$ and $\ze_b\eq\ve_a\wc\ze_b$,
the pullback of $\varphi$ takes the following form:
$$\ve_a \left[\begin{matrix}  
	1 & 0 & 0 & 0 & 0 & 0& \cdots
	\\
	0 & \wc\ka_{12}\ve_a & \ve_a^2 & \wc\zeta_b  \zeta_c  &  \wc\zeta_b  \zeta_d  & 0 & \cdots \\
\end{matrix} \right],$$
where $\wc\ka_{12}$ is the proper transform of $\ka_{12}$ as in Example~\ref{Eg:ab}.
The blowup in ($\rd_2$) that may affect this chart still has the centers locally defined by
\begin{align*}
	(\ve_a=\wc\ze_b=0)\qquad
	\tn{and}\qquad\tn{PT}(\ve_a=\ze_c=\ze_d=0).
\end{align*}
Following similar calculations as in Examples~\ref{Eg:ab} and~\ref{Eg:abcd},
one finds that after $(\rd_2)$, 
on the new chart given by $\ve_a\eq\ve'$ and $\wc\ze_b\eq\ve'\wc\ze_b'$,
the pullback of $\varphi$ takes the form
$$
\left[\begin{matrix}
	\ve' & 0\\
	0 & (\ve')^2
\end{matrix}\right]
\left[\begin{matrix}  
	1 & 0 & 0 & 0 & 0 & 0& \cdots
	\\
	0 & \wc\ka_{12} & \ve' & \wc\zeta_b'  \zeta_c  &  \wc\zeta_b'  \zeta_d  & 0 & \cdots \\
\end{matrix} \right].$$
The blowup in ($\rd_3\ph_2$) that may affect this chart  has the centers locally defined by
\begin{align}\label{Eqn:S^dot_nonempty}
	(\wc\ka_{12}=\ve'=\wc\ze_b'=0)\qquad
	\tn{and}\qquad\tn{PT}(\wc\ka_{12}=\ve'=\ze_c=\ze_d=0).
\end{align}
Therefore,
after $(\rd_3\ph_2)$,
the pullback of $\varphi$ becomes diagonalized.
\end{exam}

\begin{exam}\label{Eg:abcd''}  
	Let $z_0=(C,D)$ be the same as in Example~\ref{Eg:abcd} with only one difference: {\it we assume that $q_a$ and $q_b$ are conjugate points on the core $F$ of $C$.} Since $F$ is assumed smooth, neither $q_a$ nor $q_b$ is Weierstrass.
	We further assume $|D\!\cap\!R_\bullet|\!\ge\!3$ for $\bullet\eq a,c,d$.
	Moreover, we assume $\de_1\inn R_a$, $\de_2\inn R_b$ and $\de_3\inn R_c$.
	
	After $(\rd_1)$, 
	on the chart given by $$\ze_a=\ve_1\ve_2,\qquad
	\ze_b=\ve_1,\qquad
	\ze_c=\ve_1\ve_2\wc\ze_c,\qquad\tn{and}\qquad
	\ze_d=\ve_1\ve_2\wc\ze_d,
	$$
	(i.e.~locally one first blows up along $(\ze_a\eq\ze_b\eq 0)$, then on the chart given by $\ze_b\eq\ve_1$ and $\ze_a\eq\ve_1\wc\ze_a$, one blows up along $(\wc\ze_a\eq\ze_c\eq\ze_d\eq 0)$ and considers the new chart given by $\wc\ze_a\eq \ve_2$, $\ze_c\eq\ve_2\wc\ze_c$ and $\ze_d\eq\ve_2\wc\ze_d$,)
	the pullback of $\varphi$ takes the following form:
	\begin{align}\label{Eqn:Eg_abcd''}
		\ve_1\ve_2 \left[\begin{matrix}  
		1 & 0 & 0 & 0 & 0 &  \cdots
		\\
		0 & \ve_1\ve_2 & \wc\ka_{12}\wc\ze_c & \wc\ka_{13}\wc\ze_d  & 0 & \cdots \\
	\end{matrix} \right],
	\end{align}
	where $\wc\ka_{ij}$ is the pullback of $\ka_{ij}$ as in Example~\ref{Eg:ab}.
	Proposition~\ref{Prp:phi_key}~\ref{Part:kappa} implies that 
	$$
	 \ka_{13}=f'\ka_{12}+g'\ze_b,
	 \qquad
	 \Longrightarrow
	 \qquad
	 \wc\ka_{13}=f\wc\ka_{12}+g\ve_1,
	$$
	where $f'$, $g'$, $f$ and $g$ are some units.
	
	The blowup in ($\rd_2$) that may affect this chart still has the centers locally defined by
	\begin{align*}
		(\ve_2=\wc\ze_c=\wc\ze_d=0)\qquad
		\tn{and}\qquad\tn{PT}(\ve_1=\wc\ze_c=\wc\ze_d=0).
	\end{align*}
	Globally, these two loci correspond to the substacks $X_{3,3}$ and $X_{2,3}$ of $\ti\fM^{\rd_2}$;
	see (\ref{e_Xkk'}) for notation and Figure~\ref{figBlowup2} for illustration.
	(The derived tree in this case is labeled by (1) in Figure~\ref{figDerivedTFMRs}.)
	
	Following the order~(\ref{Eqn:Ad_1_order}),
	we first blow up along $(\ve_2=\wc\ze_c=\wc\ze_d=0)$.
	
	On the chart given by $\ve_2\eq\ve'$,
	$\wc\ze_c\eq\ve'\wc\ze_c'$ and $\wc\ze_d\eq\ve'\wc\ze_d'$, the second row of the pullback of $\varphi$ can be rewritten as 
	$$
	\ve_1(\ve')^2\big[\ 0\ \ \ \ve_1\ \ \  \wc\ka_{12}\wc\ze_c'\ \ \  \wc\ka_{13}\wc\ze_d'\ \ \ \cdots\ \big]
	\ =\ 
	\ve_1(\ve')^2\big[\ 0\ \ \ \ve_1\ \ \  \wc\ka_{12}\wc\ze_c'\ \ \  (f\wc\ka_{12}+g\ve_1)\wc\ze_d'\ \ \ \cdots\ \big].
	$$
	The remaining blowups in $(\rd_2)$, as well as those in $(\rd_3\ph_1)$,
	that may affect this chart have centers locally given by
	$$
	\tn{PT}(\ve_1=\wc\ze_c=\wc\ze_d=0)=(\ve_1=\wc\ze_c'=\wc\ze_d'=0)\qquad
	\tn{and}\qquad
	\tn{PT}(\ve_1=\wc\ka_{12}=0),
	$$
	respectively.
	Therefore, the pullback of $\varphi$ is diagonalized after $(\rd_3\ph_1)$.
		
	On the chart given by $\wc\ze_c\eq\ve'$,
	$\ve_2\eq\ve'\wc\ve_2$ and $\wc\ze_d\eq\ve'\wc\ze_d'$, the second row of the pullback of $\varphi$ can be rewritten as 
	$$
	\ve_1(\ve')^2\big[\ 0\ \ \ \ve_1\wc\ve_2\ \ \  \wc\ka_{12}\ \ \  \wc\ka_{13}\wc\ze_d'\ \ \ \cdots\ \big]
	\ =\ 
	\ve_1(\ve')^2\big[\ 0\ \ \ \ve_1\wc\ve_2\ \ \  \wc\ka_{12}\ \ \  \ve_1\wc\ze_d'\ \ \ \cdots\ \big].
	$$
	The proper transform of $(\ve_1=\wc\ze_c=\wc\ze_d=0)$ does not affect this chart,
	whereas the relevant blow up centers in $(\rd_3\ph_1)$ are
	$$
	(\wc\ve_2=\wc\ze_d'=\wc\ka_{12}= 0)\qquad
	\tn{and}\qquad
	\tn{PT}(\ve_1=\wc\ka_{12}=0),
	$$
	respectively.
	Therefore, the pullback of $\varphi$ is diagonalized after $(\rd_3\ph_1)$.
	
	The argument on the chart given by $\wc\ze_d\eq\ve'$,
	$\ve_2\eq\ve'\wc\ve_2$ and $\wc\ze_c\eq\ve'\wc\ze_c'$ is parallel.
	
	\vsp
	At the end of this example,
	we remark that in $(\rd_2)$, if the order (\ref{Eqn:Ad_1_order}) were replaced by the usual lexicographic order on $\mathbb Z\!\times\!\mathbb Z$,
	then after the entire three-round blowups terminate, the pullback of $\varphi$ would fail to be diagonalizable on some chart.
	To see this,
	we blow up $X_{2,3}$ prior to $X_{3,3}$ in this example.
	Locally,
	this means the locus $(\ve_1\eq \wc\ze_c\eq \wc\ze_d\eq 0)$ is blown up before $(\ve_2\eq \wc\ze_c\eq \wc\ze_d\eq 0)$.
	Then, on the chart given by $\ve_1\eq\ve'$, $\wc\ze_c\eq\ve'\wc\ze_c'$ and $\wc\ze_d\eq\ve'\wc\ze_d'$, the second row of the pullback of (\ref{Eqn:Eg_abcd''}) can be rewritten as 
	$$
	\ve_2(\ve')^2\big[\ 0\ \ \ \ve_2\ \ \  \wc\ka_{12}\wc\ze_c'\ \ \  (f\wc\ka_{12}+g\ve')\wc\ze_d'\ \ \ \cdots\ \big].
	$$
	After blowing up $\tn{PT}(\ve_2\eq \wc\ze_c\eq \wc\ze_d\eq 0)=
	(\ve_2\eq \wc\ze_c'\eq \wc\ze_d'\eq 0)$,
	on the chart given by
	$\wc\ze_c'\eq\ve''$,
	$\ve_2\eq\ve''\wc\ve_2$ and $
	\wc\ze_d'\eq\ve''\wc\ze_d''$,
	the second row of the pullback of $\varphi$ can be further rewritten as 
	$$
	\ve_2(\ve'\ve'')^2\big[\ 0\ \ \ \wc\ve_2\ \ \  \wc\ka_{12}\ \ \  \ve'\wc\ze_d''\ \ \ \cdots\ \big].
	$$
	The problem is that the locus $(\wc\ve_2\eq\ve'\eq\wc\ka_{12}\eq 0)$,
	which must be blown up so as to simplify the matrix above, does not correspond to any blowup center in $(\rd_3)$.
	Indeed, such loci lie deeply in the intersections of the exceptional divisors from previous steps,
	so it seems formidable to organize them in a systematic way should such blowups be included.
	However,
	by introducing the order (\ref{Eqn:Ad_1_order}) in $(\rd_2)$, such issues can be avoided nicely.
\end{exam}

\begin{exam}
	Let $z_0\eq (C,D)\inn\fM_2^{\tn{div}}$ be such that the core $F$ of $C$ is comprised of a smooth rational subcurve $B$ and a smooth genus 1 subcurve $C_2$, meeting at two points $\fp$ and $\fq$, 
	the tail of $C$ is a smooth rational subcurve $C_a$ attached to $C_2$ at $q_a$,
	and $D$ satisfies
	$$
	D\!\cap\!B= \{\de_1,\de_2,\de_3\},\qquad
	D\!\cap\!C_2= \emptyset,\qquad
	D\!\cap\!C_a\ne \emptyset.
	$$
	Particularly, $B$ is a (maximal) non-separating bridge $\tn B[\fp,\fq]$.
	Such $(C,D)$ lies over a general point in $\ov\fM_{(3,2)}$; see Figure~\ref{figBlowup1} for illustration.
	
	By Proposition~\ref{Prp:phi_key},
	under suitable trivialization, 
	$\varphi$ on a chart $\cV$ containing $(C,D)$ can be written as
	$$
	\left[
	\begin{matrix}
		1 & 0 & 0 & 0 & 0 &\cdots\\
		0 &  f\ze_\fp\!+\!g\ze_\fq & f'\ze_\fp\!+\!g'\ze_\fq & \ze_a & 0 &\cdots
	\end{matrix}
	\right],
	$$
	where $\ze_\fp$ and $\ze_\fq$ are node smoothing parameters corresponding to $\fp$ and $\fq$,
	and $f$, $f'$, $g$, and $g'$ are invertible functions on $\cV$ satisfying 
	\begin{align}\label{Eqn:det_31}
		\det\left[
		\begin{matrix}
			f & f'\\ g&g'
		\end{matrix}\right]
		\in\Ga(\sO^*_\cV).
	\end{align}
	
	In $(\rd_1\ph_3)$,
	locally the blowup center is given by 
	$(\ze_\fp\eq\ze_\fq\eq\ze_a\eq 0)$,
	corresponding to (the proper transform of)
	$\ov\fM_{(3,2)}$.
	When $(\rd_1\ph_3)$ terminates,
	let $y$ be a fixed lift  of $x$.
	
	On the chart given by $\ze_a\eq\ve$,
	$\ze_\fp\eq\ve\wc\ze_\fp$ and $\ze_\fq\eq\ve\wc\ze_\fq$,
	the pullback of $\varphi$ is obviously diagonalizable near $y$.
	
	On the chart given by $\ze_\fp\eq\ve$,
	$\ze_a\eq\ve\wc\ze_a$ and $\ze_\fq\eq\ve\wc\ze_\fq$, the pullback of $\varphi$ can be rewritten as
	$$
	\left[
	\begin{matrix}
		1 & 0\\
		0 & \ve
	\end{matrix}
	\right]
	\left[
	\begin{matrix}
		1 & 0 & 0 & 0 & 0 &\cdots\\
		0 & f\!+\!g\wc\ze_\fq & f'\!+\!g'\wc\ze_\fq & \wc\ze_a &0 &\cdots
	\end{matrix}
	\right]\,.
	$$
	By~(\ref{Eqn:det_31}),
	$f\!+\!g\wc\ze_\fq$ and $f'\!+\!g'\wc\ze_\fq$ cannot vanish simultaneously,
	hence  the pullback of $\varphi$ is diagonalizable near $y$.
\end{exam}

\begin{exam}\label{Eg:ns_bridge}   
	Let $z_0=(C,D)$ be such that $C\eq F$ consists of a smooth genus one curve $C'$
	having weight zero and an inseparable bridge $\tn B[\fp,\fq]$ made of three smooth 
	rational curve $R_a$, $R_b$ and $R_c$ that are all positively weighted.
	We assume that $R_a$ is attached to the genus one curve at $\fp$ and
	$R_c$ is attached to the genus one curve at $\fq$. Further, we write
	$q_a\eq R_a\!\cap\!R_b$ and $q_c\eq R_b\!\cap\!R_c$;
	see Figure~\ref{Fig:ns_bridge}.
	
	\begin{figure}[htb]
		\begin{center}
			\begin{tikzpicture}
			\def\1a{
				(0,0.8)..controls (-0.9,0.8) and (-1.5,0.36)..
				(-1.5,0)..controls (-1.5,-0.36) and (-0.9,-0.8)..
				(0,-0.8)..controls (0.96,-0.8) and (1.36,-0.42)..
				(1.36,-0.36)..controls (1.36,-0.3) and (1.3,-0.2)..
				(1.2,-0.2)..controls (1.1,-0.2) and (0.96,-0.36)..
				(0.76,-0.36)..controls (0.56,-0.36) and (0.3,-0.2)..
				(0.3,0)..controls (0.3,0.2) and (0.56,0.36)..
				(0.76,0.36)..controls (0.96,0.36) and (1.1,0.2)..
				(1.2,0.2)..controls (1.3,0.2) and (1.36,0.3)..
				(1.36,0.36)..controls (1.36,0.42) and (0.96,0.8)..
				(0,0.8)
				(-0.88,0)..controls (-0.64,0.12)..(-0.4,0)
				(-0.4,0)..controls (-0.64,-0.12)..(-0.88,0)
				(-1.04,0.08)--(-0.88,0)
				(-0.4,0)--(-0.24,0.08)
			}
			\draw \1a;
			\draw[fill=black!42] 
			(1.6,-.42) circle (0.24)
			(1.6,.42) circle (0.24)
			(1.832379,0) circle (0.24);
			\draw 
			(1.6,.42) node {\tiny{$R_a$}}
			(1.832379,0) node {\tiny{$R_b$}}
			(1.6,-.42) node {\tiny{$R_c$}}
			(0,-.5) node {\tiny{$C'$}}
			(2,.35) node {\tiny{$q_a$}}
			(2,-.32) node {\tiny{$q_c$}}
			(1.3,.6) node {\tiny{$\fp$}}
			(1.3,-.6) node {\tiny{$\fq$}};
			\end{tikzpicture}
		\end{center}
		\caption{The curve $C$ in Example~\ref{Eg:ns_bridge}}\label{Fig:ns_bridge}
	\end{figure}
	
	By Lemma \ref{Lm:K_smoothness}, using a point from each of $D\!\cap\!R_a, D\!\cap\!R_b$ and $D\!\cap\!R_c$,
	we have
	$$ \left[\begin{matrix}  
		1 & 0 &  0 &  \cdots
		\\
		0 &   f_2 \zeta_\fp + g_2 \zeta_c\zeta_\fq &  f_3\zeta_\fp + g_3 \zeta_\fq  & \cdots \\
	\end{matrix} \right],$$
	where $\ze_\fp$, $\ze_\fq$ and $
	\ze_c$ are respectively the node-smoothing parameters corresponding to $\fp$, $\fq$ and $q_c$, while $f_2$, $f_3$, $g_2$ and $g_3$ are all units.
	In $(\rd_1\ph_3)$, this point lies in the blowup center locally defined by $(\zeta_\fp\eq\zeta_\fq\eq 0)$.
	Then, one calculates and finds that after $(\rd_1\ph_3)$, the pullback of $\varphi$ on any chart
	takes one of the following two forms:
	$$
	\left[
	\begin{matrix}
		1 & 0\\ 0 & \ve_\fp
	\end{matrix}\right]
	\left[\begin{matrix}  
		1& 0 & \cdots
		\\
		0 &  1 &  \cdots \\
	\end{matrix} \right] 
	\qquad\tn{or}\qquad
	\left[
	\begin{matrix}
		1 & 0\\ 0 & \ve_\fq
	\end{matrix}\right]
	\left[\begin{matrix}  
		1& 0 & 0 &  \cdots
		\\
		0 &  \wc\ze_\fp &  1 &  \cdots \\
	\end{matrix} \right],$$
	where $\wc\ze_\fp$ is the proper transform of $\ze_\fp$.
	The above calculations generalize to all cases of $\ov \fM_{(3,1)}$ (c.f.~Figure~\ref{figBlowup1}):
	only the two nodes $\fp$ and $\fq$
	in the maximal bridge $\tn B[\fp,\fq]$ play deciding  roles in the process of the blowup;
	the remaining node-smoothing parameters, such as $\zeta_a$ and $\zeta_c$ in the above example, do not
	play any decisive roles,
	even though they may appear in $\varphi$.
\end{exam}


	
	
	
	

\begin{exam}\label{Eg:sep_core}
	Let $z_0=(C,D)$ be such that the core $F$ of $C$ is made of  two smooth genus one curves
	$F_1$ and $F_2$, each of which is of weight zero, and the tails of $C$ consists of two smooth rational curves $R_a$ and $R_b$, positively weighted and attached to $F_1$ and $F_2$ at points $q_a$ and $q_b$, respectively.
	We write $q=F_1 \cap F_2$;
	see Figure~\ref{Fig:sep_core}.
	
	\begin{figure}[htb]
		\begin{center}
			\begin{tikzpicture}
			\def\g1{
				(-1,0) ellipse (1 and 0.5)
				(-1.4,0)..controls(-1,0.1)..(-0.6,0)
				(-0.6,0)..controls(-1,-0.1)..(-1.4,0)
				(-0.5,0.05)--(-0.6,0)
				(-1.4,0)--(-1.5,0.05)
			}
			
			\draw \g1;
			\draw[xshift=2cm] \g1;
			\draw[fill=black!42] 
			(-1,0.9) circle (0.4);
			\draw[fill=black!42] 
			(1,0.9) circle (0.4);
				\draw 
				(-1,0.9) node {\tiny{$R_a$}}
				(1,0.9) node {\tiny{$R_b$}}
				(-1,-.3) node {\tiny{$F_1$}}
				(1,-.3) node {\tiny{$F_2$}}
				(-1,.35) node {\tiny{$q_a$}}
				(-.12,-.02) node {\tiny{$q$}}
				(1,.35) node {\tiny{$q_b$}};
			\end{tikzpicture}
		\end{center}
		\caption{The curve $C$ in Example~\ref{Eg:sep_core}}\label{Fig:sep_core}
	\end{figure}
	
	In such a case, $q_a$ and $q_b$ cannot be conjugate. We assume neither of them
	is Weierstrass. Further, we assume that  the weights of $R_a$ and $R_b$ are both  greater than one.
	
	Then by Proposition~\ref{Prp:phi_key},  locally on some affine chart $\cV$ containing the point $z_0$, we have
	$$
	\varphi =
	\left[\begin{matrix}  
		c_{11} \zeta_a & 0 & \cdots & d_{11}\zeta_q \zeta_b & \zeta_q  \zeta_b^2 & \cdots \\
		c_{21} \zeta_q \zeta_a &  \zeta_q \zeta_a^2 & \cdots & d_{21} \zeta_b & 0 &\cdots\\
	\end{matrix} \right]
	$$  
	where $\zeta_q$  is the node-smoothing parameter for the separating node $q$,
	$\zeta_a$ (resp.~$\zeta_b$) is the node-smoothing parameter for $q_a$ (resp. $q_b$),
	and $c_{ij}$ and $d_{ij}$ are some units.

	After $(\rd_1)$, on any small affine chart over $z_0$, 
	either the pullback of $\varphi$ is already diagonalized, or else, 
	up to switching the symbols  $a$ and $b$,
	the pullback of $\varphi$ may take the form 
	$$
	\left[\begin{matrix}  
		\ve_a  & 0\\
		0 & \ve_a \ve_q
	\end{matrix} \right]
	\left[\begin{matrix}  
		1 &0 &0 &\cdots \\
		0 & \ve_a & \wc\zeta_b  & \cdots\\
	\end{matrix} \right] $$  
	where $\ve_a$ and $\ve_q$ are exceptional parameters from $(\rd_1\ph_1)$ and $(\rd_1\ph_4)$
	(corresponding to $q_a$ and $q$), respectively, and $\wc\zeta_b$ is the proper transform of $\zeta_b$.
	
	As in Case A of Example~\ref{Eg:ab},  one sees that after $(\rd_2)$, the pullback of	the above matrix becomes diagonalizable on any small chart over $z_0$.
\end{exam}

\begin{exam} \label{Eg:core_2}
	Let $z_0=(C,D)$ be such that $C$ is made of a smooth core curve $F$
	satisfying $F \cap D\eq\{\de_1, \de_2\}$ 
	and one smooth rational curve $R$ attached to $F$ at a point $q \ne \de_1, \de_2$;
	see Figure~\ref{Fig:core_2}.
	
	\begin{figure}[htb]
		\begin{center}
			\begin{tikzpicture}
			\def\g2{
				(0,0.8) arc (90:450:1.6 and 0.8)
				(0.4,0)..controls (0.64,0.12)..(0.88,0)
				(0.88,0)..controls (0.64,-0.12)..(0.4,0)       
				(0.24,0.08)--(0.4,0)
				(0.88,0)--(1.04,0.08)
				(-0.4,0)..controls (-0.64,-0.12)..(-0.88,0)
				(-0.88,0)..controls (-0.64,0.12)..(-0.4,0)       
				(-0.24,0.08)--(-0.4,0)
				(-0.88,0)--(-1.04,0.08)
			}
			\draw[fill=black!42] 
			\g2
			(0,1.2) circle (0.4);
			\filldraw 
			(-1.25,.35) circle (1pt)
			(-1.25,-.35) circle (1pt);
			\draw 
			(0,1.2) node {\small{$R$}}
			(0,-.5) node {\small{$F$}}
			(0,.6) node {\tiny{$q$}}
			(-1.38,.6) node {\tiny{$\de_1$}}
			(-1.38,-.55) node {\tiny{$\de_2$}};
			\end{tikzpicture}
		\end{center}
		\caption{The point $(C,D)$ in Example~\ref{Eg:core_2}}\label{Fig:core_2}
	\end{figure}

	When $\de_1, \de_2$ are not conjugate, $z_0$ is already a smooth point of $\fM^{\rm div}_2$.
	So we study the case when $\de_1, \de_2$ are conjugate.
	
	By Proposition~\ref{Prp:phi_key},
	locally on some affine chart $\cV$ containing the point $z_0$, we have
	$$\varphi =\left[\begin{matrix}  
		c_{11} & c_{12} & c_{13} \zeta &\cdots \\
		c_{21} & c_{22} & c_{23} \zeta &\cdots\\
	\end{matrix} \right]
	$$  
	where $\zeta$  is the node-smoothing parameter for the node $q$.
	By Proposition~\ref{Prp:phi_key}~\ref{Part:theta},
	under a suitable new basis, this matrix can be represented as
	$$ \left[\begin{matrix}  
		1 & 0 & 0 & \cdots \\
		0 & \ka_{12} & \zeta &\cdots\\
	\end{matrix} \right] $$
	where $\ka_{12}\eq 0$ defines the locus where $\cD_1(z)$ and $\cD_2(z)$ are conjugate points on the curve $\cC_z$.
	
	In  $(\rd_3\ph_3)$, 
	$z_0$ lies in the blowup center locally given by $(\ka_{12}\eq\ze\eq 0)$,
	so after $(\rd_3\ph_3)$, one sees that the pullback of
	the above matrix becomes diagonalizable on any small chart over $z_0$.
\end{exam}

\section{Local diagonalizability of the structural homomorphisms on $\ti\fM_2^{\rm div}$}
\label{SecChangeOfPhi}

\subsection{Mains statement and the structure of its proof}
\label{SubsecChangeofPhiStatement}
Throughout \S\ref{SecChangeOfPhi},
$\ti\fM_2^{\rm div}$ denotes the final stack after the three-round blowups and $\pi\!:\ti\fM_2^{\rm div}\!\to\!\fdd$ the projection. 
The following commutative diagram shows the relation between the relevant stacks,
where the horizontal arrow is given by (\ref{Eqn:div_to_wt}).
\begin{center}
	\begin{tikzpicture}
		\draw
		(4,1) node {$\mwt$}
		(1.5,1) node {$\fM_2^{\rm div}$}
		(1.5,2.5) node {$\ti\fM_2^{\rm div}$}
		;
		
		\draw[->,>=stealth']
		(1.5,2.15) to node[left]{\small{$\pi$}} (1.5,1.3);
		\draw[->,>=stealth']
		(1.9,1) to
		(3.5,1);
		(0.5,1) to (1.1,1);
		\draw[->,>=stealth']
		(1.85,2.4) to node[right,above] {\small{$\varpi$}} (3.55,1.3);

	\end{tikzpicture}	
\end{center}

In this section, we investigate the pullback $\ti\varphi$ of the structural homomorphism~$\varphi$ of (\ref{varphi}) locally.
The main statement is as follows.
	
\begin{prop}\label{PrpChangeofPhi}
	For every $\ti x\inn\ti\fM_2^{\tn{div}}$,
	there exists an affine smooth chart~$\ti\cV_{\ti x}\!\to\!\ti\fM_2^{\tn{div}}$ containing $\ti x$ so that the pullback $\ti\varphi$ of 
	$\varphi$ to $\ti\cV_{\ti x}$ takes the form
	$$
	\left[\begin{matrix} z_1   &  0 & 0& \cdots & 0 \\
		0  & z_2 &  0 & \cdots & 0
	\end{matrix} \right]
	\qquad\tn{with}\quad
	z_1|z_2\in\Ga\big(\sO_{\ti\cV_{\ti x}}\big),$$
	where $z_1$ and $z_2$ are both products of some smooth parameters corresponding to (the proper transforms of) the exceptional divisors of the sequential blowups $\ti\fM_2^{\tn{div}}\!\to\!\fM_2^{\tn{div}}$.	
	Consequently, $\ti\varphi$ is diagonalizable.
\end{prop}

Recall in~\eref{e_row},
there is a partition of $\mwt$ into five disjoint substacks:
$$
\mwt=
\fM_{(1)}^{\mn}\sqcup
\fM_{(2)}^{\mn}\sqcup
\fM_{(3)}^{\mn}\sqcup
\fM_{(4)}^{\mn}\sqcup
\fM^{\mn}.
$$
Proposition~\ref{PrpChangeofPhi} will be proved 
depending on which of the above substacks contains $\varpi(\ti x)$.
More concretely,
throughout \S\ref{SecChangeOfPhi},
we fix an arbitrary $\ti x\inn\ti\fM_2^{\tn{div}}$ and write
\begin{align*}
	\wh x=\big(\,C,\,D= \de_1+\cdots+\de_m\,\big)
	=\pi(\ti x)\in\fM_2^{\rm div},\qquad
	x=\big(C,c_1(D)\big)=\varpi(\ti x)\in\mwt.
\end{align*}
Based on the topology of $F$ and the locations of the points in~$D$,
we divide the proof of Proposition~\ref{PrpChangeofPhi} into the following cases.
\begin{enumerate}[leftmargin=*,label=Case~\Alph*]
	\item \label{Case:5}
	\!\!\!. $x\in\fM^\mn$, satisfying one of the following:
	\begin{enumerate}[leftmargin=*,label=\arabic*]
		\item \label{Case:5.1}
		\!\!\!. the core $F$ of $C$ is separable;
		\item \label{Case:5.2}
		\!\!\!. $F$ is inseparable and $\deg D\!\cap\!F\!\ge\!3$;
		\item \label{Case:5.3}
		\!\!\!. $F$ is inseparable, $\deg D\!\cap\!F\!=\!2$, and the elements of $D\!\cap\!F$ are not conjugate to each other;
		\item \label{Case:5.4}
		\!\!\!. $F$ is inseparable, $\deg D\!\cap\!F\!=\!2$, and the elements of $D\!\cap\!F$ are conjugate to each other;
		\item \label{Case:5.5}
		\!\!\!. $F$ is inseparable and $\deg D\!\cap\!F\!=\!1$.
	\end{enumerate}
	\item \label{Case:4}
	\!\!\!. $x\in\fM^\mn_{(4)}$, satisfying one of the following:
	\begin{enumerate}[leftmargin=*,label=\arabic*]
		\item \label{Case:4.1}
		\!\!\!. $\deg D\!\cap\!F\!\ge\!3$;
		\item \label{Case:4.2}
		\!\!\!. $\deg D\!\cap\!F\!=\!2$ and the elements of $D\!\cap\!F$ are not conjugate to each other;
		\item \label{Case:4.3}
		\!\!\!. $\deg D\!\cap\!F\!=\!2$ and the elements of $D\!\cap\!F$ are  conjugate to each other;
		\item \label{Case:4.4}
		\!\!\!. $\deg D\!\cap\!F\!=\!1$.
	\end{enumerate}
	\item \label{Case:3}
	\!\!\!. $x\in\fM^\mn_{(3)}$, satisfying one of the following:
	\begin{enumerate}[leftmargin=*,label=\arabic*]
		\item \label{Case:3.1}
		\!\!\!. $F$ is inseparable and $\deg D\!\cap\!F\!\ge\!3$;
		\item \label{Case:3.2}
		\!\!\!. $F$ is inseparable and $\deg D\!\cap\!F\!=\!2$;
		\item \label{Case:3.3}
		\!\!\!. $F$ is inseparable and $\deg D\!\cap\!F\!=\!1$;
		\item \label{Case:3.4}
		\!\!\!. $F$ is separable.
	\end{enumerate}
	\item \label{Case:2}
	\!\!\!. $x\in\fM^\mn_{(2)}$.
	\item \label{Case:1}
	\!\!\!. $x\in\fM^\mn_{(1)}$.
\end{enumerate} 
We remark that \ref{Case:1} is the most complicated case.
Certain notation and terminology need to be introduced before dividing it into sub-cases,
so we postpone it until \S\ref{Subsec:Case_E}.

\begin{rema} 
	Let $\fN$  be the image of $\ov M_2(\P^n,d)$ in $\fM_2\wt$
	under the morphism (\ref{MPddToWeight0}).
	Then, $\fN$ is neither open nor closed in $\fM_2\wt$.
	For instance, consider  $\fM_{(i,j)}\!\subset\!\fM_2\wt$ as in (\ref{strata});
	see Figure~\ref{figBlowup1} for illustration.
	\begin{itemize}[leftmargin=*]
		\item Every $x\inn\fM_{(2,1)}$ with weight-1 core,
		which occurs in \ref{Case:2}, is contained in $\fN$.
		However, every open neighborhood $\cV_x$ of $x$ contains points of $\fM_{(4,1)}$ with weight-1 core,
		which occurs in \ref{Case:4}.\ref{Case:4.4} and is not contained in $\fN$.
		Therefore, $\fN$ is not open.
		\item Every $x\inn\fM_{(4,3)}$ with weight-1 core, which occurs in \ref{Case:4}.\ref{Case:4.4}, is not contained in $\fN$,
		but every  open neighborhood $\cV_x$ of $x$ contains points of $\fN$ whose domain curves are smooth,
		hence  $\fM_2\wt\bsl \fN$ is not open.
	\end{itemize}
	Likewise, the preimage  $\fN^{\rm div}$ of $\fN$ is neither open nor closed in $\fM_2^{\rm div}$.
	Therefore,
	although in Proposition~\ref{PrpChangeofPhi}, $\cE_\cV$ could be replaced by its restriction to $\cV\!\cap\!\fN^{\rm div}$,
	it would be difficult to construct a blowup $\ti \fN^{\rm div}\!\to\!\fN^{\rm div}$ and to analyze the pullbacks of the structural homomorphisms.
	This is why in Proposition~\ref{PrpChangeofPhi}, we choose the base stack to be the entire $\fM_2^{\rm div}$,
	which is smooth and has local parameters that are suitable for describing the structural homomorphisms, 
	even though the points of $\fM_2\wt$ satisfying
	\ref{Case:5}.\ref{Case:5.5}, 
	\ref{Case:4}.\ref{Case:4.4}, or
	\ref{Case:3}.\ref{Case:3.3}
	are not contained in $\fN$.
\end{rema}

In \S\ref{Subsec:Case_A_simple}-\ref{Subsec:Case_E_last},
we will prove Proposition~\ref{PrpChangeofPhi} case by case.

\subsection{Proof of \ref{Case:5}.\ref{Case:5.1}-\ref{Case:5.3}}
\label{Subsec:Case_A_simple}
We begin with the three simplest sub-cases of \ref{Case:5} above.

\begin{prop}
	\label{Prp:phi_M5_simple}
	Proposition~\ref{PrpChangeofPhi} holds in \ref{Case:5}.\ref{Case:5.1}, \ref{Case:5}.\ref{Case:5.2}, and \ref{Case:5}.\ref{Case:5.3}.
\end{prop}

\begin{proof}
	Let 
	$
	\cV\to\fM_2^{\rm div}
	$ be a small affine smooth chart containing~$\wh x$.
	
	In \ref{Case:5}.\ref{Case:5.1},
	let $F_1$ and $F_2$ be the two  genus 1 inseparable components 
	of~$F$.
	W.l.o.g.~we assume
	$$
	\de_1,\,a_1\in F_1,\qquad
	\de_2,\,a_2\in F_2.
	$$
	Then, by (the proof of)  Proposition \ref{Prp:genericSmoothness},
	the structural homomorphism $\varphi$ is diagonalizable on $\cV$.
	
	In~\ref{Case:5}.\ref{Case:5.2},
	observe that an inseparable $F$ can still be reducible: e.g., a smooth rational curve attached to
	two distinct points of a smooth genus one curve. 
	Nonetheless, there always exist $\de,\de'\inn D\!\cap\! F$
	that are not conjugate to each other,
	for otherwise all the elements of $D\!\cap\!F$ would lie on a non-separating bridge of $F$,
	which would imply $x\inn\fM_{(3)}^\mn$.
	In addition, $F$ being inseparable implies that $\ze_{[\de,a_s]},\ze_{[\de',a_s]}\inn\Ga(\sO^*_\cV)$ for $s\eq 1,2$.
	Hence by (the proof of)  Proposition \ref{Prp:genericSmoothness},
	$\varphi$ is diagonalizable on~$\cV$.

	In~\ref{Case:5}.\ref{Case:5.3},
	the argument from the preceding paragraph
	still holds verbatim.
	
	Finally, observe $x$ is away from the blowup centers of $\ti\fM_2^{\tn{div}}/\fM_2^{\tn{div}}$,
	hence we can take $\ti x\eq x$ and $\ti\cV_{\ti x}\eq\cV$,
	which is an affine smooth chart.
\end{proof}

In the remaining cases,
$\varphi$ is always not diagonalizable on any neighborhood of $\wh x$,
so we will study how each phase of the three-round blowups affect the pullback of $\varphi$ on a case-by-case basis.

\subsection{Local parameters for blowups}
\label{SubsecBlowups}
We continue with the setup in \S\ref{SubsecChangeofPhiStatement} and fix a small affine smooth chart $\cV\!\to\!\Md$ containing $\wh x$.

In order to tackle~\ref{Case:5}.\ref{Case:5.4},
when $F$ is inseparable and $D\!\cap\!F$ consists of two  conjugate points,
which we w.l.o.g.~assume to be $\de_1$ and $\de_2$,
we first observe that $\de_1$ cannot be conjugate to any $\de\inn D\bsl F$,
for otherwise $\de_1$ and hence $\de_2$ would be on  a non-separating bridge,
which would imply $x\inn\fM_{(3)}^\mn$.
Thus, by Proposition~\ref{Prp:phi_key},
after suitable elementary row and column operations, $\varphi$ can be written as
\begin{equation}
	\label{e_varphiMnCs4'}
	\varphi=\left[
	\begin{matrix} 
		1 & 0 & 0 & \cdots & 0\\
		0 & \ka_{12} & \ze_{[\de_3]} & \cdots & \ze_{[\de_m]}\\
	\end{matrix} 
	\right]\,.
\end{equation} 
(Recall $(\ka_{12}\eq 0)\!\subset\!\cV$ defines the locus where $\cD_1(z)$ and $\cD_2(z)$ are conjugate,
and $\ze_{[\de_i]}$ is the product of the node-smoothing parameters $\ze_q$ for all the nodes between $\de_i$ and $F$.)
Since
$\deg D\!\cap\!F\eq 2$,
the modular blowups up to the end of $(\rd_3\ph_2)$ do not affect $\cV$.
We thus do not distinguish $\cV$ and its pullback $\ti\cV^{\rd_3\ph_2}$ when $(\rd_3\ph_2)$ completes.

Recall $N(C)$ denotes the set of the nodes of $C$.
By Lemma \ref{Lm:K_smoothness},
$\ka_{12}$ and $\ze_q$, $q\inn N(C)$,
form a subset of a system of local parameters on $\cV$.
In order to analyze the pullback of the second row of (\ref{e_varphiMnCs4'}),
we need to study how $\ka_{12}$ and $\ze_q$'s change in ($\rd_3\ph_3$).

To this end, we introduce certain notation and terminology, which will be assumed throughout \S\ref{SecChangeOfPhi}.

\vsp
Given a stack $\fM$ and a closed substack $Z\!\subset\!\fM$,
let $\ti\fM$ be the blowup of $\fM$ along $Z$, $\pi\!:\ti\fM\!\to\!\fM$ be the projection,
and $\cE$ be the exceptional divisor.
For every $\fN\!\subset\!\fM$,
we denote its \ts{proper transform} and \ts{pullback} in $\ti\fM$ by 
$$
\wc\fN=\tn{PT}(\fN) :=\ov{\pi^{-1}(\fN\bsl Z)}\qquad\tn{and}\qquad
\ti\fN:=\pi^{-1}(\fN),
$$
respectively.
Here, $\bar{~}$ denotes the closure. 

Assume $\fM$ has an affine smooth chart $\cV\!\to\!\fM$.
In addition, assume $\ze_1,\ldots,\ze_r$ is a system of local parameters on $\cV$, and 
$$
Z\!\cap\cV=
\big\{\ze_1=\cdots=\ze_\ell=0\big\}\,.
$$
Then, by writing
\begin{equation}\label{e_pullbackz}
	\ti \ze_i:= \ze_i\circ\pi\in\Ga\big(\sO_{\pi^{-1}(\cV)}\big)
	\qquad
	\forall~i\inn\lrbr r,
\end{equation}
every closed point $\ti x\inn\cE\!\cap\!\pi^{-1}(\cV)$ is in the form
we have
$$
\ti x
=\big(\,0,\ldots,0,\ti \ze_{\ell+1},\ldots,
\ti \ze_r
\,;\,[u_1,\ldots,u_\ell]\,\big),
$$
where $[u_1,\ldots,u_\ell]$ denotes the homogeneous coordinate of $\PP^{\ell-1}$.
Let
\begin{equation}\label{e_dom'}
	\Delta_{\ti x}=
	\big\{k\inn\lrbr{\ell}:u_k\!\ne\!0\big\}
	\subset\lrbr{\ell}.
\end{equation}
Obviously, $\De_{\ti x}\!\ne\!\emptyset$.

\begin{lemm}\label{LmBlowupSimpleFact}
	With $\ti x$ as above,
	there exists an affine smooth chart $\ti\cV_{\ti x}\!\subset\!\pi^{-1}(\cV)$ containing $\ti x$ such that
	$$
	\ti \ze_i\,|\,\ti \ze_j\quad\tn{and}\quad
	\tn{PT}\{\ze_j\eq 0\}\!\cap\!\ti\cV_{\ti x}=
	\big\{{\ti \ze_j}/{\ti \ze_i}= 0\big\}
	\qquad
	\forall~i\inn\De_{\ti x},~j\inn\lrbr{\ell}.
	$$
	Moreover, if $i\inn\De_{\ti x}$ and $j\inn\lrbr{\ell}\bsl\De_{\ti x},$
	then $\ti \ze_j/\ti \ze_i$ is a local parameter on $\ti\cV_{\ti x}$,
	vanishing at $\ti x$;
	if $i,j\inn\De_{\ti x}$,
	then $\ti\ze_j/\ti \ze_i$ does not vanish at $\ti x$, and $\ti \ze_j/\ti \ze_i-\big(\big(\ti \ze_j/\ti \ze_i\big)(\ti x)\big)$ is a local parameter on $\ti\cV_{\ti x}$.
\end{lemm}

\begin{proof}
	Every point in $\ti\cV_{\ti x}$ is in the form $\big(\ti \ze_1,\ldots,\ti \ze_r;[u_1,\ldots,u_\ell]\big)$ so that 
	$$\ti \ze_i\,u_j=\ti \ze_j\,u_i\qquad\forall~i,j\in\lrbr{\ell}.$$
	Shrinking $\ti\cV_{\ti x}$ if necessary, 
	we may assume~$u_i$ does not vanish on $\ti\cV_{\ti x}$ for all $i\inn\De_{\ti x}$.
	Thus,
	$\ti \ze_j\eq(u_j/u_i)\ti \ze_i$
	for all $j\inn\lrbr{\ell}$.
	This  leads to the two statements in the display.
	The last statement follows from a direct check.
\end{proof}

Let $\ti\cV_{\ti x}$ be as in Lemma~\ref{LmBlowupSimpleFact}.
For each $i\inn\De_{\ti x}$,
the zero locus of $\ti z_i$ on $\ti\cV_{\ti x}$ is the same as $\cE\!\cap\!\ti\cV_{\ti x}$.
We select an arbitrary $i\inn\De_{\ti x}$ and set 
$$
\ve=\ti \ze_i|_{\ti\cV_{\ti x}}\in\Ga\big(\sO_{\ti\cV_{\ti x}}\big).
$$ 

\begin{coro}\label{CrlCoord}
	Let
	$\ti\cV_{\ti x}$ be the affine smooth chart in Lemma~\ref{LmBlowupSimpleFact}.
	Then,
	\begin{equation}\label{e_prz}
		\ve\,;\quad
		\tfrac{\ti \ze_j}{\ve}\!-\!
		\big(\tfrac{\ti \ze_j}{\ve}(\ti x)\big),\ j\inn\De_{\ti x}\bsl \{i\}\,;\quad
		\wc \ze_j\!:=\!\tfrac{\ti \ze_j}{\ve},~j\inn\lrbr{\ell}\bsl\De_{\ti x}\,;\quad\tn{and}\quad
		\wc\ze_k\!:=\!\ti \ze_k,~k\inn\lrbr{r}\bsl\lrbr \ell
	\end{equation}
	form a system of local parameters on $\ti\cV_{\ti x}$.
	Moreover, $\ve$ and all $\wc\ze_j$, $j\inn\lrbr\ell\bsl\De_{\ti x}$, vanish at~$\ti x$,
	whereas $\ti\ze_j/\ve$ does not vanish at $\ti x$ for any $j\inn\De_{\ti x}$.
\end{coro}

\begin{proof}
	By Lemma~\ref{LmBlowupSimpleFact},
	$\ve$ and $\wc \ze_j$, $j\inn\lrbr{\ell}\bsl\De_{\ti x}$ can be considered as local parameters.
	In addition, it follows from a direct check that
	on $\pi^{-1}(\cV)$,
	$\ti \ze_k$ is a local parameter satisfying
	\begin{equation*}
		\{\ti \ze_k\eq 0\}=
		\tn{PT}\{\ze_k\eq 0\}\qquad
		\forall\,k\inn\lrbr{r}\bsl\lrbr\ell.
	\end{equation*}	
	Therefore, every regular function in (\ref{e_prz}) is a local parameter.
	It is straightforward that these local parameters are independent, i.e.~they form a subset of a system of local parameters.
	
	To see the last statement, notice the display above (\ref{e_dom'}) shows $\ti\ze_i(\ti x)\eq 0$ for the chosen $i\inn\De_{\ti x}$,
	i.e.~$\ve(\ti x)\eq 0$.
	In addition, for every $j\inn\lrbr{\ell}$,
	the last statement follows from Lemma~\ref{LmBlowupSimpleFact}.
\end{proof}

We point out that $\ti\cV_{\ti x}$ contains other local parameters describing the positions of points in
$$
\big\{\,
\ti x'\in\cE\!\cap\!\ti\cV_{\ti x}:\,
\De_{\ti x'}\eq\De_{\ti x}\,
\big\}~\subset\cE\!\cap\!\ti\cV_{\ti x}.
$$
We  will always denote by $\ti \ze_j$, $j\inn\lrbr{r}$, the \ts{pullback} of $\ze_j$ as in~\eref{e_pullbackz},
and by $\wc \ze_j$,  $j\inn\lrbr{r}\bsl\De_{\ti x}$, the \ts{proper transform} of $\ze_j$ as in~\eref{e_prz}.

\subsection{Locally tree-compatible blowups}
\label{SubsecTree}

Once again we go back to \ref{Case:5}.\ref{Case:5.4}.
In the second row of (\ref{e_varphiMnCs4'}),
observe $\ze_{[\de_i]}$ divides $\ze_{[\de_j]}$ whenever $N_{[\de_i]}\!\subset\!N_{[\de_j]}$,
so we only need to study those $N_{[\de_i]}$ that are {\it minimal} with respect to inclusion among all $N_{[\de_j]}$, $3\!\le\!j\!\le\! m$.
Also observe that in the dual graph (in the usual sense; c.f.~\cite[\S23.4]{MirSym}) of $\wh x$,
after contracting (in the usual sense of graph theory) the edges corresponding to the core nodes, we obtain a rooted tree $\tau'$ whose root corresponds to the core of $C$.
After further contracting every edge $e'$ of $\tau'$ that has a vertex carrying at least one marked point and lying between $e'$ and the root,
the resulting rooted tree~$\tau$ has the following property:
each path of $\tau$ connecting the root and a leaf corresponds to some minimal $N_{[\de_i]}$ aforementioned,
and vice versa.
Therefore, in this subsection, we introduce certain terminology about rooted trees that are needed for the pullbacks of $\varphi$.

Recall in the usual language of the graph theory, a \ts{tree} refers to a connected graph whose first Betti number is 0, and a \ts{rooted tree} is a tree along with a selected vertex known as the \ts{root}.
On the set $E$ of the edges of a rooted tree $\tau$,
consider the partial order $\preceq$ given by $e\!\preceq\!e'$  if any path containing the root and $e$ must contain $e'$ as well.
We call $\preceq$ the \ts{tree order} of $\tau$.
Notice that any elements $e'$ and $e''$ of $E$ are comparable if and only if there exists $e\inn E$ such that $e\!\preceq\!e'$ and $e\!\preceq\!e''$ hold simultaneously.

Conversely, any poset $(E,\preceq)$ satisfying the above property, up to isomorphisms of posets, uniquely determines a rooted tree $\tau$ whose set of the edges is $E$ and the corresponding tree order is $\preceq$.
The maximal and minimal elements of $E$ are connected to the root and the \ts{leaves}
of $\tau$, respectively.

The poset description turns out to be more convenient for the sequential blowups,
hence we utilize the following definition of rooted trees throughout the paper,
which is identical to that used in \cite{HN2}.

\begin{defi}\label{Dfn:Rooted_tree}
	A \ts{rooted tree} is a finite (possibly empty) partially ordered set (or poset) $$\tau=(E,\preceq),$$ satisfying for any $e',e''\inn E$,
	\begin{align}\label{Eqn:tree_order}
		\big\lgroup
		e'\ \tn{and}\ e''\ \tn{are~comparable}\big\rgroup\quad
		\Longleftrightarrow\quad
		\big\lgroup \,
		\exists\, e\in E\ \ \tn{s.t.}\ \ e\preceq e'\ \tn{and}\ e\preceq e''\big\rgroup.
	\end{align}
	We call the elements of $E$ the \ts{edges} of $\tau$, and the order $\preceq$ the \ts{tree order} of $\tau$. 
	When the context is clear,
	we identify $\tau$ with $E$ and simply write $e\inn \tau$ to refer to an edge of $\tau$.
	
	The set $\max(\tau)$ of the maximal elements of $\tau$ is called the \ts{root} of $\tau$. 
	Each minimal element of $\tau$ is called a \ts{leaf} of $\tau$.
	The set of the leaves of $\tau$ is denoted by $\min(\tau)$.
	
	A \ts{root-to-leaf} path is a linearly ordered subset of $\tau$ that is maximal with respect to inclusion.
	
	The edge-less rooted tree $(\emptyset,-)$ is said to be \ts{trivial} and denoted by $\tau_\bullet$.
	
	The set of all the rooted trees is denoted by $\bT$.
	
	Two rooted trees are said to be \ts{isomorphic} if they are isomorphic as posets.
\end{defi}

\begin{rema}\label{Rmk:rooted_tree}
	Although we consider the rooted trees as posets instead of graphs,
	it is convenient to illustrate them in the usual sense of the graph theory;
	c.f.~Figure~\ref{Fig:rooted_tree}.
	The root as per Definition~\ref{Dfn:Rooted_tree} is exactly the set of the incident edges of the  root vertex,
	which is labeled as $o$ (if needed).
	When the context is clear, the root vertex can also be called the root for conciseness.
	
	We emphasize the trivial rooted tree $\tau_\bullet$, when visualized as a graph, is the edge-less graph that has exactly one vertex, instead of the empty graph.
\end{rema}

\begin{figure}[htp]
	\begin{center}
		\begin{tikzpicture}
			\filldraw
			(0,0) circle (1.2pt)
			(0,.5) circle (1.2pt)
			(-.3,0) circle (1.2pt)
			(-.6,0) circle (1.2pt)
			(.3,0) circle (1.2pt)
			(.9,0) circle (1.2pt)
			(.3,1) circle (1.2pt)
			(.6,.5) circle (1.2pt);
			\draw
			(0,0)--(0,.5)
			(-.3,0)--(.3,1)--(.9,0)
			(-.6,0)--(0,.5)--(.3,0);
			\draw
			(0,.75) node {\tiny{$a$}}
			(-.31,.16) node {\tiny{$c$}}
			(-.5,.25) node {\tiny{$b$}}
			(-.09,.16) node {\tiny{$d$}}
			(.3,.3) node {\tiny{$f$}}
			(.92,.25) node {\tiny{$z$}}
			(.62,.75) node {\tiny{$g$}}
			(.3,1.15) node {\tiny{$o$}};
			\draw
			(2,1) node[right] {\tiny{$\tau=\{a,b,c,d,f,g,z\}$}}
			(5,1) node[right] {\tiny{satisfying $b,c,d,f\prec a;\ \ z\prec g.$}}
			(2,.5) node[right] {\tiny{$\Xi(\tau)=\{\fE,\fF,\mathfrak G,\mathfrak H\},$}}
			(5,.5) node[right] {\tiny{where $\fE:=\{a,g\},\ \fF:=\{a,z\},\ \mathfrak G:=\{b,c,d,f,g\},\ \mathfrak H:=\{b,c,d,f,z\}$,}}
			(5,0) node[right] {\tiny{satisfying $\mathfrak H\prec \fF\prec\fE;$\ \ $\mathfrak H\prec \mathfrak G\prec\fE.$}}
			;
		\end{tikzpicture}		
	\end{center}
	\caption{A rooted tree $\tau$}\label{Fig:rooted_tree}		
\end{figure}

For every $E'\!\subset\!\tau$,
let
\begin{align}\label{Eqn:E_R}
	(E')^{\tn R}=\bigcup_{e'\in E'}\{\,e\inn\tau\,:\,e\,\tn R\, e'\,\}\qquad
	\tn{with}\quad \tn R=\,\prec,\,\preceq,\,\succ,\,\tn{and}\,\succeq.
\end{align}
For instance,
$(E')^\prec$ consists of $e\inn\tau$ satisfying $e\!\prec\!e'$ for {\it some} $e'\inn E'$ (not necessarily all $e'\inn E'$).
In Figure~\ref{Fig:rooted_tree},
for example,
$\{a,g\}^\prec\eq\{b,c,d,f,z\}$ and $\{b,z\}^\succeq\eq\{a,b,g,z\}$.

The following notion plays a key role in the local equations of the blowup centers of any phase.

\begin{defi}
	\label{Dfn:transverse_sections}
	Let $\tau\inn\bT$.
	A nonempty subset $\fE\!\subset\!\tau$ is called a \ts{transverse section} of $\tau$ if
	it meets any of the following two equivalent conditions: 
	\begin{itemize}[leftmargin=*]
		\item $\fE$ meets every root-to-leaf path of $\tau$ at exactly one edge;
		\item $\fE$ is a maximal subset of incomparable edges of $\tau$,
		i.e.~every pair of distinct edges of $\fE$ are incomparable, and every edge of $\tau$ is comparable with an edge of  $\fE$.
	\end{itemize}
	The set of all the transverse sections of $\tau$ is denoted by $\Xi(\tau)$.
\end{defi}

The set $\Xi(\tau)$ is always nonempty as long as $\tau\!\ne\!\tau_\bullet$,
because $\max(\tau)\inn\Xi(\tau)$ and $\min(\tau)\inn\Xi(\tau)$ always hold. 

The tree order on $\tau$ induces a partial order, still denoted by $\preceq$,
on $\Xi(\tau)$ such that for every $\fE,\fE'\inn\Xi(\tau)$,
\begin{align}\label{Eqn:transverse_sections_order}
	\lgroup\,\fE\prec \fE'\,\rgroup\quad\Longleftrightarrow\quad
	\lgroup\, \fE\ne \fE'\,\rgroup\ \ \tn{and}\ \ 
	\lgroup\,\fE\subset(\fE')^\preceq\,\rgroup\,.
\end{align}
If $\tau\!\ne\!\tau_\bullet$, then
the greatest and least element of $\Xi(\tau)$ are respectively $\max(\tau)$ and $\min(\tau)$.

For example,
in Figure~\ref{Fig:rooted_tree},
we have
\begin{align*}
	&\Xi(\tau)=\big\{\fE,\fF,\mathfrak G,
	\mathfrak H\big\},\qquad
	\tn{where}\\
	&\fE:=\{a,g\},\ \ 
	\fF:=\{a,z\},\ \ 
	\mathfrak G:=\{b,c,d,f,g\},\ \ 
	\mathfrak H:=\{b,c,d,f,z\}.
\end{align*}
The partial order (\ref{Eqn:transverse_sections_order}) is given by $\mathfrak H\!\prec\! \fF\!\prec\!\fE$ and $\mathfrak H\!\prec\! \mathfrak G\!\prec\!\fE$ in this case.

Let $\tau$ be a rooted tree and $\cV$ be an affine smooth chart of a stack $\fM$.
If there exists a system of local parameters on $\cV$ as follows:
\begin{equation}\label{e_ze}
	\big\{\,
	\ze_e\inn\Ga\big(\sO_{{\cV}}\big):\,
	e\inn\tau\,
	\big\}\,\sqcup \,
	\big\{\,\ze_s\inn\Ga\big(\sO_{{\cV}}\big):\,s\inn S'\,
	\big\}
	\,,
\end{equation}
then $\big\{\ze_e\inn\Ga\big(\sO_{{\cV}}\big):
e\inn\tau\big\}$ is called a \ts{$\tau$-labeled subset of local parameters} on $\cV$.
In this case,
we set
\begin{equation}\label{e_z[v]}
	\ze_{[e]}:=
	\prod_{e'\succeq e}\ze_{e'}\qquad\forall~e\in\tau.
\end{equation}
Every  $E\!\subset\!\tau$ determines a closed locus $X_E^\cV$ in $\cV$ by
$$
X_E^\cV=\{\,
\ze_e\eq 0:\,e\inn E
\,\}\quad\subset\cV.
$$
When the context is clear, we may simply write it as $X_E$.

With notation as above,
let
\begin{equation}
	\label{e_blowup}
	\pi:\ti\fM\lra \fM
	\qquad
	\tn{and}\qquad
	\pi_{(k)}:\ti\fM_{(k)}\lra\fM
\end{equation}
be the blowup of~$\fM$ successively along the proper transforms of a sequence of closed substacks 
$$
Z_1,Z_2,\cdots\,
\subset
\fM
$$
and the same blowup of $\fM$ after the $k$-th step, respectively.
We may call $\pi$  a sequential blowup to emphasize it is comprised of a sequence of blowups.

\begin{defi}\label{DfnGaAdm}
	Let $\tau\inn\bT$ and $\pi:\ti\fM\!\to\!\fM$ be as in~\eref{e_blowup}.
	Then, $\pi$ is said to be \ts{$\tau$-compatible on~$\cV$} if
	there exists a partition of $\Xi(\tau)$:
	$$
	\Xi(\tau)
	=\bigsqcup_{k\ge 1}
	\Xi_k(\tau)\,,
	$$
	and a $\tau$-labeled subset of local parameters on $\cV$ such that
	\begin{enumerate}[label=(C\arabic*),leftmargin=*]
		\item \label{ConditionLocalLoci} for every $k\!\ge\!1$,
		$$
		Z_k\cap\cV=\bigcup_{\fE\in\Xi_k(\tau)}\!\!\!X_{\fE}^\cV\;;
		$$
		\item \label{ConditionOrder} 
		if $\fE'\inn \Xi_{k'}(\tau)$, $\fE''\inn \Xi_{k''}(\tau)$, and $\fE'\!\succ\!\fE''$,
		then $k'\!<\!k''$.
	\end{enumerate}
\end{defi}

As $\Xi(\tau)$ is a finite set,
there only exist {\it finitely} many $k\inn\mathbb Z_{\ge 0}$ with $Z_k\!\cap\!\cV\!\ne\!\emptyset$.

Locally, the blowup centers are naturally given by the unions of some $X_\fE^\cV$ as in \ref{ConditionLocalLoci}. Example~\ref{Eg:G1} below will demonstrate this point.

The condition~\ref{ConditionOrder} has three immediate conclusions.
First, a $\tau$-compatible  blowup of $\cV$ must begin with $X_{\max(\tau)}$, unless $\tau$ is trivial (i.e.~edge-less).
Second, in each $\Xi_k(\tau)$, distinct elements are always not comparable.
The last is the following lemma that shows every $\tau$-compatible blowup on $\cV$ is always a {\it smooth} blowup of $\cV$,
although certain blowup centers $Z_k\!\cap\!\cV$ may be singular a priori (in fact, $Z_k\!\cap\!\cV$ is singular whenever $|\Xi_k(\tau)|\!\ge\!2$).
This is a crucial observation that guarantees the final stack $\ti\fM_2^{\rm div}$ proposed in Theorem 1 is a smooth stack.

\begin{lemm}
	\label{LmLocalLoci}
	Let $\pi\!:\ti\fM\!\to\!\fM$ be $\tau$-compatible on~$\cV$ as in Definition~\ref{DfnGaAdm}.
	Then for each $k\!\ge\!1$,
	the proper transforms of the irreducible components $X_\fE^\cV$ of  $Z_{k}\!\cap\!\cV$ after the $(k\!-\!1)$-th step are all smooth and pairwise disjoint.
	Therefore, the proper transform of $Z_k\!\cap\cV$ after the $(k\!-\!1)$-th step is smooth, hence the pullback of $\cV$ after each step is smooth.
\end{lemm}

\begin{proof}
	For distinct $\fE, \fE'\inn\Xi_{k}(\tau)$,
	$\fE$ and $\fE'$ are not comparable by~\ref{ConditionOrder}.
	Let $$\fF=\max(\fE\cup\fE'),$$ 
	which is nonempty because $\fE$ and $\fE'$ are nonempty as transverse sections.
	Distinct elements of~$\fF$ are incomparable because they are maximal elements of $\fE\!\cup\!\fE'$.
	Finally,
	for every $e'\inn\tau$,
	there exist $e\inn\fE$  comparable with $e'$,
	as well as $e''\inn\max(\fE\!\cup\!\fE')$ satisfying $e\!\preceq\!e''$ (because $\fE\!\cup\!\fE'$ is finite).
	Therefore,
	if $e'\!\preceq\!e$, then $e'\!\preceq\!e''$;
	if $e'\!\succ\!e$,
	then 
	$e'$ and $e''$ are still comparable because of (\ref{Eqn:tree_order}).
	In sum, $\fF$ satisfies the latter criterion (hence both criteria) of Definition~\ref{Dfn:transverse_sections},
	hence $\fF\inn\Xi(\tau)$.
	Indeed, $\fF$ is the least common successor of $\fE$ and $\fE'$ with respect to the order~\eref{Eqn:transverse_sections_order}.
	
	Since $\fF\!\subset\!\fE\!\cup\!\fE'$,
	we have
	$$ X_{\fE}\cap X_{\fE'}=X_{\fE\cup\fE'}\subset X_\fF.
	$$
	(Here, we drop the superscript $\cV$ of $X_\fE$'s for conciseness.)
	The restriction of the normal bundle of $X_\fF$ in $\cV$ to $X_\fE\!\cap\!X_{\fE'}$ can be written as
	$$
	\Big(\!
	\bigoplus_{e\in\fF\cap\fE\cap\fE'}\!\!\!\!\!\!\!\big(\sO_\cV(X_{e})\big|_{X_\fE\cap X_{\fE'}}\big)\!
	\Big)
	\oplus
	\Big(\!
	\bigoplus_{e\in\fF\cap(\fE\bsl\fE')}\!\!\!\!\!\!\!\big(\sO_\cV(X_{e})\big|_{X_\fE\cap X_{\fE'}}\big)\!
	\Big)
	\oplus
	\Big(\!
	\bigoplus_{e\in\fF\cap(\fE'\bsl\fE)}\!\!\!\!\!\!\!\!\big(\sO_\cV(X_{e})\big|_{X_\fE\cap X_{\fE'}}\big)\!
	\Big),
	$$
	where the leftmost direct sum is in the normal direction of both $X_{\fE}$ and $X_{\fE'}$, while
	the middle (resp.~rightmost) sum is in the normal direction of $X_{\fE}$ (resp.~$X_{\fE'}$) and tangent direction of $X_{\fE'}$ (resp.~$X_{\fE}$).
	
	Let $h\!\ge\!1$ be such that $X_\fF\!\subset\! Z_{h}$.
	By~\ref{ConditionOrder},  $h\!<\!k$.
	The last paragraph implies the proper transforms of $X_\fE$ and $X_{\fE'}$ becomes disjoint after $X_\fF$ is blown up in the $h$-th step.
	The smoothness of the proper transforms of $X_\fE$ and $X_{\fE'}$ after step $h$ follows from direct computation;
	see~\cite[Lemma~2.3]{VZ08} for more details.	
	
	Since we have shown the blowup center in each step is smooth, the last statement of Lemma~\ref{LmLocalLoci} holds immediately.
\end{proof}

For each $k\!\ge\!$ and $\fE\inn\Xi_k(\tau)$,
the exceptional divisor obtained by blowing up $\wc X_\fE^\cV$ is written as
\begin{equation}\label{e_excdiv}
	\cE_\fE\subset\pi_{(k)}^{-1}(\cV).
\end{equation}
By Lemma~\ref{LmLocalLoci},
these $\cE_\fE$, $\fE\inn\Xi_k(\tau)$, are disjoint.
Moreover, by a direct check, we see each $\cE_\fE$ is transverse to the proper transforms of $X_{\fE'}^\cV$ for all $\fE'\inn\Xi_{k'}(\tau)$, $k'\!>\!k$.
Thus, the proper transform of $\cE_E$ in $\pi^{-1}(\cV)$ is the same as its pullback, i.e.
\begin{equation}
	\label{e_wcEtiE}
	\wc\cE_E=\ti\cE_E\ \ \big(\subset\pi^{-1}(\cV)\big)\qquad
	\forall~\fE\inn\Xi(\tau).
\end{equation}
Finally, we comment that the notions of $\cE_\fE$ and $\wc\cE_\fE$ are intrinsic to $\cV$ and $\tau$. They are independent of the choice of the $\tau$-compatible blowup on $\cV$.

\begin{exam}\label{Eg:G1}
	An example of locally tree-compatible blowups is the the blowup $\ti\fM_1^{\tn{wt}}$
	of the Artin stack $\fM_1^{\tn{wt}}$ in~\cite{HL10},
	which is performed successively along the proper transforms of the closed loci
	$\Theta_k$, $k\eq 1,2,\cdots$.
	Here, a general point of each $\Th_k$ consists of a weight-0 genus 1 core and $k$ positively weighted tails. 
	In other words, $\Th_k$'s are the genus 1 counterparts of  $\ov\fM_{(1,k)}$'s in Figure~\ref{figBlowup1}.
	
	For every $x\eq(C,\mathbf{ w})\inn\fM_1^\tn{wt}$,
	let $\tau$ be its terminally weighted tree in the sense of~\cite[Paragraph 3.16]{HL10}.
	Then, there exists an affine smooth chart $\cV\!\to\!\fM_1^{\rm wt}$ containing $x$ so that the set of node-smoothing parameters
	$\{\ze_e\}_{e\in\tau}$ is a $\tau$-labeled subset of local parameters on $\cV$.
	Each $\Th_k$ satisfies
	$$
	\Th_k\cap \cV=\!
	\bigcup_{
		\begin{subarray}{c}
			\fE\in\Xi_k(\tau)
		\end{subarray}
	}\!\!\!
	X_\fE^\cV\,,\qquad
	\tn{where}\quad\Xi_k(\tau)=\{\fE\inn\Xi(\tau):|\fE|\eq k\}\,.
	$$
	The stability of $\fM_g^{\rm wt}$ implies for every $\fE,\fE'\inn\Xi(\tau)$ with $\fE\!\succ\!\fE'$,
	we have $|\fE|\!<\!|\fE'|$.
	It is thus a direct check that the blowup successively along the proper transforms of $\Th_1,\Th_2,\cdots$ is $\tau$-compatible on $\cV$,
	hence the sequential blowup in~\cite{HL10} is locally tree-compatible.
	
	We point out it is due to the stability of $\fM_g^{\rm wt}$  that the above partition of $\Xi(\tau)$ by the cardinals of the transverse sections actually leads to a locally $\tau$-compatible blowup on $\cV$.
	Generally, given an arbitrary rooted tree $\tau'$,
	there may exist $\fE,\fE'\inn\Xi(\tau')$ satisfying $\fE\!\succ\!\fE'$ yet $|\fE|\eq|\fE'|$;
	see the leftmost diagram of Figure~\ref{figTree}.
	
	\begin{figure}[htp]
		\begin{center}
		\begin{tikzpicture}
			\draw (0,0)--(1.5,1.5);
			\draw (0,0)--(-.75,.75);
			\draw (-.75,.75)--(-1.5,1.5);
			\filldraw (0,0) circle (2pt)
			(.75,.75) circle (2pt)
			(-.75,.75) circle (2pt)
			(1.5,1.5) circle (2pt)
			(-1.5,1.5) circle (2pt)
			;
			\draw (0,0) node[right] {$o$}
			(.3,.3) node[right] {$e_b$}
			(-.25,.3) node[left] {$e_a$}
			(1.05,1.05) node[right] {$e_d$}
			(-1,1.05) node[left] {$e_f$}
			;
			
			\draw[xshift=5cm] (0,0)--(1.5,1.5);
			\draw[xshift=5cm] (0,0)--(-1.5,1.5);
			\draw[xshift=5cm] (.75,.75)--(.5,1.5);
			\draw[xshift=5cm] (-.75,.75)--(-.5,1.5);
			\filldraw[xshift=5cm]
			(0,0) circle (2pt)
			(.75,.75) circle (2pt)
			(-.75,.75) circle (2pt)
			(1.5,1.5) circle (2pt)
			(.5,1.5) circle (2pt)
			(-.5,1.5) circle (2pt)
			(-1.5,1.5) circle (2pt);
			\draw[xshift=5cm]
			(0,0) node[right] {$o$}
			(.3,.3) node[right] {$e_b$}
			(-.25,.3) node[left] {$e_a$}
			(1.1,1.1) node[right] {$e_d$}
			(.75,1.1) node[left] {$e_c$}
			(-.7,1.1) node[right] {$e_g$}
			(-1.05,1.1) node[left] {$e_f$}
			;	
		\end{tikzpicture}	
		\end{center}
		\caption{Rooted trees}\label{figTree}
	\end{figure}
	For a concrete example,
	consider the tree~$\tau$ illustrated in the rightmost diagram of Figure~\ref{figTree}.
	The sequential blowup of $\fM_1^{\tn{wt}}$ determines a partition
	\begin{equation}\label{Eqn:Xi_G1}
		\Xi(\tau)=
		\big\{\fE_2\!=\!\{e_a,e_b\}\big\}\sqcup
		\big\{\fE_{3}\!=\!\{e_a,e_c,e_d\},\,\fE_3'\!=\!\{e_b,e_f,e_g\}\big\}\sqcup
		\big\{\fE_4\!=\!\{e_f,e_g,e_c,e_d\}\big\}	
	\end{equation}
	such that
	\begin{gather*}
		\Th_2\!\cap\!\cV
		=X_{\fE_2}^\cV,\quad
		\Th_3\!\cap\!\cV
		= X_{\fE_3}^\cV\!\cup\!X_{\fE_3'}^\cV,\quad
		\Th_4\!\cap\!\cV
		= X_{\fE_4}^\cV,\quad
		\fE_2\succ\fE_3,\fE_3'\succ\fE_4,
	\end{gather*}
	thus~\ref{ConditionLocalLoci} and~\ref{ConditionOrder} in Definition~\ref{DfnGaAdm} are satisfied.
	Notice that $\Th_3\!\cap\!\cV\eq X_{\fE_3}^\cV\!\cup\!X_{\fE_3'}^\cV$ has normal crossing singularities at $X_{\fE_3}^\cV\!\cap\!X_{\fE_3'}^\cV$ $\big(=\! X_\tau^\cV\big)$.
	Nonetheless,
	after blowing up along $X_{\fE_2}^\cV$,
	the proper transforms of $X_{\fE_3}^\cV$ and $X_{\fE_3'}^\cV$ are smooth and meet the exceptional divisor
	$\cE_{\fE_2}$ at the loci given by $\{[u_a,u_b]\eq[0,1]\}$ and $\{[u_a,u_b]\eq[1,0]\}$, respectively,
	hence they become disjoint.
	Here, $[u_a,u_b]$ denotes the homogeneous coordinate on $\P^1$ satisfying $\ti\ze_{e_a}u_b\eq\ti\ze_{e_b}u_a$.
	
	As this example reveals, given a blowup center $Z_k$ and a chart $\cV$, $Z_k\!\cap\!\cV$ is possibly a union of $X_\fE^\cV$ for some $\fE\inn\Xi(\tau)$. This explains \ref{ConditionLocalLoci} of Definition~\ref{DfnGaAdm}. 
	Nonetheless, this example also demonstrates
	that the proper transform of each blowup center $Z_k$ becomes smooth  after the $(k\!-\!1)$-th step,
	which is proved in Lemma~\ref{LmLocalLoci}.
\end{exam}

\begin{exam}
	In Definition~\ref{DfnGaAdm}, \ref{ConditionLocalLoci} implies
	the partition of $\Xi(\tau)$ is solely determined by the sequence of the blowup centers $Z_k$, $k\!\ge\!1$.
	As suggested by the following example, however, the same blowup $\ti\fM/\fM$ may be achieved by choosing a different sequence of blowup centers,
	which may yield a different partition of $\Xi(\tau)$.
	
	Let $\cV\!\to\!\fM_1^{\rm wt}$ and $\tau$ be as in Example~\ref{Eg:G1}.
	In this example,
	we take the base stack $\fM$ of  Definition~\ref{DfnGaAdm} to be $\cV$ instead of $\fM_1^{\rm wt}$, and consider $\ti\cV/\cV$, the restriction of the blowup $\ti\fM_1^{\tn{wt}}/\fM_1^{\tn{wt}}$ to $\cV$.
	In Example~\ref{Eg:G1}, we already know
	$\ti\cV/\cV$ can be achieved by blowing up $\cV$ along the proper transforms of $\Th_2\!\cap\!\cV$,
	$\Th_3\!\cap\!\cV$, and $\Th_4\!\cap\!\cV$ successively,
	which gives the partition (\ref{Eqn:Xi_G1}) of $\Xi(\tau)$.
	On the other hand, notice
	$\ti\cV/\cV$ can also be achieved by blowing up $\cV$ successively along the proper transforms of $X_{\fE_2}^\cV$,
	$X_{\fE_3}^\cV$, $X_{\fE_3'}^\cV$, and $X_{\fE_4}^\cV$,
	where the four transverse sections are the same as in (\ref{Eqn:Xi_G1}).
	This also determines a partition of $\Xi(\tau)$ given by 
	\begin{align*}
		\Xi(\tau)=\{\fE_2\}\sqcup\{\fE_3\}\sqcup\{\fE_3'\}\sqcup \{\fE_4\},
	\end{align*}
	which is different from (\ref{Eqn:Xi_G1}).
	
	It is straightforward that the above two sequential blowups are both $\tau$-compatible on $\cV$ and indeed yield the same $\ti\cV/\cV$.
	As $\fE_3$ and $\fE_3'$ are not comparable, this example particularly shows	the converse of \ref{ConditionOrder} need not to be true.
\end{exam}

At the end of this subsection,
we describe a sequential blowup induced from a $\tau$-compatible sequential blowup on $\cV$.
This is related to the modular blowups in $(\rd_3)$ and is needed for proving many cases of Proposition~\ref{PrpChangeofPhi}, including \ref{Case:5}.\ref{Case:5.4}.

\vsp
For every rooted tree $\tau\eq(E,\preceq)$, 
let $\ex(\tau)$ be the rooted tree
\begin{gather*}
	\ex(\tau)=\big(\,E\sqcup\{e_\ex\}\,,\,\preceq\,\big),
\end{gather*}
where $\preceq$ is the relation on $\ex(\tau)$ given by the tree order on $\tau$ as well as the following:
every $e\inn E$ is {\it not} comparable with $e_\ex$.
It is a direct check that $\preceq$ is a tree order as per Definition~\ref{Dfn:Rooted_tree},
and $e_\ex$ is both maximal and minimal.
The rooted tree $\ex(\tau)$ is uniquely determined by $\tau$ (up to an isomorphism of rooted trees).

Intuitively,
$\ex(\tau)$ is obtained from $\tau$ by \ts{grafting} a new vertex to the root via a new  edge $e_\ex$.
An example is illustrated in Figure~\ref{Fig:ga_+}.
	
\begin{figure}[htp]
	\begin{center}
		\begin{tikzpicture}
			\draw 
			(0,0)--(1.5,1.5)
			(.75,.75)--(0,1.5);
			\filldraw 
			(0,0) circle (2pt)
			(.75,.75) circle (2pt)
			(1.5,1.5) circle (2pt)
			(0,1.5) circle (2pt)
			;
			\draw 
			(0,0) node[left] {$o$}
			(.3,.3) node[right] {$e_b$}
			(1.05,1.05) node[right] {$e_d$}
			(.5,1.05) node[left] {$e_c$}
			;
			
			\draw[xshift=5cm]
			(0,0)--(1.5,1.5)
			(0,0)--(-.75,.75)
			(.75,.75)--(0,1.5);
			\filldraw [xshift=5cm]
			(0,0) circle (2pt)
			(.75,.75) circle (2pt)
			(-.75,.75) circle (2pt)
			(1.5,1.5) circle (2pt)
			(0,1.5) circle (2pt)
			;
			\draw [xshift=5cm]
			(0,0) node[left] {$o$}
			(.3,.3) node[right] {$e_b$}
			(-.3,.3) node[left] {$e_\ex$}
			(1.05,1.05) node[right] {$e_d$}
			(.5,1.05) node[left] {$e_c$}
			;
			\draw 
			(0,-.3) node[below] {$\tau$}
			(5,-.3) node[below] {$\ex(\tau)$}
			;
		\end{tikzpicture}
	\end{center}
	\caption{An example of grafting}\label{Fig:ga_+}
\end{figure}

Assume $\pi\!:\ti\fM\!\to\!\fM$ as in (\ref{e_blowup}) is $\tau$-compatible on $\cV$.
A blowup $$\pi_\ex:\ti\fM_\ex\lra\fM$$
successively along the proper transforms of closed substacks
$Z_{\ex,k},$ $k\eq 1,2,\cdots$, of $\fM$ is called a \ts{grafted blowup} with respect to $\pi$ on $\cV$ if there exists a local parameter $\ze_{\ex}$
so that $\{\ze_e\!:e\inn\tau\}\!\sqcup\!\{\ze_{\ex}\}$ is a subset of a system of local parameters on $\cV$, and
$$
Z_{\ex,k}\cap\cV
=
\bigcup_{\fE\in\Xi_k(\tau)}
\!\!\!X_{\fE\sqcup\{e_\ex\}}^\cV
\qquad
\forall~k\ge 1.
$$

Observe that for each $\fE\inn\Xi(\tau)$, we have
$\fE\!\sqcup\!\{e_\ex\}\in\Xi(\ex(\tau))$;
conversely, every $\fE'\inn\Xi(\ex(\tau))$ must contain $e_\ex$, and $\fE'\bsl\{e_\ex\}\in\Xi(\tau)$.
This leads to the following statement immediately.

\begin{lemm}\label{LmSupplementary}
	Let $\pi\!:\ti\fM\!\to\!\fM$ and $\pi_\ex\!:\ti\fM_\ex\!\to\!\fM$ be two sequential blowups,
	and $\cV\!\to\!\fM$ be a small affine smooth chart.
	If $\pi$ is $\tau$-compatible on~$\cV$, and $\pi_\ex$ is a grafted blowup with respect to $\pi$ on~$\cV$,
	then $\pi_\ex$ is $\ex(\tau)$-compatible on $\cV$.
\end{lemm}

\subsection{Root-to-leaf sequences, dominant edges, and the proof of \ref{Case:5}.\ref{Case:5.4}}
\label{SubsecDom}
We continue with the notation from \S\ref{SubsecTree}.
Let $\pi\!:\ti\fM\!\to\!\fM$ be the blowup in (\ref{e_blowup}) and $\tau\inn\mathbf T$.
Assume $\pi$ is $\tau$-compatible on~$\cV$ as per Definition~\ref{DfnGaAdm}.
For every $\fE\inn\Xi(\tau)$, we still use $
\wc\cE_\fE$ to denote 
the proper transform of the exceptional divisor $\cE_\fE$ in~\eref{e_excdiv},
which by~\eref{e_wcEtiE}, is the same as the pullback of $\cE_\fE$.

In this subsection, we fix $\wh x\inn X_\tau^\cV\!\subset\!\cV$ (i.e.~$\ze_e(\wh x)\eq 0$ for all $e\inn\tau$) as well as a lift $\ti x$ of $\wh x$ in $\ti\cV\eq\pi^{-1}(\cV).$
If $\tau\eq\tau_\bullet$, which implies $\Xi(\tau)\eq\emptyset$, then the blowup centers do not meet $\cV$,
so the blowup of $\cV$ is equal to $\cV$.

If $\tau\!\ne\!\tau_\bullet$, 
then 
\begin{align}
	\label{Eqn:RL_seq}
	\ov \bE:=\big\{\,\fE\inn\Xi(\tau):\,\ti x\inn\wc\cE_\fE\,\big\}
\end{align}
is nonempty.

\begin{lemm}
	\label{Lm:RL_union}
	Assume $\tau\!\ne\!\tau_\bullet$.
	Then, we can order the elements of $\ov\bE$ so that $\ov\bE\eq\{\fE_1,\ldots,\fE_\cht\}$, satisfying
	\begin{equation}\label{e_ascendingE}
		\fE_1\succ\cdots\succ \fE_\cht
		\qquad
		\tn{and}\qquad
		\bigcup_{i\in\lrbr k}\! \fE_i=\fE_k^\succeq\quad\forall~k\in\lrbr{\cht}.
	\end{equation} 
\end{lemm}

Hereafter, we call $\ov \bE$ the \ts{root-to-leaf sequence (RLS)} of $\ti x$, 
and treat $\ov\bE$ both as a subset of $\Xi(\tau)$ and as a finite sequence of the elements of $\Xi(\tau)$, should no ambiguity occur.
The notion of RLS can be generalized to the $\tau\eq\tau_\bullet$ case,
in which 
we set $\cht\eq 0$ and call $\emptyset$ the RLS of $\ti x$.

\begin{proof}[Proof of Lemma~\ref{Lm:RL_union}]
	We will begin with an inductive construction of a sequence $\fE_1,\ldots,\fE_\cht$ satisfying $\ti x\inn\bigcap_{i\in\lrbr\cht}\wc\cE_{\fE_i}$ as well as (\ref{e_ascendingE}).
	We will then prove $\ti x\!\notin\!\wc\cE_\fE$ for any other $\fE\inn\Xi(\tau)$.
	
	First, notice any $\tau$-compatible blowup on $\cV$ must begin with the blowup $\cV$ along $X_{\max(\tau)}^\cV$.
	More precisely,
	in Definition~\ref{DfnGaAdm},
	by setting $k_1:=\min\{k\!\ge\!1:\Xi_k(\tau)\!\ne\!\emptyset\}$,
	we see $\Xi_{k_1}(\tau)$ contains the root $\max(\tau)$ as its unique element,
	because $\max(\tau)$ is the greatest element of $\Xi(\tau)$.
	By writing $$\fE_1=\max(\tau),$$
	we have $\fE_1\eq\fE_1^\succeq$.
	In addition, since $x\inn X_\tau^\cV\!\subset\!X_{\fE_1}^\cV$,
	we have $\ti x\inn\ti\cE_{\fE_1}$, which is equivalent to $\ti x\inn\wc\cE_{\fE_1}$
	by (\ref{e_wcEtiE}).
	
	Next, assume we have constructed $\fE_1\!\succ\!\cdots\!\succ\!\fE_k$ satisfying $\ti x\inn\bigcap_{i\in\lrbr{k}}\!\wc\cE_{\fE_i}$ and $\bigcup_{i\in\lrbr{k}}\fE_i\eq\fE_{k}^\succeq$ for some $k\!\ge\!1$.
	We denote by $h_k$ the unique step of the sequential blowup such that
	$
	\fE_k\inn\Xi_{h_k}(\tau).$
	Recall $\pi_{(h_k)}\!:\ti\fM_{(h_k)}\!\to\!\fM$ denotes the blowup after the $h_k$-th step.
	Let $\ti x_{k}$ be the image of $\ti x$ in $\ti\fM_{(h_k)}$.
	By the inductive hypothesis,
	we have $\ti x_{k}\inn\cE_{\fE_k}$.
	We thus define $$\De_{\ti x,k}\subset \fE_k$$ analogously to~\eref{e_dom'},
	with $\ti x$ and $\lrbr{\ell}$ in~\eref{e_dom'} replaced by $\ti x_{k}$ and $\fE_k$, respectively.
	\begin{itemize}[leftmargin=*]
		\item 
		If $\De_{\ti x,k}\!\cap\!\min(\tau)\!\ne\!\emptyset$,
		we simply set $\cht\eq k$ and $\ov\bE\eq\{\fE_1,\ldots,\fE_k\}$.
		
		\item 
		If $\De_{\ti x,k}\!\cap\!\min(\tau)\!=\!\emptyset$,
		then for every $e\inn \De_{\ti x,k}$,
		the set $\{e\}^\prec$ $\big(=\!\{e'\inn\tau:e'\!\prec\! e\}\big)$ is nonempty.
		Let
		\begin{align}\label{Eqn:E_k+1}
			\fE_{k+1}:=
			\big(\fE_k\bsl \De_{\ti x,k}\big)\sqcup\!\bigsqcup_{e\in \De_{\ti x,k}}\!\!\!\max\big(\{e\}^\prec\big).
		\end{align}
		We claim this is the  $(k\!+\!1)$-th term of the proposed sequence $\fE_1,\ldots,\fE_\cht$.
		Indeed, by direct check, we have
		\begin{align*}
			\fE_{k+1}\in\Xi(\tau),\qquad
			\fE_{k+1}\prec\fE_k,\qquad\tn{and}\qquad
			\bigcup_{i\in\lrbr{k+1}}\!\!\!\fE_i
			=\fE_k^\succeq\!\cup\!\fE_{k+1}=\fE_{k+1}^\succeq.
		\end{align*}		
		To see $\ti x\inn\ti\cE_{\fE_{k+1}}$,
		notice on a small neighborhood $\ti\cV_{\ti x_k}$ of  $\ti x_{k}$,
		the proper transform $\wc X_{\fE_{k+1};(h_k)}$ of $X_{\fE_{k+1}}$ is given by
		\begin{align*}
			\wc X_{\fE_{k+1};(h_k)}\cap\ti\cV_{\ti x_k}=
			\big\{\,\wc\ze_{k;e}=0:\,
			e\inn\fE_{k+1}\,\big\},
		\end{align*}
		where the local parameters $\wc\ze_{k;e}$ are analogous to (\ref{e_prz}).
		By the last statement of  Corollary~\ref{CrlCoord},
		we have $\wc\ze_{k;e}$ for all $e\inn\fE_k\bsl\De_{\ti x,k}$;
		in addition,
		for every $e\inn\De_{\ti x,k}$ and $e'\inn\max(\{e\}^\prec)$,
		we have $\wc\ze_{k;e'}\eq\ti \ze_{k;e'}\eq 0$,
		where $\ti\ze_{k;e'}$ is the pullback of $\ze_{e'}$ to $\ti \cV_{\ti x_k}$.
		In sum, we obtain $\ti x_k\inn\wc X_{\fE_{k+1};(h_k)}$,
		which implies $\ti x\inn \ti\cE_{\fE_{k+1}}$.
	\end{itemize}
	
	It remains to show for every  $\fE\inn\Xi(\tau)\bsl\ov\bE$, we have $\ti x\!\notin\!\wc\cE_\fE$.
	To see this,
	we claim there exists $j\inn\lrbr{\cht}$ satisfying
	\begin{align}\label{Eqn:fE_j}
		\fE_j\succeq\fE\qquad\tn{and}\qquad
		\De_{\ti x,j}\cap\fE\ne\emptyset.
	\end{align}
	Indeed, the set $\{\fE'\inn\ov\bE:\fE'\!\succeq\!\fE\}$ is nonempty because $\fE_1\eq\max(\tau)\inn\ov\bE$;
	we take $\fE_j$ to be the least element (w.r.t.~the order (\ref{e_ascendingE})) of this set.
	Then, $\fE_j\!\succeq\!\fE$ but $\fE_{j+1}\!\not\succeq\!\fE$. 
	\begin{itemize}[leftmargin=*]
		\item If $j\!<\!\cht$, then $\De_{\ti x,j}\eq \fE_{j}\bsl\fE_{j+1}$.
		So if $\De_{\ti x,j}\!\cap\!\fE\!=\!\emptyset$,
		then by (\ref{Eqn:E_k+1}),
		we would have $\fE_{j+1}\!\succeq\!\fE$,
		contradicting the fact $\fE_{j+1}\!\not\succeq\!\fE$.
		Therefore, $\De_{\ti x,j}\!\cap\!\fE\!\ne\!\emptyset$.
		
		\item If $j\!=\!\cht$, notice the inductive construction of $\fE_1,\ldots,\fE_\cht$ implies $\De_{\ti x,\cht}\!\cap\!\min(\tau)\!\ne\!\emptyset$.
		For any $e'\inn\De_{\ti x;\cht}\!\cap\!\min(\tau)$,
		we have $e'\inn\fE$ because $e'$ is minimal and $\fE\!\preceq\!\fE_j$.
		Therefore, $\De_{\ti x,j}\!\cap\!\fE\!\ne\!\emptyset$.
	\end{itemize}
	
	Now that (\ref{Eqn:fE_j}) is established,
	consider the aforementioned lift $\ti x_j$ of $x$ and its neighborhood $\ti\cV_{\ti x_j}$. 
	On $\ti\cV_{\ti x_j}$,
	the pullback of $X_{\fE}^\cV$ takes the form
	\begin{align}\label{Eqn:X_pullback}
		\big\{\,
		\ve_j\eq 0;~
		\wc\ze_{j;e}\eq 0~\forall~e\inn \fE\bsl \De_{\ti x,j}
		\,\big\}
		\quad
		\subset\;\{\ve_j\eq 0\}
		=\cE_{\fE_j}\cap\ti\cV_{\ti x_j},
	\end{align}
	where the local parameters $\ve_j$ and $\wc\ze_{j;e}$ are analogous to (\ref{e_prz}).	
	Therefore,
	the proper transform of $X_{\fE}^\cV$ is empty on $\ti\cV_{\ti x_j}$,
	hence does not contain $\ti x_j$, the image of $\ti x$ on $\ti\fM_{(h_j)}$.
	Therefore, $\ti x\!\notin\!\wc\cE_\fE$.
\end{proof}

The inductive construction of  $\fE_1,\ldots,\fE_\cht$ in the proof of Lemma~\ref{Lm:RL_union} particularly implies
$$
\De_{\ti x,\cht}\cap\min(\tau)\ne\emptyset
\qquad
\big(\Longrightarrow \fE_
\cht\cap\min(\tau)\ne\emptyset\big)\,.$$

\begin{defi}\label{DfnDominant}
	The edges in $$
	\De_{\ti x}:=\bigsqcup_{i\in\lrbr \cht}\!
	\De_{\ti x,i}\qquad\tn{and}\qquad
	\De^{\min}_{\ti x}:=
	\De_{\ti x}\cap\min(\tau)
	\ \big(=\De_{\ti x,\cht}\cap\min(\tau)\big)$$ are respectively called the \ts{dominant edges} and \ts{dominant leaves} of $\ti x$.
\end{defi}

Here, recall the leaves of a rooted tree refer to the minimal edges as per Definition~\ref{Dfn:Rooted_tree}.

For every $k\inn\lrbr{\cht}$,
the exceptional divisor $\cE_{\fE_k}$  is the zero locus of a local parameter $\ve_k$ on an affine smooth chart $\ti\cV_{\ti x_k}$ containing $\ti x_{k}$.
The proper transform 
$\wc\cE_{\fE_k}$ is thus the zero locus of a local parameter $\wc\ve_k$ on an affine smooth chart $\ti\cV_{\ti x}$ containing $\ti x$.

\begin{lemm}\label{LmCoordPT}
	Let $\tau$, $\cV$, and $\pi\!:\ti\fM\!\to\!\fM$ be as in Lemma~\ref{LmLocalLoci},
	$\ti x\inn\pi^{-1}(X_\tau^\cV)$,
	and $\ov\bE\eq\{\fE_1,\ldots,\fE_\cht\}$ be the RLS of $\ti x$.
	Then,
	\begin{enumerate}
		[label=(\arabic*),leftmargin=*]
		
		\item \label{Claim1CoordPT}
		for each $i\inn\lrbr{\cht}$,
		$\ve_i$ can be taken as the pullback $\ti 
		\ze_{\se_i}\inn\Ga\big(\sO_{\ti\cV_{\ti x_i}}\big)$ of $\ze_{\se_i}\inn\Ga(\sO_{\cV})$ with an arbitrary edge $\se_i\inn \De_{\ti x,i}$;
		
		\item \label{Claim2CoordPT}
		for each $i\inn\lrbr{\cht}$, 
		$\wc\ve_i$ can be taken as the pullback $\ti\ve_i$ of $\ve_i$;
		
		\item \label{Claim3CoordPT}
		there is an affine smooth chart
		$\ti\cV_{\ti x}\!\subset\!\pi^{-1}(\cV)$ containing $\ti x$ such that
		\begin{align*}
		&\wc\ve_1,\ldots,\wc\ve_\cht\,;\qquad
		\frac{\ti\ze_e}{\prod_{i\in\lrbr\cht\,\tn{s.t.}\,e\in\fE_i}\wc\ve_i}-
		\Big(\frac{\ti\ze_e}{\prod_{i\in\lrbr\cht\,\tn{s.t.}\,e\in\fE_i}\wc\ve_i}(\ti x)\Big),\ \  e\inn\De_{\ti x}\bsl\{\se_1,\cdots,\se_\cht\}\,;
		\\
		&
		\wc \ze_e,\ \ e\inn \fE_\cht\bsl\De_{\ti x,\cht}\,;\qquad 
		\wc\ze_e\eq\ti \ze_e,\ \ e\inn \fE_\cht^\prec\,;\qquad
		\wc\ze_s\eq\ti\ze_s,\ \ s\inn S'\,;
		\end{align*}
		form a system of local parameters on $\ti\cV_{\ti x}$,
		where $\ze_s$, $s\inn S'$, are as in (\ref{e_ze});
		
		\item \label{Claim4CoordPT}
		$\wc\ve_i$, $i\inn\lrbr\cht$, and
		$\wc\ze_e$, $e\inn\big((\fE_\cht\bsl\De_{\ti x,\cht})\!\sqcup\!\fE_\cht^\prec\big)$ all vanish at $\ti x$,
		whereas $\ti\ze_e\big/\big(\prod_{i\in\lrbr\cht\,\tn{s.t.}\,e\in\fE_i}\!\wc\ve_i\big)$ does not vanish at $\ti x$ for any $e\inn\De_{\ti x}\bsl\{\se_1,\ldots,\se_\cht\}$.
	\end{enumerate}
\end{lemm}

\begin{proof}
	The statements~\ref{Claim1CoordPT}, \ref{Claim3CoordPT}, and~\ref{Claim4CoordPT} can be deduced from Corollary~\ref{CrlCoord}.
	The statement~\ref{Claim2CoordPT} follows from~\eref{e_wcEtiE}. 
\end{proof}

The next statement plays a crucial rule in analyzing the change of the structural homomorphism $\varphi$ throughout \S\ref{SecChangeOfPhi}.
Let $\tau$, $\cV$, $\pi$, $\ti x$, $\ov\bE$, and $\ti\cV_{\ti x}$ be as in Lemma~\ref{LmCoordPT}.
With the functions $\ze_{[e]}$ as in~\eref{e_z[v]},
we write $$\ti \ze_{[e]}:= \pi^*\big(\ze_{[e]}\big)
=
\prod_{e'\succeq e}\ti\ze_{e'}~
\in\Ga(\sO_{\ti\cV_{\ti x}})\qquad
\forall~e\in\tau.$$

\begin{prop}\label{PrpDominating}
	With notation as above,
	there exist 
	invertible functions $u_e\inn\Ga\big(\sO^*_{\ti\cV_{\ti x}}\big)$, $e\inn\fE_{\cht}$, such that 
	\begin{gather*}
		\ti \ze_{[e]}
		= u_e\cdot
		\wc\ve_1\cdots
		\wc\ve_\cht,
		\qquad
		\big(\,\tn{resp.}\ \ 
		{\ti \ze_{[e]}}
		=u_e\cdot
		\wc\ve_1\cdots
		\wc\ve_\cht\cdot
		\wc \ze_{e}\,\big)
	\end{gather*}
	for every $e\inn\De_{\ti x,\cht}$ (resp.~$e\inn\fE_\cht\bsl\De_{\ti x,\cht}$).
	Particularly,  $\ti \ze_{[e]}\big|\,\ti \ze_{[e']}$ for all $e\inn\De^{\min}_{\ti x}$ and  $e'\inn\min(\tau)$.
\end{prop}

\begin{proof}
	By (\ref{Eqn:E_k+1}),
	we have $\De_{\ti x,k}\eq\fE_{k}\bsl\fE_{k+1}$ for any $k\inn\lrbr{\cht\!-\!1}$.
	Along with the equality in (\ref{e_ascendingE}),
	this implies
	\begin{align*}
		\fE_\cht^\succeq=
		\big(\fE_\cht\bsl\De_{\ti x,\cht}\big)\sqcup\bigsqcup_{k\in\lrbr\cht}\!\!\De_{\ti x,k}\,.
	\end{align*}
	Hence for every $e\inn\fE_\cht^\succeq$,
	by applying Lemma~\ref{LmCoordPT} and the middle equation in~\eref{e_prz} repeatedly,
	we obtain
	\begin{align*}
		\ti \ze_{e}=
		\begin{cases}
			u_e'\cdot\prod_{i\in\lrbr{\cht}~\tn{s.t.}~e\in\fE_i}\wc\ve_i
			&
			\tn{if}~e\inn\bigsqcup_{k\in\lrbr\cht}\De_{\ti x,k},\\
			u_e'\cdot\wc\ze_e\cdot\prod_{i\in\lrbr{\cht}~\tn{s.t.}~e\in\fE_i}\wc\ve_i
			&
			\tn{if}~e\inn \fE_\cht\bsl\De_{\ti x,\cht},
		\end{cases}
	\end{align*}
	where $u_e'\inn\Ga(\sO^*_{\ti\cV_{\ti x}})$.
	This gives rise to the formulae of all $\ti\ze_{[e]}$, $e\inn\fE_\cht$.
	The last statement then follows from the fact that $\De_{\ti x}^{\min}\!\ne\!\emptyset$.
\end{proof}

\begin{rema}
	The statements of Proposition~\ref{PrpDominating} are consistent with the displayed formulae in the paragraph after~\cite[(5.23)]{HL10},
	but are more precise.
	In~\cite{HL10},
	it suffices to know the last statement of Proposition~\ref{PrpDominating}.
\end{rema}

\vsp
Now, we are ready to prove \ref{Case:5}.\ref{Case:5.4}.

\begin{prop}
	\label{Prp:phi_M5_core_wt2}
	Proposition~\ref{PrpChangeofPhi} holds in \ref{Case:5}.\ref{Case:5.4}.
\end{prop}

\begin{proof}
	At the beginning of \S\ref{SubsecBlowups},
	we already observe the blowups in $(\rd_1\ph_1)$-$(\rd_3\ph_2)$ do not affect $\cV$,
	hence we can identify $\cV$ with its pullback $\ti\cV^{\rd_3\ph_2}$.
	The structural homomorphism takes the form of (\ref{e_varphiMnCs4'}),
	where $\ka_{12}$ and $\ze_q$, $q\inn N(C)$,
	form a subset of a system of local parameters on $\cV$.
	Moreover,
	we only need to study the case when
	\begin{align*}
		m\ge 3\,,
	\end{align*}
	for otherwise (\ref{e_varphiMnCs4'}) is already diagonalized,
	and $(\rd_3\ph_3)$-$(\rd_3\ph_4)$ do not affect $\cV$ either.
	
	Next, we interpret the rooted tree of the first paragraph of \S\ref{SubsecTree} using the language of Definition~\ref{Dfn:Rooted_tree}.
	Intuitively, the underlying set of the rooted tree is the union of those $N_{[\de]}$ with $\de\inn D$ satisfying $\de$ is on a tail and there is no other $\de'\inn D$ on the same tail that is closer to the core $F$ than $\de$.
	Precisely,
	we set
	\begin{equation}\begin{split}
		\label{Eqn:tau_CaseA.4}
		&
		D_{\im}:=\big\{\,\de\inn D\bsl F:\,
		N_{[\de]}~\tn{is~minimal~w.r.t.~inclusion~among~all}~N_{[\de']},\,\de'\inn D\bsl F\,\big\},
		\\
		&
		E:=\!\bigcup_{\de\in D_{\im}}\!\!\!N_{[\de]}\ \  \big(\subset N(C)\big).
	\end{split}
	\end{equation}
	Here, the subscript of $D_{\im}$ is short for ``inclusion-minimal''; see Definition~\ref{Dfn:inc_min} below.
	The set $E$ is the underlying set of the proposed rooted tree.

	Recall we have assumed $D\!\cap\!F\eq\{\de_1,\de_2\}$ and $m\!\ge\!3$.
	Since $\de_3,\ldots,\de_m$ are all on tails, we have $N_{[\de_i]}\!\ne\!\emptyset$ for any $3\!\le\!i\!\le\!m$,
	thus $E\!\ne\!\emptyset$.
	On the set $E$,
	consider the relation $\preceq$ given by $e\!\preceq\!e'$ if and  only if any connected subcurve of $C$ containing $F$ and $e$ must contain $e'$.
	The fact that $F$ is of arithmetic genus 2 guarantees $\preceq$ satisfies (\ref{Eqn:tree_order}), i.e.
	\begin{align}\label{Eqn:tree_A.4}
		\tau_2:=(E,\preceq)
	\end{align}
	is a rooted tree,
	whose root-to-leaf paths are exactly $N_{[\de]}$, $\de\inn D_{\im}$.
	In Figure~\ref{Fig:CaseA.4},
	an example of such constructed $\tau_2$ of a point $\wh x$ satisfying the hypothesis of \ref{Case:5}.\ref{Case:5.4} is provided.
	
	\begin{figure}[htp]
		\begin{center}
			\begin{tikzpicture}
				\def\g2left{
					(0,0.8) arc (90:270:1.6 and 0.8)
					(-1.04,0.08)--(-0.88,0)
					..controls (-0.64,-0.12)..(-0.4,0)
					--(-0.24,0.08)
					(-0.88,0)..controls (-0.64,0.12)..(-0.4,0)
				}
				\def\crs{
					(-.03,-.03)--(.03,.03)
					(-.03,.03)--(.03,-.03)
				}
				
				\draw[thick] \g2left
				[xscale=-1] \g2left;
				
				\draw[thick]
				(.8,1.13) circle (0.4)
				(-.8,1.13) circle (0.4)
				(-.8,1.93) circle (0.4);
				
				\draw[thick]
				(-.8,2.63) circle (0.3)
				(-1.5,1.13) circle (0.3)
				(-1.5,1.93) circle (0.3)
				(.8,1.83) circle (0.3)
				(1.5,1.13) circle (0.3)
				(-1.9,0) circle (0.3);
				
				\draw [xshift=1.35cm,yshift=.3cm]
				\crs;
				\draw [xshift=1.35cm,yshift=-.3cm]
				\crs;
				\draw [xshift=1.7cm,yshift=1.15cm]
				\crs;
				\draw [xshift=.95cm,yshift=1.95cm]
				\crs;
				\draw [xshift=.65cm,yshift=1.95cm]
				\crs;
				\draw [xshift=-.65cm,yshift=2.8cm]
				\crs;
				\draw [xshift=-.95cm,yshift=2.8cm]
				\crs;
				\draw [xshift=-.8cm,yshift=2.58cm]
				\crs;
				\draw [xshift=-1.65cm,yshift=1.85cm]
				\crs;
				\draw [xshift=-1.5cm,yshift=2cm]
				\crs;
				\draw [xshift=-.6cm,yshift=1.85cm]
				\crs;
				\draw [xshift=-1.6cm,yshift=1.3cm]
				\crs;
				\draw [xshift=-2.02cm,yshift=.1cm]
				\crs;
				\draw [xshift=-2.02cm,yshift=-.1cm]
				\crs;
				\draw [xshift=2.2cm,yshift=2.6cm]
				\crs;
				
				\draw[<->,>=stealth]
				(1.6,.33)..controls (1.8,.15) and (1.8,-.15)..(1.6,-.33);
				
				\draw
				(-.8,1.4) node {\tiny{$e$}}
				(-1.05,1.13) node {\tiny{$d$}}
				(-.8,2.2) node {\tiny{$g$}}
				(-1.05,1.93) node {\tiny{$f$}}
				(0,-.5) node {\tiny{$F$}}
				(-1.45,0) node {\tiny{$h$}}
				(-.64,.58) node {\tiny{$r$}}
				(.7,.56) node {\tiny{$c$}}
				(.8,1.38) node {\tiny{$p$}}
				(1.03,1.13) node {\tiny{$q$}}
				(1.75,0) node[right] {\tiny{conjugate}}
				(2.2,2.6) node[right] {\tiny{:~points of $D$}}
				(2,-.7) node[right] {\small{$\wh x\eq(C,D)$}};
				
				\filldraw[xshift=7.5cm]
				(0,2) circle (1.2pt)
				(0,1) circle (1.2pt)
				(1,1) circle (1.2pt)
				(-1,1) circle (1.2pt)
				(-.4,0) circle (1.2pt)
				(.2,0) circle (1.2pt)
				(.8,0) circle (1.2pt)
				(1.4,0) circle (1.2pt);
				\draw[xshift=7.5cm]
				(.8,0)--(1,1)
				(0,2)--(0,1)
				(-1,1)--(0,2)--(1,1)--(1.4,0)
				(.2,0)--(0,1)--(-.4,0);
				\draw[xshift=7.5cm]
				(-.5,1.5) node[left] {\tiny{$h$}}
				(.08,1.5) node[left] {\tiny{$r$}}
				(.5,1.5) node[right] {\tiny{$c$}}
				(-.2,.5) node[left] {\tiny{$d$}}
				(.05,.5) node[right] {\tiny{$e$}}
				(.95,.5) node[left] {\tiny{$p$}}
				(1.2,.5) node[right] {\tiny{$q$}}
				(0,2) node[above] {\tiny{$o$}}
				(0,-.7) node {\small{$\tau_2$}};
			\end{tikzpicture}
		\end{center}
		\caption{An example of the rooted tree $\tau_2$ in \ref{Case:5}.\ref{Case:5.4}}\label{Fig:CaseA.4}
	\end{figure}
	
	
	As described in \S\ref{rd3ph3},
	the blowup centers of $(\rd_3\ph_3)$ locally are given by
	\begin{equation}\begin{split}
		\label{Eqn:r3p3_tree_cmptb}
		&\cH_k^{\rd_3\ph_2}\!\cap\ti\cV^{\rd_3\ph_2}
		=
		\cH_k\cap\cV
		=\!\bigcup_{\fE\in\Xi_k(\tau_2)}\!\!\!\!
		\{\,\ka_{12}\eq 0;~
		\ze_e\eq 0~\forall~e\inn\fE\,\},
		\\
		&
		\tn{where}\qquad
		 \Xi_k(\tau_2)=\{\fE\inn\Xi(\tau_2):|\fE|\eq k\}.
	\end{split}\end{equation}
	It is thus a direct check that on $\cV$, $(\rd_3\ph_3)$ is a grafted blowup with respect to a $\tau_2$-compatible blowup of $\cV$,
	hence by Lemma~\ref{LmSupplementary},
	$(\rd_3\ph_3)$ is $\ex(\tau_2)$-compatible on $\cV$, where the grafted edge corresponds to $\ka_{12}$.
	Then, by Proposition~\ref{PrpDominating},
	after $(\rd_3\ph_3)$, we have
	\begin{itemize}[leftmargin=*]
		\item either 
		$\ti\ka_{12}$ divides all $\ti\ze_{[e]}$, $e\inn\min(\tau_2)$, hence divides all $\ti\ze_{[\de_i]}$, $3\!\le\!i\!\le\!m$;
		\item or there exists $e\inn\De^{\min}_{\ti x}\,\big(\subset\!\min(\tau_2)\big)$ such that $\ti\ze_{[e]}$ divides $\ti\ka_{12}$ as well as all $\ti\ze_{[e']}$, $e'\inn\min(\tau_2)$,
		hence there exists $3\!\le\!i\!\le\!m$ such that $\ti\ze_{[\de_i]}$ divides $\ti\ka_{12}$ as well as all $\ti\ze_{[\de_j]}$, $3\!\le\!j\!\le\!m$.
	\end{itemize}
	In sum, the pullback of the matrix (\ref{e_varphiMnCs4'}) becomes diagonalized after  $(\rd_3\ph_3)$ finishes.
	
	Finally, by Lemma~\ref{LmLocalLoci},
	we see the pullback $\ti\cV^{\rd_3\ph_3}$ of $\cV$ is smooth.
	In addition, since the blowup centers of $(\rd_3\ph_4)$ lie in the exceptional divisors of $(\rd_1\ph_5)$,
	and $(\rd_1\ph_5)$ does not affect $\cV$,
	we conclude that the pullback of $\cV$ in $\ti\fM_2^{\rm div}$ is equal to $\ti\cV^{\rd_3\ph_3}$,
	hence is smooth.
\end{proof}

We remark that in (\ref{Eqn:tau_CaseA.4}),
the reason for choosing {\it minimal} $N_{[\de]}$'s is the same as that for introducing the {\it terminally weighted trees} in \cite[\S3]{HL10} (except here the core is of weight 2 instead of~0):
if $N_{[\de]}\!\subset\! N_{[\de']}$,
then $\ze_{[\de]}$ divides $\ze_{[\de']}$ in the equation $(\ref{e_varphiMnCs4'})$,
hence $\ti\ze_{[\de]}$ divides $\ti\ze_{[\de']}$.
In order to distinguish the minimality with respect to inclusion from other types of minimality, e.g.~with respect to the tree order on a rooted tree,
we introduce the following notion for later use.

\begin{defi}
	\label{Dfn:inc_min}
	Given a set $I$ of sets,
	we say $N\inn I$ is  \ts{inclusion-minimal} if it is minimal with respect to the inclusions of subsets among all $N'\inn I$.
\end{defi}

\subsection{Proof of~\ref{Case:5}.\ref{Case:5.5}, Part~I}
\label{Subsec:Case5.5}
We continue with the setup in \S\ref{SubsecChangeofPhiStatement} and fix a small affine smooth chart $\cV\!\to\!\Md$ containing $\wh x\eq(C,D)$.
Recall the image of $\wh x$ in $\fM_2\wt$ is denoted by $x$.

Now we aim to prove~\ref{Case:5}.\ref{Case:5.5} of Proposition~\ref{PrpChangeofPhi},
when $F$ is inseparable and $D\!\cap\!F$ consists of one point,
which w.l.o.g.~is assumed to be $\de_1$.
Such a point $\wh x\inn \Md$ is possibly the image of a point of the boundary of $\ov M_2(\P^n,d)$,
e.g.~when the core $F$ of $C$ consists three smooth rational subcurves $C_1$, $C_2$ and $C_3$,
satisfying $C_1\!\cap\!C_2\eq\{e_1,e_2\}$,
$C_3\!\cap\!C_2\eq\{e_3,e_4\}$,
$e_1,\ldots,e_4$ are all distinct,
$D\!\cap\!C_1\eq D\!\cap\!C_3\eq\emptyset$ (there are tails attached to $C_1$ and $C_3$ by stability),
and $D\!\cap\!C_2\eq\{\de_1\}$.
As mentioned in Remark~\ref{rem:r1p5},
such $\wh x$ must be considered so to resolve the entire $\ov M_2(\P^n,d)$.

By Proposition~\ref{Prp:phi_key},
the structural homomorphism can be written as
\begin{equation}\label{e_Mn5}
	\varphi=\left[\;
	\begin{matrix} 
		1 & 0 &\cdots&0\\
		0&
		\ka_{12}\ze_{[\de_2]}
		& \cdots & \ka_{1m}\ze_{[\de_m]}
	\end{matrix}\;
	\right]\,.
\end{equation}
Since $x\inn\fM^\mn$ and $\deg D\!\cap\!F\eq 1$,
the phases $(\rd_1\ph_1)$-$(\rd_1\ph_4)$ do not affect $\cV$.
Moreover,
we only need to study the case when 
\begin{align*}
	m\ge 3\,,
\end{align*}
for otherwise (\ref{e_Mn5}) is already diagonalized, and
$(\rd_1\ph_5)$-$(\rd_3\ph_4)$ do not affect $\cV$ either.

By defining the subsets $D_{\im}\!\subset\!D$ and $E\!\subset\!N(C)$ exactly in the same way as (\ref{Eqn:tau_CaseA.4}),
and defining the relation $\preceq$ on $E$ given by $e\!\preceq\!e'$ if and  only if any connected subcurve of $C$ containing $F$ and $e$ must contain $e'$,
we see
\begin{align}\label{Eqn:tree_A.5}
	\tau_1:=(E,\preceq)
\end{align}
is a rooted tree,
whose root-to-leaf paths are exactly $N_{[\de]}$, $\de\inn D_{\im}$.
Although the constructions of $\tau_1$ here and $\tau_2$ in (\ref{Eqn:tree_A.4}) are analogous, we emphasize $D_{\im}\!\subset\!\{\de_2,\ldots,\de_\ell\}$ for $\tau_1$,
whereas 
$D_{\im}\!\subset\!\{\de_3,\ldots,\de_\ell\}$ for $\tau_2$.
An example of $\tau_1$ is illustrated in Figure~\ref{Fig:CaseA.5_tau}.

\begin{figure}[htb]
	\begin{center}
		\begin{tikzpicture}
			\def\g2left{
				(0,0.8) arc (90:270:1.6 and 0.8)
				(-1.04,0.08)--(-0.88,0)
				..controls (-0.64,-0.12)..(-0.4,0)
				--(-0.24,0.08)
				(-0.88,0)..controls (-0.64,0.12)..(-0.4,0)
			}
			\def\crs{
				(-.03,-.03)--(.03,.03)
				(-.03,.03)--(.03,-.03)
			}
			
			\draw[thick] \g2left
			[xscale=-1] \g2left;
			
			\draw[thick]
			(.8,1.13) circle (0.4)
			(-.8,1.13) circle (0.4)
			(-.8,1.93) circle (0.4);
			
			\draw[thick]
			(-.8,2.63) circle (0.3)
			(-1.5,1.13) circle (0.3)
			(-1.5,1.93) circle (0.3)
			(.8,1.83) circle (0.3)
			(1.5,1.13) circle (0.3)
			(-1.9,0) circle (0.3);
			
			\crs;
			\draw [xshift=1.25cm,yshift=-.3cm]
			\crs;
			\draw [xshift=1.7cm,yshift=1.15cm]
			\crs;
			\draw [xshift=.95cm,yshift=1.95cm]
			\crs;
			\draw [xshift=.65cm,yshift=1.95cm]
			\crs;
			\draw [xshift=-.65cm,yshift=2.8cm]
			\crs;
			\draw [xshift=-.95cm,yshift=2.8cm]
			\crs;
			\draw [xshift=-.8cm,yshift=2.58cm]
			\crs;
			\draw [xshift=-1.65cm,yshift=1.85cm]
			\crs;
			\draw [xshift=-1.5cm,yshift=2cm]
			\crs;
			\draw [xshift=-.6cm,yshift=1.85cm]
			\crs;
			\draw [xshift=-1.6cm,yshift=1.3cm]
			\crs;
			\draw [xshift=-2.02cm,yshift=.1cm]
			\crs;
			\draw [xshift=-2.02cm,yshift=-.1cm]
			\crs;
			\draw [xshift=2.2cm,yshift=2.6cm]
			\crs;
			
			
			\draw
			(-.8,1.4) node {\tiny{$e$}}
			(-1.05,1.13) node {\tiny{$d$}}
			(-.8,2.2) node {\tiny{$g$}}
			(-1.05,1.93) node {\tiny{$f$}}
			(0,-.5) node {\tiny{$F$}}
			(-1.45,0) node {\tiny{$h$}}
			(-.64,.58) node {\tiny{$r$}}
			(.7,.56) node {\tiny{$c$}}
			(.8,1.38) node {\tiny{$p$}}
			(1.03,1.13) node {\tiny{$q$}}
			(2.2,2.6) node[right] {\tiny{:~points of $D$}}
			(2,-.7) node[right] {\small{$\wh x\eq(C,D)$}};
			
			\filldraw[xshift=7.5cm]
			(0,2) circle (1.2pt)
			(0,1) circle (1.2pt)
			(1,1) circle (1.2pt)
			(-1,1) circle (1.2pt)
			(-.4,0) circle (1.2pt)
			(.2,0) circle (1.2pt)
			(.8,0) circle (1.2pt)
			(1.4,0) circle (1.2pt);
			\draw[xshift=7.5cm]
			(.8,0)--(1,1)
			(0,2)--(0,1)
			(-1,1)--(0,2)--(1,1)--(1.4,0)
			(.2,0)--(0,1)--(-.4,0);
			\draw[xshift=7.5cm]
			(-.5,1.5) node[left] {\tiny{$h$}}
			(.08,1.5) node[left] {\tiny{$r$}}
			(.5,1.5) node[right] {\tiny{$c$}}
			(-.2,.5) node[left] {\tiny{$d$}}
			(.05,.5) node[right] {\tiny{$e$}}
			(.95,.5) node[left] {\tiny{$p$}}
			(1.2,.5) node[right] {\tiny{$q$}}
			(0,2) node[above] {\tiny{$o$}}
			(0,-.7) node {\small{$\tau_1$}};
		\end{tikzpicture}
	\end{center}
	\caption{An example of the rooted tree $\tau_1$ in \ref{Case:5}.\ref{Case:5.5}}\label{Fig:CaseA.5_tau}
\end{figure}

Mimicking the argument of the proof of Proposition~\ref{Prp:phi_M5_core_wt2},
we see $(\rd_1\ph_5)$ is $\tau$-compatible on~$\cV$,
with the partition $\Xi(\tau_1)\eq\bigsqcup_k\Xi_k(\tau_1)$ still given by $\Xi_k(\tau_1)=\{\fE\inn\Xi(\tau_1):|\fE|\eq k\}$.
Lemma~\ref{LmLocalLoci} then implies the pullback $\ti\cV^{\rd_1\ph_5}$ of $\cV$ is smooth.

After $(\rd_1\ph_5)$,
we fix an arbitrary lift $y\inn \ti\cV^{\rd_1\ph_5}$ of $\wh x$, as well as a small neighborhood $\cV_y$ of $y$ in $\ti\cV^{\rd_1\ph_5}$.
By Definition~\ref{DfnDominant}, there exists $i\!\ge\!2$ such that $N_{[\de_i]}\!\subset\!\De_y$,
i.e.~every node lying between~$\de_i$ and the core $F$ is dominant;
such $\de_i$ is naturally contained in $D_{\im}$.
For convenience,
hereafter we  write
\begin{align}\label{Eqn:I_y}
	D(\De_y):=
	\big\{\,\de\inn D\bsl F:\,
	N_{[\de]}\!\subset\!\De_y
	\big\}\quad
	\big(=
	\big\{\,\de_i\inn D:\,
	i\!\ge\!2,~
	N_{[\de_i]}\!\subset\!\De_y
	\big\}\,\big)\,.
\end{align}

If in addition, 
there exists $\de_i\inn D(\De_y)$ that is not conjugate to $\de_1$ (i.e.~$\ka_{1i}$ is invertible as $\cV$ is assumed small),
then by Proposition~\ref{PrpDominating},
on $\cV_y$,
the $i$-th entry of the pullback of the second row of (\ref{e_Mn5}) divides the other entries of the second row,
hence the pullback of $\varphi$ is diagonalized after ($\rd_1\ph_5$).
Moreover,
$(\rd_2)$-$(\rd_3\ph_2)$ do not affect $\cV_y$ because $(\rd_1\ph_1)$ does not affect $\cV$;
and
$(\rd_3\ph_3)$-$(\rd_3\ph_4)$ do not affect $\cV_y$ because $\de_i$ is not conjugate to $\de_1$.
Therefore,
the pullback of $\cV_y$ in $\ti\fM_2^{\rm div}$ is equal to $\cV_y$,
hence is smooth.	
In sum, 
we have established the following.
\begin{prop}\label{Prp:Case_A.5.1}
	Proposition~\ref{PrpChangeofPhi} holds in \ref{Case:5}.\ref{Case:5.5} if there exists $\de\inn D(\De_y)$ that is not conjugate to the point $D\!\cap\!F$.
\end{prop}
	
If $\de_i$ and $\de_1$ are conjugate for all $\de_i\inn D(\De_y)$,
then $(\rd_3\ph_3)$ is involved (and so is $(\rd_3\ph_4)$ if such $i$ is unique).
We shall deal with this situation in the following two subsections,
which will complete the proof of \ref{Case:5}.\ref{Case:5.5}.

\subsection{First-order derived trees and the proof of~\ref{Case:5}.\ref{Case:5.5}, Part~II}
\label{SubsecTree'}
We continue with the proof of~\ref{Case:5}.\ref{Case:5.5} under the setting of \S\ref{Subsec:Case5.5}.
Recall we have already assumed
$D\!\cap\!F\eq\{\de_1\}$,
and fixed a lift $y$ of~$\wh x\eq(C,D)$ as well as a small neighborhood $\cV_y$ of $y$ in $\ti\cV^{\rd_1\ph_5}$.
The RL-sequence  of $y$ is written as $\ov\bE\eq\{\fE_1,\ldots,\fE_\cht\}$,
and the  set of dominant edges of $y$ is denoted by $\De_y$.

To complete the proof of~\ref{Case:5}.\ref{Case:5.5},
it remains to handle the case when $\de_i$ and $\de_1$ are conjugate for all $\de_i\inn D(\De_y)$.
We will tackle the complicated $|D(\De_y)|\!=\!1$ sub-case in the succeeding subsection;
in this subsection,
we assume
\begin{align*}
	\big|D(\De_y)\big|\ge 2\,.
\end{align*}
W.l.o.g.~we further assume $D(\De_y)\eq\{\de_2,\ldots,\de_\ell\}$ for some $\ell\!\ge \!3$.

The above assumptions imply that there exists a unique tail $T$ containing $\de_2,\ldots,\de_\ell$,
for otherwise there would be multiple tails whose pivotal nodes are conjugate with $\de_1$,
which could only occur when these tails and $\de_1$ are on a non-separating bridge of $F$,
hence the image of $\wh x$ in $\fM_2\wt$ would belong to $\fM^{\mn}_{(3)}$ instead of $\fM^{\mn}$,
contradicting the hypothesis of \ref{Case:5}.
Therefore,
neither $\de_1$ nor the pivotal node $p$ ($=\!\lr{\de_i}$ for all $2\!\le\!i\!\le\!\ell$) of $T$ is on any non-separating bridge of $F$.

By Proposition~\ref{PrpDominating},
under suitable trivialization,
the pullback of (\ref{e_Mn5}) can be written as
\begin{align*}
	\ti\varphi^{\rd_1\ph_5}\!=
	\left[\;
	\begin{matrix} 
		1 & 0 \\
		0&
		\prod_{k\in\lrbr\cht}\!\wc\ve_k
	\end{matrix}\;
	\right]
	\left[\;
	\begin{matrix} 
		1 & 0 &\cdots&0 & 0 &\cdots&0\\
		0&
		\wc\ka_{12} & \cdots &
		\wc\ka_{1\ell} &
		\wc\ka_{1,\ell+1}\prod_{e\in N_{[\de_{\ell+1}]}\bsl\De_y}\!\!\wc\ze_e 
		&\cdots
		&\wc\ka_{1m}\prod_{e\in N_{[\de_{m}]}\bsl\De_y}\!\!\wc\ze_e 
	\end{matrix}\;
	\right].
\end{align*}
Here, $\wc\ve_k$'s are the local parameters corresponding to the proper transforms $\wc\cE_{\fE_k}$'s of the exceptional divisors obtained in ($\rd_1\ph_5$);
see Lemma~\ref{LmCoordPT}.

Recall for every $i\inn\lrbr m$, $N_{[\de_2\wedge\de_i]}\eq N_{[\de_2]}\!\cap\!N_{[\de_i]}$ as in (\ref{Eqn:wedge}).
Hence by Proposition~\ref{Prp:phi_key}~\ref{Part:kappa},
there exist invertible $u'_i$ and $v'_i$ such that
\begin{align}\label{Eqn:[m]_wedge}
	\wc\ka_{1i}=
	u'_i\cdot
	\wc\ka_{12}+v'_i\cdot\!\!\prod_{k\in{\lrbr \cht}_{[\de_2\wedge\de_i]}}\!\!\!\!\!\!\!\wc\ve_k\,,\quad\tn{where}\quad
	{\lrbr \cht}_{[\de_i\wedge\de_j]}:=
	\{k\inn\lrbr\cht:\,N_{[\de_i\wedge\de_j]}\!\cap\!\fE_k\!\ne\!\emptyset\}.
\end{align}
In addition, by (\ref{Eqn:I_y}), we have
\begin{align}\label{Eqn:I_y'}
	N_{[\de_i]}\bsl\De_y=\emptyset\qquad
	\forall~2\!\le\!i\!\le\!\ell.
\end{align} 
Therefore,
after taking elementary column operations,
we can rewrite the second row of $\ti\varphi^{\rd_1\ph_5}$ as
\begin{align*}
	\big(\!\prod_{k\in\lrbr\cht}\!\!\wc\ve_k\,\big)
	\left[\;
	\begin{matrix} 
		0&
		\wc\ka_{12} &
		\eta'_3
		&\cdots
		&
		\eta'_m
	\end{matrix}\,
	\right],
	\qquad\tn{where}\quad
	\eta'_i:=
	\big(\!\prod_{k\in{\lrbr\cht}_{[\de_2\wedge\de_i]}}\!\!\!\!\!\!\!\wc\ve_k\,\big)
	\big(\!\prod_{e\in N_{[\de_{i}]}\bsl\De_y}\!\!\!\!\!\wc\ze_e \,\big).
\end{align*}

Since $\de_2$ is an arbitrary point of $D(\De_y)$,
there should be a way to rewrite the second row of $\ti\varphi^{\rd_1\ph_5}$ so that the terms are independent of the choice of $\de_2$ (after suitable elementary column operations).
This follows from the lemma below.

\begin{lemm}
	\label{Lm:wedge}
	There exists $3\!\le\!j\!\le\!\ell$ such that
	\begin{align*}
		\big(N_{[\wedge D(\De_y)]}:=\big)\quad
		\bigcap_{\de\in D(\De_y)} \!\!\!\!\!N_{[\de]}
		=N_{[\de_2\wedge\de_j]}.
	\end{align*}
\end{lemm}

\begin{proof}
	Notice the left-hand side of the above equality is equal to $\bigcap_{i=2}^\ell\big(N_{[\de_2]}\!\cap\!N_{[\de_i]}\big)$.
	In addition, notice that
	given three smooth points $\de,\de',\de''$,
	one of $N_{[\de]}\!\cap\!N_{[\de']}$ and  $N_{[\de]}\!\cap\!N_{[\de'']}$ must contain the other as a subset.
	Therefore, $\cN'_2:=\big\{N_{[\de_2]}\!\cap\!N_{[\de_i]}:
	2\!\le\!\ell\big\}$ has a natural linear order given by the inclusion of subsets.
	As $\cN'_2$ is a finite set, there exists $3\!\le\!j\!\le\!\ell$ such that $N_{[\de_2]}\!\cap\!N_{[\de_j]}$ is contained in any element of $\cN'_2$.
	Then, 
	\begin{align*}
		\bigcap_{i=2}^\ell N_{[\de_i]}
		=
		\bigcap_{i=2}^\ell\big(N_{[\de_2]}\!\cap\!N_{[\de_i]}\big)
		=
		N_{[\de_2]}\!\cap\!N_{[\de_j]}
		=
		N_{[\de_2\wedge\de_j]},
	\end{align*}
	where the last equality follows from (\ref{Eqn:wedge}).
\end{proof}

Let
\begin{equation}\begin{split}\label{Eqn:[m]_de_wedge}
	&{\lrbr \cht}_{[\wedge D(\De_y)]}:=
	\big\{\,k\inn\lrbr\cht:\,N_{[\wedge D(\De_y)]}\!\cap\!\fE_k\!\ne\!\emptyset\,\big\},
	\\
	&{\lrbr \cht}_{[\de_i\wedge D(\De_y)]}:=
	\big\{\,k\inn\lrbr\cht:\,N_{[\de_i]}\!\cap\! N_{[\wedge D(\De_y)]}\!\cap\!\fE_k\!\ne\!\emptyset\,\big\}\ \subset {\lrbr \cht}_{[\wedge D(\De_y)]},\quad
	\forall~2\!\le\!i\!\le\!m\,.
\end{split}\end{equation}
By Lemma~\ref{Lm:wedge} and (\ref{Eqn:I_y'}),
there exists $3\!\le\!j\!\le\!\ell$ such that
\begin{align*}
	\eta'_j
	=\prod_{k\in{\lrbr \cht}_{[\de_j\wedge D(\De_y)]}}\!\!\!\!\!\!\!\wc\ve_k\ 
	=\prod_{k\in{\lrbr \cht}_{[\wedge D(\De_y)]}}\!\!\!\!\!\wc\ve_k\,.
\end{align*}
Moreover,
for every $3\!\le\!i\!\le\!m$ other than $j$,
\begin{itemize}[leftmargin=*]
	\item either $N_{[\de_2\wedge\de_i]}\!\supset\!N_{[\de_2\wedge\de_j]}$,
	which, by Lemma~\ref{Lm:wedge},
	implies $\eta_j'|\eta_i'$,
	\item or $N_{[\de_2\wedge\de_i]}\!\subsetneq\!N_{[\de_2\wedge\de_j]}$,
	which, also by Lemma~\ref{Lm:wedge},
	implies
	\begin{align*}
		\eta'_i:=
		\big(\!\prod_{k\in{\lrbr \cht}_{[\de_i\wedge D(\De_y)]}}\!\!\!\!\!\!\!\wc\ve_k\,\big)
		\big(\!\prod_{e\in N_{[\de_{i}]}\bsl\De_y}\!\!\!\!\!\wc\ze_e \,\big).
	\end{align*}
\end{itemize}
So, under suitable elementary column operations,
we further rewrite the second row of $\ti\varphi^{\rd_1\ph_5}$ as
\begin{align}\label{Eqn:phi_r1p5}
	\big(\!\prod_{k\in\lrbr\cht}\!\!\wc\ve_k\,\big)
	\left[\;
	\begin{matrix} 
		0&
		\wc\ka_{12} &
		\eta_3
		&\cdots
		&
		\eta_m
	\end{matrix}\,
	\right],
	\qquad\tn{where}\quad
	\eta_i:=
	\big(\!\prod_{k\in{\lrbr \cht}_{[\de_i\wedge D(\De_y)]}}\!\!\!\!\!\!\!\!\!\wc\ve_k\ \big)
	\big(\!\prod_{e\in N_{[\de_{i}]}\bsl\De_y}\!\!\!\!\!\wc\ze_e \,\big).
\end{align}

By Lemmas~\ref{Lm:K_smoothness}~(\ref{Cond:K_param}) and~\ref{LmCoordPT}~\ref{Claim3CoordPT},
we see $\wc\ka_{12}$, $\wc\ve_k$, $k\inn{\lrbr\cht}$, and $\wc\ze_e$, $e\inn N_{[\de_i]}\bsl\De_y$, $3\!\le\!i\!\le\!m$, together form a subset of a system of local parameters on $\cV_y$.
Since $|D\!\cap\!F|\!>\!0$ and $|D(\De_y)|\!\ge\!2$,
we conclude that $(\rd_2)$-$(\rd_3\ph_3)$ do not affect $\cV_y$.
Similar to \ref{Case:5}.\ref{Case:5.4},
we aim to construct a rooted tree $\varrho_y$ whose root-to-leaf paths are exactly those inclusion-minimal ${\lrbr \cht}_{[\de_i\wedge D(\De_y)]}\!\sqcup\! \big(N_{[\de_i]}\bsl\De_y\big)$ among all $3\!\le\!i\!\le\!m$,
and show $(\rd_3\ph_4)$ is $\ex(\varrho_y)$-compatible on $\cV_y$.

\vsp
We begin with $\varrho_y$.
Similar to (\ref{Eqn:tau_CaseA.4}),
we set
\begin{align*}
	D_{y,\im}:=\big\{\,\de\inn D\bsl F:\;
	&
	{\lrbr \cht}_{[\de\wedge D(\De_y)]}\!\sqcup\! \big(N_{[\de]}\bsl\De_y\big)~\tn{is~inclusion-minimal~among~all}\\
	&{\lrbr \cht}_{[\de_i\wedge D(\De_y)]}\!\sqcup\! \big(N_{[\de_i]}\bsl\De_y\big),~2\!\le\!i\!\le\!m\,\big\}.
\end{align*}
It is straightforward that $D_{y,\im}\!\subset\!D_{\im}$.
So by writing
\begin{align*}
	E_y:=\!\bigcup_{\de\in D_{y,\im}}\!\!\!\!\Big({\lrbr \cht}_{[\de\wedge D(\De_y)]}\!\sqcup\! \big(N_{[\de]}\bsl\De_y\big)\Big),
\end{align*}
we see the subset $E_y\big\bsl {\lrbr \cht}_{[\wedge D(\De_y)]}$ is also a subset of the underlying set of the rooted tree $\tau_1$ of (\ref{Eqn:tree_A.5}).
On $E_y$,
consider the partial order $\preceq_y$ determined by
\begin{itemize}[leftmargin=*]
	\item for every $k,k'\inn{\lrbr \cht}_{[\wedge D(\De_y)]}$,
	$k\!\prec_y\!k'$ if and only if $k\!>\!k'$;
	\item for every $e, e'\inn\bigcup_{\de\in D_{y,\im}}\big(N_{[\de]}\bsl\De_y\big)$,
	$e\!\preceq_y\!e'$ if and  only if any connected subcurve of $C$ containing $F$ and $e$ must contain $e'$;
	\item 
	for every $k\inn{\lrbr \cht}_{[\wedge D(\De_y)]}$ and $e\inn \bigcup_{\de\in D_{y,\im}}\big(N_{[\de]}\bsl\De_y\big)$,
	$e\!\prec_y \!k$ if and only if $e\inn\big(N_{[\wedge D(\De_y)]}\!\cap\!\fE_k\big)^\prec$ in $\tau_1$.
\end{itemize}
It is a direct check that $\preceq_y$ is a tree order on $E_y$ as per (\ref{Eqn:tree_order}).
We call the rooted tree 
\begin{align}\label{Eqn:tree_A.5.2}
	\varrho_y:=(E_y,\preceq_y)
\end{align}
the \ts{first-order derived tree} at $y$.
Notice that restriction of the tree order $\prec_y$ to $E_y\big\bsl {\lrbr \cht}_{[\wedge D(\De_y)]}$ is consistent with the restriction of the tree order of (\ref{Eqn:tree_A.5}).
Also notice the root-to-leaf paths of the
are exactly those inclusion-minimal ${\lrbr \cht}_{[\de_i\wedge D(\De_y)]}\!\sqcup\! \big(N_{[\de_i]}\bsl\De_y\big)$ among all $2\!\le\!i\!\le\!m$;
particularly,
${\lrbr \cht}_{[\wedge D(\De_y)]}$ is a root-to-leaf path of $\varrho_y$.

In the following example,
we provide illustration of the notion of first-order derived trees.
The rooted trees labeled as $\varrho_{y,1}$ and $\varrho_{y,2}$ in Figure~\ref{figDerivedTFMRs} are irrelevant here;
they will be introduced in \S\ref{Subsubsec:Case_E.1.e_1}.

\begin{figure}[htp]
	\begin{center}
		\begin{tikzpicture}
			\draw[xshift=-4cm]
			(0,0)--(1.2,1.2)
			(0,0)--(-.6,.6)
			(.6,.6)--(0,1.2)
			;
			\filldraw[xshift=-4cm]
			(0,0) circle (2pt)
			(.6,.6) circle (2pt)
			(-.6,.6) circle (2pt)
			(1.2,1.2) circle (2pt)
			(0,1.2) circle (2pt)
			;
			\draw[xshift=-4cm]
			(.2,.2) node[right] {\scriptsize{$e_b$}}
			(-.15,.2) node[left] {\scriptsize{$e_a$}}
			(.85,.85) node[right] {\scriptsize{$e_d$}}
			(.4,.85) node[left] {\scriptsize{$e_c$}}
			;
			
			\draw[xshift=-.5cm] (0,0)--(-.6,.6)--(-.6,1.2)
			(0,0)--(.6,.6)
			(0,0)--(0,.6);
			\filldraw[xshift=-.5cm] (0,0) circle (2pt)
			(.6,.6) circle (2pt)
			(-.6,.6) circle (2pt)
			(-.6,1.2) circle (2pt)
			(0,.6) circle (2pt);
			\draw[xshift=-.5cm]
			(.45,.2) node {\scriptsize{$e_d$}}
			(-.4,.2) node {\scriptsize{$1$}}
			(-.45,.9) node {\scriptsize{$2$}}
			(-.12,.4) node {\scriptsize{$e_c$}}
			;
			
			\draw[xshift=1.8cm] (0,0)--(-.6,.6);
			\draw[xshift=1.8cm] (0,0)--(.6,.6);
			\draw[xshift=1.8cm] (0,0)--(0,.6);
			\filldraw[xshift=1.8cm] (0,0) circle (2pt)
			(.6,.6) circle (2pt)
			(-.6,.6) circle (2pt)
			(0,.6) circle (2pt);
			\draw[xshift=1.8cm]
			(.45,.2) node {\scriptsize{$e_d$}}
			(-.4,.2) node {\scriptsize{$1$}}
			(-.12,.4) node {\scriptsize{$e_c$}}
			;
			
			\draw[xshift=3.6cm] (0,0)--(1.2,1.2);
			\draw[xshift=3.6cm] (0,0)--(-.6,.6);
			\draw[xshift=3.6cm] (.6,.6)--(0,1.2);
			\filldraw[xshift=3.6cm] (0,0) circle (2pt)
			(.6,.6) circle (2pt)
			(-.6,.6) circle (2pt)
			(1.2,1.2) circle (2pt)
			(0,1.2) circle (2pt);
			\draw[xshift=3.6cm]
			(.45,.2) node {\scriptsize{$e_b$}}
			(-.4,.2) node {\scriptsize{$1$}}
			(1.05,.8) node {\scriptsize{$e_d$}}
			(.2,.8) node {\scriptsize{$e_c$}}
			;
			
			\draw[xshift=6cm] (0,0)--(.6,.6) 
			(0,0)--(-.6,.6);
			\filldraw[xshift=6cm] 
			(0,0) circle (2pt)
			(.6,.6) circle (2pt)
			(-.6,.6) circle (2pt)
			;
			\draw[xshift=6cm]
			(.45,.2) node {\scriptsize{$1$}}
			(-.4,.2) node {\scriptsize{$e_a$}}
			;
			
			\draw[xshift=8.1cm] (0,0)--(1.2,1.2);
			\draw[xshift=8.1cm] (0,0)--(-.6,.6);
			\draw[xshift=8.1cm] (.6,.6)--(0,1.2);
			\filldraw[xshift=8.1cm] (0,0) circle (2pt)
			(.6,.6) circle (2pt)
			(-.6,.6) circle (2pt)
			(1.2,1.2) circle (2pt)
			(0,1.2) circle (2pt);
			\draw[xshift=8.1cm]
			(.45,.2) node {\scriptsize{$1$}}
			(-.4,.2) node {\scriptsize{$e_a$}}
			(1.05,.8) node {\scriptsize{$e_d$}}
			(.2,.8) node {\scriptsize{$2$}}
			;
			
			\draw[xshift=.5cm,yshift=-1.3cm] (0,0)--(-.6,.6);
			\draw[xshift=.5cm,yshift=-1.3cm] (0,0)--(.6,.6);
			\draw[xshift=.5cm,yshift=-1.3cm] (0,0)--(0,.6);
			\filldraw[xshift=.5cm,yshift=-1.3cm] (0,0) circle (2pt)
			(.6,.6) circle (2pt)
			(-.6,.6) circle (2pt)
			(0,.6) circle (2pt);
			\draw[xshift=.5cm,yshift=-1.3cm]
			(.45,.2) node {\scriptsize{$e_d$}}
			(-.4,.2) node {\scriptsize{$1$}}
			(-.12,.4) node {\scriptsize{$e_c$}}
			;
			
			\draw[xshift=-1.5cm,yshift=-1.3cm] (0,0)--(-.6,.6);
			\draw[xshift=-1.5cm,yshift=-1.3cm] (0,0)--(.6,.6);
			\draw[xshift=-1.5cm,yshift=-1.3cm] (0,0)--(0,.6);
			\filldraw[xshift=-1.5cm,yshift=-1.3cm] (0,0) circle (2pt)
			(.6,.6) circle (2pt)
			(-.6,.6) circle (2pt)
			(0,.6) circle (2pt);
			\draw[xshift=-1.5cm,yshift=-1.3cm]
			(.45,.2) node {\scriptsize{$e_d$}}
			(-.4,.2) node {\scriptsize{$2$}}
			(-.12,.4) node {\scriptsize{$e_c$}}
			;
			
			\draw[xshift=7.4cm,yshift=-1.3cm] (0,0)--(-.6,.6);
			\draw[xshift=7.4cm,yshift=-1.3cm] (0,0)--(.6,.6);
			\draw[xshift=7.4cm,yshift=-1.3cm] (0,0)--(0,.6);
			\filldraw[xshift=7.4cm,yshift=-1.3cm] (0,0) circle (2pt)
			(.6,.6) circle (2pt)
			(-.6,.6) circle (2pt)
			(0,.6) circle (2pt);
			\draw[xshift=7.4cm,yshift=-1.3cm]
			(.45,.2) node {\scriptsize{$e_d$}}
			(-.4,.2) node {\scriptsize{$2$}}
			(-.12,.4) node {\scriptsize{$e_a$}}
			;
			
			\draw[xshift=9.4cm,yshift=-1.3cm] (0,0)--(-.6,.6);
			\draw[xshift=9.4cm,yshift=-1.3cm] (0,0)--(.6,.6);
			\filldraw[xshift=9.4cm,yshift=-1.3cm] (0,0) circle (2pt)
			(.6,.6) circle (2pt)
			(-.6,.6) circle (2pt)
			;
			\draw[xshift=9.4cm,yshift=-1.3cm]
			(.45,.2) node {\scriptsize{$e_a$}}
			(-.4,.2) node {\scriptsize{$1$}}
			;
			
			\draw 
			(-4.1,-.2) node {\tiny{$o$}}
			(-3.5,-.4) node {\scriptsize{$\tau$}}
			(4.1,1.95) node {\scriptsize{possible $\varrho_y$}}
			(-1.3,-1.7) node {\scriptsize{$\varrho_{y,2}$}}
			(.7,-1.7) node {\scriptsize{$\varrho_{y,1}$}}
			(7.6,-1.7) node {\scriptsize{$\varrho_{y,2}$}}
			(9.6,-1.7) node {\scriptsize{$\varrho_{y,1}$}}
			(.1,1.2) node {\scriptsize{(1)}}
			(1.2,1.2) node {\scriptsize{(2)}}
			(3,1.2) node {\scriptsize{(3)}}
			(5.4,1.2) node {\scriptsize{(4)}}
			(7.5,1.2) node {\scriptsize{(5)}}
			;
			\draw [loosely dashed]
			(-1.3,1.4) rectangle (.3,-.2)
			(1,1.4) rectangle (2.6,-.2)
			(2.8,1.4) rectangle (5,-.2)
			(5.2,1.4) rectangle (6.8,-.2)
			(7.3,1.4) rectangle (9.5,-.2)
			;
			\draw[decoration={brace,amplitude=8pt},decorate]
			(-1.35,1.5) -- (9.55,1.5);
			\draw[decoration={brace,amplitude=6pt},decorate]
			(-2.2,-.5) -- (1.2,-.5);
			\draw[decoration={brace,amplitude=6pt},decorate]
			(6.7,-.5) -- (10.1,-.5);
		\end{tikzpicture}
	\end{center}
	\caption{First-order derived trees in Example~\ref{EgTree'}}\label{figDerivedTFMRs}
\end{figure}

\begin{exam}\label{EgTree'}
	Let $\tau$  be the rooted tree given by leftmost graph in Figure~\ref{figDerivedTFMRs},
	whose root is $\{e_a,e_b\}$ (i.e.~the root vertex $o$ is the lowest one).
	As in Example~\ref{Eg:G1},
	$\tau$ has two transverse sections:
	$$
	\fE_{1}\eq \{e_a,e_b\}\succ
	\fE_{2}\eq \{e_a,e_c,e_d\}.
	$$
	Given $\wh x\inn X^\cV_{\tau}$ as well as $y$ lifting $\wh x$, we analyze $\varrho_y$ for all possible RLS $\ov\bE$ and $\De_{y}$ of $y$.
	
	\begin{enumerate}
		[label=(\arabic*),leftmargin=*]
		\item If $\De_y\eq \{e_a,e_b\}$ and $\ov \bE\eq\{\fE_1,\fE_2\}$,
		then $\varrho_y$ is given by Graph~$(1)$ of Figure~\ref{figDerivedTFMRs}.
		
		\item If $\De_y\eq \{e_a,e_b\}$ and $\ov \bE\eq\{\fE_1\}$,
		then $\varrho_y$ is given by Graph~$(2)$ of Figure~\ref{figDerivedTFMRs}.
		
		\item If $\De_y\eq \{e_a\}$ and $\ov \bE\eq\{\fE_1\}$,
		then $\varrho_y$ is given by Graph~$(3)$ of Figure~\ref{figDerivedTFMRs}.
		
		\item If $\De_y\eq \{e_b,e_c,e_d\}$ and $\ov \bE\eq\{\fE_1,\fE_2\}$,
		then $\varrho_y$ is given by Graph~$(4)$ of Figure~\ref{figDerivedTFMRs}.
		
		\item If $\De_y\eq \{e_b,e_c\}$ and $\ov \bE\eq\{\fE_1,\fE_2\}$
		then $\varrho_y$ is given by Graph~$(5)$ of Figure~\ref{figDerivedTFMRs}.
		The case that $\De_y\eq \{e_b,e_d\}$ and $\ov \bE\eq\{\fE_1,\fE_2\}$ is parallel.
		
		\item In all other cases (e.g.~$\De_y\eq \{e_a,e_b,e_c\}$ and $\ov \bE\eq\{\fE_1,\fE_2\}$),
		$\varrho\eq\tau_\bullet$ because $N_{[\wedge D(\De_y)]}\eq\emptyset$.
	\end{enumerate}
\end{exam}

\vsp
In order to show $(\rd_3\ph_4)$ is $\ex(\varrho_y)$-compatible on $\cV_y$,
first notice that each transverse section $\fE$ of $\varrho_y$ contains a unique element $k_\fE\inn{\lrbr \cht}_{[\wedge D(\De_y)]}$, and determines a transverse section
\begin{align*}
	\fE\dual:=
	\big(\fE\bsl\{k_\fE\}\big)\sqcup 
	\big(N_{[\wedge D(\De_y)]}\!\cap\!\fE_{k_\fE}\big)
\end{align*}
of the rooted tree $\tau_1$.
Since $\big(\fE\bsl\{k_\fE\}\big)\!\subset\!(\tau_1\bsl\De_y)$,
we have $\fE\dual\!\preceq\!\fE_{k_\fE}$,
hence $|\fE|\eq|\fE\dual|\!\ge\!|\fE_{k_\fE}|$.

For every $1\!\le\!h\!\le\!h'$,
with $\cQ_{h,h'}$ as in (\ref{rd3ph4-pi++}),
we denote by $\wc\cQ_{h,h'}$ its proper transform after ($\rd_1\ph_5$).
Then, it is a direct check that
\begin{align}\label{Eqn:r3p4_tree_cmptb}
	\wc\cQ_{h,h'}\!\cap\!\cV_y=
	\bigcup_{
	\fE\in\Xi(\varrho_y)~\tn{s.t.}~
	|\fE_{k_\fE}|= h,~
	|\fE|= h'
	}\!\!\!\!\!\!\!\!\!\!\!\!
	\big\{\,
	\wc\ka_{12}\eq 0;~
	\wc\ve_{k_\fE}\eq 0;~
	\wc\ze_e\eq 0~\forall~e\inn \fE\bsl\{k_\fE\}
	\big\}\,.
\end{align}
This gives rise to a partition of $\Xi(\varrho_y)$, hence of $\Xi\big(\ex(\varrho_y)\big)$.
Observe that for every $\fE,\fE'\inn\Xi(\varrho_y)$ with $\fE\!\succ\!\fE'$,
either we have $k_{\fE}\!\succ_y\!k_{\fE'}$,
i.e.~$k_{\fE}\!<\!k_{\fE'}$, hence $|\fE_{k_\fE}|\!<\!|\fE_{k_{\fE'}}|$;
or we have $k_{\fE}\!=\!k_{\fE'}$ and $\fE\dual\!\succ\!(\fE')\dual$ in $\Xi(\tau_1)$,
hence $|\fE|\eq|\fE\dual|\!\le\!|(\fE')\dual|\eq|\fE'|.$
Therefore, the above partition of  $\Xi\big(\ex(\varrho_y)\big)$ satisfies
Definition~\ref{DfnGaAdm},
i.e.~($\rd_3\ph_4$) is $\ex(\varrho_y)$-compatible on $\cV_y$.

So after ($\rd_3\ph_4$),  Proposition~\ref{PrpDominating} implies either the pullback of $\wc\ka_{12}$ or the pullback of some $\eta_i$, $3\!\le\!i\!\le\!m$,
divides the other terms of the second row of the pullback $\ti\varphi^{\rd_3\ph_4}$ of $\ti\varphi^{\rd_1\ph_5}$,
i.e.~$\ti\varphi^{\rd_3\ph_4}$ is diagonalized.
Moreover,
Lemma~\ref{LmLocalLoci} ensures the pullback of $\cV_y$ is smooth after ($\rd_3\ph_4$).
To summarize,
in this subsection,
we have established the following.
\begin{prop}\label{Prp:Case_A.5.2}
	Proposition~\ref{PrpChangeofPhi} holds in \ref{Case:5}.\ref{Case:5.5} if $|D(\De_y)|\!\ge\!2$ and every point of $D(\De_y)$ is  conjugate to the point $D\!\cap\!F$.
\end{prop}

\subsection{Proper transforms of blowups and the proof of \ref{Case:5}.\ref{Case:5.5}, Part III}
\label{SubsecTwoTrees}

\subsubsection{Proof of \ref{Case:5}.\ref{Case:5.5}, Part III, up to ($\rd_3\ph_3$)}
\label{Subsubsec:Case_A.5.III}

We continue with the proof of~\ref{Case:5}.\ref{Case:5.5} under the setting of \S\ref{Subsec:Case5.5} and \S\ref{SubsecTree'}.
In particular,
recall we have assumed
$D\!\cap\!F\eq\{\de_1\}$,
and fixed a lift $y$ of~$\wh x\eq(C,D)$ as well as a small neighborhood $\cV_y$ of $y$ in $\ti\cV^{\rd_1\ph_5}$.
The RL-sequence  of $y$ is written as $\ov\bE\eq\{\fE_1,\ldots,\fE_\cht\}$,
and the  set of dominant edges of $y$ is denoted by $\De_y$.

In this subsection,
we aim to show the only sub-case of~\ref{Case:5}.\ref{Case:5.5} that has not been established,
namely when $|D(\De_y)|\!=\!1$ and the point of $D(\De_y)$ is conjugate to $\de_1$.
W.l.o.g.~we further assume $D(\De_y)\eq\{\de_2\}$.

Under the above assumptions,
the pullback $\ti\varphi^{\rd_1\ph_5}$ of the structural homomorphism $\varphi$ is analogous to that in \S\ref{SubsecTree'},
but simpler:
\begin{align}\label{Eqn:phi_A.5.III}
	&
	\ti\varphi^{\rd_1\ph_5}\!=
	\left[\;
	\begin{matrix} 
		1 & 0 \\
		0&
		\prod_{k\in\lrbr\cht}\!\wc\ve_k
	\end{matrix}\;
	\right]
	\left[\;
	\begin{matrix} 
		1 & 0 & 0 &\cdots&0 \\
		0&
		\wc\ka_{12} &
		\eta_3
		&\cdots
		&
		\eta_m
	\end{matrix}\;
	\right],\\
	&
	\tn{where}\qquad
	\eta_i:=
	\big(\!\prod_{k\in{\lrbr \cht}_{[\de_2\wedge\de_i]}}\!\!\!\!\!\!\!\wc\ve_k\; \big)
	\big(\!\prod_{e\in N_{[\de_{i}]}\bsl\De_y}\!\!\!\!\!\wc\ze_e \,\big)\,.\nonumber
\end{align}

The blowups of
$(\rd_2)$-$(\rd_3\ph_2)$ do not affect $\cV_y$ because $|D\!\cap\!F|\!>\!0$,
so we identify $\cV_y$ with its pullback in $\ti\fM^{\rd_3\ph_2}$.
The hypotheses $|D\!\cap\!F|\!=\!1$, $|D(\De_y)|\eq 1$, and $\de_1$ and $\de_2$ are conjugate together imply that $y$ lies in the intersection of certain $\cH_h^{\rd_3\ph_2}$ as in (\ref{Eqn:r3p3_blowup_center}), as well as the proper transforms $\wc\cQ_{h,h'}$ of certain $\cQ_{h,h'}$ as in (\ref{rd3ph4-pi++}).

Precisely, consider the following two  rooted trees:
one is $\tau_2$ as in (\ref{Eqn:tree_A.4}),
with $D_{\im}$ of (\ref{Eqn:tau_CaseA.4}) replaced by
\begin{align*}
	\big\{\,\de_i\inn D:\,
N_{[\de_i]}\bsl\De_y~\tn{is~inclusion-minimal~among~all}~3\!\le\!i\!\le\!m\,\big\};
\end{align*}
the other is the first order derived tree $\varrho_y$ as in (\ref{Eqn:tree_A.5.2}).
Notice $\varrho_y$ is still well-defined when $|D(\De_y)|\eq 1$;
in this case
\begin{align*}
	N_{[\wedge D(\De_y)]}=N_{[\de_2]},\qquad
	{\lrbr \cht}_{[\wedge D(\De_y)]}=\lrbr\cht,\qquad
	{\lrbr \cht}_{[\de_i\wedge D(\De_y)]}=
	{\lrbr \cht}_{[\de_2\wedge\de_i]}\,.
\end{align*}
It is then a direct check that (\ref{Eqn:r3p3_tree_cmptb}) and (\ref{Eqn:r3p4_tree_cmptb}) remain valid on $\cV_y$.
Therefore,
($\rd_3\ph_3$) is  $\ex(\tau_2)$-compatible on $\cV_y$,
and in particular,
the pullback $\ti\cV_y^{\rd_3\ph_3}\!\subset\!\ti\fM^{\tn{div},\rd_3\ph_3}$ of $\cV_y$ is smooth.
Moreover, the sequential blowup
\begin{align}
	\label{Eqn:r3p4_proto}
	(\ti\fM^{\tn{div},\rd_3\ph_4})'\lra
	\ti\fM^{\tn{div},\rd_3\ph_2}
\end{align}
successively (w.r.t.~the usual lexicographic order) along the proper transforms of $\cQ_{h,h'}$ in $\ti\fM^{\tn{div},\rd_3\ph_2}$,
which serves as the
``prototype'' of ($\rd_3\ph_4$),
is $\ex(\varrho_y)$-compatible on $\cV_y$.

After ($\rd_3\ph_3$),
fix an arbitrary lift $z\inn\ti\fM^{\tn{div},\rd_3\ph_3}$ of $y$, which determines $\De_z\!\subset\!\ex(\tau_2)$ as in Definition~\ref{DfnDominant}.
Let $\cV_z\!\subset\!\ti\cV_y^{\rd_3\ph_3}$ be a small neighborhood of $z$.

\begin{lemm}
	\label{Lm:PT_blowup_center}
	For every $\fE\inn\Xi\big(\ex(\varrho_y)\big)$ satisfying $\fE\!\cap\!\De_z\!\ne\!\emptyset$,
	the proper transform (on $\ti\cV_y^{\rd_3\ph_3}$) of the vanishing locus $X^{\cV_y}_\fE$  is disjoint from  $\cV_z$.
\end{lemm}

\begin{proof}
	The statement holds because the pullback to $\cV_z$ of $X^{\cV_y}_\fE$ is contained in the proper transform of some exceptional divisor of $\cV_z/\cV_y$; this is similar to~(\ref{Eqn:X_pullback}).
\end{proof}

By Lemma~\ref{Lm:PT_blowup_center}, on $\cV_z$,
the proper transforms of $\cQ_{h,h'}$ of (\ref{rd3ph4-pi++}) on $\cV_z$, which are exactly the blowup centers $(\rd_3\ph_4)$,
are all unions of the proper transforms of $X^{\cV_y}_\fE$ with $\fE\inn\Xi\big(\ex(\varrho_y)\big)$ satisfying $\fE\!\cap\!\De_z\!=\!\emptyset$.
We thus expect such transverse sections of $\ex(\varrho_y)$ form a new rooted tree,
with which ($\rd_3\ph_4$) is compatible on $\cV_z$.

\subsubsection{Proper transforms of blowups and rooted trees}\label{Subsubsec:PT}
In order to study the rooted tree that  ($\rd_3\ph_4$) is compatible with on $\cV_z$,
we take a detour and introduce the following notions under a more general setting.

\begin{defi}
	\label{Dfn:PT}
	Let
	$\pi:\ti\fM\!\to\!\fM$ and $\pi':\ti\fM'\!\to\!\fM$ respectively be the sequential blowups of a stack $\fM$ successively along the proper transforms of the closed substacks 
	$Z_1,Z_2,\ldots\!\subset\!\fM$ and 
	$Z_1',Z_2',\ldots\!\subset\!\fM$.
	Let 
	$\wc Z_1',\wc Z_2',\ldots\!\subset\!\ti\fM$  be the proper transforms of $Z_1',Z_2',\ldots$,
	respectively.
	We call the sequential blowup of $\ti\fM$ successively along the proper transforms of $\wc Z_1',\wc Z_2',\ldots$ the \ts{proper transform of} $\pi'$ (\ts{with respect to $\pi$}),
	and denote it by
	\begin{align*}
		\PT_{\pi}(\pi'):\ti\fM''\lra\ti\fM\,.
	\end{align*}
\end{defi}

For example, $(\rd_3\ph_4)$ is the proper transform of the sequential blowup (\ref{Eqn:r3p4_proto}) with respect to $(\rd_3\ph_3)$.

With notation as above,
assume there exist an affine smooth chart $\cV$ of $\fM$, two rooted trees $\tau$ and $\tau'$,
and a $\tau\!\cup\!\tau'$-labeled subset of a system of local parameters on $\cV$ as per (\ref{e_ze}), such that
$\pi$ and $\pi'$ are respectively $\tau$- and $\tau'$-compatible.
Here in  $\tau\!\cup\!\tau'$, $\tau$ and $\tau'$ are considered as sets instead of posets.
The intersection $\tau\!\cap\!\tau'$ may or may not be empty.

Fix a point $y\inn X^\cV_\tau\!\cap\!X^\cV_{\tau'}$,
as well as a lift $z\inn\ti\fM$ of $y$ (i.e.~$\pi(z)\eq y$). 
By Lemma~\ref{LmCoordPT},
$\{\wc\ze_e:e\inn\tau'\bsl\De_z\}$ form a subset of a system of local parameters on a neighborhood $\cV_z\!\subset\!\ti\fM$ of $z$.
Below,
we construct the rooted tree that $\PT_\pi(\pi')$ is compatible with on $\cV_z$.

Let
\begin{align}\label{Eqn:PT_tree}
	\Xi^\complement_z(\tau'):=
	\big\{\,\fE\inn\Xi(\tau'):\,
	\fE\!\cap\!\De_z\eq\emptyset\,\big\}\,,\qquad
	E':=\!
	\bigcup_{\fE\in \Xi^\complement_z(\tau')}\!\!\!\!\fE\quad \big(\subset(\tau'\bsl \De_z)\subset\tau'\big)\,.
\end{align}
We denote by $\preceq'$ the tree order of $\tau'$, whose restriction  to the subset $E'$ is also written as $\preceq'$.
It is straightforward that
\begin{align*}
	\PT_{z}(\tau')=
	\PT_{\pi,z}(\tau'):=
	\big(E',\preceq'\big)
\end{align*}
is a rooted tree, called the \ts{proper transform of $\tau'$ (with respect to $\pi$) at $z$}. (Indeed, every subposet of a rooted tree is naturally a rooted tree; c.f.~\cite[Definition~2.10]{HN2}.)

An example of the proper transform of a rooted tree is provided in Example~\ref{Eg:PT} at the end of \S\ref{SubsecTwoTrees}.

\begin{lemm}
	\label{Lm:PT_trans_sect}
With notation as above,
we have $\Xi\big(\PT_{z}(\tau')\big)\eq 
\Xi^\complement_z(\tau')$.
\end{lemm}

\begin{proof}
W.l.o.g.~we assume $\Xi^\complement_z(\tau')\!\ne\!\emptyset$,
which is equivalent to $\PT_{z}(\tau')\!\ne\!\tau_\bullet$.

Given $\fE\inn\Xi^\complement_z(\tau')$,
any two edges of $\fE$ are incomparable in $\tau'$, hence in $\PT_{z}(\tau')$.
Moreover, since $\PT_{z}(\tau')\!\subset\!\tau'$ as sets,
we see  every edge of $\PT_{z}(\tau')$ is  an edge of $\tau'$, hence is comparable with some edge of $\fE$.
The second criterion of Definition~\ref{Dfn:transverse_sections} then implies $\fE\inn \Xi\big(\PT_{z}(\tau')\big)$,
so $\Xi^\complement_z(\tau')\!\subset\!\Xi\big(\PT_{z}(\tau')\big)$.

Conversely, given $\fE\inn\Xi\big(\PT_{z}(\tau')\big)$,
any two edges of $\fE$ are incomparable in $\PT_{z}(\tau')$, hence in $\tau'$.
It remains to show
for every $e'\inn \tau'$,
there exists $e\inn\fE$ comparable with $e'$.
Let $\wp'$ be an arbitrary root-to-leaf path of $\tau'$ that contains $e'$.
By the first criterion of Definition~\ref{Dfn:transverse_sections},
each $\fE'\inn \Xi^\complement_z(\tau')$ meets $\wp'$ at a unique edge $e_{\fE'}$.
Since $\wp'$, along with the restriction of $\preceq'$,
is a linear set,
we denote by $e_0$ the least element of the set $\{e_{\fE'}\}$,
which is comparable with $e'$.
In addition, observe that $\wp'\!\cap\!\PT_{z}(\tau')$ is a root-to-leaf path of $\PT_{z}(\tau')$,
so it meets $\fE$ at a unique edge $e$.
Then, $e_0\!\preceq'\!e$.
\begin{itemize}[leftmargin=*]
	\item If $e'\!\preceq'\!e_0$,
	then $e'\!\preceq'\!e$.
	\item If $e'\!\succ'\!e_0$,
	then (\ref{Eqn:tree_order}) implies $e'$ and $e$ are comparable.
\end{itemize}
In sum, we find $e\inn\fE$ comparable with $e'$.
By the second criterion of Definition~\ref{Dfn:transverse_sections},
we thus have $\fE\inn \Xi(\tau')$.
Taking $\fE\inn E'\!\subset\!\tau'\bsl\De_z$ into consideration, we conclude that $\Xi\big(\PT_{z}(\tau')\big)\!\subset\!\Xi^\complement_z(\tau')$.
\end{proof}

\begin{coro}
	\label{Crl:PT}
	$\PT_{\pi}(\pi')$ is $\PT_{z}(\tau')$-compatible on $\cV_z$.
\end{coro}

\begin{proof}
	Let $\Xi(\tau')\eq\bigsqcup_{k\ge 1}\Xi_k(\tau')$ be the partition determined by $\pi'$ as per Definition~\ref{DfnGaAdm}.
	It induces a partition of $\Xi\big(\PT_{z}(\tau')\big)$:
	\begin{align*}
		\Xi\big(\PT_{z}(\tau')\big)=
		\Xi^\complement_z(\tau')=
		\bigsqcup_{k\ge 1}\Xi_{z,k}^\complement(\tau'),\qquad
		\tn{where}\quad
		\Xi_{z,k}^\complement(\tau'):=
		\{\,\fE\inn\Xi_k(\tau'):\,
		\fE\!\cap\!\De_z\eq\emptyset\,\}\,.
	\end{align*}
	The first identity above follows from Lemma~\ref{Lm:PT_trans_sect}.
	By the analogue of Lemma~\ref{Lm:PT_blowup_center} in the current setting,
	we see each $\wc Z_k'$ is the union of $X_\fE^{\cV_z}$, $\fE\inn \Xi_{z,k}^\complement(\tau')$,
	which justifies \ref{ConditionLocalLoci}  of Definition~\ref{DfnGaAdm} for $\PT_{\pi}(\pi')$.
	Since 
	\ref{ConditionOrder} naturally holds for $\PT_{\pi}(\pi')$,
	the statement of the corollary is established.
\end{proof}

\begin{prop}
	\label{Prp:PT_rl_path}
	Every root-to-leaf path of $\PT_{z}(\tau')$ is an inclusion-minimal element of 
	\begin{align*}
		S':=
		\big\{\,
		\wp'\bsl\De_z:\,
		\wp'~\tn{is~a~root-to-leaf~path~of~} \tau'
		\,
		\big\}\,,
	\end{align*}
	and vice versa.
\end{prop}

We remark that
each root-to-leaf path of $\PT_{z}(\tau')$,
as a subposet of $\PT_z(\tau')$,
is naturally a subposet of $\tau'$,
because so is $\PT_z(\tau')$.

\begin{proof}[Proof of Proposition~\ref{Prp:PT_rl_path}]
	If $\De_z$ contains a root-to-leaf path of $\tau'$,
	then on the one hand, we have $\emptyset\inn S'$, hence the only inclusion-minimal element of $S'$ is $\emptyset$;
	on the other hand, we have $\Xi^\complement_z(\tau')\eq\emptyset$ by Definition~\ref{Dfn:transverse_sections},
	so $\PT_z(\tau')$ is trivial,
	thus the only root-to-leaf path is the trivial poset. 
	So Proposition~\ref{Prp:PT_rl_path} is true when $\De_z$ contains a root-to-leaf path of $\tau'$.
	
	Below, we assume $\De_z$ does not contain any root-to-leaf path of $\tau'$ (which implies $\tau'$ is not trivial).
	
\begin{enumerate}[leftmargin=*,label=(\alph*)]
	\item 
	\label{Part:im_S'_subset_PT}
	We begin with showing every inclusion-minimal element of $S'$ is a root-to-leaf path of $\PT_z(\tau')$.
	
	\begin{itemize}[leftmargin=*]
		\item 
		To this end,
		we first show every inclusion-minimal element $\ti\wp\bsl\De_z$ of $S'$ is contained in $\PT_z(\tau')$.
		
		Notice each element of $S'$ is nonempty and thus has a unique least element (w.r.t.~the tree order),
		hence
		\begin{align*}
			\fE':=
			\big\{\,
			\min(\varsigma'):\,
			\varsigma'~\tn{is~an~inclusion-minimal~element~of}~S'
			\,\big\}
		\end{align*}
		is a well-defined nonempty subset of $\tau'\bsl\De_z$.
		We claim $$\fE'\in\Xi(\tau').$$
		
		In fact,
		for every root-to-leaf path $\wp'$ of $\tau'$,
		there exists a root-to-leaf path $\wp''$ of $\tau'$ such that 
		\begin{align*}
			\wp''\bsl \De_z~\tn{is~an~inlucsion-minimal~elment~of}~S'\qquad\tn{and}\qquad
			(\wp''\bsl\De_z)\subset(\wp'\bsl \De_z)\,.
		\end{align*}
		With $e''$ denoting the least element of $\wp''\bsl\De_z$, we have $$e''\in\big(\fE'\!\cap\!(\wp''\bsl\De_z)\big)
		\subset \big(\fE'\!\cap\!(\wp'\bsl\De_z)\big)
		\subset 
		\big(\fE'\!\cap\!\wp'\big).$$
		Suppose $\fE'\!\cap\!\wp'$ contains another edge $\wh e$ distinct from $e''$.
		Then, 
		either $\wh e\!\succ\!e''$ or $\wh e\!\prec\!e''$, because
		$e''$ and $\wh e$ belong to the linearly ordered set $\wp'$.
		However, $\wh e\inn\fE'$ implies
		there exists a root-to-leaf path $\wh\wp$ of $\tau'$ such that $\wh\wp\bsl\De_z$ is inclusion-minimal in $S'$,
		and $\wh e$ is the least element of $\wh\wp\bsl\De_z$.
		Consequently, one of $\wp''\bsl\De_z$ and $\wh\wp\bsl\De_z$ must contain the other as a strict subset,
		hence is not inclusion-minimal.
		This contradiction shows $e''$ is the only element of $\fE'\!\cap\!\wp'$,
		which by Definition~\ref{Dfn:transverse_sections} implies $\fE'\inn\Xi(\tau')$.
		
		Now, consider an arbitrary inclusion-minimal element $\ti\wp\bsl\De_z$ of $S'$, where $\ti\wp$ is a root-to-leaf path  of $\tau'$.
		For every $\ti e\inn\ti\wp\bsl\De_z$,
		we claim 
		\begin{align}\label{Eqn:ti_e}
			\ti e\in
			\ti\fE'\in\Xi(\tau'),\qquad
			\tn{where}\quad \ti\fE':=\max\big(\fE'\!\cup\! \{\ti e\}\big)\,.
		\end{align}
		\begin{itemize}[leftmargin=*]
			\item 
			The reason that $\ti\fE'\inn\Xi(\tau')$ is as follows.
			Notice that distinct maximal elements of $\fE'\!\cup\! \{\ti e\}$ are incomparable.
			In addition,
			for every $e\inn\tau'$,
			there exist $e'\inn\fE'$ comparable with $e$ because $\fE'\inn\Xi(\tau')$,
			as well as $e''\inn\ti\fE'$ satisfying $e''\!\succeq\!e'$.
			If $e'\!\succeq\!e$, then $e''\!\succeq\!e$;
			if $e'\!\prec\!e$, then by (\ref{Eqn:tree_order}),
			$e$ and $e''$ are comparable.
			In sum, we have $\ti\fE'\inn\Xi(\tau')$.
			
			\item 
			Suppose $\ti e\!\notin\!\ti\fE'$,
			then there would exist $\ti e'\inn\ti\fE'$ comparable with $\ti e$.
			From the definition of $\ti\fE'$ in (\ref {Eqn:ti_e}),
			we conclude that $\ti e\!\prec\!\ti e'$ and $\ti e'\inn\fE'$.
			However,
			the least element $\ti e^\flat$ of $\ti\wp\bsl\De_z$
			is contained in $\fE'$, satisfying 
			$\ti e^\flat\!\preceq\!\ti e\!\prec\!\ti e'$,
			which contradicts the fact that distinct elements of the transverse section $\fE'$ are incomparable.
			Therefore, $\ti e\!\in\!\ti\fE'$.
		\end{itemize}
		The above two paragraphs then give rise to  (\ref{Eqn:ti_e}).
		Taking $\fE'\!\cap\!\De_z\eq\emptyset$ and $\ti e\!\notin\!\De_z$ into consideration, we obtain $$\ti e\in\ti\fE'\in\Xi^\complement_z(\tau'),$$
		hence $\ti e\inn\PT_z(\tau')$.
		In other words,
		the inclusion-minimal element $\ti\wp\bsl\De_z$ of $S'$ is contained in $\PT_z(\tau')$.

		\item We then show the above $\ti\wp\bsl\De_z$ is a root-to-leaf path of $\PT_z(\tau')$,
		i.e.~a  linearly ordered subset of $\PT_z(\tau')$ that is maximal with respect to inclusion.
		
		Indeed,
		it is linearly ordered because so is $\ti\wp$.
		
		Suppose it is not maximal with respect to inclusion,
		then there exists a root-to-leaf path $\wp$ of $\PT_z(\tau')$ containing $\ti\wp\bsl\De_z$ as a strict subset.
		Take $e\inn\wp$ that is not in $\ti\wp\bsl\De_z$.
		By (\ref{Eqn:PT_tree}),
		there exists $\fE\inn\Xi(\tau')$ that is disjoint from $\De_z$ and contains $e$.
		Since $\ti\wp$ is a root-to-leaf path of $\tau'$,
		it meets $\fE$ at a unique edge $e_0$.
		\begin{itemize}[leftmargin=*]
			\item If $e_0\inn\De_z$,
			then $e_0\inn(\fE\!\cap\!\De_z)$,
			contradicting the fact that $\fE$ is disjoint from $\De_z$.
			
			\item If $e_0\inn(\ti\wp\bsl\De_z)$,
			then $e_0\!\ne\!e$ because $e\!\notin\!(\ti\wp\bsl\De_z)$.
			However, by the choice of $\wp$, we have
			$e_0\inn \wp$. 
			Since $\wp$ is linearly ordered, $e_0$ and $e$ are  comparable,
			contradicting the fact that they are distinct elements of the transverse section $\fE$. 
		\end{itemize}
		The above contradictions indicate $\ti\wp\bsl\De_z$ is a root-to-leaf path of $\PT_z(\tau')$.
		This completes the proof of Part~\ref{Part:im_S'_subset_PT}.
	\end{itemize}
	
	\item We next show every root-to-leaf path $\wp$ of $\PT_{z}(\tau')$ is an inclusion-minimal element of $S'$.
	
	To see this, notice that $\wp$
	is a priori a linearly ordered subset of $\tau'$.
	Therefore, there exists a maximal linearly ordered subset $\wp'$ of $\tau'$, i.e.~a root-to-leaf path of $\tau'$,
	such that $\wp\!\subset\!\wp'$.
	Since the underlying set of $\PT_{z}(\tau')$ is contained in $\tau'\bsl\De_z$,
	we have $\wp\!\subset\!(\wp'\bsl\De_z)$.
	Then, there exists a root-to-leaf path $\wp''$ of $\tau'$ such that
	\begin{align*}
		\wp''\bsl \De_z~\tn{is~an~inlucsion-minimal~elment~of}~S'\qquad\tn{and}\qquad
		(\wp''\bsl\De_z)\subset(\wp'\bsl \De_z)\,.
	\end{align*}
	We will show $\wp\eq(\wp''\bsl\De_z)$.
	
	Recall by Part~\ref{Part:im_S'_subset_PT},
	$\wp''\bsl\De_z$ is a root-to-leaf path of $\PT_z(\tau')$.
	Therefore, in order to show $\wp\eq(\wp''\bsl\De_z)$,
	it suffices to show $\wp\!\subset\!(\wp''\bsl\De_z)$.
	Indeed,
	for every $e\inn\wp$ ($\subset\!\PT_{z}(\tau')$),
	by (\ref{Eqn:PT_tree}),
	there exists $\fE\inn\Xi^\complement_z(\tau')$ containing $e$.
	As a transverse section of $\tau'$,
	$\fE$ meets $\wp''$ at a unique edge $e''$.
	Since $\fE\!\cap\!\De_z\eq\emptyset$,
	we have 
	\begin{align*}
		e\in\wp\subset 
		(\wp'\bsl\De_z),\qquad
		e''\in (\wp''\bsl\De_z)\subset 
		(\wp'\bsl\De_z)\,.
	\end{align*}
	This implies
	$e$ and $e''$ both belong to $\wp'\!\cap\!\fE$,
	which is a singleton by Definition~\ref{Dfn:transverse_sections}.
	Therefore, $e\eq e''$,
	thus $\wp\!\subset\!(\wp''\bsl\De_z)$.
	\qedhere		
\end{enumerate}	
\end{proof}

\subsubsection{Proof of \ref{Case:5}.\ref{Case:5.5}, Part III,  ($\rd_3\ph_4$)}
\label{Subsubsec:Case_A.5.III_2}
We apply the results of \S\ref{Subsubsec:PT} to the fixed lift $z\inn\ti\fM^{\tn{div},\rd_3\ph_3}$ of $y\inn\ti\fM^{\tn{div},\rd_1\ph_5}$.
By Corollary~\ref{Crl:PT},
$(\rd_3\ph_4)$ is $\PT_{z}\big(\ex(\varrho_y)\big)$-compatible on $\cV_z$,
so after $(\rd_3\ph_4)$, the pullback of $\cV_z$ is smooth.

It remains to verify the pullback of (\ref{Eqn:phi_A.5.III}) is diagonalized after $(\rd_3\ph_4)$.
\begin{itemize}[leftmargin=*]
\item 
If $\De_z$ contains the grafted edge (corresponding to the parameter $\wc\ka_{12}$) of $\ex(\tau_2)$,
then the pullback of (\ref{Eqn:phi_A.5.III}) to $\cV_z$ is already diagonalized.

\item 
If $\De_z$ does not contain the grafted edge of $\ex(\tau_2)$,
then the second row of the pullback of
(\ref{Eqn:phi_A.5.III}) takes the form
\begin{align*}
	\big(\!\prod_{k\in\lrbr\cht}\!\!\wc\ve_k\,\big)
	\big(\!\prod_{k\in\lrbr{\cht^{\rd_3\ph_3}}}\!\!\!\!\wc\ve_k^{\,\rd_3\ph_3}\,\big)
	\left[\;
	\begin{matrix} 
		0&
		\wc\ka_{12} &
		\wc \eta_3
		&\cdots
		&
		\wc \eta_m
	\end{matrix}\,
	\right],
	\quad
	\tn{where}\ \ 
	\wc\eta_i:=
	\big(\!\prod_{k\in{\lrbr \cht}_{[\de_2\wedge \de_i]}}\!\!\!\!\!\!\!\!\wc\ve_k\ \big)
	\big(\!\prod_{e\in N_{[\de_{i}]}\bsl(\De_y\sqcup\De_z)}\!\!\!\!\!\!\!\!\!\wc\ze_e~\big).
\end{align*}
Here, the superscript $\rd_3\ph_3$ indicates the parameters corresponding to the exceptional divisors obtained in ($\rd_3\ph_3$).
Notice that
for each $3\!\le\!i\!\le\!m$ with
$N_{[\de_i]}\bsl\De_y\!\subset\!\De_z$,
the corresponding $\wc\eta_i\eq \prod_{k\in{\lrbr \cht}_{[\de_2\wedge \de_i]}}\wc\ve_k$.
So by Proposition~\ref{Prp:PT_rl_path},
the inclusion-minimal elements of the set 
\begin{align*}
	\{e_\ex\}\sqcup
	\big\{\,
	{\lrbr \cht}_{[\de_2\wedge \de_i]}\sqcup 
	\big(N_{[\de_{i}]}\bsl(\De_y\sqcup\De_z)\big):\,
	3\!\le\!i\!\le\!m
	\,\big\}
\end{align*}
are exactly the root-to-leaf paths of $\PT_z\big(\ex(\varrho_y)\big)$.
Hence by Proposition~\ref{PrpDominating},
on the pullback of $\cV_z$,
the pullback of
(\ref{Eqn:phi_A.5.III}) is diagonalized.
\end{itemize}

To summarize,
in this subsection,
we have established the following.
\begin{prop}\label{Prp:Case_A.5.3}
	Proposition~\ref{PrpChangeofPhi} holds in \ref{Case:5}.\ref{Case:5.5} if $|D(\De_y)|\!\ge\!1$ and the point of $D(\De_y)$ is  conjugate to the point $D\!\cap\!F$.
\end{prop}

Propositions~\ref{Prp:Case_A.5.1}, \ref{Prp:Case_A.5.2} and~\ref{Prp:Case_A.5.3} together give rise to the following.
\begin{prop}
	\label{Prp:phi_M5_core_wt1}
	Proposition~\ref{PrpChangeofPhi} holds in \ref{Case:5}.\ref{Case:5.5}.
\end{prop}

We conclude \S\ref{SubsecTwoTrees} with an example that demonstrates the idea of proper transforms of blowups as well as the structure of the proof of Part III.

\begin{exam}\label{Eg:PT}
	Let $\wh x\eq(C,D)\inn\fM_2^{\tn{div}}$ be such that $C$ is the same rooted tree as in Figure~\ref{Fig:abcd}, but the distribution of the elements of $D$ is different:
	$$
	D\!\cap\!F\eq\{\de_1\},\quad
	D\!\cap\!C_b\eq\{\de_2\},\quad
	D\!\cap\!C_a\!\ni\!\{\de_3\},\quad 
	D\!\cap\!C_c\!\ni\!\{\de_4\},\quad 
	D\!\cap\!C_d\!\ni\!\{\de_5\}.
	$$
	Moreover,
	we assume $\de_1$ and $\de_2$ are conjugate.
	This implies $\de_1$ and $\de_3$ cannot be conjugate,
	nor can $\de_3$ be conjugate to $\de_4$ or $\de_5$.
	
	By Proposition~\ref{Prp:phi_key}, under suitable trivialization, 
	$\varphi$ on a chart $\cV$ containing $(C,D)$ can be written as
	$$
	\left[
	\begin{matrix}
		1 & 0 & 0 & 0 & 0 & 0 & \cdots\\
		0 & {\ka}_{12}\ze_b & \ze_a & \ze_b^2\ze_c & \ze_b^2\ze_d & 0 & \cdots
	\end{matrix}
	\right]\, .
	$$
	
	As shown in \S\ref{Subsec:Case5.5},
	$(\rd_1\ph_5)$ is compatible with $\tau_1$ (\ref{Eqn:tree_A.5}) near $\wh x$.
	Here,
	$\tau_1$ is given by the set $\{q_a,q_b\}$ along with the trivial order.
	
	After $(\rd_1\ph_5)$,
	we take a lift $y$ of $\wh x$ whose RLS and dominant edges are respectively given by
	\begin{align*}
		\ov\bE=\big\{\{q_a,q_b\}\big\},\qquad
		\De_y=\{q_b\}.
	\end{align*}
	Locally near $y$,
	the pullback (\ref{Eqn:phi_A.5.III}) of $\varphi$ can be further simplified as
	$$
	\left[
	\begin{matrix}
		1 & 0 \\
		0 & \ve_1^{\rd_1\ph_5}
	\end{matrix}
	\right]
	\left[
	\begin{matrix}
		1 & 0 & 0 & 0 & 0 & 0 & \cdots\\
		0 & \wc{\ka}_{12} & \wc\ze_a & \ve_1^{\rd_1\ph_5}\ze_c & \ve_1^{\rd_1\ph_5}\ze_d & 0 & \cdots
	\end{matrix}
	\right]\, .
	$$
	Here, $\ve_1^{\rd_1\ph_5}$ corresponds to the exceptional divisor after blowing up along $(\ze_a\eq\ze_b\eq 0)$.
	
	Let $\tau_{2}$ be the rooted tree given by the set $\{q_a,q_c,q_d\}$ along with the trivial order, which is the analogue of (\ref{Eqn:tree_A.4}) in this situation.
	In addition,
	let $\varrho_y$ be the first-order derived tree at $y$ as in (\ref{Eqn:tree_A.5}),
	which is given by the set $\{1,q_a\}$ along with the trivial order.
	
	As shown in the beginning of \S\ref{Subsubsec:Case_A.5.III},
	$(\rd_3\ph_3)$ is compatible with $\ex(\tau_2)$ near $y$,
	where the grafted edge corresponds to $\wc\ka_{12}$.
	After $(\rd_3\ph_3)$, we fix a lift $z$ of $y$.
	As shown in \S\ref{Subsubsec:Case_A.5.III_2},
	$(\rd_3\ph_4)$ is compatible with $\PT_z\big(\ex(\varrho_y)\big)$ near $z$,
	where the grafted edge corresponds to $\wc\ka_{12}$.
	
	If either $\wc\ka_{12}$ or $q_a$ is contained in $\De_z$,
	then on the one hand,
	we immediately see the pullback of $\varphi$ becomes diagonalized near any lift of $z$.
	On the other hand,
	we have $\PT_z\big(\ex(\varrho_y)\big)\eq\tau_\bullet$ because the only transverse section of $\ex(\varrho_y)$ is $\{e_\ex,1,q_a\}$;
	in other words, $(\rd_3\ph_4)$ does not affect a neighborhood of $z$.
	
	If neither $\wc\ka_{12}$ nor $q_a$ is contained in $\De_z$,
	then at least one between $q_c$ and $q_d$,
	say $q_c$, is in $\De_z$.
	On the one hand,
	near $z$, the pullback of $\varphi$ can be rewritten as 
	$$
	\left[
	\begin{matrix}
		1 & 0 \\
		0 & \ve_1^{\rd_1\ph_5}\ve_1^{\rd_3\ph_3}
	\end{matrix}
	\right]
	\left[
	\begin{matrix}
		1 & 0 & 0 & 0 & 0 &  \cdots\\
		0 & \wc{\ka}_{12} & \wc\ze_a & \ve_1^{\rd_1\ph_5} &  0 & \cdots
	\end{matrix}
	\right]\, .
	$$
	Here, $\ve_1^{\rd_1\ph_5}$ corresponds to the exceptional divisor after blowing up along $(\wc\ze_a\eq\ze_c\eq\ze_d\eq 0)$.
	(By abuse of notation, we still write the proper transform of $\ve_1^{\rd_1\ph_5}$ as $\ve_1^{\rd_1\ph_5}$.)
	On the other hand,
	we have $\PT\big(\ex(\varrho_y)\big)\eq\ex(\varrho_y)$ because the only transverse section $\{e_\ex,1,q_a\}$ of $\ex(\varrho_y)$ does not contain $q_c$ or $q_d$.
	In $(\rd_3\ph_4)$, the locus $(\wc\ka_{12}\eq\ve_1^{\rd_1\ph_5}\eq\wc\ze_a\eq 0)$ is blown up,
	hence the pullback of $\varphi$ becomes diagonalized near any lift of $z$.
\end{exam}

\subsection{Proof of \ref{Case:4}}
\label{Subsec:Case_B}
In this subsection, we prove Proposition~\ref{PrpChangeofPhi} in \ref{Case:4}.

Let 
$
\cV\to\fM_2^{\rm div}
$ be a small affine smooth chart containing~$\wh x\eq(C,D)$.
In~\ref{Case:4},
we assume~$x$, the image  of $\wh x$ in $\fM_2\wt$, lies in $\fM^\mn_{(4)}$.
This implies the core $F$ of $\wh x$ can be written as
\begin{align*}
	F=F_1\cup \tn B\cup F_2,
\end{align*}
where $F_1$ and $F_2$ are both genus 1 inseparable components, and $\tn B$ is either a separating bridge or a node, such that any two among $F_1$, $F_2$ and $\tn B$ do not share any common irreducible component.
W.l.o.g. we assume that 
\begin{gather*}
	a_1\in F_1\,,\qquad
	a_2\in F_2\,,\qquad\tn{and}\qquad
	D\!\cap\! F
	=\{\de_1,\cdots,\de_\ell\}
	\subset F_1\!\cup\! \tn B
\end{gather*}
for some $\ell\inn\lrbr m$.

Notice that $D\!\cap\! F$ cannot be entirely contained in $\tn B$, for otherwise $x$ would be in $\fM_{(2)}^\mn$.
Hence $D$ has an element, say $\de_1$, such that
\begin{align}\label{Eqn:r1p4_de_1}
	\de_1\inn F_1,\qquad
	\tn{i.e.}\qquad
	N_{[\de_1,a_1]}=\emptyset\,.
\end{align}
So by Proposition~\ref{Prp:phi_key},
after suitable elementary row and column operations, $\varphi$ can be written as
\begin{equation}
	\label{Eqn:Case_B_phi}
	\varphi=\left[
	\begin{matrix} 
		1 & 0 &  \cdots & 0\\
		0 & \ka_{12}\ze_{[\de_2,a_2]} & \cdots 
		& \ka_{1m}\ze_{[\de_m,a_2]}
	\end{matrix} 
	\right]\,.
\end{equation}
Also notice $D\!\cap\! F$ cannot belong to any non-separating bridge of $F_1$ (should there be any),
for otherwise $x$ would be in $\fM_{(3)}^\mn$.
Therefore,
($\rd_1\ph_1$)-($\rd_1\ph_3$) do not affect $\cV$.

Mimicking (\ref{Eqn:tau_CaseA.4}),
we set
\begin{align*}
	&
	D_{\im}:=\big\{\,\de\inn D:\,
	N_{[\de,a_2]}~\tn{is~inclusion-minimal~among~all}~N_{[\de',a_2]},~\de'\inn D\,\big\},
	\\
	&
	E:=\!\bigcup_{\de\in D_{\im}}\!\!\!N_{[\de,a_2]}\ \  \big(\subset N(C)\big).
\end{align*}
On the set $E$,
we define the relation $\preceq$ such that $e\!\preceq\!e'$ if and  only if any connected subcurve of $C$ containing $F_2$ and $e$ must contain $e'$,
which once again satisfies (\ref{Eqn:tree_order}),
i.e.~the poset
\begin{align}\label{Eqn:tree_B}
	\tau^\dag:=(E,\preceq)
\end{align}
is a rooted tree.
The root-to-leaf paths of $\tau^\dag$ are exactly $N_{[\de,a_2]}$, $\de\inn D_{\im}$.

An example of $\tau^\dag$ is provided in Figure~\ref{Fig:B}.
\begin{figure}[htb]
	\begin{center}
		\begin{tikzpicture}
			\def\g1{
				(-1,0) ellipse (1 and 0.5)
				(-1.4,0)..controls(-1,0.1)..(-0.6,0)
				(-0.6,0)..controls(-1,-0.1)..(-1.4,0)
				(-0.5,0.05)--(-0.6,0)
				(-1.4,0)--(-1.5,0.05)
			}
			\def\halfg1{
				(-.32,0)..controls(-.32,-.1) and (-.4,-.2)..(-.6,-.2)
				(-.6,-.2)..controls(-.8,-.2) and (-1,-.25)..(-1,-.35)
				(-.6,-.5)..controls(-.8,-.5) and (-1,-.45)..(-1,-.35)
				(-.6,-.5)..controls(-.25,-.5) and (0,-.35)..(0,0)
				(-.32,0)..controls(-.32,.1) and (-.4,.2)..(-.6,.2)
				(-.6,.2)..controls(-.8,.2) and (-1,.25)..(-1,.35)
				(-.6,.5)..controls(-.8,.5) and (-1,.45)..(-1,.35)
				(-.6,.5)..controls(-.25,.5) and (0,.35)..(0,0)
			}
			\def\crs{
				(-.03,-.03)--(.03,.03)
				(-.03,.03)--(.03,-.03)
			}
			
			\draw\g1;
			\draw[xshift=-1cm]
			(-1.4,0) circle (.4cm)
			;
			\draw[xshift=-2.8cm]
			\halfg1;
			\draw[xshift=-4.8cm,xscale=-1]
			\halfg1;
			
			\draw
			(-2.4,.8) circle (.4cm)
			(-1,.9) circle (.4cm)
			;
			
			\draw[xshift=-4.5cm,yshift=.3cm]
			\crs			
			;
			\draw[xshift=-4cm,yshift=-.32cm]
			\crs			
			;
			\draw[xshift=-2.5cm,yshift=1cm]
			\crs			
			;
			\draw[xshift=-.8cm,yshift=.95cm]
			\crs			
			;
			\draw[xshift=-1.1cm,yshift=1.05cm]
			\crs			
			;
			\draw[xshift=-3cm,yshift=-.1cm]
			\crs			
			;
			
			\draw[xshift=4cm]
			(0,0)--(1.2,1.2)
			(0,0)--(-.6,.6)
			(.6,.6)--(0,1.2)
			;
			\filldraw[xshift=4cm]
			(0,0) circle (2pt)
			(.6,.6) circle (2pt)
			(-.6,.6) circle (2pt)
			(1.2,1.2) circle (2pt)
			(0,1.2) circle (2pt)
			;
			\draw[xshift=4cm]
			(.2,.2) node[right] {\scriptsize{$p$}}
			(-.15,.2) node[left] {\scriptsize{$a$}}
			(.85,.85) node[right] {\scriptsize{$b$}}
			(.4,.85) node[left] {\scriptsize{$q$}}
			;
			
			\draw [xshift=2cm]
			(1.9,-.2) node {\tiny{$o$}}
			(2.5,-.36) node {{$\tau^\dag$}};
			
			\draw
			(0,-.4) node[right] {{$x\eq (C,D)$}}
			(-2.05,0) node[right] {\tiny{$p$}}
			(-2.85,0) node[right] {\tiny{$q$}}
			(-3.8,.4) node[above] {\tiny{$r$}}
			(-3.8,-.4) node[below] {\tiny{$s$}}
			(-1,.45) node[above] {\tiny{$a$}}
			(-2.4,.35) node[above] {\tiny{$b$}}
			;			
		\end{tikzpicture}
	\end{center}
	\caption{An example of $\tau^\dag$}\label{Fig:B}
\end{figure}

\begin{lemm}
	\label{Lm:r1p4_tree}
	Given $\de\inn D\bsl F$,
	let $\lr\de\inn F$ be as in (\ref{Eqn:lr}).
	If $\lr\de\inn F_1$,
	then $\de\!\notin\!D_{\im}$.
\end{lemm}

\begin{proof}
	If $\de\!\notin\! F$ and $\lr\de\inn F_1$,
	then $N_{[\de]}\!\ne\!\emptyset$ and $N_{[\de,a_2]}\eq N_{[\de]}\!\sqcup\!N_{[a_2,F_1]}$.
	Since $\de_1\inn D\!\cap\!F$,
	we have $N_{[\de_1,a_2]}\!\subset\!N_{[a_2,F_1]}\!\subsetneq\!N_{[\de,a_2]}$,
	thus $\de\!\notin\!D_{\im}$.
\end{proof}

As described in \S\ref{rd1},
the blowup centers of $(\rd_1\ph_4)$ are locally given by
\begin{align*}
	&\ov\fM_{(4,k)}\!\cap\!\cV=\!\!
	\bigcup_{\fE\in\Xi_k(\tau^\dag)}\!\!
	\{\,
	\ze_e\eq 0:e\inn\fE\,\},
	\qquad
	\tn{where}\qquad
	\Xi_k(\tau^\dag)=\{\fE\inn\Xi(\tau^\dag):|\fE|\eq k\}.
\end{align*}
From the same argument as in the proof of Proposition~\ref{Prp:phi_M5_core_wt2},
we conclude that $(\rd_1\ph_4)$ is $\tau^\dag$-compatible on $\cV$.
Lemma~\ref{LmLocalLoci} then implies the pullback $\ti\cV^{\rd_1\ph_4}$ of $\cV$ is smooth.

After $(\rd_1\ph_4)$,
we fix an arbitrary lift $y\inn \ti\cV^{\rd_1\ph_4}$ of $\wh x$, as well as a small neighborhood $\cV_y$ of $y$ in $\ti\cV^{\rd_1\ph_4}$.
The RL sequence of $y$ is denoted by $\ov \bE\eq \{\fE_1,\ldots,\fE_\cht\}$ as before.
By Definition~\ref{DfnDominant}, there exists $i\inn\lrbr m$ such that $N_{[\de_i,a_2]}\!\subset\!\De_y$;
such $\de_i$ is naturally contained in $D_{\im}$.
Mimicking the notation of (\ref{Eqn:I_y}),
we write
\begin{align*}
	D(\De_y):=
	\big\{\,\de\inn D:\,
	N_{[\de,a_2]}\!\subset\!\De_y
	\big\}\quad
	\big(=
	\big\{\,\de\inn D_{\im}:\,
	N_{[\de,a_2]}\!\subset\!\De_y
	\big\}\,\big)\,.
\end{align*}

\begin{itemize}[leftmargin=*]
	\item 
If $D(\De_y)\!\not\subset\! F$,
there exists $i\!>\!\ell$ such that $\de_i\inn D(\De_y)\!\subset\!D_{\im}$.
Lemma~\ref{Lm:r1p4_tree} then implies $\lr{\de_i}\!\notin\!F_1$.
From (\ref{Eqn:r1p4_de_1}) and Lemma~\ref{Lm:K_position},
we conclude that $\de_1$ and $\de_i$ are not conjugate,
i.e.~the proper transform of $\ka_{1i}$ (which is equal to the pullback in this case) does not vanish at $y$,
hence is invertible on the small neighborhood $\cV_y$.
So by Proposition~\ref{PrpDominating},
the pullback of $\ka_{1i}\ze_{[\de_i,a_2]}$ divides all the other terms of the second row of (\ref{Eqn:Case_B_phi}).
That is, after suitable elementary column operations,
(\ref{Eqn:Case_B_phi}) becomes diagonalized.
Moreover,
it is straightforward that ($\rd_1\ph_5$)-($\rd_3\ph_4$) do not affect $\cV_y$,
so the pullback of $\cV_y$ to $\ti\fM_2^{\rm div}$ is smooth.

\item
If
$D(\De_y)\!\subset\! F$,
then all the separating nodes of $F$ are in $\De_y$,
hence 
\begin{align*}
	N_{[\de_i,a_2]}\bsl\De_y=N_{[\de_i]}\bsl\De_y\qquad
	\forall~i\inn\lrbr m\,.
\end{align*}
Therefore, 
the pullback of (\ref{Eqn:Case_B_phi}) to $\cV_y$,
after suitable elementary column operations, can be written as
\begin{equation}\begin{split}
	\label{Eqn:Case_B_phi'}
	&
	\left[
	\begin{matrix} 
		1 & 0\\
		0 & \prod_{k\in\lrbr\cht}\wc\ve_k
	\end{matrix} 
	\right]
	\left[
	\begin{matrix} 
		1 & 0 &  \cdots & 0\\
		0 & \wc\ka_{12}\prod_{e\in N_{[\de_{2}]}\bsl\De_y}\!\!\wc\ze_e & \cdots & 
		\wc\ka_{1m}\prod_{e\in N_{[\de_{m}]}\bsl\De_y}\!\!\wc\ze_e
	\end{matrix} 
	\right]\,,		
\end{split}\end{equation}
where $\wc\ve_k$ are the local parameters on $\cV_y$ corresponding to the exceptional divisors obtained in $(\rd_1\ph_4)$;
c.f.~Lemma~\ref{LmCoordPT}.
W.l.o.g.~we assume there exists $j\inn\lrbr\ell$ such that
\begin{align*}
	&
	D(\De_y)=\{\de_h,\cdots,\de_\ell\}\quad\big(\subset F\big)\,.
\end{align*}
Then, we have
\begin{align}
	\label{Eqn:r1p4_dom_core}
	\prod_{e\in N_{[\de_{i}]}\bsl\De_y}\!\!\!\!\!\!\wc\ze_e
	\in\Ga(\sO^*_{\cV_y})\qquad\forall~
	h\!\le\!i\!\le\!\ell\,.
\end{align}
\begin{itemize}[leftmargin=*]
	\item 
In \ref{Case:4}.\ref{Case:4.1},
recall $D\!\cap\! F$ cannot belong to any non-separating bridge of $F_1$ (should there be any),
for otherwise $x$ would be in $\fM_{(3)}^\mn$.
Therefore,
there exists $h\!\le\!j\!\le\!\ell$ such that $\de_j$ is not conjugate to $\de_1$,
i.e.~$\wc\ka_{1j}$ is invertible on $\cV_y$.
Along with (\ref{Eqn:r1p4_dom_core}),
this implies in the second row of (\ref{Eqn:Case_B_phi'}),
the $j$-th entry divides the other entries,
hence (\ref{Eqn:Case_B_phi'}) is diagonalized on $\cV_y$.
Moreover,
it is straightforward that ($\rd_1\ph_5$)-($\rd_3\ph_4$) do not affect $\cV_y$,
so the pullback of $\cV_y$ to $\ti\fM_2^{\rm div}$ is smooth.

\item 
In~\ref{Case:4}.\ref{Case:4.2},
the same argument as above can be applied verbatim.

\item 
In~\ref{Case:4}.\ref{Case:4.3},
we have $h\eq 1$, $\ell\eq 2$, and $\wc\ka_{12}$ vanishes at $y$.
Once again,
every $\de_i$, $i\!\ge\!3$, is not conjugate to $\de_1$ because $\de_1$ and $\de_2$ are not on any non-separating bridge.
Taking (\ref{Eqn:r1p4_dom_core}) into consideration,
we can rewrite (\ref{Eqn:Case_B_phi'}) as
\begin{equation}\begin{split}\label{Eqn:Case_B_phi''}
	&
	\left[
	\begin{matrix} 
		1 & 0\\
		0 & \prod_{k\in\lrbr\cht}\wc\ve_k
	\end{matrix} 
	\right]
	\left[
	\begin{matrix} 
		1 & 0 & 0 & \cdots & 0\\
		0 & \wc\ka_{12} & \prod_{e\in N_{[\de_{3}]}\bsl\De_y}\!\!\wc\ze_e & \cdots & 
		\prod_{e\in N_{[\de_{m}]}\bsl\De_y}\!\!\wc\ze_e
	\end{matrix} 
	\right]\,,		
\end{split}\end{equation}
In this situation, it is straightforward that ($\rd_1\ph_5$)-($\rd_3\ph_2$) do not affect $\cV_y$.
Let $\tau_2$ be the rooted tree as in (\ref{Eqn:tree_A.4}),
whose root-to-leaf paths are those inclusion-minimal $N_{[\de_i]}$ among all $3\!\le\!i\!\le\!m$.
Then, by Corollary~\ref{Crl:PT} and the proof of Proposition~\ref{Prp:phi_M5_core_wt2}, 
we see $(\rd_3\ph_3)$ is $\PT_y\big(\ex(\tau_2)\big)$-compatible on $\cV_y$,
so the pullback of $\cV_y$ to $\ti\fM^{\tn{div},\rd_3\ph_3}$,
denoted by $\ti\cV_y^{\rd_3\ph_3}$, is smooth.
In addition, by Propositions~\ref{Prp:PT_rl_path} and~\ref{PrpDominating},
the pullback of (\ref{Eqn:Case_B_phi''}) is diagonalized on $\ti\cV_y^{\rd_3\ph_3}$.
Moreover,
it is straightforward that ($\rd_3\ph_4$) does not affect $\ti\cV_y^{\rd_3\ph_3}$ because ($\rd_1\ph_5$) does not affect $\cV_y$;
therefore,  the pullback of $\cV_y$ to $\ti\fM^{\tn{div}}_2$ is smooth.

\item 
In~\ref{Case:4}.\ref{Case:4.4},
we have $h\eq \ell\eq 1$.
In this situation, 
let $\tau_1$ be the rooted tree as in (\ref{Eqn:tree_A.5}).
Then, by the proof of Proposition~\ref{Prp:phi_M5_core_wt1}, we see $(\rd_1\ph_5)$ is $\PT_y(\tau_1)$-compatible on $\cV_y$.
Consequently, the three cases of the proof of Proposition~\ref{Prp:phi_M5_core_wt1},
namely those of Propositions~\ref{Prp:Case_A.5.1}, \ref{Prp:Case_A.5.2}, and \ref{Prp:Case_A.5.3},
can be applied to show
the pullback of $\cV_y$ to $\ti\fM^{\tn{div}}_2$ is smooth,
on which the pullback of (\ref{Eqn:Case_B_phi'}) is diagonalized.
\end{itemize}
\end{itemize}

To summarize, 
we have shown the following.

\begin{prop}\label{PrpChangeofPhiM4}
	Proposition~\ref{PrpChangeofPhi} holds in \ref{Case:4}. 
\end{prop}

\subsection{Proof of \ref{Case:3}}
\label{Subsec:Case_C}

In this subsection, we prove Proposition~\ref{PrpChangeofPhi} in \ref{Case:3}.

Let 
$
\cV\to\fM_2^{\rm div}
$ be a small affine smooth chart containing~$\wh x\eq(C,D)$.
In~\ref{Case:3},
we assume~$x$, the image  of $\wh x$ in $\fM_2\wt$, lies in $\fM^\mn_{(3)}$.
This implies the core $F$ of $x$ contains a 
{\it maximal}
non-separating bridge $\tn B \eq \tn B[\fp,\fq]$ so that
$$
D\cap \tn B = D\cap F.
$$
W.l.o.g.~we assume there exist $1\!\le\!h\!\le\!\ell\!\le\! m$ such that
\begin{align*}
	D\!\cap\! F
	=\{\de_1,\ldots,\de_h\},\qquad
	\{\de\inn D:\lr{\de}\inn\tn B\}
	=\{\de_1,\cdots,\de_\ell\}\,,
\end{align*}
where $\lr\de$ are as in (\ref{Eqn:lr}).
Furthermore,
we assume
\begin{align*}
	N_{[\de_1,\fp]}\!\subset\!N_{[\de_i,\fp]}\quad\forall~i\inn\lrbr{h}.
\end{align*}
In addition, if $|D\!\cap\!F|\!\ge\!2$ (i.e.~$h\!\ge\!2$),
then we assume
\begin{align*}
	N_{[\de_2,\fq]}\!\subset\!N_{[\de_i,\fq]}\quad\forall~i\inn\lrbr{h}.
\end{align*}

Let $F_1$ be the subcurve of $F$ as in Definition~\ref{Dfn:bridge}, determined by 
$\tn B\!\cup\!F_1\eq F$ and $\tn B\!\cap\!F_1\eq\{\fp,\fq\}$. 
The subcurve $F_1$ contains a unique genus 1 inseparable component,
denoted by $T_1$.
Notice $T_1$ is strictly smaller than $F_1$ if and only if $F$ is separable.

By Part~\ref{AssI} of
Assumption~\ref{Ass:basic}, 
if $F$ is inseparable,
then $a_1$ and $a_2$ are both on $T_1$ ($=\! F_1$);
if $F$ is separable,
w.l.o.g.~we assume $a_1$ and  $\tn B$ belong to the same genus 1 inseparable component,
whereas $a_2\inn T_1$.
In either case,
we have $N_{[\de_1,a_1]}\eq\emptyset$,
hence after suitable row and column operations, the structural homomorphism
$\varphi$ still takes the form of (\ref{Eqn:Case_B_phi});
that is,
\begin{equation}
	\label{Eqn:Case_C_phi}
	\varphi=\left[
	\begin{matrix} 
		1 & 0 &  \cdots & 0\\
		0 & \ka_{12}\ze_{[\de_2,a_2]} & \cdots 
		& \ka_{1m}\ze_{[\de_m,a_2]}
	\end{matrix} 
	\right]\,.
\end{equation}
Moreover, in either case,
\begin{align}\label{Eqn:sep_nod_Case_3}
	N_{[\de_i,a_2]}=N_{[a_1,a_2]}\qquad
	\forall~i\inn\lrbr h.
\end{align}

For the proof of \ref{Case:3},
we delve more deeply into (\ref{Eqn:Case_C_phi}).
Notice that 
\begin{align*}
	\ka_{1i}\in\Ga(\sO^*_\cV)\quad\forall~\ell\!<\!i\!\le\!m,\qquad
	\ka_{1i}(\wh x)\eq 0\quad\forall~i\inn\lrbr{\ell}.
\end{align*}
More precisely, for $i\inn\lrbr{\ell}$, we have the follows.
Recall for every $\de\inn D$,
$z_{\lr\de}$ is a smooth point on $\tn B$ that is sufficiently close to $\lr\de$;
see Lemmas~\ref{Lm:W_smoothness} and~\ref{Lm:K_smoothness}.
\begin{itemize}[leftmargin=*]
	\item When $h\!\ge\!2$,
	by Proposition~\ref{Prp:phi_key}~\ref{Part:kappa},
	we have
	\begin{equation}\begin{split}\label{Eqn:ka_Case_C}
		&\ka_{12}=
		f_2\ze_{[\de_1,\fp]}+g_2\ze_{[\de_2,\fq]}\,,\qquad
		\ka_{1i}=
		v_i\ka_{12}+
		g_i'\ze_{[\de_2,\fq]}\,,
		\quad 2\!<\!i\!\le\!h\,,\\
		&\ka_{1i}=
		v_i\ka_{12}+g_i'
		\be_i,\quad h\!<\!i\!\le\!\ell,
		\qquad\tn{where}\quad
		\be_i:=\begin{cases}
			\ze_{[z_{\lr{\de_i}},\fp]}
			&\tn{if}~[z_{\lr{\de_i}},\fp]\!\subsetneq\![\de_1,\fp],\\
			\ze_{[z_{\lr{\de_i}},\fq]}
			&\tn{if}~[z_{\lr{\de_i}},\fq]\!\subsetneq\![\de_2,\fq],\\
			\ze_{[\de_2,\fq]}
			&\tn{otherwise},
		\end{cases}\end{split}
	\end{equation}
	and $f_2$, $g_2$, $v_i$ and $g'_i$ are invertible functions on $\cV$.
	Taking (\ref{Eqn:sep_nod_Case_3}) into consideration, we can apply suitable column operations to 
	(\ref{Eqn:Case_C_phi}) and obtain the follows.
	\begin{itemize}[leftmargin=*]
		\item If $h\!\ge\!3$, then the structural homomorphism $\varphi$ can be rewritten as 
		\begin{align}
			\left[
			\begin{matrix} 
				1 & 0 & 0 & 0 & \cdots & 0 &
				0 & \cdots & 0 & \vec{\bf{0}}\\
				0 & \ka_{12}\ze_{[a_1,a_2]} & 
				\ze_{[\de_2,\fq]}\ze_{[a_1,a_2]}
				& \be_{h+1}\ze_{[\de_{h+1},a_2]}
				&\cdots
				& \be_{\ell}\ze_{[\de_\ell,a_2]}
				&\ze_{[\de_{\ell+1},a_2]}
				&\cdots
				&\ze_{[\de_{m},a_2]}
				& \vec{\bf{0}}
			\end{matrix} 
			\right],\label{Eqn:Case_C_phi_3+}
		\end{align}
		where $\vec{\bf{0}}$ denotes the  zero matrix of size $1\!\times\!(h\!-\!3)$.
		\item 
		If $h\!=\!2$, then $\varphi$ can be rewritten as 
		\begin{align}
			&\varphi=\left[
			\begin{matrix} 
				1 & 0 & 0 & \cdots & 0
				& 0 &\cdots & 0
				\\
				0 & \ka_{12}\ze_{[a_1,a_2]} 
				& \be_{h+1}\ze_{[\de_{h+1},a_2]}
				&\cdots
				& \be_{\ell}\ze_{[\de_m,a_2]}
				&\ze_{[\de_{\ell+1},a_2]}
				&\cdots
				&\ze_{[\de_{m},a_2]}
			\end{matrix} 
			\right].\label{Eqn:Case_C_phi_2}
		\end{align}
	\end{itemize}
	
	\item When $h\!=\!1$,
	the points of $D\bsl\{\de_1\}$ are all on tails,
	so before applying blowups, there is no obvious choice of $\de_2\inn D\bsl\{\de_1\}$ such that $\ka_{1i}$, $i\!\ge\!3$, should be compared with $\ka_{12}$.
	In this case, we postpone the comparison of the $\ka_{1i}$'s until $(\rd_1\ph_3)$ is complete.
\end{itemize}

We are ready to analyze the effects of the sequential blowups on $\cV$.
Since $D\!\cap\! F$ belongs to a non-separating bridge of $F$,
($\rd_1\ph_1$) and ($\rd_1\ph_2$) do not affect $\cV$.
Next, we construct the rooted tree $\tau^\ddag$ that $(\rd_1\ph_3)$ is compatible with on $\cV$.
Let
\begin{align*}
	S'_\de:=
	\begin{cases}
		\big\{N_{[\de]}\big\}
		&\tn{if}~\lr{\de}\!\notin\!\tn B,\\
		\big\{\,
		N_{[\de]}\!\sqcup\!N_{[z_{\lr\de},\fp]},\,
		N_{[\de]}\!\sqcup\!N_{[z_{\lr\de},\fq]}\,
		\big\}
		&\tn{if}~\lr{\de}\!\in\!\tn B,
	\end{cases}
	\quad\de\inn D;
	\qquad
	S':=\bigcup_{\de\in D}S'_\de.
\end{align*}
Particularly,
we observe that $N_{[\de_1,\fp]}$ and $N_{[\de_2,\fq]}$ are inclusion-minimal in $S'$ if $h\!\ge\!2$,
while $N_{[\de_1,\fp]}$ and $N_{[\de_1,\fq]}$ are inclusion-minimal in $S'$ if $h\!=\!1$.

We take the underlying set $E$ of the proposed rooted tree $\tau^\ddag$ to be the union of the inclusion-minimal elements of $S'$, i.e.
\begin{align*}
	E=\bigcup_{\tn{inclusion-minimal}~N'\in S'}\!\!\!\!\!\!\!N'
	\quad\qquad\big(\subset\!N(C)\big)\,.
\end{align*}
Let $C_1$ be the maximal (w.r.t.~inclusion) connected genus-1 subcurve of $C$ that is disjoint from $D\!\cap\!F$.
Then, $E\!\subset\!C_1$.
We define the relation $\preceq$ on $E$ so that $e\!\preceq\!e'$ if and only if 
every connected subcurve $C'$ of $C_1$ containing $e$ and $F_1$ must contain $e'$.
Intuitively, the two criteria above indicate $e'$ is ``closer'' to $F_1$ than $e$ on $N'$,
which shares the same idea as in the constructions of $\tau_1$, $\tau_2$ and $\tau^\dag$ in \S\ref{SubsecDom}, \S\ref{Subsec:Case5.5} and \S\ref{Subsec:Case_B}, respectively.
The relation $\preceq$ satisfies (\ref{Eqn:tree_order}), thus the poset
\begin{align}\label{Eqn:tree_C}
	\tau^\ddag:=(E,\prec)
\end{align}
is a rooted tree.
The root-to-leaf paths of $\tau^\ddag$ are exactly the inclusion-minimal elements of $S'$.

	An example of $\tau^\ddag$ is provided in Figure~\ref{Fig:B}.
	Notice that $\tau^\dag$ as in (\ref{Eqn:tree_B}) is also defined in this case, and is involved in the proof of \S\ref{Subsubsec:Case_C.4}.
	\begin{figure}[htb]
		\begin{center}
			\begin{tikzpicture}
				\def\g1{
					(-1,0) ellipse (1 and 0.5)
					(-1.4,0)..controls(-1,0.1)..(-0.6,0)
					(-0.6,0)..controls(-1,-0.1)..(-1.4,0)
					(-0.5,0.05)--(-0.6,0)
					(-1.4,0)--(-1.5,0.05)
				}
				\def\halfg1{
					(-.32,0)..controls(-.32,-.1) and (-.4,-.2)..(-.6,-.2)
					(-.6,-.2)..controls(-.8,-.2) and (-1,-.25)..(-1,-.35)
					(-.6,-.5)..controls(-.8,-.5) and (-1,-.45)..(-1,-.35)
					(-.6,-.5)..controls(-.25,-.5) and (0,-.35)..(0,0)
					(-.32,0)..controls(-.32,.1) and (-.4,.2)..(-.6,.2)
					(-.6,.2)..controls(-.8,.2) and (-1,.25)..(-1,.35)
					(-.6,.5)..controls(-.8,.5) and (-1,.45)..(-1,.35)
					(-.6,.5)..controls(-.25,.5) and (0,.35)..(0,0)
				}
				\def\crs{
					(-.03,-.03)--(.03,.03)
					(-.03,.03)--(.03,-.03)
				}
				
				\draw\g1;
				\draw[xshift=-1cm]
				(-1.4,0) circle (.4cm)
				;
				\draw[xshift=-2.8cm]
				\halfg1;
				\draw[xshift=-4.8cm,xscale=-1]
				\halfg1;
				
				\draw
				(-2.4,.8) circle (.4cm)
				(-1,.9) circle (.4cm)
				;
				
				\draw[xshift=-4.5cm,yshift=.3cm]
				\crs			
				;
				\draw[xshift=-4cm,yshift=-.32cm]
				\crs			
				;
				\draw[xshift=-2.5cm,yshift=1cm]
				\crs			
				;
				\draw[xshift=-.8cm,yshift=.95cm]
				\crs			
				;
				\draw[xshift=-1.1cm,yshift=1.05cm]
				\crs			
				;
				
				\draw[xshift=4cm]
				(0,0)--(1.2,1.2)
				(0,0)--(-.6,.6)
				(.6,.6)--(0,1.2)
				;
				\filldraw[xshift=4cm]
				(0,0) circle (2pt)
				(.6,.6) circle (2pt)
				(-.6,.6) circle (2pt)
				(1.2,1.2) circle (2pt)
				(0,1.2) circle (2pt)
				;
				\draw[xshift=4cm]
				(.2,.2) node[right] {\scriptsize{$p$}}
				(-.15,.2) node[left] {\scriptsize{$a$}}
				(.85,.85) node[right] {\scriptsize{$b$}}
				(.4,.85) node[left] {\scriptsize{$q$}}
				;
				
				\draw [xshift=2cm]
				(1.9,-.2) node {\tiny{$o$}}
				(2.5,-.36) node {{$\tau^\dag$}};
				
				\draw[xshift=7cm]
				(-.9,1.2)--(0,0)--(.9,1.2)
				(-.3,1.2)--(0,0)--(.3,1.2)
				;
				\filldraw[xshift=7cm]
				(0,0) circle (2pt)
				(.3,1.2) circle (2pt)
				(-.3,1.2) circle (2pt)
				(.9,1.2) circle (2pt)
				(-.9,1.2) circle (2pt)
				;
				\draw[xshift=7cm]
				(.3,.8) node {\scriptsize{$b$}}
				(.75,.8) node {\scriptsize{$a$}}
				(-.3,.8) node {\scriptsize{$s$}}
				(-.8,.8) node {\scriptsize{$r$}}
				;
				
				\draw [xshift=5cm]
				(1.9,-.2) node {\tiny{$o$}}
				(2.5,-.36) node {{$\tau^\ddag$}};
				
				\draw
				(0,-.4) node[right] {{$x\eq (C,D)$}}
				(-2.05,0) node[right] {\tiny{$p$}}
				(-2.85,0) node[right] {\tiny{$q$}}
				(-3.8,.4) node[above] {\tiny{$r$}}
				(-3.8,-.4) node[below] {\tiny{$s$}}
				(-1,.45) node[above] {\tiny{$a$}}
				(-2.4,.35) node[above] {\tiny{$b$}}
				;			
			\end{tikzpicture}
		\end{center}
		\caption{An example of $\tau^\ddag$}\label{Fig:C}
	\end{figure}

As described in \S\ref{rd1},
the blowup centers of $(\rd_1\ph_3)$ are locally given by
\begin{align*}
	&\ov\fM_{(3,k)}\!\cap\!\cV=\!\!
	\bigcup_{\fE\in\Xi_k(\tau^\ddag)}\!\!\!
	\{\,
	\ze_e\eq 0:e\inn\fE\,\},
	\qquad
	\tn{where}\quad
	\Xi_k(\tau^\ddag)=\{\fE\inn\Xi(\tau^\ddag):|\fE|\eq k\!+\!1\},\quad
	k\ge 1.
\end{align*}
By the stability of $\fM_2^{\rm div}$,
we conclude that $(\rd_1\ph_3)$ is $\tau^\ddag$-compatible on $\cV$.
Lemma~\ref{LmLocalLoci} then implies the pullback $\ti\cV^{\rd_1\ph_3}$ of $\cV$ is smooth.

After $(\rd_1\ph_3)$,
we fix an arbitrary lift $y\inn \ti\cV^{\rd_1\ph_3}$ of $\wh x$, as well as a small neighborhood $\cV_y$ of $y$ in $\ti\cV^{\rd_1\ph_3}$.
The RL sequence of $y$ is denoted by $\ov \bE\eq \{\fE_1,\ldots,\fE_\cht\}$ as before.

By Definition~\ref{DfnDominant}, there exists at least one inclusion-minimal $N'\!\in\!S'$ such that $N'\!\subset\!\De_y$,
hence there exists $\de\inn D$ such that $N'\inn S'_\de$.
Mimicking  (\ref{Eqn:I_y}),
we set
\begin{align*}
	S'(\De_y):=
	\big\{\,
	N'\inn S':\,N'\!\subset\!\De_y
	\,\big\},\qquad
	D(\De_y):=
	\big\{\,\de\inn D:\,
	S'_\de\!\cap\!S'(\De_y)\!\ne\!\emptyset\,
	\big\}\,.
\end{align*}

\subsubsection{\ref{Case:3}.\ref{Case:3.1}}
\label{Subsubsec:Case_C.1}
In this sub-case, we assume $h\!\ge\!3$ and $F$ is inseparable,
so we have 
\begin{align*}
	N_{[a_1,a_2]}=\emptyset\qquad
	\big(\Longrightarrow\ze_{[a_1,a_2]}= 1\big).
\end{align*}
Moreover, $h\!\ge\!3$ implies $(\rd_1\ph_4)$-$(\rd_3\ph_4)$ do not affect $\cV_y$,
hence the pullback of $\cV_y$ to the final stack is smooth.

It remains to show the pullback of (\ref{Eqn:Case_C_phi}) is diagonalized after after $(\rd_1\ph_3)$.
We divide the proof into three sub-cases as follows.
\begin{enumerate}[leftmargin=*,label=(\arabic*)]
\item \label{Part:Case_3_1_1}
If $D(\De_y)\!\subset\!F$,
then there exists $i\inn\lrbr h$ such that 
at least one between $N_{[\de_i,\fp]}$ and $N_{[\de_i,\fq]}$,
say $N_{[\de_i,\fp]}$,
is contained in $\De_y$.
By the assumption on $\de_1$,
we then have $N_{[\de_1,\fp]}\!\subset\!\De_y$.
Along with (\ref{Eqn:ka_Case_C}) and Proposition~\ref{PrpDominating},
this implies on $\cV_y$, 
the pullbacks of $\ka_{12}$ and $\ze_{[\de_2,\fq]}$ respectively take the form
\begin{align*}
	\ti\ka_{12}=\big(\prod_{k\in\lrbr{\cht}}\!\!\wc\ve_k\big)\cdot
	\big(\ti f_2+\ti g_2\wc\ze_{[\de_2,\fq]}
	\big),\qquad
	\ti\ze_{[\de_2,\fq]}=
	\big(\prod_{k\in\lrbr{\cht}}\!\!\wc\ve_k\big)\cdot\wc\ze_{[\de_2,\fq]}\,,
\end{align*}
where $\ti f_2$ and $\ti g_2$ are invertible functions,
and $\wc\ze_{[\de_2,\fq]}$ denotes the proper transform of $\ze_{[\de_2,\fq]}$.
Since $\cV_y$ is assumed small,
we conclude that either $\wc\ze_{[\de_2,\fq]}$ or $\big(\ti f_2+\ti g_2\wc\ze_{[\de_2,\fq]}\big)$ is invertible.

Moreover, for each $h\!<\!i\!\le\!\ell$,
notice that
\begin{align*}
	N_{[z_{\lr{\de_i}},\fp]}\!\sqcup\! N_{[\de_{i},a_2]}=
	N_{[\de_i]}\!\sqcup\!N_{[z_{\lr{\de_i}},\fp]}\in S',\quad
	N_{[z_{\lr{\de_i}},\fq]}\!\sqcup\! N_{[\de_{i},a_2]}=
	N_{[\de_i]}\!\sqcup\!N_{[z_{\lr{\de_i}},\fq]}\in S',
\end{align*}
hence by (\ref{Eqn:ka_Case_C}) and Proposition~\ref{PrpDominating}, we see $\prod_{k\in\lrbr{\cht}}\wc\ve_k$ divides the pullback of $\be_i\ze_{[\de_i,a_2]}$.
Similarly,
for each $i\!>\!\ell$,
notice that $N_{[\de_i,a_2]}\eq N_{[\de_i]}\inn S'$,
hence $\prod_{k\in\lrbr{\cht}}\wc\ve_k$ divides the pullback of $\ze_{[\de_i,a_2]}$.

To summarize, we conclude that either $\ti\ka_{12}$ or $\ti\ze_{[\de_2,\fq]}$ divides all the entries of the second row of the pullback of (\ref{Eqn:Case_C_phi_3+}),
hence the pullback of $\varphi$ is diagonalized after $(\rd_1\ph_3)$.

\item \label{Part:Case_3_1_2}
If $D(\De_y)\!\not\subset\!F$, and $\lr\de\inn\tn B$ for all $\de\inn D(\De_y)$,
then w.l.o.g.~we assume 
\begin{align*}
	\de_{h+1}\in D(\De_y)\qquad\tn{and}\qquad
	N_{[z_{\lr{\de_{h+1}}},\fp]}\subsetneq
	N_{[\de_1,\fp]}\,,
\end{align*}
which implies 
\begin{align*}
	N_{[\de_{h+1}]}\!\sqcup\! N_{[z_{\lr{\de_{h+1}}},\fp]} 
	\in 
	S'(\De_y)\,.
\end{align*}
Notice that $N_{[z_{\lr{\de_{h+1}}},\fp]}\!\sqcup\!N_{[\de_{h+1},a_2]}\eq N_{[z_{\lr{\de_{h+1}}},\fp]}\!\sqcup\!N_{[\de_{h+1}]}$.
So by (\ref{Eqn:ka_Case_C}),
we see the pullback of $\be_{h+1}\ze_{[\de_{h+1},a_2]}$ takes the form
\begin{align*}
	\ti \be_{h+1}\ti\ze_{[\de_{h+1},a_2]}=
	u_{h+1}\cdot
	\big(\prod_{k\in\lrbr{\cht}}\!\!\wc\ve_k\big),\qquad
	\tn{where}\quad u_{h+1}\inn\Ga\big(\sO^*_{\cV_y}\big)\,.
\end{align*}
Meanwhile,
since
\begin{align*}
	&N_{[\de_1,\fp]},\qquad
	N_{[\de_2,\fq]},\qquad
	N_{[z_{\lr{\de_i}},\fp]}\!\sqcup\!N_{[\de_i,a_2]}\eq
	N_{[z_{\lr{\de_i}},\fp]}\!\sqcup\!N_{[\de_i]},\ \ h\!<\!i\!\le\!\ell,\\
	&
	N_{[z_{\lr{\de_i}},\fq]}\!\sqcup\!N_{[\de_i,a_2]}\eq
	N_{[z_{\lr{\de_i}},\fq]}\!\sqcup\!N_{[\de_i]},\ \ h\!<\!i\!\le\!\ell,\qquad\tn{and}\quad 
	N_{[\de_j,a_2]}\eq N_{[\de_j]},\ \ \ \ell\!<\!j\!\le\!m
\end{align*}
are all in $S'$,
we conclude from Proposition~\ref{PrpDominating} that $\prod_{k\in\lrbr{\cht}}\wc\ve_k$, consequently $\ti \be_{h+1}\ti\ze_{[\de_{h+1},a_2]}$, divides all the entries of the second row of the pullback of (\ref{Eqn:Case_C_phi_3+}),
hence the pullback of $\varphi$ is diagonalized after $(\rd_1\ph_3)$.

\item \label{Part:Case_3_1_3}
Otherwise, there exists $\de\inn D(\De_y)$ such that $\lr\de\!\notin\!\tn B$ (which implies $\de\!\notin\!F$).
W.l.o.g.~we assume 
\begin{align*}
	\de_{\ell+1}\inn D(\De_y)\,.
\end{align*}
Since $N_{[\de_{\ell+1},a_2]}\eq N_{[\de_{\ell+1}]}$, the pullback of $\ze_{[\de_{\ell+1},a_2]}$ takes the form
\begin{align*}
	\ti\ze_{[\de_{\ell+1},a_2]}=
	u_{\ell+1}\cdot
	\big(\prod_{k\in\lrbr{\cht}}\!\!\wc\ve_k\big),\qquad
	\tn{where}\quad u_{\ell+1}\inn\Ga\big(\sO^*_{\cV_y}\big)\,.
\end{align*}
The same argument as in Part~\ref{Part:Case_3_1_2} once again implies the pullback of $\varphi$ is diagonalized after $(\rd_1\ph_3)$. 
\end{enumerate}

In sum, the proof of \ref{Case:3}.\ref{Case:3.1} is complete.

\subsubsection{\ref{Case:3}.\ref{Case:3.2}}
\label{Subsubsec:Case_C.2}
In this sub-case, we assume $h\!=\!2$ and $F$ is inseparable,
so we still have 
\begin{align*}
	N_{[a_1,a_2]}=\emptyset\qquad
	\big(\Longrightarrow\ze_{[a_1,a_2]}= 1\big).
\end{align*}
Moreover, it is a direct check that $(\rd_1\ph_4)$-$(\rd_3\ph_2)$ and $(\rd_3\ph_4)$ do not affect $\cV$.

Below,
we divide the proof of \ref{Case:3}.\ref{Case:3.2} into three sub-cases in the same way as in \ref{Case:3}.\ref{Case:3.1}.
First, observe that the arguments of Parts~\ref{Part:Case_3_1_2} and~\ref{Part:Case_3_1_3} of \ref{Case:3}.\ref{Case:3.1} apply to  \ref{Case:3}.\ref{Case:3.2} verbatim,
so after $(\rd_1\ph_3)$, $\cV_y$ is smooth, on which the pullback of $\varphi$ is already diagonalized.
In addition, under the assumption of Part \ref{Part:Case_3_1_2} or~\ref{Part:Case_3_1_3},
there exists $i\!\ge\!3$ such that $N_{[\de_i]}\!\subset\!\De_y$,
so $(\rd_3\ph_3)$ does not affect $\cV_y$ either, so the pullback of $\cV_y$ to the final stack is isomorphic to itself, hence is smooth.

It remains to show \ref{Case:3}.\ref{Case:3.2} when $D(\De_y)\!\subset\!F$.
W.l.o.g.~we assume $$
N_{[\de_1,\fp]}\subset\De_y.$$
By (\ref{Eqn:ka_Case_C}) and Parts~\ref{Claim3CoordPT} and~\ref{Claim4CoordPT} of Lemma~\ref{LmCoordPT},
we can write the pullback $\ti\ka_{12}$ and the proper transform $\wc\ka_{12}$ of $\ka_{12}$ on $\cV_y$ as
\begin{align*}
	\ti\ka_{12}=
	\big(\prod_{k\in\lrbr{\cht}}\!\wc\ve_k\big)\cdot\wc\ka_{12},\qquad
	\wc\ka_{12}=f_2'\cdot\!\big(\!\!\!\prod_{e\in N_{[\de_1,\fp]}}\!\!\!(1\!+\!c_e\xi_e)\big)+\,
	g_2'\cdot\big(\!\!\prod_{e\in N_{[\de_2,\fq]}\cap\De_y}\!\!\!\!\!\!\!\!\!(1\!+\!c_{e}\xi_{e})\big)\cdot 
	\big(\!\prod_{e\in N_{[\de_2,\fq]}\bsl\De_y}\!\!\!\!\!\!\!\wc\ze_e\;\big)\,,
\end{align*}
where $f_2'$ and $g_2'$ are invertible functions,
$\xi_e$'s are independent local parameters on $\cV_y$ vanishing at $y$, and $c_e$'s are constants satisfying $c_e\eq 0$ if and only if $e\inn\{\se_1,\ldots,\se_\cht\}$,
where  $\se_k\inn\De_{y,k}$ are as in Part~\ref{Claim1CoordPT} of Lemma~\ref{LmCoordPT}.
So particularly,
by taking $e_1$  to be the unique element of $N_{[\de_1,\fp]}\!\cap\!\fE_\cht$ and $e_2$  to be the unique element of $N_{[\de_2,\fq]}\!\cap\!\fE_\cht$,
we conclude that $c_{e_1}$ and $c_{e_2}$ cannot be zero simultaneously.
\begin{itemize}[leftmargin=*]
	\item If $\wc\ka_{12}(y)\!\ne\!0$,
	then $\wc\ka_{12}$ is invertible on $\cV_y$ because $\cV_y$ is small.
	Therefore, parallel to Part~\ref{Part:Case_3_1_1} of \ref{Case:3}.\ref{Case:3.1},
	we conclude that on $\cV_y$,
	among all the entries of of second row of the pullback of (\ref{Eqn:Case_C_phi_2}),
	the second entry divides the remaining entries,
	hence the pullback of $\varphi$ is still diagonalized after $(\rd_1\ph_3)$.
	
	Moreover, $\wc\ka_{12}$ is invertible on $\cV_y$ implies the proper transforms of the blowup centers of $(\rd_3\ph_3)$ are disjoint from $\cV_y$,
	hence $(\rd_3\ph_3)$ does not affect $\cV_y$ either.
	In this way, we see the pullback of $\cV_y$ to the final stack is smooth.
	
	\item
	Assume $\wc\ka_{12}(y)\eq 0$.
	Then, the above equation of $\wc\ka_{12}$ implies $N_{[\de_2,\fq]}\!\subset\!\De_y$,
	hence  $\wc\ka_{12}$ is a local parameter on $\cV_y$ that vanishes at $y$;
	moreover, $\wc\ka_{12}$ and $\wc\ze_e$, $e\inn N(C)\bsl\De_y$ together form a subset of a system of local parameters on $\cV_y$.
	
	To study the effect of $(\rd_3\ph_3)$ on $\cV_y$,
	consider the sequential blowup of $$\pi':(\fM_2^{\rm div})'\lra\fM_2^{\rm div}$$ along the proper transforms of $Z'_1,Z_2',\ldots$,
	where
	$Z'_j$ is the closed substack of $\fM_2^{\rm div}$ whose general points are pairs $(C,D)$ such that $C$ has $j$ tails attached to its core $F$, and $|D\!\cap\!F|\eq 2$.
	Recall $\tau_2$ is the rooted tree of (\ref{Eqn:tree_A.4}),
	whose root-to-leaf paths are those inclusion-minimal $N_{[\de_i]}$ among all $3\!\le\!i\!\le\!m$.
	It is then a direct check that $\wh x\inn X^\cV_{\tau_2}$ and $\pi'$ is $\tau_2$-compatible on $\cV$.
	Corollary~\ref{Crl:PT} thus implies
	$\PT_y(\tau')$ is $\PT_y(\tau_2)$-compatible on $\cV_y$.
	
	We emphasize that with $\cH_k\!\subset\!\fM_2^{\rm div}$ as in \S\ref{rd3ph3},
	the sequential blowup along the proper transforms of $\cH_k$, $k\!\ge\!1$,
	which serves as the ``prototype'' of ($\rd_3\ph_3$),
	is {\it not} $\ex(\tau_2)$-compatible on $\cV$,
	because $\ka_{12}$, as a sum of two products of local parameters in (\ref{Eqn:ka_Case_C}), is not always a local parameter on $\cV$.
	
	Nonetheless, the assumptions of \ref{Case:3}.\ref{Case:3.2} as well as $D(\De_y)\!\subset\!F$ imply that $y$ lies in the blowup centers of $\PT_y(\pi')$.
	Taking $\wc\ka_{12}(y)\eq 0$ into consideration, we conclude that 
	\begin{align*}
		y\in X^{\cV_y}_{\ex(\PT_y(\tau_2))}
		\,,
	\end{align*}
	where the grafted edge corresponds to $\wc\ka_{12}$.
	Moreover, notice that for every $j\!\ge\!1$,
	with $\wc Z'_j$ denoting the proper transform of $Z_j'$, and $\cH_j^{\rd_3\ph_2}\!\subset\!\ti\fM^{{\rm div},\rd_3\ph_2}$ denoting the proper transform of $\cH_j$ as in (\ref{Eqn:r3p3_blowup_center}), we have
	\begin{align*}
		\{\wc\ka_{12}\eq 0\}\!\cap\!\wc Z'_j\!\cap\!\cV_y
		=
		\cH_j^{\rd_3\ph_2}\!\cap\!\cV_y\,.
	\end{align*}
	Therefore,
	$(\rd_3\ph_3)$ is $\ex\big(\PT_y(\tau_2)\big)$-compatible on $\cV_y$,
	hence the pullback of $\cV_y$ to $\ti\fM^{{\rm div},\rd_3\ph_3}$, hence to the final stack, is smooth.
	
	Finally, we check the pullback of (\ref{Eqn:Case_C_phi_2}) is diagonalized on the pullback of $\cV_y$.
	In fact, on $\cV_y$, the pullback of (\ref{Eqn:Case_C_phi_2}) takes the form
	\begin{align}
		&
		\left[
		\begin{matrix} 
			1 & 0
			\\
			0 & \prod_{k\in\lrbr\cht}\!\wc\ve_k
		\end{matrix} 
		\right]
		\left[
		\begin{matrix} 
			1 & 0 & 0 & \cdots & 0
			\\
			0 & \wc\ka_{12}
			&\prod_{e\in N_{[\de_{3}]}\bsl\De_y}\!\wc\ze_e
			&\cdots
			& \prod_{e\in N_{[\de_{m}]}\bsl\De_y}\!\wc\ze_e
		\end{matrix} 
		\right].\label{Eqn:Case_C_phi_2'}
	\end{align}
	By Proposition~\ref{Prp:PT_rl_path},
	we see the root-to-leaf paths of $\PT_y(\tau_2)$ are exactly the inclusion-minimal elements of $\{N_{[\de_i]}\bsl\De_y:3\!\le\!i\!\le\!m\}$.
	Taking $(\rd_3\ph_3)$ is $\ex\big(\PT_y(\tau_2)\big)$-compatible on $\cV_y$ into consideration,
	we conclude that the pullback of (\ref{Eqn:Case_C_phi_2'}) is diagonalized after $(\rd_3\ph_3)$.
\end{itemize}

In sum, the proof of \ref{Case:3}.\ref{Case:3.2} is complete.

\subsubsection{\ref{Case:3}.\ref{Case:3.3}}
In this sub-case, we assume $h\!=\!1$ (i.e.~$D\!\cap\!F\eq\{\de_1\}$) and $F$ is inseparable,
so we still have 
\begin{align*}
	N_{[a_1,a_2]}=\emptyset\qquad
	\big(\Longrightarrow\ze_{[a_1,a_2]}= 1\big).
\end{align*}
Moreover, it is a direct check that $(\rd_1\ph_4)$ and $(\rd_2)$-$(\rd_3\ph_2)$ do not affect $\cV$.

\begin{enumerate}[leftmargin=*,label=(\arabic*)]
	\item If there exists $\de\inn D(\De_y)$ such that $\lr\de\!\notin\!\tn B$,
	then the argument of Part~\ref{Part:Case_3_1_3} of \ref{Case:3}.\ref{Case:3.1} still applies to this situation verbatim.
	
	\item If $\lr\de\!\in\!\tn B$ for all $\de\inn D(\De_y)$, and $D(\De_y)\!\not\subset\! F$,
	then w.l.o.g.~we assume
	\begin{align*}
		\de_{2}\inn D(\De_y)\bsl F\qquad\tn{and}\qquad
		N_{[z_{\lr{\de_{2}}},\fp]}\!\subsetneq\!
		N_{[\de_1,\fp]},
	\end{align*}
	which implies
	\begin{align*}
		N_{[\de_{2}]}\!\sqcup\!
		N_{[z_{\lr{\de_{2}}},\fp]}\in S'(\De_y),\qquad
		\ka_{12}=f_2'\ze_{[z_{\lr{\de_2}},\fp]}+g_2'\ze_{[\de_1,\fq]}\quad
		\tn{for~some}~f_2',g_2'\inn\Ga(\sO^*_\cV),
	\end{align*}
	and $(\rd_1\ph_4)$-$(\rd_3\ph_2)$ and $(\rd_3\ph_4)$ do not affect (the pullback of) $\cV_y$.
	Here,
	the expression of $\ka_{12}$ above follows from  Proposition~\ref{Prp:phi_key}~\ref{Part:kappa}.
	
	As $\de_2\inn D(\De_y)$ and $F$ is inseparable,
	the pullback of the second entry of the second row of (\ref{Eqn:Case_C_phi}),
	i.e.~$\ti\ka_{12}\ti\ze_{[\de_2,a_2]}$,
	can be written as 
	\begin{align*}
		\ti\ka_{12}\ti \ze_{[\de_2,a_2]}=
		\ti\ka_{12}\ti \ze_{[\de_2]}=
		u_2\cdot\big(\prod_{k\in\lrbr{\cht}}\!\!\wc\ve_k\big)
		\cdot
		\Big(f''_2+
		g''_2\!\cdot\!\ti \ze_{[\de_2]}\!\cdot\!\big(\!\!\prod_{e\in N_{[\de_1,\fq]}\bsl\De_y}\!\!\!\!\!\!\!\wc\ze_e\,\big)\Big),
	\end{align*}
	where $u_{2},$ $f''_2,$ and $g''_2$ are invertible functions on $\cV_y$.
	Notice that $\de_2$ is on a tail, so
	$\ti\ze_{[\de_2]}$ vanishes at $y$;
	Taking $f''_2$ is invertible and $\cV_y$ is small into consideration,
	we see $\ti\ka_{12}\ti\ze_{[\de_2,a_2]}$ is equal to $\prod_{k\in\lrbr{\cht}}\wc\ve_k$ up to a unit.
	Particularly, this implies 
	$\wc\ka_{12}$ is invertible on ${\cV_y}$,
	hence $(\rd_3\ph_3)$ does not affect $\cV_y$.
	
	Meanwhile,
	since
	\begin{align*}
	&
	N_{[z_{\lr{\de_i}},\fp]}\!\sqcup\!N_{[\de_i,a_2]}\eq
	N_{[z_{\lr{\de_i}},\fp]}\!\sqcup\!N_{[\de_i]},\ \ 2\!<\!i\!\le\!\ell,\\
	&
	N_{[z_{\lr{\de_i}},\fq]}\!\sqcup\!N_{[\de_i,a_2]}\eq
	N_{[z_{\lr{\de_i}},\fq]}\!\sqcup\!N_{[\de_i]},\ \ 2\!<\!i\!\le\!\ell,\qquad\tn{and}\quad 
	N_{[\de_j,a_2]}\eq N_{[\de_j]},\ \ \ \ell\!<\!j\!\le\!m
	\end{align*}
	are all in $S'$,
	we conclude from Proposition~\ref{PrpDominating} that $\prod_{k\in\lrbr{\cht}}\wc\ve_k$, consequently $\ti\ka_{12}\ti\ze_{[\de_2,a_2]}$, divides all the entries of the second row of the pullback of (\ref{Eqn:Case_C_phi}),
	hence the pullback of $\varphi$ is diagonalized after $(\rd_1\ph_3)$.
	
	\item
	If $D(\De_y)\!\subset\!F$,
	i.e.~$D(\De_y)\eq\{\de_1\}$,
	then 
	(\ref{Eqn:Case_C_phi}) pulls back to 
	\begin{align*}
		\left[
		\begin{matrix} 
			1 & 0
			\\
			0 & \prod_{k\in\lrbr\cht}\!\wc\ve_k
		\end{matrix} 
		\right]
		\left[
		\begin{matrix} 
			1 & 0 & \cdots & 0
			\\
			0 & \wc\ka_{12}
			\prod_{e\in N_{[\de_{2}]}\bsl\De_y}\!\wc\ze_e
			&\cdots
			& \wc\ka_{1m}\prod_{e\in N_{[\de_{m}]}\bsl\De_y}\!\wc\ze_e
		\end{matrix} 
		\right]\,,
	\end{align*}
	where $N_{[\de_i]}\bsl\De_y$, $i\!\ge\!2$, are all nonempty.
	Observe that	
	$(\rd_1\ph_4)$ does not affect $\cV_y$ because $F$ is assumed inseparable.
	In addition, $(\rd_1\ph_5)$ is $\PT_y(\tau_1)$-compatible on $\cV_y$,
	where $\tau_1$ is as in (\ref{Eqn:tree_A.5}).
	
	By Proposition~\ref{Prp:PT_rl_path},
	the inclusion-minimal elements of $N_{[\de_{i}]}\bsl\De_y$, $i\!\ge\!2$, are exactly the root-to-leaf paths of $\PT_y(\tau_1)$.
	Comparing the above pullback of $\varphi$ with (\ref{e_Mn5}),
	we conclude that
	the argument for \ref{Case:5}.\ref{Case:5.5}, i.e.~the proof of Proposition~\ref{Prp:phi_M5_core_wt1},
	applies to the $D(\De_y)\!\subset\!F$ subcase,
	as long as we can show the functions $\wc\ka_{1i}$, $i\!\ge\!2$,
	behave in the same way as $\ka_{1i}$, $i\!\ge\!2$, do in  \ref{Case:5}.\ref{Case:5.5}.
	
	In fact, since $D(\De_y)\eq\{\de_1\}$,
	w.l.o.g.~we assume $$N_{[\de_1,\fp]}\subset\De_y.$$	
	Mimicking the argument that $\wc\ka_{12}$ is a local parameter in \S\ref{Subsubsec:Case_C.2},
	we conclude from Proposition~\ref{Prp:phi_key}~\ref{Part:kappa} that
	\begin{itemize}[leftmargin=*]
		\item 
		for every  $2\!\le\!i\!\le\!m$,
		the proper transform $\wc\ka_{1i}$ of $\ka_{1i}$ is either a unit or a local parameter on~$\cV_y$;
		\item 
		$\wc\ka_{1i}(y)$ and $\wc\ka_{1j}(y)$ cannot be simultaneously zero whenever $\lr{\de_i}\!\ne\!\lr{\de_j}$;
		\item 
		if $\wc\ka_{1i}(y)\eq\wc\ka_{1j}(y)\eq 0$,
		then there exist invertible functions $v_{ij}$ and $u'_{ij}$ on $\cV_y$ such that
		$$\wc\ka_{1j}
		= v_{ij}\wc\ka_{1j} +
		u'_{ij}\prod_{e\in N_{[\de_j\wedge\de_k]}\bsl\De_y}
		\!\!\!\!
		\wc\ze_e,$$
		where $N_{[\de_j\wedge\de_k]}$ is as in (\ref{Eqn:wedge}).
	\end{itemize}
	Consequently,
	the functions $\wc\ka_{1i}$, $i\!\ge\!2$,
	behave in the same way as $\ka_{1i}$, $i\!\ge\!2$, do in  \ref{Case:5}.\ref{Case:5.5},
	thus the $D(\De_y)\!\subset\!F$ sub-case follows from the argument for \ref{Case:5}.\ref{Case:5.5}.
\end{enumerate}

\subsubsection{\ref{Case:3}.\ref{Case:3.4}}
\label{Subsubsec:Case_C.4}
In this sub-case,
the underlying curve $C$ of $\wh x$ is separable,
and after $(\rd_1\ph_3)$, on $\cV_y$, (\ref{Eqn:Case_C_phi}) pulls back to 
\begin{align*}
	\ti\varphi^{\rd_1\ph_3}=
	\left[
	\begin{matrix} 
		1 & 0\\
		0 & \prod_{k\in\lrbr\cht}\wc\ve_k
	\end{matrix} 
	\right]
	\left[
	\begin{matrix} 
		1 & 0 &  \cdots & 0\\
		0 & \wc\ka_{12}\prod_{e\in N_{[\de_{2},a_2]}\bsl\De_y}\!\!\wc\ze_e & \cdots & 
		\wc\ka_{1m}\prod_{e\in N_{[\de_{m},a_2]}\bsl\De_y}\!\!\wc\ze_e
	\end{matrix} 
	\right]\,.
\end{align*}

Recall that throughout \S\ref{Subsec:Case_C},
$T_1$ refers to the genus 1 inseparable  component that is disjoint from the non-separating bridge $\tn B$, hence disjoint from $D\!\cap\!F$.
If there exists $\de_i\inn D(\De_y)$ such that $\lr{\de_i}\inn T_1$,
then $N_{[\de_{i},a_2]}\bsl\De_y\eq 
N_{[\de_{i}]}\bsl\De_y\eq\emptyset$.
Moreover, by
Lemma~\ref{Lm:K_smoothness}, we have $\ka_{1i}$ is invertible on $\cV$.
Therefore, 
$\wc\ka_{1i}\prod_{e\in N_{[\de_{i},a_2]}\bsl\De_y}\!\wc\ze_e$ is a unit on $\cV_y$, hence
$\ti\varphi^{\rd_1\ph_3}$ is diagonalized.
In addition,
($\rd_1\ph_4$)-$(\rd_3\ph_4)$ do not affect $\cV_y$.

It remains to verify the situation when $\lr{\de}\inn F\bsl T_1$ for every $\de\inn D(\De_y)$.
In such a situation,
we observe that $N_{[\de_i,a_2]}\bsl\De_y$ is nonempty for every $i\inn\lrbr m$, and ($\rd_1\ph_4$) is $\PT_y(\tau^\dag)$-compatible on $\cV_y$,
where the rooted tree $\tau^\dag$ is as in (\ref{Eqn:tree_B}).
Moreover,
by Proposition~\ref{Prp:PT_rl_path},
the inclusion-minimal elements of $N_{[\de_{i},a_2]}\bsl\De_y$, $i\inn\lrbr m$, are exactly the root-to-leaf paths of $\PT_y(\tau^\dag)$.
Furthermore,
mimicking the argument used in the $D(\De_y)\!\subset\!F$ sub-case of the proof of \ref{Case:3}.\ref{Case:3.3},
we see that for every
$\de_i$ satisfying $N_{[\de_i,a_2]}\bsl\De_y$ is inclusion-minimal among all $2\!\le\!i\!\le\!m$,
$\wc\ka_{1i}$
behaves in the same way as $\ka_{1i}$ does in  \ref{Case:4}.
Comparing the above $\ti\varphi^{\rd_1\ph_3}$ with (\ref{Eqn:Case_B_phi}),
we thus conclude that
the argument for \ref{Case:4}, i.e.~the proof of Proposition~\ref{PrpChangeofPhiM4},
applies to the current situation.

To summarize, 
we have shown the following in \S\ref{Subsec:Case_C}.

\begin{prop}\label{PrpChangeofPhiM3}
	Proposition~\ref{PrpChangeofPhi} holds in \ref{Case:3}. 
\end{prop}

\subsection{Proof of \ref{Case:2}}
\label{Subsec:Case_D}

In this subsection, we prove Proposition~\ref{PrpChangeofPhi} in \ref{Case:2}.

Let 
$
\cV\to\fM_2^{\rm div}
$ be a small affine smooth chart containing~$\wh x\eq(C,D)$.
In~\ref{Case:2},
we assume~$x$, the image  of $\wh x$ in $\fM_2\wt$, lies in $\fM^\mn_{(2)}$.
This implies the core $F$ of $x$ contains two genus 1 inseparable components $T_1$ and $T_2$, as well as a 
(maximal) separating bridge $\tn B$ so that
$$
D\cap \tn B = D\cap F\ne\emptyset.
$$
W.l.o.g.~we assume $a_1\inn T_1$, $a_2\inn T_2$, and there exists $1\!\le\!h\!\le\! m$ such that
\begin{align*}
	D\!\cap\! F
	=\{\de_1,\ldots,\de_h\}.
\end{align*}
Furthermore,
we assume
\begin{align*}
	N_{[\de_1,a_1]}\!\subset\!N_{[\de_i,a_1]}\quad\forall~i\inn\lrbr{h}.
\end{align*}
In addition, if $|D\!\cap\!F|\!\ge\!2$ (i.e.~$h\!\ge\!2$),
then we assume
\begin{align*}
	N_{[\de_2,a_2]}\!\subset\!N_{[\de_i,a_2]}\quad\forall~i\inn\lrbr{h}.
\end{align*}
Since $D\!\cap\!F\!\ne\!\emptyset$,
$(\rd_1\ph_1)$ does not affect $\cV$.

Next, 
we construct the rooted tree $\tau_\wedge$ that $(\rd_1\ph_2)$ is compatible with on $\cV$.
Let
\begin{align*}
	S'':=
	\bigcup_{\de\in D}\big\{N_{[\de,a_1]},N_{[\de,a_2]}\big\}\,.
\end{align*}
Particularly,
if $h\!\ge\!2$, then $N_{[\de_1,a_1]}$ and $N_{[\de_2,a_2]}$ are inclusion-minimal;
if $h\eq 1$, then $N_{[\de_1,a_1]}$ and $N_{[\de_1,a_2]}$ are inclusion-minimal.

We take the underlying set $E$ of the proposed rooted tree $\tau_\wedge$ to be the union of the inclusion-minimal elements of $S''$, i.e.
\begin{align*}
	E=\bigcup_{\tn{inclusion-minimal}~N''\in S''}\!\!\!\!\!\!\!N''
	\quad\qquad\big(\subset\!N(C)\big)\,.
\end{align*}
It is a direct check that $E$ decomposes into two disjoint nonempty subsets: 
$$E=E_1\sqcup E_2,$$
such that for $s\inn\lrbr 2$, $e\inn E_s$ if and only if
there exists $i\inn\lrbr m$ such that $e\inn N_{[\de_i,a_s]}\!\subset\! E$.

For $s\inn\lrbr 2$,
we define the relation $\preceq_s$ on $E_s$ such that $e\!\preceq_s\!e'$ if and  only if any connected subcurve of $C$ containing $T_s$ and $e$ must contain $e'$,
which once again satisfies (\ref{Eqn:tree_order}).
In other words,
$$ (E_1,\preceq_1),\ (E_2,\preceq_2)\in\bT.$$
The tree orders $\preceq_1$ and $\preceq_2$ together determine a tree order $\preceq$ on $E$:
$e\preceq e'$ if and only if either $e\preceq_1 e'$ or $e\preceq_2 e'$.
In other words,
\begin{align}\label{Eqn:tree_D}
	\tau_\wedge:=(E,\preceq)\in\bT\,.
\end{align}
The root-to-leaf paths of $\tau_\wedge$ are exactly the inclusion-minimal elements of $S''$.

	An example of $\tau_\wedge$ is provided in Figure~\ref{Fig:D}.
	\begin{figure}[htb]
		\begin{center}
			\begin{tikzpicture}
				\def\g1{
					(-1,0) ellipse (1 and 0.5)
					(-1.4,0)..controls(-1,0.1)..(-0.6,0)
					(-0.6,0)..controls(-1,-0.1)..(-1.4,0)
					(-0.5,0.05)--(-0.6,0)
					(-1.4,0)--(-1.5,0.05)
				}
				\def\halfg1{
					(-.32,0)..controls(-.32,-.1) and (-.4,-.2)..(-.6,-.2)
					(-.6,-.2)..controls(-.8,-.2) and (-1,-.25)..(-1,-.35)
					(-.6,-.5)..controls(-.8,-.5) and (-1,-.45)..(-1,-.35)
					(-.6,-.5)..controls(-.25,-.5) and (0,-.35)..(0,0)
					(-.32,0)..controls(-.32,.1) and (-.4,.2)..(-.6,.2)
					(-.6,.2)..controls(-.8,.2) and (-1,.25)..(-1,.35)
					(-.6,.5)..controls(-.8,.5) and (-1,.45)..(-1,.35)
					(-.6,.5)..controls(-.25,.5) and (0,.35)..(0,0)
				}
				\def\crs{
					(-.03,-.03)--(.03,.03)
					(-.03,.03)--(.03,-.03)
				}
				
				\draw\g1;
				\draw[xshift=-1cm]
				(-1.4,0) circle (.4cm)
				;
				\draw[xshift=-2.8cm]
				\halfg1;
				\draw[xshift=-4.8cm,xscale=-1]
				\halfg1;
				
				\draw
				(-5.2,0) circle (.4cm)
				(-1,.9) circle (.4cm)
				;
				
				\draw[xshift=-5.25cm,yshift=.26cm]
				\crs			
				;
				\draw[xshift=-5.2cm,yshift=-.28cm]
				\crs			
				;
				\draw[xshift=-2.2cm,yshift=-.15cm]
				\crs			
				;
				\draw[xshift=-.8cm,yshift=.95cm]
				\crs			
				;
				\draw[xshift=-1.1cm,yshift=1.05cm]
				\crs			
				;
				
				\draw[xshift=4cm]
				(-.9,1.2)--(0,0)--(.9,1.2)
				(-.3,1.2)--(0,0)--(.3,1.2)
				;
				\filldraw[xshift=4cm]
				(0,0) circle (2pt)
				(.3,1.2) circle (2pt)
				(-.3,1.2) circle (2pt)
				(.9,1.2) circle (2pt)
				(-.9,1.2) circle (2pt)
				;
				\draw[xshift=4cm]
				(.3,.8) node {\scriptsize{$p$}}
				(.75,.8) node {\scriptsize{$a$}}
				(-.3,.8) node {\scriptsize{$q$}}
				(-.8,.8) node {\scriptsize{$c$}}
				;
				
				\draw [xshift=2cm]
				(1.9,-.2) node {\tiny{$o$}}
				(2.5,-.36) node {{$\tau_\wedge$}};
				
				\draw
				(-4.75,0) node[left] {\tiny{$c$}}
				(0,-.4) node[right] {{$x\eq (C,D)$}}
				(-2.05,0) node[right] {\tiny{$p$}}
				(-2.85,0) node[right] {\tiny{$q$}}
				(-3.8,.4) node[above] {\tiny{$r$}}
				(-3.8,-.4) node[below] {\tiny{$s$}}
				(-1,.45) node[above] {\tiny{$a$}}
				;			
			\end{tikzpicture}
		\end{center}
		\caption{An example of $\tau_\wedge$}\label{Fig:D}
	\end{figure}

As described in \S\ref{rd1},
the blowup centers of $(\rd_1\ph_2)$ are locally given by
\begin{align*}
	&\ov\fM_{(2,k)}\!\cap\!\cV=\!\!
	\bigcup_{\fE\in\Xi_k(\tau_\wedge)}\!\!\!
	\{\,
	\ze_e\eq 0:e\inn\fE\,\},
	\qquad
	\tn{where}\qquad
	\Xi_k(\tau_\wedge)=\{\fE\inn\Xi(\tau_\wedge):|\fE|\eq k\!+\!1\},\quad
	k\!\ge\!1.
\end{align*}
By the stability of $\fM_2^{\rm div}$,
we conclude that $(\rd_1\ph_2)$ is $\tau_\wedge$-compatible on $\cV$.
Lemma~\ref{LmLocalLoci} then implies the pullback $\ti\cV^{\rd_1\ph_2}$ of $\cV$ is smooth.

After $(\rd_1\ph_2)$,
we fix an arbitrary lift $y\inn \ti\cV^{\rd_1\ph_2}$ of $\wh x$, as well as a small neighborhood $\cV_y$ of $y$ in $\ti\cV^{\rd_1\ph_2}$.
The RL sequence of $y$ is denoted by $\ov \bE\eq \{\fE_1,\ldots,\fE_\cht\}$ as before.
The pullback $\ti\varphi^{\rd_1\ph_2}$ of structural homomorphism
$\varphi$ to $\cV_y$ can be written as
\begin{equation}
	\label{Eqn:Case_D_phi}
	\big(\prod_{k\in\lrbr\cht}\!
	\wc\ve_k\big)\cdot
	\left[
	\begin{matrix} 
		c'_{11}\prod_{e\in N_{[\de_1,a_1]}\bsl \De_y}\!\wc\ze_e &   \cdots & c'_{1m}\prod_{e\in N_{[\de_m,a_1]}\bsl \De_y}\!\wc\ze_e\\
		c'_{21}\prod_{e\in N_{[\de_1,a_2]}\bsl \De_y}\!\wc\ze_e & \cdots 
		& c'_{2m}\prod_{e\in N_{[\de_m,a_2]}\bsl \De_y}\!\wc\ze_e
	\end{matrix} 
	\right]\,,
\end{equation}
where $\det\left[\begin{matrix}
	c_{1i} & c_{1j}\\ c_{2i} & c_{2j}\end{matrix}\right]$
is equal to $\wc\ka_{ij}\prod_{e\in N_{[\de_i\wedge\de_j]}\bsl\De_y}\!\wc\ze_e$ up to a unit.

By Definition~\ref{DfnDominant}, there exists at least one inclusion-minimal $N''\!\in\!S''$ such that $N''\!\subset\!\De_y$.
We set
\begin{align*}
	S''(\De_y):=
	\big\{\,
	N''\inn S'':\,N''\!\subset\!\De_y
	\,\big\},\qquad
	D(\De_y):=
	\big\{\,\de\inn D:\,
	\{N_{[\de,a_1]},N_{[\de,a_2]}\}\!\cap\!S''(\De_y)\!\ne\!\emptyset\,
	\big\}\,,
\end{align*}
which are both nonempty.

\begin{itemize}[leftmargin=*]
\item 
If there exist distinct $i$ and $j$ such that 
$N_{[\de_i,a_1]}\!\subset\!\De_y$, $N_{[\de_j,a_2]}\!\subset\!\De_y$,
and $\de_i$ is not conjugate to $\de_j$ (i.e.~$\ka_{ij}$ is invertible on $\cV$),
then 
$(\rd_1\ph_3)$-$(\rd_3\ph_4)$ do not affect $\cV_y$, and
the entries of the $i$-th and $j$-th columns of (\ref{Eqn:Case_D_phi}),
as well as the corresponding $2\!\times\!2$ minor,
are all invertible,
hence $\ti\varphi^{\rd_1\ph_2}$ is diagonalizable on $\cV_y$.

\item 
If there exist distinct $i$ and $j$ such that 
$N_{[\de_i,a_1]}\!\subset\!\De_y$ and $N_{[\de_j,a_2]}\!\subset\!\De_y$,
and $\de_i$ is conjugate to $\de_j$ whenever $N_{[\de_i,a_1]}\!\subset\!\De_y$ and $N_{[\de_j,a_2]}\!\subset\!\De_y$,
then by Lemma~\ref{Lm:K_position},
$\de_i$ and $\de_j$ must be on the same irreducible component of the separating bridge $\tn B$,
and there are no other points of $D$ lying on $F$.
From the assumptions on $D\!\cap\!\tn B$ in the proof of \ref{Case:2} as well as Lemma~\ref{Lm:K_smoothness},
we conclude that $D(\De_y)\eq\{\de_1,\de_2\}\eq D\!\cap\tn B$, and $\ka_{12}$ is a local parameter on $\cV$ that vanishes at $\wh x$.
Consequently, after suitable row and column operations, (\ref{Eqn:Case_D_phi}) can be rewritten as
\begin{equation*}
	\big(\prod_{k\in\lrbr\cht}\!
	\wc\ve_k\big)\cdot
	\left[
	\begin{matrix} 
		1 & 0 & 0 &   \cdots & 0\\
		0 & \wc\ka_{12} &
		\prod_{e\in N_{[\de_3]}\bsl \De_y}\!\wc\ze_e & \cdots 
		& \prod_{e\in N_{[\de_m]}\bsl \De_y}\!\wc\ze_e
	\end{matrix} 
	\right]\,.
\end{equation*}

Notice that $(\rd_1\ph_3)$-$(\rd_3\ph_2)$ and $(\rd_3\ph_4)$ do not affect (the pullback of) $\cV_y$,
and $(\rd_3\ph_3)$ is $\ex\big(\PT_y(\tau_2)\big)$-compatible on $\cV_y$, where $\tau_2$ is the rooted tree given by (\ref{Eqn:tree_A.4}).
Therefore,
the pullback of $\cV_y$ after $(\rd_3\ph_3)$, hence after the entire blowup,
is smooth,
on which the pullback of $\ti\varphi^{\rd_1\ph_2}$ is diagonalized.

\item 
If there exists a unique $\de\eq D$ such that $N_{[\de,a_1]}\!\subset\!\De_y$ and $N_{[\de,a_2]}\!\subset\!\De_y$,
then $D\!\cap\!\tn B\eq\{\de\}$,
so $h\eq 1$ and $\de\eq\de_1$.
By Lemma~\ref{Lm:K_position},
we further have $\ka_{1i}$ is invertible on $\cV$ for every $i\!\ge\!2$ with $N_{[\de_i,a_1]}\!\subset\!E$ or $N_{[\de_i,a_2]}\!\subset\!E$.
Consequently, after suitable row and column operations, (\ref{Eqn:Case_D_phi}) can be rewritten as
\begin{equation*}
	\big(\prod_{k\in\lrbr\cht}\!
	\wc\ve_k\big)\cdot
	\left[
	\begin{matrix} 
		1 & 0 &    \cdots & 0\\
		0 & \wc\ka_{12}
		\prod_{e\in N_{[\de_2]}\bsl \De_y}\!\wc\ze_e & \cdots 
		& \wc\ka_{1m}\prod_{e\in N_{[\de_m]}\bsl \De_y}\!\wc\ze_e
	\end{matrix} 
	\right]\,.
\end{equation*}

Notice in this situation, $(\rd_1\ph_3)$ and $(\rd_1\ph_4)$ do not affect $\cV_y$,
whereas $(\rd_1\ph_5)$ is $\PT_y(\tau_1)$-compatible on $\cV_y$,
where $\tau_1$ is as in (\ref{Eqn:tree_A.5}).
Moreover,
$\wc\ka_{1i}$, $i\!\ge\!2$, behave in the same way as $\ka_{1i}$, $i\!\ge\!2$, do in \ref{Case:5}.\ref{Case:5.5},
hence the argument for \ref{Case:5}.\ref{Case:5.5} (i.e.~the proof of Proposition~\ref{Prp:phi_M5_core_wt1}) apply to the this situation.

\item 
If one of $E_1$ and $E_2$, say $E_1$,
satisfies $N''\!\subset\!E_1$ for all $N''\inn S''(\De_y)$,
then re-labeling $\de_1$ if necessary, after taking suitable row and column operations, we can rewrite (\ref{Eqn:Case_D_phi}) as
\begin{equation*}
	\big(\prod_{k\in\lrbr\cht}\!
	\wc\ve_k\big)\cdot
	\left[
	\begin{matrix} 
		1 & 0 &    \cdots & 0\\
		0 & \wc\ka_{12}
		\prod_{e\in N_{[\de_2,a_2]}\bsl \De_y}\!\wc\ze_e & \cdots 
		& \wc\ka_{1m}\prod_{e\in N_{[\de_m,a_2]}\bsl \De_y}\!\wc\ze_e
	\end{matrix} 
	\right]\,.
\end{equation*}
We observe that $N_{[\de_i,a_2]}\bsl\De_y$ is nonempty for every $i\inn\lrbr m$, and ($\rd_1\ph_4$) is $\PT_y(E_2,\preceq_2)$-compatible on $\cV_y$.
Moreover,
the inclusion-minimal elements of $N_{[\de_{i},a_2]}\bsl\De_y$, $i\inn\lrbr m$, are exactly the root-to-leaf paths of $\PT_y(E_2,\preceq_2)$.
Furthermore,
mimicking the argument used in the $D(\De_y)\!\subset\!F$ sub-case of the proof of \ref{Case:3}.\ref{Case:3.3},
we see that for every
$\de_i$ satisfying $N_{[\de_i,a_2]}\bsl\De_y$ is inclusion-minimal among all $2\!\le\!i\!\le\!m$,
$\wc\ka_{1i}$
behaves in the same way as $\ka_{1i}$ does in  \ref{Case:4}.
Comparing the above $\ti\varphi^{\rd_1\ph_2}$ with (\ref{Eqn:Case_B_phi}),
we thus conclude that
the argument for \ref{Case:4}, i.e.~the proof of Proposition~\ref{PrpChangeofPhiM4},
applies to the current situation.
\end{itemize}

To summarize, 
we have shown the following in \S\ref{Subsec:Case_D}.

\begin{prop}\label{PrpChangeofPhiM2}
Proposition~\ref{PrpChangeofPhi} holds
in \ref{Case:2}.
\end{prop}

\subsection{Proof of \ref{Case:1}, Part I}
\label{Subsec:Case_E}

In this subsection, we prove Proposition~\ref{PrpChangeofPhi} in \ref{Case:1},
when the image of $\wh x$ in $\fM_2\wt$ belongs to $\fM^\mn_{(1)}$.
This implies
\begin{align*}
	D\cap F=\emptyset\,,
\end{align*}
i.e.~all the points of $D$ are on tails.
	
Let 
$
\cV\to\fM_2^{\rm div}
$ be a small affine smooth chart containing~$\wh x\eq(C,D)$.
By defining the subsets $D_{\im}\!\subset\!D$ and $E\!\subset\!N(C)$ exactly in the same way as (\ref{Eqn:tau_CaseA.4}),
and defining the relation $\preceq$ on $E$ given by $e\!\preceq\!e'$ if and  only if any connected subcurve of $C$ containing $F$ and $e$ must contain $e'$,
we see
\begin{align}\label{Eqn:tree_E}
	\tau_0:=(E,\preceq)
\end{align}
is a rooted tree,
whose root-to-leaf paths are exactly $N_{[\de]}$, $\de\inn D_{\im}$.
Although the constructions of $\tau_0$ here is parallel to that of $\tau_2$ in (\ref{Eqn:tree_A.4}) and $\tau_1$ in (\ref{Eqn:tree_A.5}), we emphasize the set $D\bsl F$ containing $D_{\im}$ are different for $\tau_0$, $\tau_1,$ and $\tau_2$.
In Figure~\ref{Fig:CaseA.4},
if the two points of $D$ on the core are removed,
then $\tau_0$ is equal to the rooted tree of  Figure~\ref{Fig:CaseA.4} (whereas $\tau_2$ is not defined in this situation).

As described in \S\ref{rd1},
the blowup centers of $(\rd_1\ph_1)$ locally are given by
\begin{equation}\begin{split}
		\label{Eqn:r1p1_tree_cmptb}
		&\ov\fM_{(1,k)}\!\cap\ti\cV
		=\!\bigcup_{\fE\in\Xi_k(\tau_0)}\!\!\!\!
		\{\,
		\ze_e\eq 0~\forall~e\inn\fE\,\},\qquad
		\tn{where}\quad
		\Xi_k(\tau_0)=\{\fE\inn\Xi(\tau_0):|\fE|\eq k\}.
\end{split}\end{equation}
It is thus a direct check that $(\rd_1\ph_1)$ is $\tau_0$-compatible on $\cV$.

After $(\rd_1\ph_1)$, we denote by $\ti\cV^{\rd_1\ph_1}$ the pullback of $\cV$, and fix an arbitrary lift $y\inn \ti\cV^{\rd_1\ph_1}$ of $\wh x$, as well as a small neighborhood $\cV_y$ of $y$ in $\ti\cV^{\rd_1\ph_1}$.
The RL sequence of $y$ is denoted by $\ov \bE\eq \{\fE_1,\ldots,\fE_\cht\}$ as before.
Parallel to (\ref{Eqn:I_y}), we set
\begin{align}\label{Eqn:D(De)_r1p1}
	D(\De_y):=
	\big\{\,\de\inn D:\,
	N_{[\de]}\!\subset\!\De_y
	\big\}\,,
\end{align}
which determines a subset of the core $F$ of $C$:
\begin{align*}
	\lr{\De_y}:=
	\big\{\,\lr{\de}:\,
	\de\inn D(\De_y)\,\big\}\,.
\end{align*}
Based on the topological type of $F$ as well as the locations of the points of $D(\De_y)$,
we divide the proof of \ref{Case:1} into the following sub-cases:
\begin{enumerate}[leftmargin=*,label=\arabic*]
	\item \label{Case:1.1}
	\!\!\!. there does not exist any genus 1 inseparable component of $F$ that is disjoint from $\lr{\De_y}$, and one of the following holds:
	\begin{enumerate}[leftmargin=*,label=\alph*]
		\item \label{Case:1.1.a}
		\!\!\!. 
		$F$ is separable;
		\item \label{Case:1.1.b}
		\!\!\!. 
		$F$ is inseparable, and $|\lr{\De_y}|\!\ge\!3$;
		\item \label{Case:1.1.c}
		\!\!\!. 
		$F$ is inseparable, $|\lr{\De_y}|\!=\!2$, and the elements of $\lr{\De_y}$ are not conjugate to each other;
		\item \label{Case:1.1.d}
		\!\!\!. 
		$F$ is inseparable, $|\lr{\De_y}|\!=\!2$, and the elements of $\lr{\De_y}$ are conjugate to each other;
		\item \label{Case:1.1.e}
		\!\!\!. 
		$F$ is inseparable, $|\lr{\De_y}|\!=\!1$, and the element of $\lr{\De_y}$ is not a Weierstrass point of $F$;
		\item \label{Case:1.1.f}
		\!\!\!. 
		$F$ is inseparable, $|\lr{\De_y}|\!=\!1$, and the element of $\lr{\De_y}$ is a Weierstrass point of $F$;
	\end{enumerate}
	\item \label{Case:1.2}
	\!\!\!. there exists a unique genus 1 inseparable component of $F$ that is disjoint from $\lr{\De_y}$,
	and there does not exist any (non-separating) bridge of $F$ that contains $\lr{\De_y}$ as a subset;
	\item \label{Case:1.3}
	\!\!\!. there exists a unique non-separating bridge of $F$ that contains $\lr{\De_y}$ as a subset;
	\item \label{Case:1.4}
	\!\!\!. there exists a unique separating bridge of $F$ that contains $\lr{\De_y}$ as a subset.
\end{enumerate}

The proof of \ref{Case:1}.\ref{Case:1.1}.\ref{Case:1.1.a}-\ref{Case:1.1.c} is parallel to that of Proposition~\ref{Prp:phi_M5_simple}.
More precisely,
\begin{itemize}[leftmargin=*]
	\item in \ref{Case:1}.\ref{Case:1.1}.\ref{Case:1.1.a},
	there exist two points of $D(\De_y)$, say $\de_1$ and $\de_2$,
	such that for $s\eq 1,2$, $\lr{\de_s}$ and $a_s$ belong to the same genus 1 inseparable component of $F$;
	\item 
	in \ref{Case:1}.\ref{Case:1.1}.\ref{Case:1.1.b} and \ref{Case:1.1.c},
	there exist two points of $D(\De_y)$, say $\de_1$ and $\de_2$,
	such that $\lr{\de_1}$ and $\lr{\de_2}$ are distinct and not conjugate to each other.
\end{itemize}
In either situation above,
by Proposition~\ref{Prp:phi_key}~\ref{Part:phi} (and \ref{Part:theta} in the latter situation) as well as Proposition~\ref{PrpDominating},
the pullback of the structural homomorphism $\varphi$ takes the following form on $\cV_y$:
\begin{align*}
	\Big(\prod_{k\in\lrbr\cht}\!\wc\ve_k\Big)\cdot
	\left[\;
	\begin{matrix} 
		1 & 0 & 0 &\cdots&0 \\
		0&
		1 &
		0
		&\cdots
		&
		0
	\end{matrix}\;
	\right]\,,
\end{align*}
i.e.~it is diagonalized after $(\rd_1\ph_1)$.
Moreover,
the blowups starting from $(\rd_1\ph_2)$ do not affect $\cV_y$, hence the pullback of $\cV_y$ in the final stack is smooth.

The remaining sub-cases of \ref{Case:1} are more complicated. 
We will verify them in later subsections.

\subsection{First-order doubly-derived trees and the proof of \ref{Case:1}, Part II}
In this subsection,
we aim to show \ref{Case:1}.\ref{Case:1.1}.\ref{Case:1.1.d},
when there are two tails $C_p$ and $C_q$ of $C$,
whose pivotal nodes are $p$ and $q$, respectively,
such that 
\begin{align*}
	D(\De_y)\subset C_p\!\sqcup\!C_q,\qquad
	D_p(\De_y):=D(\De_y)\!\cap\! C_p\ne\emptyset,\qquad
	D_q(\De_y):=D(\De_y)\!\cap\! C_q\ne\emptyset,
\end{align*}
and  $p$ is conjugate to $q$ on $F$.
W.l.o.g.~we assume there exist $1\!\le\!h\!\le\!\ell\!<\!s\!\le\!t\!\le\!m$ such that
\begin{align*}
	D(\De_y)\eq \{\de_1,\cdots\!,\de_{h}\}
	\!\sqcup\!
	\{\de_{\ell+1},\cdots\!,\de_{s}\},\qquad
	\lr{\de_1}\eq\cdots\eq \lr{\de_{\ell}}\eq p,\qquad 
	\lr{\de_{\ell+1}}\eq\cdots\eq\lr{\de_{t}}\eq q.
\end{align*}
Moreover,
we observe the assumptions of \ref{Case:1}.\ref{Case:1.1}.\ref{Case:1.1.d} imply $(\rd_1\ph_2)$-$(\rd_2)$ do not affect $\cV_y$.

Since $p$ and $q$ are not on any non-separating bridge,
by Proposition~\ref{Prp:phi_key}~\ref{Part:kappa},
we have
\begin{align*}
	\ka_{1,{j}}~\tn{is~a~local~parameter} \quad
	\forall~\ell\!<\!j\!\le\!t\,,\qquad
	\ka_{1i}\inn\Ga(\sO^*_\cV)\quad
	\forall~1\!<\!i\!\le\!\ell\,.
\end{align*}
Hence by Proposition~\ref{Prp:phi_key}~\ref{Part:theta},
after suitable elementary row and column operations,
the pullback $\ti\varphi^{\rd_1\ph_1}$ of the structural homomorphism $\varphi$ takes the following form on $\cV_y$:
\begin{align*}
	\Big(\prod_{k\in\lrbr\cht}\!\wc\ve_k\Big)\cdot
	\left[\;
	\begin{matrix} 
		1 & 0 & \cdots & 0 & 0 & 0 &\cdots&0 
		& 0 &\cdots&0 
		\\
		0&
		\undermat{\tn{DNE~if}~\ell\eq 1}{\eta'_{2}& 
		\cdots 
		&
		\eta'_{\ell}}
		&
		\wc\ka_{1,\ell+1}
		&
		\undermat{\tn{DNE~if}~t\eq \ell\!+\!1}{\eta'_{\ell+2}
		&\cdots
		&
		\eta'_{t}}
		&
		\eta'_{t+1}
		&\cdots
		&
		\eta'_{m}
	\end{matrix}\;
	\right]\,,
\end{align*}
where 
\begin{align*}
	\eta_{i}':=
	\begin{cases}
	\big(\prod_{k\in\lrbr{\cht}_{[\de_1\wedge\de_i]}}
	\!\wc\ve_k\big)\big(\prod_{e\in N_{[\de_i]}\bsl\De_y}\!\wc\ze_e\big) &\tn{if}~1\!<\!i\!\le\!\ell,
	\\
	\big(\prod_{k\in\lrbr{\cht}_{[\de_{\ell+1}\wedge\de_i]}}\!\wc\ve_k\big)\big(\prod_{e\in N_{[\de_i]}\bsl\De_y}\!\wc\ze_e\big)\quad
	&\tn{if}~\ell\!+\!1\!<\!i\!\le\!t,
	\\
	\prod_{e\in N_{[\de_i]}\bsl\De_y}\!\wc\ze_e
	&\tn{if}~t\!<\!i;
	\end{cases}
	\qquad
	\lrbr{\cht}_{[\de_i\wedge\de_j]}~\tn{are as in}~(\ref{Eqn:[m]_wedge}).
\end{align*}

Parallel to Lemma~\ref{Lm:wedge},
we see there exist $1\!\le\!j\!\le\!h$ and $\ell\!+\!1\!\le\!r\!\le\!s$ such that
\begin{align}\label{Eqn:N_wedge'}
	N_{[\wedge D_p(\De_y)]}:=
	\bigcap_{\de\in D_p(\De_y)} \!\!\!\!\!\!N_{[\de]}
	=N_{[\de_1\wedge\de_{j}]}\,,
	\qquad
	N_{[\wedge D_q(\De_y)]}:=
	\bigcap_{\de\in D_q(\De_y)}\!\!\!\!\!\!N_{[\de]}
	=N_{[\de_{\ell+1}\wedge\de_{r}]}\,.
\end{align}
Let
\begin{equation}\begin{split}\label{Eqn:cht_wedge'}
	&{\lrbr \cht}_{[\wedge_2 D(\De_y)]}:=
	\big\{\,k\inn\lrbr\cht:\,N_{[\wedge D_p(\De_y)]}\!\cap\!\fE_k\!\ne\!\emptyset,
	~N_{[\wedge D_q(\De_y)]}\!\cap\!\fE_k\!\ne\!\emptyset\,\big\},
	\\
	&{\lrbr \cht}_{[\de_i\wedge_2 D(\De_y)]}:=
	\begin{cases}
	\big\{\,k\inn{\lrbr \cht}_{[\wedge D(\De_y)]}:\,N_{[\de_i]}\!\cap\! N_{[\wedge D_p(\De_y)]}\!\cap\!\fE_k\!\ne\!\emptyset\,\big\}\quad
	&
	\tn{if}~1\!\le\!i\!\le\!\ell,
	\\
	\big\{\,k\inn{\lrbr \cht}_{[\wedge D(\De_y)]}:\,N_{[\de_i]}\!\cap\! N_{[\wedge D_q(\De_y)]}\!\cap\!\fE_k\!\ne\!\emptyset\,\big\}
	&
	\tn{if}~\ell\!+\!1\!\le\!i\!\le\!t,
	\\
	\emptyset & \tn{if}~t\!+\!1\!\le\!i\!\le\!m.
	\end{cases}
\end{split}\end{equation}
Here, the subscript $2$ of $\wedge_2$ indicates $D(\De_y)$ splits into two subsets $D_p(\De_y)$ and $D_q(\De_y)$,
and ``$\wedge$'' is taken within each of the two subsets.
Then, we obtain the following analogue of (\ref{Eqn:phi_r1p5}):
under suitable elementary column operations,
the second row of $\ti\varphi^{\rd_1\ph_1}$ can be written as
\begin{align}\label{Eqn:Case_E.1.d_phi}
	&\big(\prod_{k\in\lrbr\cht}\!\wc\ve_k\big)
	\left[\;
	\begin{matrix} 
		0&
		\undermat{\tn{DNE~if}~\ell\eq 1}{\eta_{2}& 
			\!\cdots\! 
			&
			\eta_{\ell}}
		&
		\wc\ka_{1,\ell+1}
		&
		\undermat{\tn{DNE~if}~t\eq \ell\!+\!1}{\eta_{\ell+2}
			&\!\cdots\!
			&
			\eta_{t}}
		&
		\eta_{t+1}
		&\!\cdots\!
		&
		\eta_{m}
	\end{matrix}\,
	\right],
\end{align}
where
\begin{align*}
	\eta_i\!:=
	\big(\!\!\prod_{k\in{\lrbr \cht}_{[\de_i\wedge_2 D(\De_y)]}}\!\!\!\!\!\!\!\!\!\!\!\wc\ve_k\ \; \big)
	\big(\!\prod_{e\in N_{[\de_{i}]}\bsl\De_y}\!\!\!\!\!\!\!\wc\ze_e \,\big).
\end{align*}
Below, we divide the proof into two sub-cases, depending on the cardinal of $D(\De_y)$.

\subsubsection{$|D(\De_y)|\!\ge\! 3$}\label{Subsubsec:Case_E.1.d_1}
	In this situation, we have $h\!+\!(s\!-\!\ell)\!\ge\!3$, i.e.~$\{\de_1,\de_{\ell+1}\}$ is a strict subset of $D(\De_y)$.
	Then, (\ref{Eqn:N_wedge'}) implies there exists $i_0\inn\lrbr{t} \bsl\{1,\ell\!+\!1\}$ such that 
	\begin{align*}
		{\lrbr \cht}_{[\de_{i_0}\wedge_2 D(\De_y)]}
		=
		{\lrbr \cht}_{[\wedge_2 D(\De_y)]}.
	\end{align*}
	The corresponding term $\eta_{i_0}$ in (\ref{Eqn:Case_E.1.d_phi}) can thus be written as
	$\eta_{i_0}\eq \prod_{k\in{\lrbr \cht}_{[\wedge_2 D(\De_y)]}}\!\wc\ve_k$.
	
	To see how $(\rd_3\ph_1)$ affects $\cV_y$, we mimic the construction of the first-order derived tree $\varrho_y$ as in (\ref{Eqn:tree_A.5.2}) and set
	\begin{align*}
		&D_{y,\im}:=\big\{\,\de\inn D:\;
		{\lrbr \cht}_{[\de_i\wedge_2 D(\De_y)]}\!\sqcup\! \big(N_{[\de_i]}\bsl\De_y\big)~\tn{is~inclusion-minimal~among~all}~i\!\in\!\lrbr{m}\,\big\},\\
		&
		E_y:=\!\bigcup_{\de\in D_{y,\im}}\!\!\!\!\Big({\lrbr \cht}_{[\de\wedge_2 D(\De_y)]}\!\sqcup\! \big(N_{[\de]}\bsl\De_y\big)\Big).
	\end{align*}
	On $E_y$,
	consider the partial order $\preceq_y$ determined by
	\begin{itemize}[leftmargin=*]
		\item for every $k,k'\inn{\lrbr \cht}_{[\wedge_2 D(\De_y)]}$,
		$k\!\prec_y\!k'$ if and only if $k\!>\!k'$;
		\item for every $e, e'\inn\bigcup_{\de\in D_{y,\im}}\big(N_{[\de]}\bsl\De_y\big)$,
		$e\!\preceq_y\!e'$ if and  only if every connected subcurve of $C$ containing $F$ and $e$ must contain $e'$;
		\item 
		for every $k\inn{\lrbr \cht}_{[\wedge_2 D(\De_y)]}$ and $e\inn \bigcup_{\de\in D_{y,\im}}\!\!\big(N_{[\de]}\bsl\De_y\big)$ ($\subset\!\tau_0$),
		\begin{align*}
			e\!\prec_y \!k
			\qquad\Longleftrightarrow\qquad
			e\inn\big(N_{[\wedge D_p(\De_y)]}\!\cap\!\fE_k\big)^\prec
			\sqcup 
			\big(N_{[\wedge D_q(\De_y)]}\!\cap\!\fE_k\big)^\prec
			\ \tn{in}\ \tau_0.
		\end{align*}
	\end{itemize}
	It is a direct check that $\preceq_y$ is a tree order on $E_y$ as per (\ref{Eqn:tree_order}).
	We call the rooted tree 
	\begin{align}\label{Eqn:tree_doubly_der}
		\u_y:=(E_y,\preceq_y)
	\end{align}
	\ts{first-order doubly-derived} tree.

	In Figure~\ref{Fig:double_der},
	we provide illustration of the notion of first-order doubly-derived trees.
	Here,
	$\tau$ is the same rooted tree as in Figure~\ref{figDerivedTFMRs}.
	Given
	$\wh x\inn X^\cV_{\tau}$,
	consider two lifts $y$ and $z$ of $\wh x$,
	sharing the same RLS $\ov\bE\eq\big\{\{e_a,e_b\},\{e_a,e_c,e_d\}\big\}$ and
	satisfying
	\begin{align*}
		\De_{y}=\{e_a,e_b,e_c\}\,,\qquad
		\De_{z}=\{e_a,e_b,e_c,e_d\}\,.
	\end{align*}
	The corresponding doubly-derived trees $\upsilon_{y}$ and $\upsilon_z$ are respectively illustrated in Figure~\ref{Fig:double_der}.
	
	\begin{figure}[htp]
		\begin{center}
			\begin{tikzpicture}
				\draw[xshift=-4cm]
				(0,0)--(1.2,1.2)
				(0,0)--(-.6,.6)
				(.6,.6)--(0,1.2)
				;
				\filldraw[xshift=-4cm]
				(0,0) circle (2pt)
				(.6,.6) circle (2pt)
				(-.6,.6) circle (2pt)
				(1.2,1.2) circle (2pt)
				(0,1.2) circle (2pt)
				;
				\draw[xshift=-4cm]
				(.2,.2) node[right] {\scriptsize{$e_b$}}
				(-.15,.2) node[left] {\scriptsize{$e_a$}}
				(.85,.85) node[right] {\scriptsize{$e_d$}}
				(.4,.85) node[left] {\scriptsize{$e_c$}}
				;
				
				\draw[xshift=-.5cm]
				(0,.6)--(.6,1.2)
				(0,0)--(0,.6)--(-.6,1.2);
				\filldraw[xshift=-.5cm] 
				(0,0) circle (2pt)
				(.6,1.2) circle (2pt)
				(-.6,1.2) circle (2pt)
				(0,.6) circle (2pt);
				\draw[xshift=-.5cm]
				(.45,.8) node {\scriptsize{$e_d$}}
				(-.4,.8) node {\scriptsize{$2$}}
				(-.12,.3) node {\scriptsize{$1$}}
				;
				
				\draw[xshift=2cm] (0,0)--(0,.6);
				\filldraw[xshift=2cm] (0,0) circle (2pt)
				(0,.6) circle (2pt);
				\draw[xshift=2cm]
				(-.12,.3) node {\scriptsize{$1$}}
				;
				
				\draw 
				(-4.1,-.2) node {\tiny{$o$}}
				(-3.5,-.4) node {\scriptsize{$\tau$}}
				(0,-.4) node {\scriptsize{$\upsilon_y$}}
				(2.5,-.4) node {\scriptsize{$\upsilon_z$}};
			\end{tikzpicture}
		\end{center}
		\caption{First-order doubly-derived trees}\label{Fig:double_der}
	\end{figure}

	Notice the root-to-leaf paths of $\u_y$ are those ${\lrbr \cht}_{[\de\wedge_2 D(\De_y)]}\!\sqcup\! \big(N_{[\de]}\bsl\De_y\big)$ with $\de\inn D_{y,\im}$.
	Particularly,
	${\lrbr \cht}_{[\wedge_2 D(\De_y)]}$ gives a root-to-leaf path of $\u_y$;
	in other words,
	every transverse section $\fF$ of $\u_y$ contains a unique element of ${\lrbr \cht}_{[\wedge_2 D(\De_y)]}$,
	denoted by $k_\fF$.
	
	Recall (the prototypes of) the blowup centers $K_{k,k'}$, $2\!\le\!k\!\le\!k'$, of ($\rd_3\ph_1$) are described in \S\ref{rd3ph1}.
	It is then a direct check that
	\begin{align*}
		\wc K_{k,k'}\cap\cV_y
		=
		\bigcup_{
			\fF\in\Xi(\u_y)~\tn{s.t.}~
			|\fE_{k_\fF}|= k,~
			|\fF|= k'-1
		}\!\!\!\!\!\!\!\!\!\!\!\!\!\!\!\!\!\!
		\big\{\,
		\wc\ka_{1,\ell+1}\eq 0;~
		\wc\ve_{k_\fF}\eq 0;~
		\wc\ze_e\eq 0~\forall~e\inn \fF\bsl\{k_\fF\}
		\big\}\,,
		\qquad
		2\!\le\!k\!\le\!k'.
	\end{align*}
	Following the same argument as for (\ref{Eqn:r3p4_tree_cmptb}),
	we see the above equations gives a partition of $\Xi(\u_y)$, hence of $\Xi\big(\ex(\u_y)\big)$,
	satisfying Definition~\ref{DfnGaAdm}.
	In other words,
	$(\rd_3\ph_1)$ is $\ex(\u_y)$-compatible on $\cV_y$,
	where the grafted edge corresponds to the local parameter $\wc\ka_{1,\ell+1}$.
	Comparing the root-to-leaf paths of $\ex(\u_y)$ and the terms of (\ref{Eqn:Case_E.1.d_phi}),
	we conclude from Proposition~\ref{PrpDominating} that
	the pullback of $\ti\varphi^{\rd_1\ph_1}$ is diagonalized after $(\rd_3\ph_1)$. 
	Moreover, by Lemma~\ref{LmLocalLoci}, the pullback of $\cV_y$ is smooth.
	Finally,
	notice $(\rd_3\ph_2)$-$(\rd_3\ph_4)$ do not affect the pullback of $\cV_y$,
	hence the proof for the $h\!+\!(s\!-\!\ell)\!\ge\!3$ situation is complete.
	
\subsubsection{$|D(\De_y)|\eq 2$}\label{Subsubsec:Case_E.1.d_2}
	In this situation, we have $h\!+\!(s\!-\!\ell)\!=\!2$, i.e.~$D_p(\De_y)\eq\{\de_1\}$ and $D_q(\De_y)\eq\{\de_{\ell+1}\}$.
	Then, 
	the above construction of the first-order doubly derived tree $\u_y$ still applies to this situation verbatim,
	and $(\rd_3\ph_1)$ is sill $\ex(\u_y)$-compatible on $\cV_y$,
	so the pullback of $\cV_y$ is still smooth.
	However,
	on the one hand, 
	${\lrbr \cht}_{[\wedge_2 D(\De_y)]}$ is a root-to-leaf path of $\u_y$;
	on the other hand,
	for every $i\inn\lrbr{m}\bsl\{1,\ell\!+\!1\}$,
	the term $\eta_i$ in (\ref{Eqn:Case_E.1.d_phi}) contains factors that are not labeled by ${\lrbr \cht}_{[\wedge_2 D(\De_y)]}$ because
	$N_{[\de_i]}\bsl \De_y$ is nonempty.
	Thus, the pullback of $\ti\varphi^{\rd_1\ph_1}$ is possibly not diagonalized after $(\rd_3\ph_1)$.
	
	Precisely, let $z$ be an arbitrary lift of $y$ after $(\rd_3\ph_1)$.
	From (\ref{Eqn:Case_E.1.d_phi}), we see the pullback of $\ti\varphi^{\rd_1\ph_1}$ is not diagonalized if and only if ${\lrbr \cht}_{[\wedge_2 D(\De_y)]}\!\subset\!\De_z$,
	and $\big(N_{[\de_i]}\bsl \De_y\big)\!\not\subset\!\De_z$ for any $i\inn\lrbr{m}\bsl\{1,\ell\!+\!1\}$.
	
	Nonetheless,
	since $|D(\De_y)|\eq 2$ and $\ka_{1,\ell+1}(\wh x)\eq 0$,
	we define the rooted tree $\tau_2$ parallel to (\ref{Eqn:tree_A.4}),
	but with two modifications:
	\begin{itemize}[leftmargin=*]
		\item each $N_{[\de]}$ should be replaced by $N_{[\de]}\bsl\De_y$, and
		\item $\de_2$ should be replaced by $\de_{\ell+1}$.
	\end{itemize}
	It is then a direct check that
	$(\rd_3\ph_2)$ does not affect the pullback of $\cV_y$,
	whereas the prototype of $(\rd_3\ph_3)$,
	i.e.~the blowup of $\fM_2^{\rm div}\!\times\!_{\fM_2^{\rm wt}}\ti\fM^{\rd_2}$ successively along the proper transforms of $\cH_k$, $k\!\ge\!0$,
	is $\ex(\tau_2)$-compatible on $\cV_y$.
	Hence by Corollary~\ref{Crl:PT},
	$(\rd_3\ph_3)$ is $\PT_z\big(\ex(\tau_2)\big)$-compatible,
	or equivalently $\ex\big(\PT_z(\tau_2)\big)$-compatible, on a small (smooth) neighborhood $\cV_z$ of $z$.
	Particularly, this implies $\cV_z$ is smooth.
	Moreover,
	the root-to-leaf paths of $\tau_2$ are those inclusion-minimal $N_{[\de_i]}\bsl(\De_y\!\sqcup\!\De_z)$,
	along with the grafted edge corresponding to $\wc\ka_{1,\ell+1}$.
	Therefore,
	the pullback of $\ti\varphi^{\rd_3\ph_1}$ becomes diagonalized on $\cV_z$.
	Finally, notice $(\rd_3\ph_4)$ does not affect $\cV_z$, so the proof of the $h\!+\!(s\!-\!\ell)\!=\!2$ situation is complete.

\subsection{Proof of \ref{Case:1}, Part III}
In this subsection,
we aim to show \ref{Case:1}.\ref{Case:1.1}.\ref{Case:1.1.e},
when $D(\De_y)$ is contained in a unique tail,
whose pivotal node is not a Weierstrass point.
W.l.o.g.~we assume there exist $1\!\le\!h\!\le\!\ell\!\le\!m$ such that
\begin{align*}
	D(\De_y)=\{\de_1,\cdots,\de_h\},\qquad
	\lr{\de_i}=\lr{\de_1}\quad\forall~1\!\le\!i\!\le\!\ell,\qquad
	\lr{\de_i}\ne\lr{\de_1}\quad\forall~\ell\!<\!i\!\le\!m.
\end{align*}

Since $\lr{\de_1}$ is not a Weierstrass point of $F$,
by Proposition~\ref{Prp:phi_key}~\ref{Part:theta} and the argument for (\ref{Eqn:phi_r1p5}),
after suitable elementary row and column operations,
the pullback $\ti\varphi^{\rd_1\ph_1}$ of the structural homomorphism $\varphi$ takes the following form on $\cV_y$:
\begin{align}\label{Eqn:r2_phi}
	\Pi^{\rd_1\ph_1}
	\cdot
	\left[\;
	\begin{matrix} 
		1 & 0 & \cdots & 0  & 0 &\cdots&0 
		\\
		0&
		\undermat{\tn{DNE~if}~\ell\eq 1}{\eta_{2}& 
			\cdots 
			&
			\eta_{\ell}}
		&
		\wc\ka_{1,\ell+1}
		\eta_{\ell+1}
		&\cdots
		&
		\wc\ka_{1m}
		\eta_{m}
	\end{matrix}\;
	\right]\,,
\end{align}
where 
\begin{align*}
	\Pi^{\rd_1\ph_1}:=
	\prod_{k\in\lrbr\cht}\!\wc\ve_k,\qquad
	\eta_{i}:=
	\big(\prod_{k\in\lrbr{\cht}_{[\de_i\wedge D(\De_y)]}}
	\!\!\!\!\!\!\!\!\wc\ve_k\ \big)\big(\prod_{e\in N_{[\de_i]}\bsl\De_y}\!\!\!\wc\ze_e\ \big),
\end{align*}
and $\lrbr{\cht}_{[\de_i\wedge D(\De_y)]}$'s are as in (\ref{Eqn:[m]_de_wedge}).
Moreover, when $h\!\ge\!2$,
there exists $2\!\le\!i_0\!\le\!h$ such that
$\lrbr{\cht}_{[\de_{i_0}\wedge D(\De_y)]}\eq 
\lrbr{\cht}_{[\wedge D(\De_y)]}$.

The assumption of \ref{Case:1}.\ref{Case:1.1}.\ref{Case:1.1.e} implies $(\rd_1\ph_2)$-$(\rd_1\ph_4)$ and $(\rd_3\ph_2)$ do not affect $\cV_y$.
Depending on whether $D(\De_y)$ is a singleton (i.e.~whether $h\eq 1$),
$(\rd_1\ph_5)$, $(\rd_3\ph_3)$, and $(\rd_3\ph_4)$ may or may not be relevant,
so we divide the proof of \ref{Case:1}.\ref{Case:1.1}.\ref{Case:1.1.e} into two sub-cases as follows.

\subsubsection{$|D(\De_y)|\!\ge\!2$}
\label{Subsubsec:Case_E.1.e_1}
In this situation,
$(\rd_1\ph_5)$, $(\rd_3\ph_3)$, and $(\rd_3\ph_4)$ do not affect $\cV_y$.

To see how $(\rd_2)$ affects $\cV_y$,
we break it into ``phases''.
Precisely,
for each $j\!\ge\!1$,
we denote by $\pi'_j$ the blowup of $\fM_2^{\rm div}\!\times_{\fM_2^{\rm wt}}\!\ti\fM^{\rd_1}$ successively along the proper transforms of $$X_{j,j'}^{\rm div}:=\fM_2^{\rm div}\!\times_{\fM_2^{\rm wt}}\!X_{j,j'}, \qquad j'\!=j,\,j\!+\!1,\,\cdots.$$
As mentioned in \S\ref{rd2},
the stability of  $\fM_2^{\rm wt}$ implies there exists a finite number $j_0\!\ge\!1$ such that 
$(\rd_2)$ is equivalent to exerting $\pi'_{j_0}$ to $\cV_y$,
then exerting $\PT_{\pi'_{j_0}}\!(\pi'_{j_0-1})$ (see Definition~\ref{Dfn:PT}) to the pullback of $\cV_y$, and so forth.
Below, we find the rooted trees that these $\pi'_j$'s are compatible with on $\cV_y$.

Consider the first-order derived tree $\varrho_y$ given in (\ref{Eqn:tree_A.5.2}),
which can be defined verbatim in the current situation.
(Notice that $D\bsl F\eq D$ here, while $D\bsl F\eq D\bsl\{\de_1\}$ in (\ref{Eqn:tree_A.5.2}).)
The root-to-leaf paths of $\varrho_y$ are exactly the inclusion-minimal elements of 
\begin{align*}
	\big\{\,\lrbr{\cht}_{[\de_i\wedge D(\De_y)]}\!\sqcup\!
\big(N_{[\de_i]}\bsl\De_y\big):\,i\!\in\!\lrbr{m}\,\big\}\,.
\end{align*}
For each $k\inn\lrbr{\cht}_{[\wedge D(\De_y)]}$,
we denote by $\varrho_{y,k}$ the {\it subposet} of $\varrho_y$ determined by 
\begin{align*}
	\varrho_{y,k}:=
	\bigcup_{\fF\in\Xi(\varrho_y)\,\tn{s.t.}\,k\in\fF}\!\!\!\!\!\!
	\fF
	\qquad\subset\varrho_y\,.
\end{align*}
It is straightforward that each $\varrho_{y,k}$ is a rooted tree as per Definition~\ref{Dfn:Rooted_tree},
containing the singleton $\{k\}$ as a root-to-leaf path.
Below, we show $\pi'_{|\fE_k|}$ is $\varrho_{y,k}$-compatible on $\cV_y$.

In fact, it is a direct check that for each $j\!\ge\!1$,
the following holds. 
\begin{itemize}[leftmargin=*]
	\item If there exists $k\inn\lrbr{\cht}_{[\wedge D(\De_y)]}$ such that $|\fE_k|\eq j$ (recall $\fE_k\inn\ov\bE$ is a transverse section of $\tau_0$),
	then the stability of $\Mw$ implies this $k$ is unique.
	We denote it by $k_j$.
	In this situation, the proper transform of $X^{\rm div}_{j,j'}$ takes the following form on $\cV_y$:
	\begin{align*}
		\wc X^{\rm div,\rd_1\ph_1}_{j,j'}\!\cap\!\cV_y
		=
		\bigcup_{\fF\in \Xi(\varrho_{y,k_j})\,\tn{s.t.}\,
		|\fF|= j'}\!\!\!\!\!\!\!\!\!\!\!
		\big\{\,
		\wc\ve_{k_j}\eq 0\,;\;
		\wc\ze_e\eq 0~\forall\,e\inn\fF\bsl\{k_j\}
		\,\big\}\,.
	\end{align*}
	Here, recall  
	$\{\wc\ve_{k_j}\eq 0\}$ is the exceptional divisor on $\cV_y$ determined by $\fE_{k_j}$; see (\ref{e_excdiv}) and (\ref{e_wcEtiE}).
	Once again, by the stability of $\Mw$,
	$\pi'_j$ is $\varrho_{y,k_j}$-compatible on $\cV_y$.
	\item 
	Otherwise, we have
	\begin{align*}
		\wc X^{\rm div,\rd_1\ph_1}_{j,j'}\!\cap\!\cV_y
		=\emptyset,
	\end{align*}
	hence $\pi'_{j}$ does not affect $\cV_y$.
\end{itemize}
By Lemma~\ref{LmLocalLoci},
the pullback of $\cV_y$ becomes smooth after $(\rd_2)$.

Next, we study the pullback $\ti\varphi^{\rd_2}$ of $\varphi$ after $(\rd_2)$.
For each $k\inn\lrbr{\cht}_{[\wedge D(\De_y)]}$,
we fix a lift $y_k$ of $y$ after exerting the proper transform of $\pi'_{|\fE_k|}$, satisfying the image of $y_k$ under the proper transform of $\pi'_{|\fE_k|}$ is $y_{k+1}$;
here, when $k\eq\max\big(\lrbr{\cht}_{[\wedge D(\De_y)]}\big)$,
we set $y_{k+1}\!:=\!y$.
Particularly,
$$y_1\in\ti\fM_2^{\rm div}.$$
The RLS of each $y_k$ is denoted by $\ov\bE_k\eq\{\fE_{k,1},\ldots,\fE_{k,\cht_k}\}$.
It is noteworthy that some $\cht_k$ can be 0;
i.e.~the proper transform of $\pi'_{|\fE_k|}$ does not affect a neighborhood of $y_k$.
Near the final lift $y_1$,
the local parameters corresponding the exceptional divisors obtained in $(\rd_2)$ are denoted by $\wc\ve_{k,k'}$,
$k\inn\lrbr{\cht}_{[\wedge D(\De_y)]}$,
$k'\inn\lrbr{\cht_k}$.
Let
\begin{align}\label{Eqn:r2_prod_ve}
	\Pi^{\rd_2}:=
	\prod_{k\in\lrbr{\cht}_{[\wedge D(\De_y)]}}
	\prod_{k'\in\lrbr{\cht_k}}
	\wc\ve_{k,k'}\,.
\end{align}

The construction of $\varrho_{y,k}$ implies
\begin{align}\label{Eqn:r2_trans}
	\Xi(\varrho_{y})=\bigsqcup_{k\in \lrbr{\cht}_{[\wedge D(\De_y)]}}\!\!\!\!\!\!\!\!\Xi(\varrho_{y,k})\,.
\end{align}
So, the intersection of every root-to-leaf path of $\varrho_y$ and every $\fE_{k,k'}$ is a singleton.
Therefore, for every $2\!\le\!i\!\le\!m$,
the pullback $\ti\eta_i^{\rd_2}$ of  $\eta_i$ after $(\rd_2)$ satisfies
\begin{align}\label{Eqn:r2_Row2_etc}
	\Pi^{\rd_2}\big|\;
	\ti\eta_i^{\rd_2}
	\ \ \tn{on a neighborhood of}\ y_1.
\end{align}
Along with Lemma~\ref{Lm:PT_trans_sect},
(\ref{Eqn:r2_trans}) also implies the existence of a root-to-leaf path $\wp$ of $\varrho_y$ satisfying
\begin{align}\label{Eqn:r2_dom}
	\wp\subset\De^{\rd_2}_{y_1} :=
	\bigsqcup_{k\in \lrbr{\cht}_{[\wedge D(\De_y)]}}\!\!\!\!\!\!\!
	\De_{y_k}\,,
\end{align}
where $\De_{y_k}$ are as in Definition~\ref{DfnDominant}.
In other words, there exists $2\!\le\!i_1\!\le\!m$ such that 
\begin{align}\label{Eqn:r2_Row2_dom}
	\ti\eta_{i_1}^{\rd_2}\big/
	\Pi^{\rd_2}\ \ \tn{is invertible on a neighborhood of}\ y_1.
\end{align}
\begin{itemize}[leftmargin=*]
	\item If there exists a root-to-leaf path $\wp$ of $\varrho_y$ satisfying (\ref{Eqn:r2_dom})
	as well as $\wp\!\cap\!\lrbr\cht\!\ne\!\emptyset$,
	then the index $i_1$ above can be taken in $\lrbr{\ell}$,
	i.e.~$\lr{\de_{i_1}}\eq\lr{\de_1}$.
	So, by (\ref{Eqn:r2_phi}), (\ref{Eqn:r2_Row2_etc}) and (\ref{Eqn:r2_Row2_dom}),
	$\ti\varphi^{\rd_2}$ is diagonalized after $(\rd_2)$.
	\item 
	If every root-to-leaf path $\wp$ of $\varrho_y$ satisfying (\ref{Eqn:r2_dom}) is disjoint from $\lrbr{\cht}$,
	then we set
	\begin{align}\label{Eqn:D(r2_dom)}
		D(\De_{y_1}^{\rd_2}):=
		\big\{\,
		\de_{i_1}:\,
		2\!\le\!i_1\!\le\!m,~
		i_1~\tn{satisfies}~(\ref{Eqn:r2_Row2_dom})
		\,\big\}\,.
	\end{align}
	For every $\de_{i_1}\inn D(\De_{y_1}^{\rd_2})$, we have  $\lr{\de_{i_1}}\!\ne\!\lr{\de_1}$.
	We have the following two possibilities.
	\begin{itemize}[leftmargin=*]
		\item 
		If there exists $\de_{\ell'}\inn D(\De_{y_1}^{\rd_2})$ such that  $\lr{\de_{\ell'}}$ is not conjugate to $\lr{\de_1}$ on $F$,
		then $\wc\ka_{1,\ell'}$ is invertible on $\cV_y$.
		By (\ref{Eqn:r2_phi}), (\ref{Eqn:r2_Row2_etc}) and (\ref{Eqn:r2_Row2_dom}),
		$\ti\varphi^{\rd_2}$ is diagonalized after $(\rd_2)$.
		\item 
		If 
		$\lr{\de_{\ell'}}$ is conjugate to $\lr{\de_1}$ on $F$ for all $\de_{\ell'}\inn D(\De_{y_1}^{\rd_2})$,
		then the assumption of \ref{Case:1}.\ref{Case:1.1} implies these $\de_{\ell'}$ are on the same tail $C''$,
		whose pivotal node is distinct from and conjugate to $\lr{\de_1}$;
		in addition, $\wc\ka_{1,\ell+1}$ is a local parameter vanishing at $y_1$.
		
		Notice we can rewrite the notion ${\lrbr \cht}_{[\wedge_2 D(\De_y)]}$ of (\ref{Eqn:cht_wedge'}) in the following  equivalent way:
		\[
		{\lrbr \cht}_{[\wedge_2 D(\De_y)]}:=
		\Big\{\,k\inn\lrbr\cht:\,
		\Big|\big(\bigcup_{\de\in D(\De_y)}\!\!\!\! N_{[\de]}\big)\!\cap\!\fE_k\Big|\!=\!2\,\Big\},
		\]
		which gives rise to the same partition $D(\De_y)\eq D_p(\De_y)\!\sqcup\!D_q(\De_y)$.
		Then, we can define ${\lrbr \cht}_{[\de_i\wedge_2 D(\De_y)]}$, $i\inn\lrbr m$, likewise.
		Applying the same idea to the current situation,
		we set
		\begin{align}\label{Eqn:cht_wedge''}
		{\lrbr \cht}_{[\wedge_2 D(\De_y,\De_{y_1}^{\rd_2})]}:=
		\Big\{\,k\inn\lrbr\cht:\,
		\Big|\big(\bigcup_{\de\in D(\De_y)\sqcup D(\De_{y_1}^{\rd_2})}\!\!\!\!\!\!\!\!\!\!\!\!\! N_{[\de]}\,\big)\!\cap\!\fE_k\Big|\!=\!2\,\Big\},
		\end{align}
		and define ${\lrbr \cht}_{[\de_i\wedge_2 D(\De_y,\De_{y_1}^{\rd_2})]}$, $i\inn\lrbr m$, analogously.
		Then, by Proposition~\ref{Prp:phi_key}~\ref{Part:kappa},
		under suitable elementary column operations,
		the second row of $\ti\varphi^{\rd_2}$ can be written as
		\begin{align*}
			&
			\Pi^{\rd_1\ph_1}\!\cdot\!
			\Pi^{\rd_2}\cdot\left[\;
			\begin{matrix} 
				0&
				\wc\eta_{2}& 
					\!\cdots\! 
					&
					\wc\eta_{\ell}
				&
				\wc\ka_{1,\ell+1}
				&
				\wc\eta_{\ell+2}
				&\!\cdots\!
				&
				\wc\eta_{m}
			\end{matrix}\,
			\right],\\
			&\tn{where}\qquad\wc\eta_i\!:=
			\big(\!\!\prod_{k\in{\lrbr \cht}_{[\de_i\wedge_2 D(\De_y,\De_{y_1}^{\rd_2})]}}\!\!\!\!\!\!\!\!\!\!\wc\ve_k\ \big)
			\big(\!\prod_{e\in N_{[\de_{i}]}\bsl(\De_y\sqcup \De^{\rd_2}_{y_1})}\!\!\!\!\!\!\!\wc\ze_e \ \big)\,.
		\end{align*}
		Comparing the above expression of $\ti\varphi^{\rd_2}$ with (\ref{Eqn:Case_E.1.d_phi}),
		we see the construction of \S\ref{Subsubsec:Case_E.1.d_1} can be applied in a parallel manner,
		which implies the pullback of $\varphi^{\rd_2}$ becomes diagonalized,
		and the pullback of a neighborhood of $y_1$ is smooth, after $(\rd_3\ph_1)$.
	\end{itemize}
\end{itemize}

\subsubsection{$|D(\De_y)|\!=\!1$}
In this situation,
$(\rd_1\ph_5)$ is involved.
By defining the subsets $D_{\im,y}\!\subset\!D$ and $E_y\!\subset\!N(C)$ exactly in the same way as $D_{\im}$ and $E$ in (\ref{Eqn:tau_CaseA.4}),
and defining the relation $\preceq_y$ on $E_y$ given by $e\!\preceq_y\!e'$ if and  only if any connected subcurve of $C$ containing $F$ and $e$ must contain $e'$,
we see
\begin{align*}
	\tau_{1,y}:=(E_y,\preceq_y)
\end{align*}
is a rooted tree,
whose root-to-leaf paths are exactly $N_{[\de]}\bsl\De_y$, $\de\inn D_{\im,y}$.
It is then a direct check that $(\rd_1\ph_5)$ is $\tau_{1,y}$-compatible on $\cV_y$.
After $(\rd_1\ph_5)$,
we fix a lift $z$ of $y$ in $\ti\fM^{\rd_1\ph_5}$ and define $D(\De_z)$ to be the analogue of (\ref{Eqn:I_y}).
\begin{itemize}[leftmargin=*]
	\item If there exists $\de\inn D(\De_z)$ such that $\lr{\de}\!\ne\!\lr{\de_1}$,
	then the proof of \ref{Case:5}.\ref{Case:5.5} can be applied to this situation in a parallel way;
	c.f.~\S\ref{Subsec:Case5.5}-\S\ref{SubsecTwoTrees}.
	\item If $\lr{\de}\!=\!\lr{\de_1}$ for all $\de\inn D(\De_z)$,
	then the argument in \S\ref{Subsubsec:Case_E.1.e_1} can be applied to this situation in a parallel way.
\end{itemize}

\subsection{Second-order derived trees and the proof of \ref{Case:1}, Part IV}
In this subsection,
we aim to show \ref{Case:1}.\ref{Case:1.1}.\ref{Case:1.1.f},
when $D(\De_y)$ is contained in a unique tail,
whose pivotal node is a Weierstrass point.
W.l.o.g.~we assume there exists $1\!\le\!\ell\!\le\!m$ such that
\begin{align*}
	D(\De_y)\subset\{\de_1,\cdots,\de_\ell\},\qquad
	\lr{\de_i}=\lr{\de_1}\quad\forall~1\!\le\!i\!\le\!\ell,\qquad
	\lr{\de_i}\ne\lr{\de_1}\quad\forall~\ell\!<\!i\!\le\!m.
\end{align*}

Since $\lr{\de_1}$ is a Weierstrass point of $F$,
we have $\ka_{1i}(\ti x)\eq 0$ for all $i\inn\lrbr\ell$,
whereas $\ka_{1j}$, $\ell\!<\!j\!\le\!m$,
are all invertible.
By Proposition~\ref{Prp:phi_key}~\ref{Part:theta} and the argument for (\ref{Eqn:phi_r1p5}),
after suitable elementary row and column operations,
the pullback $\ti\varphi^{\rd_1\ph_1}$ of the structural homomorphism $\varphi$ takes the following form on $\cV_y$:
\begin{align}\label{Eqn:r3p2_phi}
	\ti\varphi^{\rd_1\ph_1}=
	\Pi^{\rd_1\ph_1}
	\cdot
	\left[\;
	\begin{matrix} 
		1 & 0 & \cdots & 0  & 0 &\cdots&0 
		\\
		0&
		\undermat{\tn{DNE~if}~\ell\eq 1}{\wc\ka_{12}\eta_{2}'& 
			\cdots 
			&
			\wc\ka_{1\ell}\eta_{\ell}'}
		&
		\eta_{\ell+1}'
		&\cdots
		&
		\eta_{m}'
	\end{matrix}\;
	\right]\,,
\end{align}
where 
\begin{align*}
	\Pi^{\rd_1\ph_1}:=
	\prod_{k\in\lrbr\cht}\!\wc\ve_k,\qquad
	\eta_{i}':=
	\big(\prod_{k\in\lrbr{\cht}_{[\de_1\wedge \de_i]}}
	\!\!\!\!\!\!\wc\ve_k\ \big)\big(\prod_{e\in N_{[\de_i]}\bsl\De_y}\!\!\!\!\!\wc\ze_e\ \big),
\end{align*}
and $\lrbr{\cht}_{[\de_1\wedge \de_i]}$'s are as in (\ref{Eqn:[m]_wedge}).

The assumption of \ref{Case:1}.\ref{Case:1.1}.\ref{Case:1.1.f} implies $(\rd_1\ph_2)$-$(\rd_1\ph_4)$ do not affect $\cV_y$.
Depending on the cardinal of $D(\De_y)$,
$(\rd_1\ph_5)$, $(\rd_2)$, $(\rd_3\ph_1)$, $(\rd_3\ph_3)$, and $(\rd_3\ph_4)$ may or may not be relevant,
so we divide the proof of \ref{Case:1}.\ref{Case:1.1}.\ref{Case:1.1.f} into the following sub-cases.

\subsubsection{$|D(\De_y)|\!\ge\!2$}
\label{Subsubsec:Case_E.1.f_1}
In this situation,
we have $h\!\ge\!2$,
hence $(\rd_1\ph_5)$ and $(\rd_3\ph_4)$ do not affect $\cV_y$.

As shown in \S\ref{Subsubsec:Case_E.1.f_1},
the pullback of $\cV_y$ is smooth after $(\rd_2)$.
In addition,
we obtain a sequence $y_k$, $k\inn{\lrbr\cht}_{[\wedge D(\De_y)]}$, of the lifts of $y$,
with $y_1\inn\ti\fM_2^{\rm div}$.
For $\Pi^{\rd_2}$ as in (\ref{Eqn:r2_prod_ve}),
the analogue of (\ref{Eqn:r2_Row2_etc}) is still valid in the current case,
with each $\ti\eta_i^{\rd_2}$ replaced with the pullback of $\eta_i'$.
Let $\De_{y_1}^{\rd_2}$ and $D(\De_{y_1}^{\rd_2})$ be as in 
(\ref{Eqn:r2_dom}) and (\ref{Eqn:D(r2_dom)}),
respectively,
and ${\lrbr \cht}_{[\wedge_2 D(\De_y,\De_{y_1}^{\rd_2})]}$ be as in  (\ref{Eqn:cht_wedge''}).

\begin{itemize}[leftmargin=*]
	\item If there exists $\ell\!<\!j\!\le\!m$ such that  $\de_j\inn D(\De_{y_1}^{\rd_2})$,
	then the pullback of $\eta_i'$ is equal to $\Pi^{\rd_2}$ (up to a unit) locally near $y_1$.
	Therefore,
	$\ti\varphi^{\rd_2}$ is diagonalized.
	
	\item if $D(\De_{y_1}^{\rd_2})\!\subset\!\{\de_1,\ldots,\de_\ell\}$,
	then the following analogue of Lemma~\ref{Lm:wedge} holds:
	there exists $2\!\le\!i\!\le\!h$,
	say $i\eq 2$, such that
	\begin{align*}
		N_{[\de_{1}\wedge \de_2]}= 
		\bigcap_{\de\in D(\De_y)\sqcup D(\De_{y_1}^{\rd_2})} \!\!\!\!\!\!\!\!\!N_{[\de]}.
	\end{align*}
	Applying Proposition~\ref{Prp:phi_key}~\ref{Part:kappa},
	we see after suitable elementary row and column operations,
	the second row of the pullback $\ti\varphi^{\rd_2}$ of (\ref{Eqn:r3p2_phi}) takes the following form near $y_1$:
	\begin{align*} 
		&\Pi^{\rd_1\ph_1}\cdot
		\Pi^{\rd_2}\cdot	
		\left[\;
		\begin{matrix} 
			0&
			\wc\ka_{12}
			&
			\eta_3''
			&\cdots
			&
			\eta_{m}''
		\end{matrix}\;
		\right]\,,
		\qquad
		\tn{where}\\
		&
		\eta_i'':=
		\Big(\prod_{
			k\in\lrbr{\cht}_{[\de_i\wedge D(\De_y,\De_{y_1}^{\rd_2})]}}
		\prod_{ 
			k'\in\lrbr{\cht_k}}
		\!\!\wc\ve_{k,k'}\Big)
		\Big(
		\prod_{k\in\lrbr{\cht}_{[\de_i\wedge_2 D(\De_y,\De_{y_1}^{\rd_2})]}}
		\!\!\!\!\!\!\wc\ve_k\Big)
		\Big(\prod_{e\in N_{[\de_{i}]}\bsl(\De_y\sqcup \De^{\rd_2}_{y_1})}\!\!\!\!\!\!\!\wc\ze_e \ \Big)\,,
	\end{align*}
	where \begin{align*}
		{\lrbr \cht}_{[\wedge D(\De_y,\De_{y_1}^{\rd_2})]}:=
		\Big\{\,k\inn\lrbr\cht:\,
		\Big|\big(\bigcup_{\de\in D(\De_y)\sqcup D(\De_{y_1}^{\rd_2})}\!\!\!\!\!\!\!\!\!\!\!\!\! N_{[\de]}\,\big)\!\cap\!\fE_k\Big|\!=\!1\,\Big\},
	\end{align*}
	${\lrbr \cht}_{[\de_i\wedge D(\De_y,\De_{y_1}^{\rd_2})]}$ are analogous to (\ref{Eqn:cht_wedge'}),
	and $\cht_k$ are as in \S\ref{Subsubsec:Case_E.1.e_1}.
	
	To see how ($\rd_3\ph_1$) and ($\rd_3\ph_2$) affects a neighborhood of $y_1$,
	first, we observe the inclusion-minimal elements of the following set
	\begin{align*}
		\Big\{\,
		\lrbr{\cht}_{[\de_i\wedge_2 D(\De_y,\De_{y_1}^{\rd_2})]}\sqcup 
		\big(N_{[\de_{i}]}\bsl(\De_y\sqcup \De^{\rd_2}_{y_1})\big)\,:\;
		3\!\le\!i\!\le\!m
		\,\Big\}
	\end{align*} 
	are exactly the root-to-leaf path of the rooted tree $\u_{y_1}$,
	and $(\rd_3\ph_1)$ is $\ex(\u_{y_1})$-compatible near $y_1$,
	where the grafted edge corresponding to the local parameter $\wc\ka_{12}$.
	
	Next,
	we mimic the first-order derived tree $\varrho_y$ of (\ref{Eqn:tree_A.5.2}) and construct the \ts{second-order derived tree} $\varrho_{y_1}^{(2)}$ at $y_1$.
	Precisely,
	let $E'_{y_1}$ be the union of the inclusion-minimal elements of the following set
	\begin{align}\label{Eqn:2nd_der_inc_min}
		\Big\{\,
		\big\{(k,k'):k\inn\lrbr{\cht}_{[\de_i\wedge D(\De_y,\De_{y_1}^{\rd_2})]},\,
		k'\in\lrbr{\cht_k}
		\big\}\sqcup 
		\big(N_{[\de_{i}]}\bsl(\De_y\sqcup \De^{\rd_2}_{y_1})\big)\,:\;
		3\!\le\!i\!\le\!m
		\,\Big\}\,.
	\end{align}
	This is the analogue of $E_y$ in (\ref{Eqn:tree_A.5.2}).
	Then, we endow $E'_{y_1}$ with the partial order $\prec_{y_1}'$ such that 
	\begin{itemize}[leftmargin=*]
		\item for every $(k,k'), (h,h')\inn E_{y_1}'$, we have
		$(k,k')\!\prec_{y_1}'\!(h,h')$ if and only if either $k\!>\!h$, or $k\eq h$ and $k'\!>\!h'$;
		\item for every $e, e'\inn E'_{y_1}\!\cap\!N(C)$, we have
		$e\!\preceq_y\!e'$ if and  only if any connected subcurve of $C$ containing $F$ and $e$ must contain $e'$;
		\item 
		for every $(k,k')$ and $e$ in $E'_{y_1}$,
		$e\!\prec_{y_1}' \!(k,k')$ if and only if $e\inn\big(\lrbr{\cht_k}\!\cap\!\fE_{k,k'}\big)^{\prec_y}$ in $\varrho_{y,k}$,
		where $\varrho_{y,k}$ is as in \S\ref{Subsubsec:Case_E.1.e_1}.
	\end{itemize}
	It is a direct check that $\prec'_{y_1}$ is a tree order as per (\ref{Eqn:tree_order}).
	Then, we set $$\varrho_{y_1}^{(2)}:=
	\big(\,E'_{y_1}\,,\,\prec'_{y_1}\big)\,.$$
	 
	In Figure~\ref{Fig:2nd_der}, we provide an example illustrating both first-order doubly-derived and second-order derived trees.
	Here,
	$\tau$ is the same rooted tree as in Figure~\ref{figDerivedTFMRs}.
	Given
	$\wh x\inn X^\cV_{\tau}$,
	consider a lift  $y$  of $\wh x$ after $(\rd_1\ph_1)$, whose RLS and dominant edges are given by
	satisfying
	\begin{align*}
		\ov\bE\eq\big\{\{e_a,e_b\},\{e_a,e_c,e_d\}\big\},\qquad
		\De_{y}=\{e_b,e_c\}\,.
	\end{align*}
	The first-order derived tree $\varrho_y$ is the same as Graph $(5)$ of Figure~\ref{figDerivedTFMRs}.
	
	Let $y_1$ be a lift of $y$ after $(\rd_2)$,
	satisfying
	\begin{align*}
		\De^{\rd_2}_{y_1}=\{e_d,1\}\,.
	\end{align*}
	The corresponding doubly-derived tree $\upsilon_{y_1}$ and second-order derived tree $\varrho^{(2)}_{y_1}$ are respectively illustrated in Figure~\ref{Fig:2nd_der}.
	
	\begin{figure}[htp]
		\begin{center}
			\begin{tikzpicture}
				\draw[xshift=-4cm]
				(0,0)--(1.2,1.2)
				(0,0)--(-.6,.6)
				(.6,.6)--(0,1.2)
				;
				\filldraw[xshift=-4cm]
				(0,0) circle (2pt)
				(.6,.6) circle (2pt)
				(-.6,.6) circle (2pt)
				(1.2,1.2) circle (2pt)
				(0,1.2) circle (2pt)
				;
				\draw[xshift=-4cm]
				(.2,.2) node[right] {\scriptsize{$e_b$}}
				(-.15,.2) node[left] {\scriptsize{$e_a$}}
				(.85,.85) node[right] {\scriptsize{$e_d$}}
				(.4,.85) node[left] {\scriptsize{$e_c$}}
				;
				
				\draw[xshift=-.5cm] (0,0)--(1.2,1.2);
				\draw[xshift=-.5cm] (0,0)--(-.6,.6);
				\draw[xshift=-.5cm] (.6,.6)--(0,1.2);
				\filldraw[xshift=-.5cm] (0,0) circle (2pt)
				(.6,.6) circle (2pt)
				(-.6,.6) circle (2pt)
				(1.2,1.2) circle (2pt)
				(0,1.2) circle (2pt);
				\draw[xshift=-.5cm]
				(.45,.2) node {\scriptsize{$1$}}
				(-.4,.2) node {\scriptsize{$e_a$}}
				(1.05,.8) node {\scriptsize{$e_d$}}
				(.2,.8) node {\scriptsize{$2$}}
				;
				
				\draw[xshift=3cm]
				(-.6,.6)--(0,0)--(.6,.6);
				\filldraw[xshift=3cm] (0,0) circle (2pt)
				(.6,.6) circle (2pt)
				(-.6,.6) circle (2pt);
				\draw[xshift=3cm]
				(.45,.2) node {\scriptsize{$e_a$}}
				(-.4,.2) node {\scriptsize{$2$}}
				;
				
				\draw[xshift=6cm]
				(-.6,.6)--(0,0)--(.6,.6);
				\filldraw[xshift=6cm] (0,0) circle (2pt)
				(.6,.6) circle (2pt)
				(-.6,.6) circle (2pt);
				\draw[xshift=6cm]
				(.45,.2) node {\scriptsize{$e_a$}}
				(-.65,.2) node {\scriptsize{$(1,1)$}}
				;
				
				\draw 
				(-4.1,-.2) node {\tiny{$o$}}
				(-3.5,-.4) node {\scriptsize{$\tau$}}
				(0,-.4) node {\scriptsize{$\varrho_{y}$}}
				(3.5,-.4) node {\scriptsize{$\upsilon_{y_1}$}}
				(6.5,-.4) node {\scriptsize{$\varrho^{(2)}_{y_1}$}};
			\end{tikzpicture}
		\end{center}
		\caption{First-order doubly-derived and second-order derived trees}\label{Fig:2nd_der}
	\end{figure}

	Observe that the root-to-leaf paths of $\varrho_{y_1}^{(2)}$ are exactly the inclusion-minimal elements of (\ref{Eqn:2nd_der_inc_min}).	
	In addition, recall in \S\ref{rd3ph2},
	(the prototypes of) the blowup centers of ($\rd_3\ph_2$) are denoted by $W_{k,k',k''}$, $1\!\le\!k\!\le\!k'\!\le\!k''$.
	It is also a direct check that near $y_1$,
	the proper transform of each $W_{k,k',k''}$ takes the form
	\begin{align*}
		\bigcup_{
			\fF\in\Xi(\varrho_{y_1}^{\rd_2})~\tn{s.t.}~
			|\fE_{k_\fF}|= k,~
			|\fF|= k''
		}\!\!\!\!\!\!\!\!\!\!\!\!\!\!\!\!\!\!
		\big\{\,
		\wc\ka_{12}\eq 0;~
		\wc\ve_{k_\fF,k'}\eq 0;~
		\wc\ze_e\eq 0~\forall~e\inn \fF\bsl\{(k_\fF,k')\}
		\big\}\,,
		\qquad
		1\!\le\!k\!\le\!k'\!\le\!k''.
	\end{align*}
	Following the same argument as for (\ref{Eqn:r3p4_tree_cmptb}),
	we see the above equations gives a partition of $\Xi(\varrho_{y_1}^{\rd_2})$, hence of $\Xi\big(\ex(\varrho_{y_1}^{\rd_2})\big)$,
	satisfying Definition~\ref{DfnGaAdm}.
	In other words, the prototype of
	$(\rd_3\ph_2)$ is $\ex(\varrho_{y_1}^{\rd_2})$-compatible near $y_1$,
	where the grafted edge corresponds to the local parameter $\wc\ka_{12}$.
	Consequently, for any lift $z$ of $y_1$ after $(\rd_3\ph_1)$, Corollary~\ref{Crl:PT} ensures
	$(\rd_3\ph_2)$ is $\PT_{z}\big(\ex(\varrho_{y_1}^{\rd_2})\big)$-compatible near $z$.
	
	From the above two paragraphs,
	we see the pullback of a neighborhood of $y_1$ is smooth after $(\rd_3\ph_2)$. 
	Taking the above equations for $\eta''_i$ into consideration,
	we conclude that the pullback of  $\ti\varrho^{\rd_2}$ is diagonalized after $(\rd_3\ph_2)$.
\end{itemize}

\subsubsection{$|D(\De_y)|\!=\!1$}

In this situation,
$(\rd_1\ph_5)$ is involved,
which is locally $\tau_{y,1}$-compatible on $\cV_y$.
Here, $\tau_{y,1}$ is the analogue of $\tau_1$ as in (\ref{Eqn:tree_A.5}),
with $N_{[\de]}$ replaced by $N_{[\de]}\bsl\De_y$.
After $(\rd_1\ph_5)$,
we fix an arbitrary lift $z$ of $y$ and define $D(\De_z)$ in the same way as (\ref{Eqn:I_y}).

If there exists $\ell\!+\!1\!\le\!i\!\le\!m$ such that $\de_i\inn D(\De_z)$,
then from (\ref{Eqn:r3p2_phi}) we conclude that the pullback of $\varphi$ is diagonalized on a neighborhood $\cV_z$ of $z$, and the blowups starting from $(\rd_2)$ do not affect $\cV_z$.

Otherwise, the points of $D(\De_z)$ are all on the tail that contains $\{\de_1\}$ ($=\!D(\De_y)$).
The argument is then parallel to \S\ref{Subsubsec:Case_E.1.f_1}.
We omit further details.

\subsection{Proof of \ref{Case:1}, Part V}
\label{Subsec:Case_E_last}
In this subsection,
we aim to show \ref{Case:1}.\ref{Case:1.2}, 
\ref{Case:1}.\ref{Case:1.3}, and \ref{Case:1}.\ref{Case:1.4}.

\subsubsection{Proof of \ref{Case:1}.\ref{Case:1.2}}
In this case, we write 
\begin{align*}
	F=F_1\cup \tn B\cup F_2
\end{align*} as in the beginning \S\ref{Subsec:Case_B},
and assume that $a_1\inn F_1$, $a_2\inn F_2$,
and the pivotal nodes of the tails containing points of $D(\De_y)$ are on $F_1\!\cup\tn B$.
It is a direct check that $(\rd_1\ph_2)$ and $(\rd_1\ph_3)$ do not affect $\cV_y$, but $(\rd_1\ph_4)$ is involved,
which is locally $\tau_{y}^\dag$-compatible on $\cV_y$.
Here, $\tau_{y}^\dag$ is the analogue of $\tau^\dag$ as in (\ref{Eqn:tree_B}),
with $N_{[\de,a_2]}$ replaced by $N_{[\de,a_2]}\bsl\De_y$.
After $(\rd_1\ph_4)$,
we fix an arbitrary lift $z$ of $y$ and define $D(\De_z)$ in the same way as \S\ref{Subsec:Case_B}.

If there exists $\de\inn D(\De_z)$ such that the pivotal node of the tail containing $\de$ is on $F_2$,
then we can mimic the argument of  \S\ref{Subsec:Case_B} and conclude that the pullback of $\varphi$ is diagonalized on a neighborhood $\cV_z$ of $z$, and the blowups starting from $(\rd_1\ph_5)$ do not affect $\cV_z$.

Otherwise, the pivotal nodes of the tails containing points of $D(\De_z)$ are on $F_1\!\cup\tn B$ as well.
The argument is then parallel to the proof of  \ref{Case:1}.\ref{Case:1.1}.
We omit further details.

\subsubsection{Proof of \ref{Case:1}.\ref{Case:1.3}}

In this case, we write 
\begin{align*}
	F=F_1\cup \tn B
\end{align*} as in the beginning \S\ref{Subsec:Case_C},
and assume the pivotal nodes of the tails containing points of $D(\De_y)$ are on $\tn B$.
It is a direct check that $(\rd_1\ph_2)$ does not affect $\cV_y$, but $(\rd_1\ph_3)$ is involved,
which is locally $\tau_{y}^\ddag$-compatible on $\cV_y$.
Here, $\tau_{y}^\ddag$ is the analogue of $\tau^\ddag$ as in (\ref{Eqn:tree_C}),
with $N'$ replaced by $N'\bsl\De_y$.
After $(\rd_1\ph_3)$,
we fix an arbitrary lift $z$ of $y$ and define $D(\De_z)$ in the same way as \S\ref{Subsec:Case_C}.

Depending on the position of the points of $D(\De_z)$,
we have one of the following: 
\begin{itemize}[leftmargin=*]
	\item the pullback of $\varphi$ is diagonalized on a neighborhood $\cV_z$ of $z$, and the blowups starting from $(\rd_1\ph_4)$ do not affect $\cV_z$; or
	\item  the argument is parallel to  the proof of  \ref{Case:1}.\ref{Case:1.2}; or
	\item  the argument is parallel to  the proof of  \ref{Case:1}.\ref{Case:1.1}.
\end{itemize}
We omit further details.

\subsubsection{Proof of \ref{Case:1}.\ref{Case:1.4}}

In this case, we write 
\begin{align*}
	F=T_1\cup \tn B\cup T_2
\end{align*} as in the beginning \S\ref{Subsec:Case_D},
and assume that $a_1\inn F_1$, $a_2\inn F_2$, the pivotal nodes of the tails containing points of $D(\De_y)$ are on $\tn B$.
It is a direct check that $(\rd_1\ph_3)$ does not affect $\cV_y$, but $(\rd_1\ph_2)$ is involved,
which is locally $\tau_{y,\wedge}$-compatible on $\cV_y$.
Here, $\tau_{y,\wedge}$ is the analogue of $\tau_{\wedge}$ as in (\ref{Eqn:tree_D}),
with $N''$ replaced by $N''\bsl\De_y$.
After $(\rd_1\ph_2)$,
we fix an arbitrary lift $z$ of $y$ and define $D(\De_z)$ in the same way as \S\ref{Subsec:Case_C}.

Depending on the position of the points of $D(\De_z)$,
we have one of the following: 
\begin{itemize}[leftmargin=*]
	\item the pullback of $\varphi$ is diagonalized on a neighborhood $\cV_z$ of $z$, and the blowups starting from $(\rd_1\ph_3)$ do not affect $\cV_z$; or
	\item  the argument is parallel to  the proof of  \ref{Case:1}.\ref{Case:1.2}; or
	\item  the argument is parallel to  the proof of  \ref{Case:1}.\ref{Case:1.1}.
\end{itemize}
We omit further details.

\vsp
To summarize, in \S\ref{Subsec:Case_E}-\ref{Subsec:Case_E_last},
we have shown the following.
\begin{prop}\label{PrpChangeofPhiM1}
Proposition~\ref{PrpChangeofPhi} holds
in \ref{Case:1}.
\end{prop}

\vsp
\section{Resolutions of $\ov  M_2(\P^n,d)$}
\label{localEquations}
\label{SecLocalEqns}

In this section, we construct the stacks $\ti M_2(\P^n,d)$ and $\ti M_2(\P^n,d;k)$, $k\!\ge\!2$, that are proposed in \S\ref{SecIntro}, and prove Theorems 1 and 4 as well as Corollaries 2 and 3 of \S\ref{SecIntro}.

\subsection{The construction of $\ti M_2(\P^n,d)$}
We continue with the notation as in \S \ref{subsect:localEq}
and the previous sections.
Let
$\wv\eq\cV \!\times_{\fdd}\! \ti\fM^{\rm div}_2.$
Then
$$\sE_{\wv}=\sE_\cV\times_\cV\wv
$$
is the total space of the pullback bundle $\beta\sta
\rho\lsta\sL(\cA)^{\oplus n}$, 
where $\beta: \wv \to \cV$ denotes the projection. The tautological restriction
homomorphism $``$rest$"$  of \eqref{restHom-0} pullbacks to give
rise to \beq\label{pullbackRest}  \widetilde{\text{rest}}: \beta^* \rho_*
\sL(\cA)^{\oplus n}\lra \beta^*\rho_* (\sL(\cA)^{\oplus n}|_\cA).
\eeq The section $\Phi$ of  \eqref{sec1-0}  pullbacks to give rise
to \beq \tilde\Phi \in \Gamma (\sE_{\wv}\,,\;{\tilde p}^*\beta^*\rho_*
(\sL(\cA)^{\oplus n}|_{\cA})), \eeq 
where $\tilde p\!: \sE_{\wv}\!\to\!
\wv$ denotes the projection. 
Given small open $U\!\subset\!\ov M_2(\P^n,d)$ and $H\inn\bH_U$ as in (\ref{Eqn:H}),
the morphism $f_{U,H}:U\!\to\!\cV\,(\to\!\fdd)$ determines
\begin{align}
	\label{Eqn:cU}
\widetilde\cU=\widetilde\cU_{U,H}:=\wv\times_{\cV;f_{U,H}} U.
\end{align}
The immersion $\cU\!\to\!\sE_\cV$ of
Theorem \ref{thm:immersion} then naturally induces an open immersion
\begin{align*}
	\widetilde\cU \lra \big(\ti\Phi=0 \big) \sub \sE_{\wv}.
\end{align*}

\begin{lemm}\label{LmNormalCrossing}
With notation as above,
$\widetilde\cU $ has smooth irreducible components and normal crossing singularities.
\end{lemm}

\begin{proof} We trivialize $\beta^*
\rho_*\sM(\cD)$ and $\beta^* \rho_* \sO_{\cA}(\cA)$ so that
$$\sE_{\wv} \cong \wv \times  (\AA^{d+1})^n.$$
By Proposition~\ref{PrpChangeofPhi}, 
$\ti\cV$ can be covered by $\{\ti\cV'\}$ so that on each $\ti\cV'\!\subset\!\ti\cV$,
the homomorphism $$\beta^*\varphi: \beta^*
\rho_*\sM(\cD) \lra \beta^* \rho_* \sO_{\cA}(\cA)$$ can be
diagonalized as
$$
\left[\begin{matrix} z_1   &  0 & 0& \cdots & 0 \\
	0  & z_2 &  0 & \cdots & 0
\end{matrix} \right]
\qquad\tn{with}\quad
z_1|z_2\in\Ga\big(\sO_{\ti\cV'}\big),$$
where $z_1$ and $z_2$ are both products of some smooth parameters corresponding to (the proper transforms of) the exceptional divisors.

Let $w_j^i \inn \Ao$ for all $1 \!\le\! i \!\le\! n$ and $0 \!\le\! j \!\le\! d$.
The equation
$\ti\Phi =0$ is then equivalent to the following equations
\beq\label{finalEq'}
z_1 w_1^i =0,\qquad z_2 w_2^i=0, \qquad 1 \le i \le n.
\eeq
Thus, $\ti\cU$ has normal crossing singularities.

Since the main component of $
\ti\Phi^{-1}(0)\!\cap\!\ti\cV'$ consists of general points
with non-vanishing $z_1$ and $z_2$, we conclude from (\ref{finalEq'}) that the main component is defined by
\begin{align*} 
w_1^i = w_2^i=0,\qquad1 \le i \le n.
\end{align*}
Thus, the main component of $\ti\Phi^{-1}(0)\!\cap\!\ti\cV'$ is smooth.

Each of the remaining irreducible component $X$ of $\ti\Phi^{-1}(0)\!\cap\!\ti\cV'$ is contained in the proper transform of an exceptional divisor,
so locally there exists a unique local parameter $\xi$ (whose zero locus locally defines the proper transform of the given exceptional divisor) such that $X\!\subset\!\{\xi\eq 0\}(\subset\!\sE_{\ti\cV})$.
This $\xi$
is either a factor of $z_1$ (hence $z_2$),
or a factor of $z_2$ but not $z_1$.
In the former case,
(\ref{finalEq'}) implies $X$ is simply given by
$\xi\eq 0$; there are no constraints on $w_1^i$ or $w_2^i$.
In the latter case,
(\ref{finalEq'}) also implies $X$ is given by
\begin{align*}
	\xi=
	w_1^i=0,\qquad
	1\le i\le n.
\end{align*}
Either way, $X$ is smooth.
\end{proof}

Next, we show all $\ti\cU_{U,H}$ as in (\ref{Eqn:cU}) glue to form the desired $\tMPdd$.
Consider arbitrary  small open $U_s/\cV_s$, $s\eq 1,2$, of $\ov M_2(\P^n,d)$, along with the morphism $f_{U_s,H_s}:U_s\!\to\!\cV_s$, as in \S\ref{sheafStructures},
satisfying $U_1\!\cap\!U_2\!\ne\!\emptyset$.
For each step $\ell\!:=\!\rd_i\ph_j\st_t$ of the sequential blowup of $\fdd$,
we denote by $\ell\!+\!1$  the step immediately after $\ell$, and by $\ti\fM_2^{{\rm div},\ell}$ the corresponding blowup stack.
Let $\ti\cV_s^\ell\eq\cV_s\!\times_{\fdd}\!\ti\fM_2^{{\rm div},\ell}$, $s\eq 1,2$.
In addition,
we set $\ti\fM_2^{{\rm div},0}\!:=\!\ti\fM_2^{{\rm div}}$ and $\ti\cV_s^0\!:=\!\cV_s$.

\begin{lemm}
	\label{Lm:U_pullback}
With notation as above, 
there exists an isomorphism
\begin{align*}
	\big(\ti\cV_1^\ell\!\times_{\cV_1;f_{U_1,H_1}}\!\!(U_1\!\cap\!U_2)\big)
	\big/(U_1\!\cap\!U_2)
	\lra 
	\big(\ti\cV_2^\ell\!\times_{\cV_2;f_{U_2,H_2}}\!\!(U_1\!\cap\!U_2)\big)
	\big/(U_1\!\cap\!U_2)
\end{align*}
between the above two blowups of $U_1\!\cap\!U_2$.
\end{lemm}

The statement of Lemma~\ref{Lm:U_pullback} would be trivial if there existed $H\inn\breve{\P}^n$ lying in the common intersection of $\bH_U$ (c.f.~(\ref{Eqn:H})) for all small $U\!\subset\MPdd$,
which would give rise to a natural global morphism from $\MPdd$ to $\fM_2^{\rm div}$.
Since such $H$ does not exist,
we introduce Lemma~\ref{Lm:U_pullback} to ensure that $\ti\cU_{U,H}$ glue.

\begin{proof}[Proof of Lemma~\ref{Lm:U_pullback}]
For conciseness,
in this proof,
we write 
\begin{align*}
	\Up= U_1\!\cap\!U_2,\quad
\ti{\Up}_{s}^\ell= \ti\cV_s^\ell\!\times_{\cV_s;f_{U_s,H_s}}\!\!\Up,\quad 
f_s=f_{U_s,H_s}|_{\Up}:\Up\lra\cV_s,\quad 
s\eq 1,2.
\end{align*}
In addition, we denote by
\begin{align*}
	\tn p_s:\ti{\Up}_{s}^\ell\lra \Up\quad
	\tn{and}\quad
	f_s^\ell:\ti{\Up}_{s}^\ell\lra\ti\cV_s^\ell,\qquad
	s\eq 1,2,
\end{align*}
the natural morphisms.
We aim to show by induction over $\ell$ that there exist isomorphisms
$\phi^\ell_{21}:\ti\Up_{1}^\ell\!\to\!\ti\Up_{2}^\ell$ and $\phi^\ell_{12}:\ti\Up_{2}^\ell\!\to\!\ti\Up_{1}^\ell$, fitting in the Cartesian diagram:
\begin{center}
	\begin{tikzpicture}{h}
		\draw 
		(2.5,1.2) node {$\ti \Up_{1}^{\ell}\eq\ti\cV_{1}^{\ell} \!\times_{\cV_{1};f_1}\!\Up$}
		(5.5,1.2) node {$\ti\cV_{1}^{\ell}$}
		(-2.5,1.2) node {$\ti \Up_{2}^{\ell}\eq \ti\cV_{2}^{\ell}\!\times_{\cV_{2};f_2}\!\Up$}
		(-5.5,1.2) node {$\ti\cV_{2}^{\ell}$}
		(0,0) node {$\Up$}
		(5.5,0) node {$\cV_{1}$}
		(-5.5,0) node {$\cV_{2}$}
		(2.7,.2) node {\scriptsize{$f_{1}$}}
		(-2.7,.2) node {\scriptsize{$f_{2}$}}
		(4.6,.9) node {\scriptsize{$f_{1}^\ell$}}
		(-4.6,.9) node {\scriptsize{$f_{2}^\ell$}}
		(-.9,.55) node {\scriptsize{${\tn p}_2$}}
		(.96,.52) node {\scriptsize{${\tn p}_1$}}
		(0,1.45) node {\tiny{$\phi^\ell_{1,2}$}}
		(0,.9) node {\tiny{$\phi^\ell_{2,1}$}};
		\draw[->,>=stealth'] 
		(4.1,1.2)--(5.15,1.2);
		\draw[->,>=stealth'] 
		(-4.1,1.2)--(-5.15,1.2);
		\draw[->,>=stealth'] 
		(.33,0)--(5.15,0);
		\draw[->,>=stealth']
		(-.35,0)--(-5.15,0);
		\draw[->,>=stealth'] 
		(5.4,.9)--(5.4,.3);
		\draw[->,>=stealth'] 
		(-5.6,.9)--(-5.6,.3);
		\draw[->,>=stealth'] 
		(.93,.9)--(.23,.23);
		\draw[->,>=stealth'] 
		(-.9,.9)--(-.23,.23);
		\draw[dashed,->,>=stealth'] 
		(-.8,1.25)--(.8,1.25);
		\draw[dashed,->,>=stealth'] 
		(.8,1.15)--(-.8,1.15);
	\end{tikzpicture}		
\end{center}

The base case ($\ell\eq 0$) is when no blowup is exerted on $U$, which is trivial.

Assume the above diagram holds for the unique isomorphisms $\phi^\ell_{21}$ and $\phi^{\ell}_{12}$ after the $\ell$-th step.
For $s\eq 1,2$, let
\begin{align*}
	Z_s\subset\ti\cV_s^\ell\qquad\tn{and}\qquad
	\tn q_s:\ti\Up_s^{\ell+1}\eq 
	\ti\Up_s^\ell\!\times_{\ti\cV_s^\ell}\!
	\ti\cV_s^{\ell+1}
	\lra\ti\Up_s^{\ell}
\end{align*}
be the center of the blowup $\ti\cV_s^{\ell+1}/\ti\cV_s^\ell$ and the natural morphism, respectively.

Notice that for every node $e$ of $C$,
the node smoothing parameters $\ze_e^{\cV_1}$ and $\ze_e^{\cV_2}$ satisfy
\begin{align}\label{Eqn:gluing_ze}
	\ze_e^{\cV_1}\circ f_1=0
	\quad\Longleftrightarrow\quad
	\ze_e^{\cV_2}\circ f_2=0\qquad
	\tn{on}~\Up\,.
\end{align}
In other words,
$\ze_e^{\cV_1}\!\circ\! f_1$ and $\ze_e^{\cV_2}\!\circ\! f_2$ differ by a (multiplicative) unit on $\Up$.

Moreover,
if the $(\ell\!+\!1)$-th step is in ($\rd_3$), and the blowups $\ti\cV^{\ell+1}_s/\ti\cV^\ell_s$, $s\eq 1,2$, are not trivial,
then for every  $z_s\inn\Up_s^\ell$ whose image in $\ti\cV_s^\ell$ is contained in $Z_s$,
by writing $(C,u)\eq \tn p_s(z_s)$,
we have $\deg(u|_{F})\!\le\!2$, where $F$ denotes the core of $C$. 
Let 
$$\Up':=\big\{\,(C,u)
\inn\Up:\,
\deg(u|_{F})\!\le\!2\,\big\}\quad\subset\Up\,.
$$ 
This implies
\begin{align*} 
	(f_s^\ell)^{-1}(Z_s)\subset \tn p_s^{-1}(\Up'),\qquad
	s\eq 1,2.
\end{align*}
Notice for every $(C,u)\inn\Up'$ and irreducible component $\Si$ of $C$,
the divisors $u^{-1}(H_1)\!\cap\!\Si$ and $u^{-1}(H_2)\!\cap\!\Si$ are linearly equivalent,
hence
\begin{align}\label{Eqn:gluing_ka}
	\ka_{ij}\circ f_1=0 \quad\Longleftrightarrow\quad
	\ka_{ij}\circ f_2=0\qquad
	\tn{on}~\Up'\,.
\end{align} 
In other words,
$\ka_{ij}\!\circ\! f_1$ and $\ka_{ij}\!\circ\! f_2$ differ by a (multiplicative) unit on $\Up'$.

Given $(C,u)\inn\Up$,
let $\wh x_1$ and $\wh x_2$ be 
its images in $\fM_2^{\rm div}$ under $f_1$ and $f_2$, respectively;
i.e.~$\wh x_1\eq\big(C,u^{-1}(H_1)\big)$ and $\wh x_2\eq\big(C,u^{-1}(H_2)\big)$.
As shown in \S\ref{SecChangeOfPhi},
near the respective pullbacks of $\wh x_1$ and $\wh x_2$,
the blowup centers of any step $\ell'$ of the sequential blowups are locally given by the {\it same} equations
\begin{align}\label{Eqn:blowup_center_match'}
	\{\ga_{\ell';i}\eq 0: i\inn I_{\ell'}\},
\end{align}
where $I_{\ell'}$ is an index set,
and each $\ga_{\ell';i}$ is a quotient of monomials of the pullbacks of $\ze_e^{\cV_s}$ and/or $\ka_{ij}$.
Hence by (\ref{Eqn:gluing_ze}) and (\ref{Eqn:gluing_ka}), we have
\begin{align}\label{Eqn:blowup_center_match}
	\big(f_1^\ell\circ\phi^\ell_{1,2}\big)^{-1}(Z_1)=
	(f_2^\ell)^{-1}(Z_2).
\end{align}

For $s\eq 1,2$, let $\tn N_s^\ell$ 
be the union of all the irreducible components of $\ti\Up_s^\ell$ whose images in $\ti\cV_s^\ell$ under $f_s^\ell$ are contained in $Z_s$, and $\tn M_s^\ell$ be~the union of the remaining irreducible components of $\ti\Up_s^\ell$.
Then, we have
$$
\ti\Up_s^\ell=\tn M_s^\ell\cup\tn N_s^\ell,\qquad
\phi^\ell_{1,2}(\tn M_2^\ell)=\tn M_1^\ell,\qquad
\phi^\ell_{1,2}(\tn N_2^\ell)=\tn N_1^\ell.
$$
We  then set 
$$\ti {\tn M}_s^{\ell+1}=\tn q_s^{-1} (\tn M_s^\ell),\qquad
\ti{\tn N}_s^{\ell+1}=\tn q_s^{-1} (\tn N_s^\ell).$$

By (\ref{Eqn:blowup_center_match}), we have
\begin{align*}
	\big(f_1^\ell\circ\phi^\ell_{1,2}\circ \tn q_2\big)^{-1}(Z_1)=
	(f_2^\ell\circ \tn q_2)^{-1}(Z_2).
\end{align*}
Observe that the intersection of $(f_2^\ell\circ \tn q_2)^{-1}(Z_2)$ and $\ti{\tn M}_2^{\ell+1}$ is a Cartier divisor of $\ti{\tn M}_2^{\ell+1}$,
because the pullback of $Z_2$ to $\ti\cV_2^{\ell+1}$ is the exceptional divisor, and the image of every irreducible component of $\tn M_2^\ell$ under $f_2^\ell$ is not contained in $Z_2$.
Therefore, the intersection of $\big(f_1^\ell\circ\phi^\ell_{1,2}\circ \tn q_2\big)^{-1}(Z_1)$ and $\ti{\tn M}_2^{\ell+1}$
is a Cartier divisor of $\ti{\tn M}_2^{\ell+1}$.
The universality of blowing up thus induces a unique morphism 
\begin{align}\label{Eqn:f_12^ell+1}
	f_{1,2}^{\ell+1}:\ti{\tn M}_2^{\ell+1}\lra
	\ti\cV_1^{\ell+1}
\end{align}
lifting the restriction to $\ti{\tn M}_2^{\ell+1}$ of the morphism $f_1^\ell\!\circ\!\phi^\ell_{1,2}\!\circ\!\tn q_2$.

To extend $f_{1,2}^{\ell+1}$ to the entire $\ti\Up_2^{\ell+1}$,
notice that
$$\ti{\tn N}^{\ell+1}_s= {\tn N}^{\ell}_s\!\times_{f_s^\ell,Z_s}\cE_s^{\ell+1},\qquad
s\eq 1,2,
$$
where $\cE_s^{\ell+1}$ denotes the exceptional divisor of $\ti\cV_s^{\ell+1}/\ti\cV_s^{\ell}$.
Therefore, (\ref{Eqn:gluing_ze}), (\ref{Eqn:gluing_ka}),
and the sentence containing (\ref{Eqn:blowup_center_match'}) naturally induce an isomorphism 
$$
\phi^{\ell+1}_{1,2;\tn N}:
\ti{\tn N}^{\ell+1}_2\stackrel{\sim}{\lra}
\ti{\tn N}^{\ell+1}_1,
$$
satisfying that $f_1^{\ell+1}\!\circ\!\phi^{\ell+1}_{1,2;\tn N}$ matches (\ref{Eqn:f_12^ell+1}) on $\ti{\tn M}^{\ell+1}_2\!\cap\!\ti{\tn N}^{\ell+1}_2$.
In other words,
we obtain 
\begin{align*}
	f_{1,2}^{\ell+1}:\ti{\Up}_2^{\ell+1}\lra
	\ti\cV_1^{\ell+1}
\end{align*}
extending (\ref{Eqn:f_12^ell+1}), whose restriction to $\ti{\tn N}^{\ell+1}_2$ equals $f_1^{\ell+1}\!\circ\!\phi^{\ell+1}_{1,2;\tn N}$.

Then,
the universality of pullbacks induces a unique morphism  
$$\phi_{1,2}^{\ell+1}:\ti\Up_2^{\ell+1}\lra
\ti\Up_1^{\ell+1}$$
that $f_{1,2}^{\ell+1}$ and $\phi^\ell_{1,2}\!\circ\!\tn q_2$ respectively factor through.
By symmetry, we analogously obtain a unique morphism
$$\phi_{2,1}^{\ell+1}:\ti\Up_1^{\ell+1}\lra
\ti\Up_2^{\ell+1}$$
that $\phi^\ell_{2,1}\!\circ\!\tn q_1$ factors through.
In sum, we have the following commutative diagram:
\begin{center}
	\begin{tikzpicture}{h}
		\draw 
		(2.5,2.4) node {$\ti \Up_{1}^{\ell+1}\eq\ti\cV_{1}^{\ell+1} \!\times_{\cV_{1};f_1}\!\Up$}
		(5.5,2.4) node {$\ti\cV_{1}^{\ell+1}$}
		(-2.5,2.4) node {$\ti \Up_{2}^{\ell+1}\eq \ti\cV_{2}^{\ell+1}\!\times_{\cV_{2};f_2}\!\Up$}
		(-5.5,2.4) node {$\ti\cV_{2}^{\ell+1}$}		
		(2.5,1.2) node {$\ti \Up_{1}^{\ell}\eq\ti\cV_{1}^{\ell} \!\times_{\cV_{1};f_1}\!\Up$}
		(5.5,1.2) node {$\ti\cV_{1}^{\ell}$}
		(-2.5,1.2) node {$\ti \Up_{2}^{\ell}\eq \ti\cV_{2}^{\ell}\!\times_{\cV_{2};f_2}\!\Up$}
		(-5.5,1.2) node {$\ti\cV_{2}^{\ell}$}
		(0,0) node {$\Up$}
		(5.5,0) node {$\cV_{1}$}
		(-5.5,0) node {$\cV_{2}$}
		(2.7,.2) node {\scriptsize{$f_{1}$}}
		(-2.7,.2) node {\scriptsize{$f_{2}$}}
		(4.6,.95) node {\scriptsize{$f_{1}^\ell$}}
		(-4.6,.95) node {\scriptsize{$f_{2}^\ell$}}
		(-.92,.56) node {\scriptsize{${\tn p}_2$}}
		(1.35,.53) node {\scriptsize{${\tn p}_1$}}
		(-2.88,1.85) node {\scriptsize{${\tn q}_2$}}
		(1.58,1.85) node {\scriptsize{${\tn q}_1$}}
		(0,1.45) node {\tiny{$\phi^\ell_{1,2}$}}
		(0,.9) node {\tiny{$\phi^\ell_{2,1}$}}
		(0,2.65) node {\tiny{$\phi^{\ell+1}_{1,2}$}}
		(0,2.1) node {\tiny{$\phi^{\ell+1}_{2,1}$}}
		(1.025,3.13) node {\tiny{$f_{1,2}^{\ell+1}$}};
		\draw[->,>=stealth'] 
		(3.9,1.2)--(5.15,1.2);
		\draw[->,>=stealth'] 
		(-3.9,1.2)--(-5.15,1.2);
		\draw[->,>=stealth'] 
		(4.3,2.4)--(5.05,2.4);
		\draw[->,>=stealth'] 
		(-4.3,2.4)--(-5.05,2.4);
		\draw[->,>=stealth'] 
		(.33,0)--(5.15,0);
		\draw[->,>=stealth']
		(-.35,0)--(-5.15,0);
		\draw[->,>=stealth'] 
		(5.4,.9)--(5.4,.3);
		\draw[->,>=stealth'] 
		(-5.6,.9)--(-5.6,.3);
		\draw[->,>=stealth'] 
		(5.4,2.1)--(5.4,1.5);
		\draw[->,>=stealth'] 
		(-5.6,2.1)--(-5.6,1.5);
		\draw[->,>=stealth'] 
		(1.85,2.1)--(1.85,1.5);
		\draw[->,>=stealth'] 
		(-3.16,2.1)--(-3.16,1.5);
		\draw[->,>=stealth'] 
		(1.8,.9)--(.23,.23);
		\draw[->,>=stealth'] 
		(-3.1,.9)--(-.23,.23);
		\draw[->,>=stealth'] 
		(-1,1.25)--(1,1.25);
		\draw[->,>=stealth'] 
		(1,1.15)--(-1,1.15);
		\draw[dashed,->,>=stealth'] 
		(-.6,2.45)--(.6,2.45);
		\draw[dashed,->,>=stealth'] 
		(.6,2.35)--(-.6,2.35);
		\draw[dashed,->,>=stealth'] 
		(-3.15,2.75) .. controls (-1.95,3.6) and (4,3.6) .. (5.2,2.7);
	\end{tikzpicture}		
\end{center}

Notice the composition $\phi_{2,1}^{\ell+1}\!\circ\!\phi_{1,2}^{\ell+1}$
restricts to the identity morphism on $\ti{\tn M}^{\ell+1}_2\big\bsl (f_2^\ell\!\circ\!{\tn q}_2)^{-1}(Z_2)$,
which is open and dense in $\ti{\tn M}^{\ell+1}_2$,
hence  $$\phi_{1,2}^\ell|_{\ti{\tn M}^{\ell+1}_2}\qquad\tn{and}\qquad\phi_{2,1}^\ell|_{\ti{\tn M}^{\ell+1}_1}$$ are inverse to each other.
As 
$$\phi_{1,2}^\ell|_{\ti{\tn N}^{\ell+1}_2}\qquad\tn{and}\qquad\phi_{2,1}^\ell|_{\ti{\tn N}^{\ell+1}_1}$$ are respectively $\phi_{1,2;\tn N}^\ell$ and its inverse,
we conclude that $\phi_{1,2}^\ell$ and $\phi_{2,1}^\ell$ are inverse to each other.
This completes the induction.
\end{proof}

Lemma~\ref{Lm:U_pullback} implies that we can glue $\ti\cU_{U,H}$ for all small open $U\!\subset\!\MPdd$ and all $H\inn\bH_U$ together.
The resulting  Deligne-Mumford stack is denoted by $$\ti M_2(\P^n,d).$$ 
The corollary below follows directly from Lemma~\ref{Lm:U_pullback}.

\begin{coro}
	\label{Crl:Gluing}
For every small open $U\!\subset\!\ov M_2(\P^n,d)$ and every $H\inn \bH_U$,
there exists a commutative diagram
\[
\begin{tikzcd}[column sep=.5em,row sep=1.5em]
	\ti \cU_{U,H}  \arrow[dr,""] \arrow[rr,"\sim"] && U\times_{\ov M_2(\P^n,d)}\ti M_2(\P^n,d)
	\arrow[dl,""]
	\\
	& U 
\end{tikzcd}
\]
where the horizontal arrow is an  isomorphism of stacks over $U$,
and the other two arrows are natural morphisms.
\end{coro}

The stack $\ti M_2(\P^n,d)$ enjoys the following properties.
Recall the main component of $\widetilde{M}_2(\Pn,d)$ is denoted by $\widetilde{M}_2(\Pn,d)^{\rm mc}$.

\begin{theo} \label{inducedImmersion} 
$\tMPdd$ has smooth irreducible components and normal crossing singularities.
Moreover, for $d\!>\!2$, $\widetilde{M}_2(\Pn,d)^{\rm mc}$  is of expected dimension.
\end{theo}

\begin{proof}
The first statement follows from  Lemma~\ref{LmNormalCrossing}.
As for the dimension, note that $\dim \sE_{\wv}= 3+d + (d+1)n$. Hence the dimension of the main component of  $\widetilde\cU$
is $$3+d +  (d+1)n - 2n = d(n+1) -n+ 3$$ which is the virtual dimension of $\MPdd$.
\end{proof}

\vsp
When $d=2$, 
the interior of $\widetilde{M}_2(\Pn,2)^{\rm mc}$ consists of stable maps that are double covers of smooth rational curves,
which is smooth and of dimension
$2n+4$.
Since the virtual dimension of $\overline{M}_2(\Pn,2)$ is $n+5$, the main component is of wrong dimension unless $n\eq 1$.
The entire moduli can be treated by hand.

	\vsp
\subsection{The desingularization of $\pi_*\ff^*\sO_{\P^n}(k)$}
In this subsection, we study the pullback of the derived object $\bR\pi_*\ff^*\sO_{\P^n}(k)$ to $\widetilde{M}_2(\Pn,d)$.
Recall $(\pi,\ff)$ is the universal family of $\MPdd$; see (\ref{Eqn:univ_fami})

\begin{prop}
	\label{Prp:Diag}	
For every $k\inn\mathbb Z_{>0}$,
the pullback of  $\bR\pi_*\ff^*\sO_{\P^n}(k)$ to $\ti M_2(\P^n,d)^{\rm mc}$ is locally diagonalizable.	
Moreover, for $k\eq 1$, the pullback of  $\bR\pi_*\ff^*\sO_{\P^n}(1)$ to the entire $\ti M_2(\P^n,d)$ is locally diagonalizable.
\end{prop}

\begin{proof}
Given $k\inn\mathbb Z_{>0}$ and $z\eq[C_0,u_0]\inn\MPdd$,
we can choose a sufficiently small open $U\!\subset\!\ov M_2(\P^n,d)$ containing $z$,
a hyperplane $H\inn\breve{P}^n$, and
a degree $k$ morphism
$\psi_k:\P^n\!\lra\!\P^n$
such that 
\[(C_0,u_0^{-1}\!\circ\!\psi_k^{-1}(H)\big)
\in\fdd\,.
\]
Given a small chart $\cV_{kd}\!\lra\!\fdd$ containing the above point,
since $U$ is assumed to be sufficiently small, we see that
\begin{align}\label{Eqn:f_k}
	f_k=f_{U,H,\psi_k}:
	U\lra\cV_{kd}\,,\qquad
	[C,u]\mapsto
	\big(C,u^{-1}\!\circ\!\psi_k^{-1}(H)\big),
\end{align}
is well-defined.

By Proposition~\ref{PrpChangeofPhi} and Lemma~\ref{lemm:usefulFacts},
we see the pullback of $\bR\rho_*\sO_\cC(\cD)$ to the blowup $\ti\cV_{kd}$ of $\cV_{kd}$ is locally diagonalizable.
Therefore,
the pullback of $\bR\pi_*\ff^*\sO_{\P^n}(k)$  is locally diagonalizable on 
\begin{align}\label{Eqn:U_kd}
	\ti\cU_{kd}:=
	U\times_{f_{k};\cV_{kd}}\ti\cV_{kd}\,.
\end{align}

When $k\eq 1$,
by setting $\psi_1$ to be the identity morphism,
we see $f_{1}$ is equal to $f_{U,H}$ of (\ref{toP0}),
hence
$\ti\cU_d\eq\ti\cU_{U,H}$ is an open neighborhood of the preimage of $z$ in $\ti M_2(\P^n,d)$.
Since $z$ is arbitrary,
we obtain the latter statement of Proposition~\ref{Prp:Diag}.

To justify the former statement of Proposition~\ref{Prp:Diag},
we assume $k\!\ge\!2$.
Following the notation in the proof of Lemma~\ref{Lm:U_pullback},
for each step $\ell\!:=\!\rd_i\ph_j\st_t$ of the sequential blowups of $\fdd$,
we denote by $\ell\!-\!1$  the step immediately before $\ell$, and by $\ti\fM_2^{{\rm div},\ell}$ the corresponding blowup stack.
In addition,
we set 
$$\ti\fM_2^{{\rm div},0}:=\ti\fM_2^{{\rm div}},
\qquad
\ti\cV_{kd}^\ell:=\cV_{kd}\!\times_{\fdd}\!\ti\fM_2^{{\rm div},\ell},
\qquad
\ti\cU_{kd}^\ell:=
U\times_{f_{k};\cV_{kd}}\ti\cV_{kd}^\ell.
$$
Similarly, by writing $\ell\eq \rd_i$ (resp.~$\ell\eq\rd_i\ph_j$),
we consider $\ell$ as the last step of $(\rd_i)$ (resp.~$(\rd_i\ph_j)$) that affects the local chart.
Moreover,
we denote by
\begin{align*}
	\ti\cU^{\ell,{\rm mc}}_{kd}\subset
	\ti\cU^{\ell}_{kd}
\end{align*}
the main component of $\ti\cU^{\ell}_{kd}$,
i.e.~the irreducible component whose general points have smooth domain curves.
Since interior of the main components remain unchanged during the sequential blowups,
we have
\begin{align}\label{Eqn:MC_inclusion}
	\tn p_{k}^{\ell-1,\ell}\big(\ti\cU^{\ell,{\rm mc}}_{kd}\big)\subset \ti\cU^{\ell-1,{\rm mc}}_{kd}
\end{align}
for any step $\ell$, 
where $\tn p_{k}^{\ell-1,\ell}:\ti\cU_{kd}^\ell\lra\ti\cU_{kd}^{\ell-1}$ is the natural morphism.

For every step $\ell$ in $(\rd_1\ph_1)$-$(\rd_1\ph_4)$,
we claim 
there exists an isomorphism $\phi^\ell_k:\ti \cU^\ell_{d}\!\lra\!\ti \cU^\ell_{kd}$,
fitting in the following Cartesian diagram:
\begin{center}
	\begin{tikzpicture}{h}
		\draw 
		(2.5,1.2) node {$\ti \cU_{kd}^{\ell}\eq \cU\!\times_{f_{k};\cV_{kd}}\!\ti\cV_{kd}^{\ell}$}
		(5.5,1.2) node {$\ti\cV_{kd}^{\ell}$}
		(-2.3,1.2) node {$\ti \cU_{d}^{\ell}\eq \cU\!\times_{f_{1};\cV_{d}}\!\ti\cV_{d}^{\ell}$}
		(-5.3,1.2) node {$\ti\cV_{d}^{\ell}$}
		(0,0) node {$U$}
		(5.5,0) node {$\cV_{kd}$}
		(-5.3,0) node {$\cV_{d}$}
		(2.7,.2) node {\scriptsize{$f_{k}$}}
		(-2.6,.2) node {\scriptsize{$f_{1}$}}
		(4.6,.9) node {\scriptsize{$f_{k}^\ell$}}
		(-4.3,.9) node {\scriptsize{$f_{1}^\ell$}}
		(-.9,.55) node {\scriptsize{${\tn p}_1^\ell$}}
		(.96,.52) node {\scriptsize{${\tn p}_k^\ell$}}
		(0,1.45) node {\tiny{$\phi^\ell_{k}$}}
		;
		\draw[->,>=stealth'] 
		(4.1,1.2)--(5.15,1.2);
		\draw[->,>=stealth'] 
		(-3.8,1.2)--(-4.95,1.2);
		\draw[->,>=stealth'] 
		(.33,0)--(5.15,0);
		\draw[->,>=stealth']
		(-.35,0)--(-4.95,0);
		\draw[->,>=stealth'] 
		(5.4,.9)--(5.4,.3);
		\draw[->,>=stealth'] 
		(-5.4,.9)--(-5.4,.3);
		\draw[->,>=stealth'] 
		(.93,.9)--(.23,.23);
		\draw[->,>=stealth'] 
		(-.9,.9)--(-.23,.23);
		\draw[dashed,->,>=stealth'] 
		(-.8,1.2)--(.8,1.2);
	\end{tikzpicture}		
\end{center}
where $f_1^\ell$, $\tn p_1^\ell$, $f_k^\ell$, and $\tn p_k^\ell$ are respectively the natural morphisms.

The above claim follows from an inductive argument (over $\ell$) similar to the proof of Lemma~\ref{Lm:U_pullback}.
The key point is as follows.
For every $[C,u]\inn U$,
let $(C,D_{kd})$ and $(C,D_{d})$ be its images in $\fdd$ under $f_{k}$ and $f_{1}$, respectively.
On the one hand, the difference between $D_{kd}$ and $D_d$ is that on each irreducible component $\Si$ of $C$,
we have
\begin{align}\label{Eqn:k_vs_1}
	|D_{kd}\cap \Si|=k\,|D_{d}\cap \Si|\,.
\end{align}
On the other hand, notice the blowup centers in $(\rd_1\ph_1)$-$(\rd_1\ph_4)$ are all determined by whether the weight of an irreducible component of a curve is zero or positive, regardless of the precise value of the weight.
Therefore,
for each step $\ell$ in $(\rd_1\ph_1)$-$(\rd_1\ph_4)$,
the pullbacks of the blowup centers respectively via $f_k^{\ell-1}\!\circ\!\phi_k^{\ell-1}$ and $f_1^{\ell-1}$ are identical on $\ti\cU_{d}^{\ell-1}$.
The remainder of the inductive argument utilizes the universalities of blowups and pullbacks, which is parallel to the proof of Lemma~\ref{Lm:U_pullback}, hence is omitted.

Depending on the location of $z\inn\MPdd$,
the blowups in ($\rd_1\ph_5$) may affect  $\ti\cU_d^{\rd_1\ph_4}$, but they cannot affect $\ti\cU_{kd}^{\rd_1\ph_4}$ because of (\ref{Eqn:k_vs_1}).
Although this still leads to a morphism $\ti\cU_d^{\rd_1\ph_5}/U\!\lra\!\ti\cU_{kd}^{\rd_1\ph_5}/U$,
the blowup centers of $(\rd_2)$ do not always match,
so there may not exist any morphism of $\ti\cU_d^{\rd_2}/U$ to $\ti\cU_{kd}^{\rd_2}/U$.
Nonetheless,
we will show ($\rd_1\ph_5$) does not affect the main component of  $\ti\cU_d^{\rd_1\ph_4}$ as follows.

For every $$\ti z\,\in\,\ti\cU_d^{\rd_1\ph_4},$$
we denote by $y^{\rd_1\ph_j}$, $1\!\le\!j\!\le\!4$, its images in $\ti\cV_d^{\rd_1\ph_j}$,
respectively.
If $y^{\rd_1\ph_4}$ lies in some blowup centers of $(\rd_1\ph_5)$,
then there are only two possibilities:
\begin{itemize}[leftmargin=*]
	\item
	either the core $F$ of $C$ does not contain any point of $D_d$, and there exists a unique point of $D_d$, say $\de_1$, such that $D(\De_{y^{\rd_1\ph_1}})$ as in (\ref{Eqn:D(De)_r1p1}) is equal to $\{\de_1\}$,
	and for every $\de\inn D_d\bsl\{\de_1\}$,
	\begin{align}\label{Eqn:N_de_other}
		N_{[\de]}\not\subset
		\De_{y^{\rd_1\ph_1}}\sqcup 
		\De_{y^{\rd_1\ph_2}}\sqcup 
		\De_{y^{\rd_1\ph_3}}\sqcup 
		\De_{y^{\rd_1\ph_4}}\,;
	\end{align}
	\item 
	or
	$D_d\!\cap\!F$ is a singleton, whose unique element, denoted by $\de_1$, belongs to a separating bridge of $F$,
	and
	(\ref{Eqn:N_de_other}) holds for  every $\de\inn D_d\bsl\{\de_1\}$.
\end{itemize}
In either situation,
the homomorphism 
$\varphi$,
given by the $2\!\times\!d$ matrix $\left[c_{si}\ze_{[\de_i,a_s]}\right]_{1\le s\le 2, 1\le i\le d}$ as in Proposition~\ref{HomFirstOrder},
pulls back to the following form on $\ti\cV_d^{\rd_1\ph_4}$:
\begin{align*}
	\left[\begin{matrix} \xi_1   &  0 \\
		0  & \xi_2 
	\end{matrix} \right]
	\left[\begin{matrix} \ti c_{11}   &  g_{12} & \cdots & g_{1d} \\
		\ti c_{21}  & g_{22} &  \cdots & g_{2d}
	\end{matrix} \right]\,,
\end{align*}
where $\ti c_{s1}$ are invertible, $\xi_s$ are the products of certain local parameters corresponding to (the proper transforms of) the exceptional divisors obtained prior to $(\rd_1\ph_5)$,
and $g_{si}$ are regular functions satisfying \begin{align}\label{Eqn:g_vanish}
	g_{si}(y^{\rd_1\ph_4})=0\,.
\end{align} 
Here, (\ref{Eqn:g_vanish}) follows from (\ref{Eqn:N_de_other}) and the fact that $\ze_{[\de_i,a_s]}$ is divisible by $\ze_{[\de_i]}$.

Similar to the proof of Lemma~\ref{LmNormalCrossing},
the pullbacks of $\rho_*\sM(\cD_i)$, $1\!\le\!i\!\le\!d$, as well as $\rho_*\sO_\cA(\cA)$ can be trivialized so that
\begin{align*}
	\sE_{\cV_d}\!\times_{\cV_d}\!\ti\cV_d^{\rd_1\ph_4}
	\cong \ti\cV_d^{\rd_1\ph_4}\! \times\!  (\AA^{d+1})^n\,.
\end{align*}
The main component $\ti\cU_d^{\rd_1\ph_4,{\rm mc}}$ of $\ti\cU_d^{\rd_1\ph_4}$ satisfies the following local equations:
\begin{align}\label{Eqn:MC_r1p4}
	\ti c_{s1} w_1^h+\sum_{i=2}^dg_{si}w_i^h =0,
	\qquad
	1\!\le\!s\!\le\!2,\quad
	1\!\le\!h\!\le\!n,
\end{align}
where each $w_i^h\inn\A^1$ is a free variable corresponding to the trivialization of the pullback of $\rho_*\sM(\cD_i)$.

Suppose $\ti z\inn\ti\cU_d^{\rd_1\ph_4,{\rm mc}}$, while $y^{\rd_1\ph_4}$ lies in some blowup centers of $(\rd_1\ph_5)$.
Then,
by (\ref{Eqn:g_vanish}) and (\ref{Eqn:MC_r1p4}),
$\ti z$ satisfies
\begin{align*}
	w_1^1=w_1^2=\cdots=w_1^n=0.
\end{align*}
However, the above equalities imply the restriction of $u$ to the irreducible component $\Si\!\subset\!C$ containing $\de_1$ is constant,
because $\Si$ does not contain any other point of $D_d$.
This contradicts the fact that $\deg(u|_\Si)\eq 1$ (because $D_d\!\cap\!\Si\eq\{\de_1\}$).
Therefore,
the image of  $\ti\cU_d^{\rd_1\ph_4,{\rm mc}}$ in $\ti\cV_d^{\rd_1\ph_4}$ misses the blowup centers of $(\rd_1\ph_5)$,
i.e.~$\ti\cU_d^{\rd_1\ph_5,{\rm mc}}/U$ is isomorphic to $\ti\cU_d^{\rd_1\ph_4,{\rm mc}}/U$.

Since $\ti\cU_{kd}^{\rd_1\ph_5,{\rm mc}}/U$ is isomorphic to $\ti\cU_{kd}^{\rd_1\ph_4,{\rm mc}}/U$,
the isomorphism $\phi^{\rd_1\ph_4}_k$ induces an isomorphism 
\begin{align*}
	\phi^{\rd_1\ph_5,{\rm mc}}_k:
	\ti\cU_d^{\rd_1\ph_5,{\rm mc}}/U\lra 
	\ti\cU_{kd}^{\rd_1\ph_5,{\rm mc}}/U\,.
\end{align*}
Moreover, just like in the first four phases of $(\rd_1)$, 
the blowups of $(\rd_2)$-$(\rd_3\ph_2)$ have the same effect on $\ti\cU_d^{\rd_1\ph_5,{\rm mc}}$ and $\ti\cU_{kd}^{\rd_1\ph_5,{\rm mc}}$ because of (\ref{Eqn:k_vs_1}).
Taking (\ref{Eqn:MC_inclusion}) into consideration,
we obtain an isomorphism 
\begin{align*}
	\phi^{\rd_3\ph_2,{\rm mc}}_k:
	\ti\cU_d^{\rd_3\ph_2,{\rm mc}}/U\lra 
	\ti\cU_{kd}^{\rd_3\ph_2,{\rm mc}}/U\,.
\end{align*}

Finally,
notice that $(\rd_3\ph_3)$ does not affect $\ti\cU_{kd}^{\rd_3\ph_2,{\rm mc}}$,
for otherwise $k\eq 2$,
and $\ti\cU_d^{\rd_3\ph_2,{\rm mc}}$ contains points whose images in $\ti\cU_d^{\rd_1\ph_4,{\rm mc}}$ belong to some blowup centers of $(\rd_1\ph_5)$, which we have shown is impossible.
In addition, $(\rd_3\ph_4)$ does not affect $\ti\cU_{kd}^{\rd_3\ph_2,{\rm mc}}$ either,
because we have shown $(\rd_1\ph_5)$ does not affect $\ti\cU_{kd}^{\rd_1\ph_4}$.
In other words, $\ti\cU_{kd}^{\rd_3\ph_4,{\rm mc}}/U$ is isomorphic to
$\ti\cU_{kd}^{\rd_3\ph_2,{\rm mc}}/U$.
After the entire sequential blowups terminate,
we thus obtain a morphism
\begin{align*}
	\phi^{{\rm mc}}_k: \big(U\!\times_{\MPdd}\ti M_2(\P^n,d)^{\rm mc}\big)\big/ U\lra
	\ti\cU_{kd}/U\,,
\end{align*}
given by the composition of the following:
\begin{align*}
	\big(U\!\times_{\MPdd}\ti M_2(\P^n,d)^{\rm mc}\big)\stackrel{\sim}{\lra}\ti\cU_d^{\rd_3\ph_4,{\rm mc}}\lra \ti\cU_d^{\rd_3\ph_2,{\rm mc}}
	\stackrel{\sim}{\lra}
	\ti\cU_{kd}^{\rd_3\ph_2,{\rm mc}}
	\stackrel{\sim}{\lra}
	\ti\cU_{kd}^{\rd_3\ph_4,{\rm mc}}
	\hookrightarrow
	\ti\cU_{kd}\,,
\end{align*}
in which the first isomorphism is as in Corollary~\ref{Crl:Gluing},
and the second isomorphism is $\phi^{\rd_3\ph_2,{\rm mc}}_k$.
As shown in the sentence containing (\ref{Eqn:U_kd}),
the pullback of $\bR\pi_*\ff^*\sO_{\P^n}(k)$ is locally diagonalizable on $\ti\cU_{kd}$,
hence is on $U\!\times_{\MPdd}\!\ti M_2(\P^n,d)^{\rm mc}$.
Since $U$ is arbitrary,
the proof of Proposition~\ref{Prp:Diag} is complete.
\end{proof}

\vsp
\begin{coro}
	\label{Crl:Diag}
Let $(\ti\pi^{\rm mc}, \ti\ff^{\rm mc})$ be the pullback to $\ti M_2(\P^n,d)^{\rm mc}$ of the universal family $(\pi, \ff)$. Then the direct image sheaf $(\ti\pi^{\rm mc})_{\ast} (\ti\ff^{\rm mc})^*
\sO_{\Pn}(k)$  is locally free for all $k\!\ge\! 1$.
It is of rank $kd\!-\!1$ if $d\!>\!2$.
\end{coro}

\begin{proof}
Given small open $U/\cV$ of $\MPdd$ and $H\inn\bH_U$, let 
$\ti\cU\eq\ti\cU_{U,H}$ be as in (\ref{Eqn:cU}).
Recall $\cX/U$ is the universal family over $U$, $(\cC,\cD)$ is the universal family over $\cV$, 
and $\cA\!\subset\!\cC$ is the auxiliary divisor as in (\ref{package}).
We set
\begin{align*}
	\ti\cX'=\cX\!\times_U\!\ti\cU=\cC\!\times_\cV\!\ti\cU,\qquad
	\ti\cD'=\cD\!\times_U\!\ti\cU=\cD\!\times_\cV\!\ti\cU,
\end{align*}
where the second equality in each set of identities follows from the universality of $\fdd$.
Let $\ti\cA'\eq \cA\!\times_{\cV}\ti\cU$ and $\ti\sL'\eq\sO_{\ti\cX'}(\ti\cD')$.

The same argument used in the proof of \cite[Theorem 2.11]{HL10}, particularly the sentence containing \cite[(5.22)]{HL10},
applies to the current situation verbatim.
More precisely, the Cartesian diagram
\begin{align*}
	\begin{CD}
		\ti\cX' @>{\ti\be}>> \cC\\
		@V{\ti\pi'}VV  @V{\rho}VV \\
		\ti\cU @>{\be}>> \cV
	\end{CD}
\end{align*}
along with the cohomology and base change theorem,
give the following Cartesian diagram
\begin{align*}
	\begin{CD}
		\ti\pi_*'\ti\sL'(\ti\cA') @>{\ti\varphi}>> \ti\pi_*'\ti\sO'_{\ti\cA'}(\ti\cA')\\
		@V{\wr}VV  @V{\wr}VV \\
		\be^*\rho_*\sL(\cA) @>{\be^*\varphi}>> \be^*\rho_*\sO_{\cA}(\cA)
	\end{CD}
\end{align*}
In the diagram above, the two vertical arrows are isomorphisms, and $\ti\varphi$ is the evaluation homomorphism.
Thus,
investigating $\ti\varphi$ is equivalent to investigating the pullback of $\varphi$ to $\ti\cU$.
Therefore, in order to show $(\ti\pi^{\rm mc})_{\ast} (\ti\ff^{\rm mc})^*
\sO_{\Pn}(k)$  is locally free,
it suffices to show the pullback of $\pi_{\ast} \ff^*
\sO_{\Pn}(k)$ to $\ti M_2(\P^n,d)^{\rm mc}$ is locally free,
which follows immediately from  Propositions~\ref{Prp:Diag} and~\ref{PrpLocallyFree}.

As for the rank, when $d\!>\!2$,
notice near a general point of $\MPdd^{\rm mc}$,
the matrix of $\varphi$, whose size is $2\!\times\!kd$, is of rank 2.
Thus, $\ker\varphi$ is of rank $kd\!-\!2$.
So by Lemma~\ref{lemm:usefulFacts}, $\ti\pi'_*\ti\sL'$ is of  rank $kd\!-\!1$.
\end{proof}

\vsp
\subsection{The construction of $\ti M_2(\P^n,d;k)$}
For every $k\inn\mathbb Z_{>0}$, following the same argument as in Lemma~\ref{Lm:U_pullback}, we observe the pullbacks $\ti\cU_{kd}$ as in (\ref{Eqn:U_kd})
can also be glued together to form a Deligne-Mumford stack,
denoted by 
\begin{align*}
	\ti M_2(\P^n,d;k)\,.
\end{align*}
For $k\!>\!1$, $\ti M_2(\P^n,d;k)$ is generally not as smooth as $\ti M_2(\P^n,d;1)\eq\ti M_2(\P^n,d)$;
nonetheless,
as far as the local diagonalization of $\bR\pi_*\ff^*\sO_{\P^n}(k)$ is concerned,
$\ti M_2(\P^n,d;k)$ behaves nicer than $\ti M_2(\P^n,d)$.

\begin{prop}
	\label{Prp:Diag_k}
For every $k\inn\mathbb Z_{>0}$, there exist a Deligne-Mumford stack $\ti M_2(\P^n,d;k)$
and a morphism $\ti M_2(\P^n,d;k)\!\lra\!\ov M_2(\P^n,d)$
such that for every small open $U\!\subset\!\ov M_2(\P^n,d)$, $H\inn\breve{\P}^n$, and degree $k$ morphism $\psi_k:\P^n\!\lra\!\P^n$ satisfying that $f_{U,H,\psi_k}$ as in (\ref{Eqn:f_k}) is well-defined,
we have the following commutative diagram
\[
\begin{tikzcd}[column sep=.5em,row sep=1.5em]
	U\times_{f_{U,H,\psi_k};\,\fM_2^{\rm div}}\ti\fM_2^{\rm div}  \arrow[dr,""] \arrow[rr,"\sim"] && U\times_{\ov M_2(\P^n,d)}\ti M_2(\P^n,d;k)
	\arrow[dl,""]
	\\
	& U 
\end{tikzcd}
\]
where the horizontal arrow is an isomorphism (of stacks over $U$),
and the other two arrows are natural morphisms.

Moreover, the pullback of the derived object $\bR \pi_* \ff^* \sO_{\Pn}(k)$ to $\ti M_2(\P^n,d;k)$ becomes locally diagonalizable. 

Furthermore,
for every irreducible component $N$ of $\ti M_2(\P^n,d;k)$,
with $(\ti\pi_N, \ti\ff_N)$ denoting the pullback to $N$ of the universal family (\ref{Eqn:univ_fami}), the direct image sheaf $(\ti\pi_N)_{\ast} (\ti\ff_N)^*
\sO_{\Pn}(k)$  is locally free.
It is of rank $kd\!-\!1$ if $d\!>\!2$ and the general points of $N$ have smooth domain curves.
\end{prop}

\begin{proof}
	The existence of $\ti M_2(\P^n,d;k)$ follows from a line by line repetition of the proof of Lemma~\ref{Lm:U_pullback},
	hence we omit the detail.
	The local diagonalizability of the pullback of $\bR \pi_* \ff^* \sO_{\Pn}(k)$ to $\ti M_2(\P^n,d;k)$ follows directly from Proposition~\ref{PrpChangeofPhi}.
	The proof of the remaining statements are parallel to the proof of Corollary~\ref{Crl:Diag},
	hence is omitted.
\end{proof}

\vsp

\end{document}